\documentclass[11pt, twoside]{report}
\usepackage{baththesis, amssymb, graphicx, parskip}

\usepackage{lipsum}%%lipsum package produces fake text for testing (below)

\usepackage{algorithm}
\usepackage[noend]{algpseudocode}
\makeatletter
\def\BState{\State\hskip-\ALG@thistlm}
\makeatother
\algdef{SE}[SUBALG]{Indent}{EndIndent}{}{\algorithmicend\ }%
\algtext*{Indent}
\algtext*{EndIndent}

\newcommand{\ba}{\begin{eqnarray}}
\newcommand{\ea}{\end{eqnarray}}
\newcommand{\nn}{\nonumber}
\newcommand{\dawesspace}{\ \hspace{-5.0cm}}
\usepackage{ulem}  % to include the \sout{} (strikeout) command.

\usepackage{enumitem} % to allow enumerate arguments. See List of Papers and Preprints.

\usepackage{natbib}
\usepackage{bibunits}
\usepackage{pgfplots}
\usepackage{tikz}
\usepackage{subcaption}
\usepackage{mathrsfs}  
\usepackage{amsfonts}
\usepackage{amsmath}
\usepackage{amssymb}
\usepackage{mathtools}
\usepackage{hyperref}
\usepackage{nomencl}
\makenomenclature
\usepackage{etoolbox}
\renewcommand\nomgroup[1]{%
  \item[\bfseries
  \ifstrequal{#1}{A}{Stochastic Dynamical Systems}{%
  \ifstrequal{#1}{B}{Deterministic Dynamical Systems}{%
  \ifstrequal{#1}{C}{Echo State Networks and reservoir maps}{%
  \ifstrequal{#1}{D}{Partial Differential Equations}{
  }}}}%
]}
\DeclareMathOperator*{\argmin}{arg\,min}

\usepackage{amsthm}
\newtheorem{theorem}{Theorem}
\newtheorem{lemma}[theorem]{Lemma}

\newtheorem{corollary}[theorem]{Corollary}
\theoremstyle{definition}
\newtheorem{defn}[theorem]{Definition}
\newtheorem{proc}[theorem]{Procedure}
\newtheorem{remark}[theorem]{Remark}
\numberwithin{theorem}{section}

\title{Reservoir Computing with Dynamical Systems}
\author{Allen G Hart}
\degree{Doctor of Philosophy}
\department{Department of Mathematical Sciences}
\degreemonthyear{\date}
\faculty{Faculty of Science}

\begin{document}

\maketitle

\chapter*{Acknowledgements}

Thank you everyone who supported me over the last four years, and for all the conversations about reservoir computing and dynamical systems. 

\begin{abstract}

A reservoir computer is a special type of neural network, where most of the weights are randomly fixed and only a subset are trained.

In this thesis we prove results about reservoir computers trained on deterministic dynamical systems, and stochastic processes. We focus mostly on a special type of reservoir computer called an Echo State Network (ESN).

In the deterministic case, we prove (under some assumptions) that if a reservoir computer has the Echo State Property (ESP), then there is a $C^1$ generalised synchronisation between the input dynamical system and the dynamics in the reservoir space. Furthermore, we prove that a reservoir computer with the local ESP in several disjoint subsets of the reservoir space will admit several distinct generalised synchronisations. In the special case that the reservoir map is linear, and has the ESP, we prove that the generalised synchronisation is generically an embedding. This result admits Takens' embedding Theorem as a special case.

We go to show that ESNs trained on scalar observations of an ergodic dynamical system can approximate an arbitrary target function, including the next step map used in time series forecasting. This universal approximation property holds despite the training process being entirely linear.

We prove analogous results for ESNs trained on observations of a stochastic process, which are not be Markovian in general. We use these results to develop supervised learning, and reinforcement learning algorithms supported by an ESN.

In the penultimate chapter of this thesis, we use a reservoir computer to numerically solve linear PDEs. In the final chapter, we conclude and discuss directions for future work.

\end{abstract}

\chapter*{List of Publications and Preprints}

This thesis is largely derived from the following publications and preprints.

\begin{enumerate}[leftmargin=*]  % removes the indentation of the enumerate environment. Requires the enumitem package, see LaTeX preamble.
    
\item Hart, A., Hook, J. \& Dawes, J. H. P. (2020), ‘Embedding and approximation theorems for echo state networks’, Neural Networks 128, 234–247.
    
Results from this paper appear in Chapters \ref{chapter::computational_methods}, \ref{chapter::SSMs}, and \ref{chapter::universal_approximation}. A G Hart is the lead author with J L Hook and J H P Dawes having supervisory roles.
    
\item Grigoryeva, L., Hart, A. \& Ortega, J.-P. (2021), ‘Chaos on compact manifolds:  Differentiable synchronizations beyond the Takens theorem’, Phys. Rev. E103, 062204.
    
The paper forms the basis of Chapter \ref{chapter::SSMs}. All authors contributed equally.

\item Grigoryeva, L., Hart, A. \& Ortega, J.-P. (2021), ‘Learning Strange Attractors with reservoir maps’, \emph{arXiv:2108.05024}. [PREPRINT]
    
This paper forms the basis of Chapter \ref{chapter::embedding}. All authors contributed equally.
    
\item Hart,  A.  G.,  Hook,  J.  L.  \&  Dawes,  J. H. P.  (2021),  ‘Echo  state  networks  trained  by Tikhonov least squares are $L^2(\mu)$ approximators of ergodic dynamical systems’, Physica D: Nonlinear Phenomena p. 132882.
    
    Results from this paper appear in Chapter \ref{chapter::universal_approximation}. A G Hart is the lead author with J L Hook and J H P Dawes having supervisory roles.
    
\item Hart, A. G., Olding, K. R., Cox, A. M. G., Isupova, O. \& Dawes, J. H. P. (2021), ‘Using Echo State Networks to Approximate Value Functions for Control’, \emph{arXiv:2004.08170}. [PREPRINT]
    
    This paper forms the basis of Chapter \ref{chapter::stochastic}. A G Hart was primarily responsible for the ESN approximation results, K R Olding was primarily responsible for derivations surrounding the market making problem, while J H P Dawes, A M G Cox, and O Isupova had supervisory roles.   
    
\end{enumerate}

\tableofcontents

\nomenclature[C]{$N \in \mathbb{N}$}{Number of neurons / dimension of the reservoir}
\nomenclature[C]{$d \in \mathbb{N}$}{Number of input channels / dimension of the input}
\nomenclature[C]{$\ell \in \mathbb{N}$}{Number of sample points}
\nomenclature[C]{$x_k \in \mathbb{R}^N$}{Reservoir state at time $k \in \mathbb{Z}$}
\nomenclature[C]{$z_k \in \mathbb{R}^d$}{Input at time $k \in \mathbb{Z}$}
\nomenclature[C]{$A \in \mathbb{M}_{N \times N}(\mathbb{R})$}{Reservoir matrix}
\nomenclature[C]{$C \in \mathbb{M}_{N \times d}(\mathbb{R})$}{Input matrix}
\nomenclature[C]{$b \in \mathbb{R}^N$}{Bias vector}
\nomenclature[C]{$\sigma : \mathbb{R}^N \to \mathbb{R}^N$}{Activation function}
\nomenclature[C]{$W \in \mathbb{R}^N$}{Output vector / linear readout layer}
\nomenclature[C]{$f_{(\phi,\omega,F)} : M \to \mathbb{R}^N$}{Generalised Synchronisation}
\nomenclature[C]{$F : \mathbb{R}^N \times \mathbb{R}^d \to \mathbb{R}^N$}{reservoir map}
\nomenclature[C]{$\mathcal{F}$}{Class of linear universal approximators}

\nomenclature[B]{$q \in \mathbb{N}$}{Dimension of the manifold $M$}
\nomenclature[B]{$M$}{$q$-dimensional manifold}
\nomenclature[B]{$\phi: M \to M$}{Evolution operator on $M$}
\nomenclature[B]{$m \in M$}{Point on $M$}
\nomenclature[B]{$T_m$}{Tangent space at $m \in M$}
\nomenclature[B]{$T_m\phi : T_m \to T_{\phi(m)}$}{Tangent map for $\phi$ at $m \in M$.}
\nomenclature[B]{$D\omega(m) : T_m \to \mathbb{R}^d$}{The differential of $\omega$ at $m \in M$}
\nomenclature[B]{$\omega : M \to \mathbb{R}^d$}{Observation function}
\nomenclature[B]{$u : M \to \mathbb{R}$}{Target function}
\nomenclature[B]{$\text{Diff}^n(M)$}{Diffeomorphism of order $n \in \mathbb{N}$ on $M$}
\nomenclature[B]{$C^n(M,\mathbb{R}^N)$}{The $n$-times differentiable maps from $M$ to $\mathbb{R}^N$}

\nomenclature[A]{$(\mathbb{R}^N)^{\mathbb{Z}}$}{The set of bi-infinite $\mathbb{R}^N$ valued sequences, i.e. the set of maps $\mathbb{Z} \mapsto \mathbb{R}^N$}
\nomenclature[A]{$T : (\mathbb{R}^N)^{\mathbb{Z}} \to (\mathbb{R}^N)^{\mathbb{Z}}$}{Time shift operator $T(z)_k := z_{k+1}$}
\nomenclature[A]{$V : (\mathbb{R}^d)^{\mathbb{Z}} \to \mathbb{R}$}{Value functional}
\nomenclature[A]{$\mathcal{R} : (\mathbb{R}^d)^{\mathbb{Z}} \to \mathbb{R}$}{Reward functional}
\nomenclature[A]{$\gamma \in [0,1)$}{Discount factor}

\nomenclature[D]{$\Delta$}{The Laplace operator $\Delta: C^2(\mathbb{R}^d, \mathbb{R}) \to C^0(\mathbb{R}^d, \mathbb{R})$}
\nomenclature[D]{$Z \in \Omega$}{Random interior point}
\nomenclature[D]{$Z' \in \partial \Omega$}{Random boundary point}
\nomenclature[D]{$\phi : \Omega \to \mathbb{R}$}{Solution to the boundary value problem}
\nomenclature[D]{$h : \partial \Omega \to \mathbb{R}$}{Boundary data}

\printnomenclature

\chapter{Introduction}
\label{chapter::introduction}

\section{Preamble}

A single layer feedforward neural network $f : \mathbb{R}^d \to \mathbb{R}$ is a map the form 
\begin{align*}
    f(z) = \sum_{k = 1}^N W_k \sigma( C_k^{\top} z + b_k )
\end{align*}
paramatrised by weights $C_k \in \mathbb{R}^d$, $W_k \in \mathbb{R}$ and biases $b_k \in \mathbb{R}$ \citep{higham2019deep}. The map $\sigma : \mathbb{R} \to \mathbb{R}$ is called the activation function. Usual choices of the activation function include
\begin{itemize}
    \item the rectified linear unit (ReLU) defined coordinate wise $\sigma(z)_i = \max(z_i,0)$
    \item the hyperbolic tangent also defined coordinate wise $\sigma(z)_i = \tanh(z_i)$.
\end{itemize}

A multilayer (or deep) neural network can be created by composing several feedforward layers together. A composition of $n$ feedforward neural networks is called an $n$-layer neural network \citep{higham2019deep}.
Both deep and shallow (single-layer) neural networks are dense in some appropriate function space so that for a function $g$ in the appropriate space, and given sufficiently many neurons $N$, we can choose weights $W_k, C_k$ and biases $b_k$ such that the feedforward neural network $f$ approximates $g$. This is made formal in the universal approximation theorem for feedforward neural networks (discussed in Chapter \ref{chapter::universal_approximation}).

We can use feedforward neural networks to solve supervised learning problems.
These problems involve a supervisor, an agent and a collection of data. The data are labelled by the supervisor, and the agent is tasked with learning the relationship between the data and its labels, in such a way that the relationship can be generalised to unlabelled data. Learning this relationship is called training. The success of the training is measured by how accurately the agent can correctly label a set of unlabelled data.

There are 2 major categories of supervised learning: regression and classification. In the former, the agent seeks a map from the data (observations) to the labels (targets) such that any unlabelled datum (observation) on a continuum can be assigned a label (target) also on a continuum. The problem of time series forecasting is often formulated as a regression problem. Classification is similar, except that the labels do not lie on a continuum but instead lie in a finite set of categories. A classification algorithm may for example classify a set of images as either being cats or dogs.

When we use a feedforward neural network to solve a supervised learning problem, the data are denoted $\{z_i \in \mathbb{R}^d \}_{i \in I}$ for some finite index $I$. We assume the existence of a relationship $g : \mathbb{R}^d \to \mathbb{R}$ from the data to the targets $\{g(z_i) \}_{i \in I}$. Then our goal is to find weights $W_k, C_k$ and biases $b_k$ such that the feedforward neural network $f$ approximates $g$ by finding $W_k, C_k, b_k$ that minimise the so called loss $\mathcal{L}$ (sometimes) defined
\begin{align*}
    \mathcal{L} = \sum_{i = 1}^{| I |} \lVert f(z_i) - g(z_i) \rVert^2.
\end{align*}
We can view the loss $\mathcal{L}$ as a smooth function of the weights $W_k,C_k$ and biases $b_k$, so we can try to minimise the loss with smooth optimisation techniques like gradient descent or stochastic gradient descent. This optimisation problem is generally non-convex, so finding a global minimum is generally difficult.

Many supervised learning problems involve data drawn from a time series where the data has an important temporal structure. Such problems encompass speech recognition, the forecasting of chaotic systems, and decision making in robotics.
For such problems, it is often better to replace the feedforward neural network with a \emph{reservoir computer}.

In this thesis we focus on mathematical results that hold for reservoir computers. We focus on reservoir computers trained on a discrete, equally spaced time series that is either:
\begin{enumerate}
    \item A (partial) observation of a system of ODEs, or
    \item A realisation of a stochastic process.
\end{enumerate}

A reservoir computer is a pair of maps $(F,h)$ where $F : \mathbb{R}^N \times \mathbb{R}^d \to \mathbb{R}^N$ is called the reservoir map and takes as input a reservoir state vector $x \in \mathbb{R}^N$ and observation vector $z \in \mathbb{R}^d$, then returns a new reservoir state vector $x \in \mathbb{R}^N$. Given an initial state $x_0 \in \mathbb{R}^N$ and a time series of observations $z_1, z_2, \ldots \in \mathbb{R}^d$ the reservoir computer creates a sequence of states $x_1, x_2, \ldots \in \mathbb{R}^N$ like so
\begin{align*}
    x_{k+1} = F(x_k,z_k).
\end{align*}
The state $x_{k+1}$ then depends on all observations $\ldots, z_{k-2}, z_{k-1}, z_k$ prior to and including timestep $k$. The second part of the reservoir computer is a map $h : \mathbb{R}^N \to \mathbb{R}^s$ which takes a reservoir state $x \in \mathbb{R}^N$ and returns an output. The map $h$ could be a feedforward neural network or simply a row vector $W^\top$. The maps $F$ and $h$ together form a reservoir computer, which takes a time series $z_1, z_2, \ldots \in \mathbb{R}^d$ of observations and computes an output. We can define a reservoir computer formally as follows:
\begin{defn} (Reservoir Computer)
    We call $F : \mathbb{R}^N \times \mathbb{R}^d \to \mathbb{R}^N$ a reservoir map, $h : \mathbb{R}^N \to \mathbb{R}^s$ a readout map, and the pair $(F,h)$ a reservoir computer.
\end{defn}

Reservoir computing (RC) exploded in popularity at the start of the early 2000s after the seminal papers of \cite{Jaeger2001} and \cite{doi:10.1162/089976602760407955}. RC has since been applied to both signal processing and machine learning, and has captured the interest of researchers in computer science and robotics \citep{10.1007/978-3-540-25940-4_14}, physics \citep{Inubushi2017} \citep*{chaos_on_compacta}, mathematics \citep{GRIGORYEVA2018495}, \citep{CENI2020132609}, and electrical engineering \citep{TANAKA2019100}. Furthermore, there is some biology and neuroscience literature suggesting that a sensory input stimulating a nervous system can be described as a reservoir computer driven by a time series \citep{izhikevich2007dynamical}, \citep{TANAKA2019100}. Rather ambitious applications of reservoir computing include ``building by 2050, a team of fully autonomous
humanoid robots to beat the human winning team of the FIFA Soccer World Cup''
\citep{10.1007/978-3-540-25940-4_14}, predicting magnetic storms \citep{kataokareconstructing}, and forecasting electricity sales \citep{li2020research}.

Reservoir computing can be performed on an ordinary digital computer, which is often referred to reservoir computing \emph{in silico}. Alternatively, reservoir computing can be performed on purpose build hardware; often known as reservoir computing \emph{in materio}. Reservoirs have been physically built from magnets, photonic node arrays \citep{TANAKA2019100}, and field programmable gate arrays \citep{patent} \citep{Apostel2021}, and (possibly) occur naturally in the form of biological brains. In this thesis, all reservoir computing experiments are performed in silico, though the mathematical results hold more generally. Some authors have argued that there is potential to develop extremely efficient reservoir computers in materio, which could eventually outperform ordinary computers \citep{9106624}.

The goal in this thesis is to further develop a rigorous mathematical description of reservoir computing on time series, which we hope will serve the highly multidisciplinary reservoir computing community.

\section{ESNs and dynamical systems}

In this section we will go through a numerical experiment which was undertaken at the start of the PhD. The outcome of the experiment is illuminating, and motivates many of the results that appear in subsequent chapters in this thesis.

Suppose we have a dynamical system, but we do not know its governing equations, and instead have access only to a low dimensional observation of a trajectory. We could have for instance the celebrated Lorenz system \citep{doi:10.1175/1520-0469(1963)020<0130:DNF>2.0.CO;2}
\begin{align}
    \dot{\xi} &= 10(\upsilon - \xi) \nonumber \\ 
    \dot{\upsilon} &= \xi(28 - \zeta) - \upsilon \label{eqn::Lorenz} \\
    \dot{\zeta} &= \xi \upsilon - (8/3)\zeta \nonumber
\end{align}
(where we have fixed the parameters to their `usual values') and have access to only the $\xi$ component, at regularly spaced time points $t_0, t_1, \ldots t_{\ell-1}$. Given this finite time series, the key question is: can we predict the future trajectory without knowledge of the underlying equations? We note that the observation at time $t_\ell$ is not determined from the observation at time $t_{\ell-1}$ alone, so we will need to use many observations from the past $\ldots , t_{\ell-2}, t_{\ell-1}$ to have any hope of estimating the observation at $t_\ell$.

We approach this time series forecasting problem with a reservoir computer called an Echo State Network (ESN). An ESN takes a time series of observations $z_0, z_1, \ldots \in \mathbb{R}^d$ and produces a sequence of reservoir state vectors $x_0, x_1, \ldots \in \mathbb{R}^N$ defined like so:
\begin{align}
    x_{k+1} = F(x_k,z_k) := \sigma(Ax_k + Cz_k + b) \label{eqn::ESN}
\end{align}
where 
\begin{itemize}
    \item $\sigma : \mathbb{R}^N \to \mathbb{R}^N$ is the activation function defined to act coordinate-wise: $\sigma_i(x) = \tanh(x_i)$ for $i = 1, \ldots N$
    \item $A \in \mathbb{M}_{N \times N}(\mathbb{R})$ is the $N \times N$ real reservoir matrix
    \item $C \in \mathbb{M}_{N \times d}(\mathbb{R})$ is the $N \times d$ real input matrix 
    \item $b \in \mathbb{R}^N$ is the real bias vector.
\end{itemize}
We will sometimes call the Lorenz system the \emph{drive system} and the dynamics of the reservoir states the \emph{response system}. 

The recursive nature of \eqref{eqn::ESN} ensures that the reservoir state $x_{k+1}$ depends on all past observations $\ldots , z_{k-1} , z_k$. The ESN therefore has a sort of memory, because its representation of the world at time step $k$ is $x_k$ which depends on everything it has previously seen. Equation \eqref{eqn::ESN} is sometimes called a single layer (or shallow) ESN in contrast to multilayer (or deep) ESNs studied by \cite{gallicchio2018design} and others.

It seems natural that observations far in the past should have less impact on the reservoir state $x_k$ than those observations in the near past. This is called the Fading Memory Property (FMP) defined in \cite{GRIGORYEVA2018495}. The FMP is closely related to the Echo State Property (ESP) presented by \cite{Jaeger2001} and is introduced and discussed in Chapters \ref{chapter::SSMs} and \ref{chapter::stochastic}. We can ensure that an ESN has the ESP by demanding that the activation function $\sigma$ and reservoir matrix $A$ are contracting. The matrix $A$ is called the reservoir matrix (and the field called \emph{reservoir computing}) because the entries of $A$ are fixed and represent a sort of computational reservoir that is capable of encoding a rich variety of inputs with a reservoir state vector $x$.

ESNs are often used to approximate the relationship between a time series of observations $z_0, z_1 , \ldots \in \mathbb{R}^d$ and a time series of targets $u_0, u_1, \ldots \in \mathbb{R}$. In the special case that $d=1$ and the targets $u_k = z_k$ then the problem is time series forecasting. There is substantial numerical evidence in the literature suggesting that ESNs are very good at this. They have performed remarkably well on problems ranging from seizure detection, to robot control,
handwriting recognition, and financial forecasting, where ESNs have won competitions \citep{LUKOSEVICIUS2009127},  \citep{LukoseviciusMantas2012RCT,RodanA2011MCES,NIPS2010_4056}. Impressively, ESNs outperformed recurrent neural networks (RNNs) and long short term memory networks (LSTMs) at a chaotic time series prediction task by a factor of over 2400 \citep{Jaeger78}. ESNs have also proved themselves competitive in reinforcement learning \citep{10.1007/11840817_86} and control \citep{Pietz_2021}.

To approximate the relationship between the observations and the targets, we seek a readout vector $W^*$ that minimises (over $W$) the regularised least squares difference between the reservoir states (which represent observations) and targets, so
\begin{align*}
    W^* = \argmin_{W \in \mathbb{R}^N} \bigg( \sum_{k = 0}^{\ell - 1} \rVert W^{\top}x_k - u_k \lVert^2 + \lambda \lVert W \rVert^2 \bigg)
\end{align*}
where $\lambda > 0$ is the Tikhonov regularisation parameter. We can find the optimal $W^*$ easily using linear regression.
\label{subsec} 
 
\section{Training an ESN on the Lorenz system}

\subsection{Creating the reservoir states}

To demonstrate everything we have discussed in section \ref{subsec}, we will present two numerical experiments. In section \ref{zeta_from_xi}, we use an ESN to learn a mapping from the discrete time series of values of the $\xi$ component of a trajectory in the Lorenz attractor (observations) to the $\zeta$ component (targets). In section \ref{predict_the_future}, we learn a mapping from the $\xi$ component of the same discrete trajectory in the Lorenz attractor (observations) to the next value of the discretely-sampled $\xi$ component (targets). We then use this next step map to generate a future trajectory for the $\xi$ components. To this end, let $\phi:\mathbb{R}^3 \to \mathbb{R}^3$ denote a discretisation of the Lorenz system \eqref{eqn::Lorenz} with time step $\tau$ i.e.
a discrete-time map of the form
\begin{align*}
    \phi(\xi,\upsilon,\zeta) = (\xi,\upsilon,\zeta) + \int_0^{\tau} (\dot{\xi},\dot{\upsilon},\dot{\zeta}) \ dt.
\end{align*}
We set the timestep $\tau = 0.01$ and initial condition $(\xi_0,\upsilon_0,\zeta_0) = (0, 1.0, 1.05)$. The timestep is small in comparison to the maximal Lyapunov exponent (see Section \ref{section::lyapunov}) of the Lorenz system, $\lambda_{\text{max}} \approx 0.9056$ \citep{Sprott_2003} and the initial condition is close to the attractor. For these initial conditions and the parameter values as in \eqref{eqn::Lorenz}, we computed a trajectory for a 200 time units (i.e. $\ell = 20000$ timesteps), illustrated in Figure~\ref{Lorenz_fig}.

\begin{figure}
  \centering
    \includegraphics[width=0.95\textwidth]{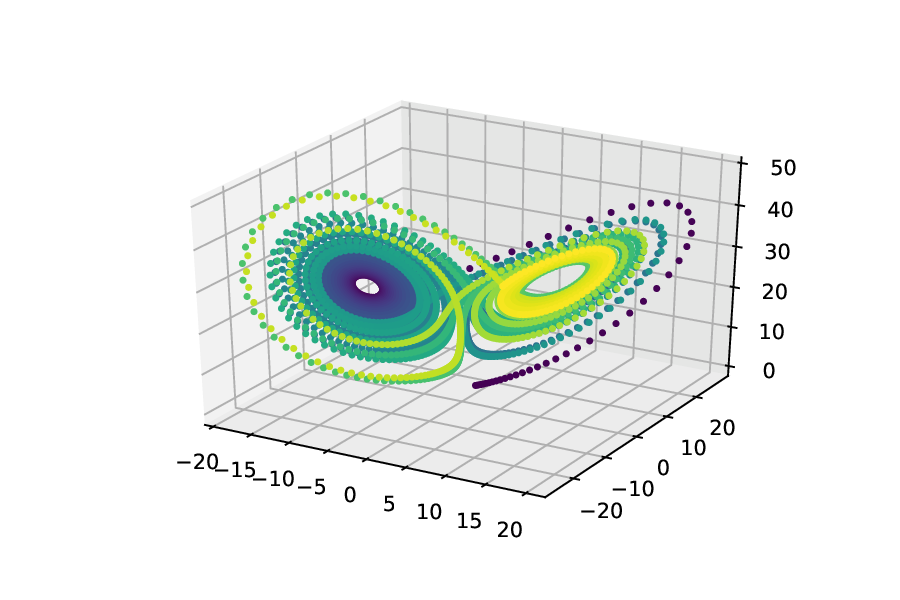}
    \caption{A typical trajectory of the Lorenz system \eqref{eqn::Lorenz} computed for 4000 timesteps, represented by the individual dots at time intervals $\tau=0.01$. Colour indicates the direction of travel along the trajectory: darkest colours (blue) at the earliest times and lightest colours (yellow) at the most recent times.}
    \label{Lorenz_fig}
\end{figure}

The observation and target functions are the first and third components of the Lorenz system $\xi$ and $\zeta$ i.e. we have an observation function $\omega(\xi,\upsilon,\zeta) = \xi$ so that the observations $z_k$ are the $\xi$ components of the trajectory at the sampled time points $k\tau$, so that
\begin{align*}
    z_k = \omega\circ\phi^k(\xi_0,\upsilon_0,\zeta_0).
\end{align*}
The target function is $u(\xi,\upsilon,\zeta) = \zeta$ so the targets $u_k$ are the $\zeta$ components of the trajectory:
\begin{align*}
    u_k = u\circ\phi^k(\xi_0,\upsilon_0,\zeta_0).
\end{align*}
The trajectories of these two components of observations and targets are shown in Figure \ref{xz_fig}(a) and (b), respectively. 
\begin{figure}
    \centering
    \begin{subfigure}[b]{0.95\textwidth}
        \includegraphics[width=\textwidth]{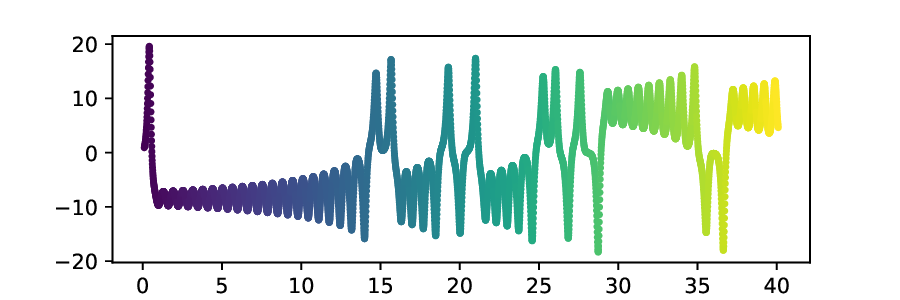}
        \caption{The $\xi$-component of the Lorenz trajectory (vertical axis) plotted against time (horizontal axis).}
    \end{subfigure}
    ~ %add desired spacing between images, e. g. ~, \quad, \qquad, \hfill etc. 
      %(or a blank line to force the subfigure onto a new line)
    \begin{subfigure}[b]{0.95\textwidth}
        \includegraphics[width=\textwidth]{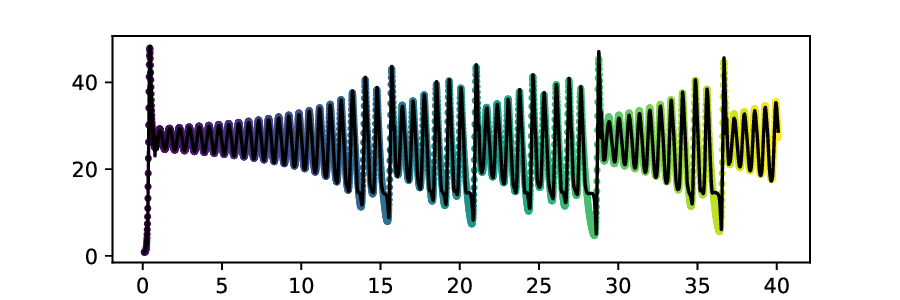}
        \caption{The $\zeta$-component of the Lorenz trajectory (vertical axis) plotted against time (horizontal axis). The black line at the $k^\mathrm{th}$ timestep indicates the approximation to this target time series given by $W^{*\top}x_k$.}
    \end{subfigure}
    \caption{Observations $z_k$ and targets $u_k$ drawn from the Lorenz trajectory.}
    \label{xz_fig}
\end{figure}

Our goal is to use an ESN to predict the targets based on the observations. So, we set up an ESN with the following parameters:
\begin{itemize}
    \item Reservoir size: $N = 300$,
    \item Input matrix $C$ and bias vector $\zeta$: i.i.d uniform random variables $\sim U[-0.05,0.05]$,
    \item Reservoir matrix $A$: i.i.d uniform random variables rescaled so that the matrix 2-norm $\lVert A \rVert_2 = 1$,
    \item Regularisation parameter $\lambda = 10^{-9}$,
\end{itemize}
which we obtained by hand tuning so that the experiment went smoothly. We have (so far) little theory for choosing the parameters.

Iterating the ESN with observations $z_k$ creates a discrete-time sequence of reservoir states $x_k$, illustrated in Figure \ref{Reservoir_Lorenz_fig}, which shows a projection of the reservoir states onto their first 3 principal components. To compute the components, we created a matrix $X^{\top} \in \mathbb{M}_{N \times \ell}(\mathbb{R})$ with $k$th column $x_k$, and took the singular value decomposition (SVD) $X = U \Sigma V^{\top}$. The first 3 principal components are the first 3 columns of $V$ which we denote $V_1, V_2, V_3$. Then $(V_1^{\top}X, V_2^{\top}X, V_3^{\top}X)$ is the projection of the reservoir states onto the first 3 principal components.

\begin{figure}
  \centering
    \includegraphics[width=0.9\textwidth]{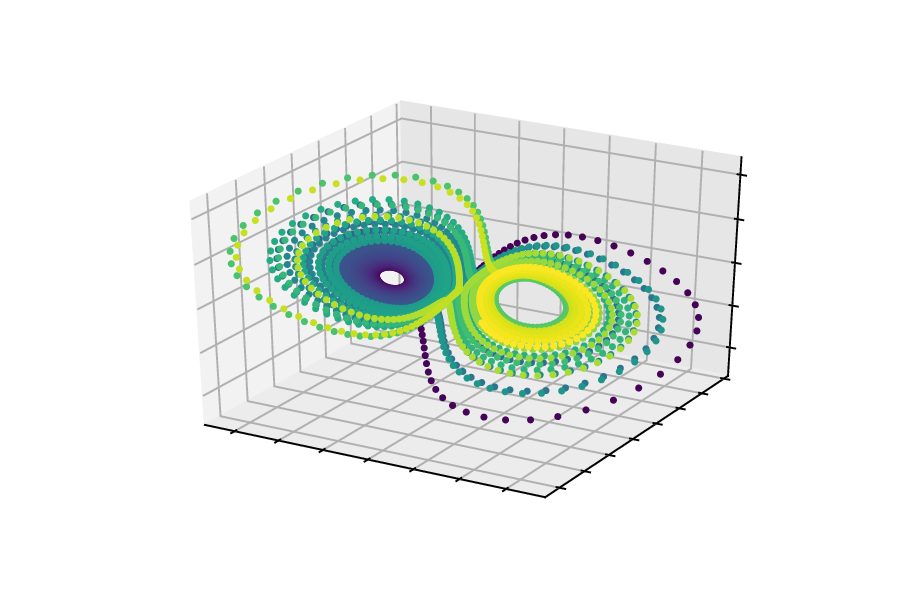}
    \caption{Illustration of the reservoir states of the ESN driven by inputs $z_k$ being the discrete-time samples observed from a trajectory of the Lorenz system. The figure shows the projection of the reservoir states onto their first 3 principal components.}
    \label{Reservoir_Lorenz_fig}
\end{figure}

The Lorenz system in Figure \ref{Lorenz_fig} and reservoir dynamics in Figure \ref{Reservoir_Lorenz_fig} look rather similar. It appears that the ESN has mapped the Lorenz attractor $\mathcal{A}$ directly into the reservoir space via a map $f :(\mathcal{A} \subset \mathbb{R}^3) \to \mathbb{R}^N$. The smoothness of the Lorenz system, ESN, and reservoir dynamics together suggest that $f$ is at least continuously differentiable ($C^1$). We show in Chapter \ref{chapter::SSMs} that $f$ is a continuously differentiable mapping of the Lorenz system into the reservoir space. Consequently, the Lorenz dynamics are synchronised to the reservoir dynamics, via the map $f$, which is called a generalised synchronisation (GS), in the sense of \cite{PhysRevLett.76.1816}. The map $f$ is essentially `learned' during training, and this is discussed in detail in Chapter \ref{chapter::SSMs}.

The image of the GS $f$ appears to smoothly and injectively reproduce the Lorenz dynamics. This suggests that properties such as Lyapunov spectra, eigenvalues of fixed points, and homology groups may be preserved under $f$ and replicated in the reservoir dynamics. We explore these claims numerically using computational methods detailed in Chapter \ref{chapter::computational_methods}. From a purely mathematical perspective, we know these properties are preserved under diffeomorphism. A map which is diffeomorphic to its image is called an embedding. Since the GS $f$ appears to preserve these properties we propose that $f$ is an embedding, and analyse this claim in Chapter \ref{chapter::embedding}.

These observations appear to hold over many different realisations of the ESN (which has random weights and biases) suggesting the properties hold with probability close to 1 or perhaps exactly 1. This possibility is studied in Chapter \ref{chapter::embedding}.

\subsection{Learning the targets $\zeta$ from the observations $\xi$}
\label{zeta_from_xi}

Having made our observations about the reservoir dynamics, we proceed to solve the least squares problem
\begin{align}
    \argmin_{W \in \mathbb{R}^N}\sum_{k=0}^{\ell-1} \lVert W^{\top} x_{k} - u_k \rVert^2 + \lambda\lVert W \rVert^2 \label{eqn::ls}
\end{align}
to determine the output layer $W^*$ using the SVD once again. In particular, if we denote the $k$th singular value by $\sigma_k$, the $k$th column of $U$ by $U_k$ and define a vector $Y \in \mathbb{R}^\ell$ with $\ell$th component $u_\ell$ then it is a standard calculation \citep{Deblurring_2006} that 
\begin{align}
    W^* = \sum_{k=1}^N \frac{\sigma_k U_k^{\top}Y}{\sigma_k^2 + \lambda}. \label{eqn::SVD}
\end{align}
With this $W^*$, we found a good approximation of a mapping from the $\xi$ component to the $\zeta$ component, and the results are shown in Figure \ref{xz_fig}.

The good fit of the nonlinear relationship between the observations and the targets suggests that the ESN has a sort of universal approximation property, despite the fitting process being entirely linear. This remarkable feature of ESNs is analysed in Chapter \ref{chapter::universal_approximation}.

\subsection{Predicting the future}
\label{predict_the_future}

Having found a mapping from $\xi$ to $\zeta$, we are now interested in finding the next step map for time series forecasting. We can feed our predicted next step back into the ESN to recursively generate a trajectory into the future. To make this more explicit we define the ESN autonomous phase 
\begin{align}
    x_{k+1} = \sigma(Ax_k + C(W^{*\top}x_k) + b) \label{eqn::auto}
\end{align}
and recognise that if the ESN is `well trained' (i.e. we have made a good choice of $W^*$) then $W^{*\top}x_k \approx u_k = z_k$ for some sufficiently long future time, yielding a good forecast for the time series. So we define the target function 
\begin{align*}
    u(\xi,\upsilon, \zeta) := \omega\circ\phi(\xi,\upsilon,\zeta)
\end{align*}
where the observation function remains $\omega(\xi,\upsilon,\zeta) = \xi$. We solve the least squares problem \eqref{eqn::ls} using the SVD \eqref{eqn::SVD} once again, then generated a 4000 timestep (40 time units) future trajectory, starting at $t=40$, using the ESN autonomous phase \eqref{eqn::auto}. The future $\xi$ components and reservoir states generated by the autonomous ESN are shown in Figures \ref{fig::future_states} and \ref{fig::future} respectively.

\begin{figure}
    \centering
        \begin{subfigure}[b]{1.0\textwidth}
        \hspace{-0.5cm}\includegraphics[width=1.1\textwidth]{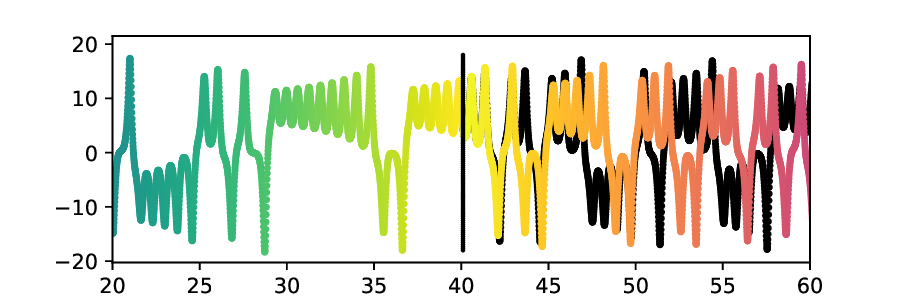}
        \caption{$\xi$ component of the Lorenz system (vertical axis) evolves in time (horizontal axis) up until time $t = 40$. After this point, the autonomous ESN predicts a future trajectory (multicoloured), which is compared to the true future trajectory (black).}
        \label{fig::future}
    \end{subfigure}
    ~
    \begin{subfigure}[b]{0.95\textwidth}
        \includegraphics[width=\textwidth]{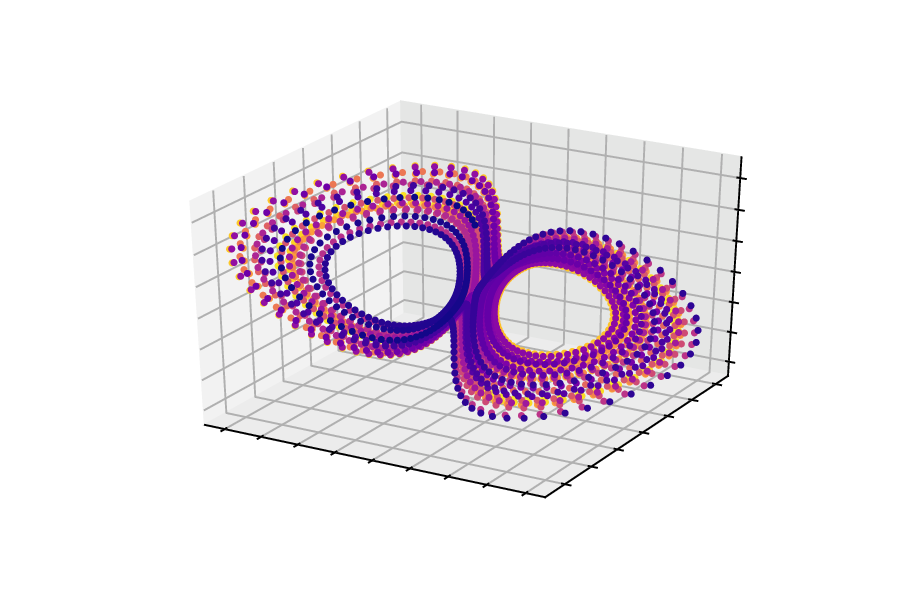}
        \caption{Illustration of the reservoir dynamics of the autonomous ESN, projected onto the first 3 principal components of the (non-autonomous) reservoir dynamics.}
        \label{fig::future_states}
    \end{subfigure}
        \caption{ESN autonomous phase.}
    \label{fig::auto}
\end{figure}

The ESN autonomous dynamics appear to have replicated a trajectory that is diffeomorphic to the Lorenz dynamics, even as the trajectories diverge in absolute value. 

This diffeomorphism is not guaranteed by a universal approximation result, which does not rule out the possibility of small errors accumulating until the autonomous dynamics no longer resemble the Lorenz dynamics at all. We explore in Chapter \ref{chapter::universal_approximation} how the \emph{structural stability} of the Lorenz system ensures that qualitative features of the dynamics are preserved even as the trajectories diverge.

\section{The remainder of the thesis}

Predicting the future trajectory of the Lorenz system is a supervised learning problem. In almost all real applications of supervised learning, the data is noisy and there is uncertainty. This is in contrast to the deterministic evolution of the Lorenz system we considered in the previous sections. Thus, in Chapter \ref{chapter::stochastic} we develop stochastic analogues for many of the deterministic results that appear in Chapters \ref{chapter::SSMs}-\ref{chapter::universal_approximation} and that were touched on here. 

In Chapter \ref{chapter::stochastic} we extend the stochastic set-up to encompass reinforcement learning problems. In the reinforcement learning paradigm an agent explores its environment by executing actions and collects rewards influenced by the environment and its actions. The goal of the agent is to learn the impact of its actions on the environment, and develop a policy that maximises its expected sum of future rewards. Learning about the environment and refining the policy often happen at the same time. We prove that ESNs are especially well suited to reinforcement learning problems where the environment evolves randomly, in a generally non-Markovian manner, and the ESN's memory of the past is crucial for making good decisions.

In Chapter \ref{chapter::PDEs}, we use reservoir computers to solve linear PDEs. The assumptions made in this Chapter are rather different to those made in the preceding Chapters, where we imagine that our agent does not know the equations that govern the environment. In previous chapters, the agent must learn a model from scratch using a reservoir computer trained on data. However in Chapter \ref{chapter::PDEs} we allow the agent access to the PDE explicitly, and use a reservoir computer to approximate the solution.

\chapter{Computational Methods}

\label{chapter::computational_methods}

The numerical results in Chapter \ref{chapter::introduction} suggests that under appropriate conditions the ESN autonomous dynamics are diffeomorphic to the dynamics of the drive system. In this case, the autonomous dynamics should inherit qualitative properties of the drive system, for example the existence of fixed points, eigenvalues of linearisations around them, and for chaotic attractors, their Lyapunov spectra and homology groups. Properties like these can be constructed from a time series using an assortment of techniques. In particular the computation of homology groups is a central idea in the field known loosely as \emph{computational topology} also known as \emph{algorithmic topology} or \emph{applied topology}. We will demonstrate some of these methods by recovering the topological properties of the Lorenz system using an ESN. The presentation in this chapter closely follows the numerical experiments in \cite{embedding_and_approximation_theorems}. To demonstrate these computational methods, we set up an ESN with the following parameters:

\begin{itemize}
    \item Reservoir size: $N = 300$,
    \item Elements of the input matrix $C$ and bias vector $\zeta$: i.i.d uniform Gaussian variables $\sim \mathcal{N}(0,0.1^2)$,
    \item Reservoir matrix $A$: Erd\H{o}s-R\'{e}nyi matrix with mean $6$ and connection weights (where they are non-zero) i.i.d Gaussian, re-scaled such that the spectral radius $\rho = 1$.
    \item Regularisation parameter $\lambda = 10^{-6}$.
\end{itemize}

\section{Fixed Points}

If the Generalised Synchronisation (GS) $f$ is an embedding, then $f$ will embed the fixed points of the Lorenz system into the reservoir space. Moreover if the autonomous ESN approximates the embedded Lorenz dynamics sufficiently well, the autonomous dynamics will have fixed points with the same linearisation as the Lorenz system. To verify this, we searched for the autonomous ESN's fixed points using Newton's method.

In particular, we found the fixed point of
\begin{align}
    \psi(x) := \sigma(Ax + C(W^{\top}x) + b) \label{autonomous_ESN}
\end{align}
using Newton iterations 
\begin{align*}
    x_{k+1} = x_k - (J(x_k)-I)^{-1}(\psi(x_k) - x_k)
\end{align*}
where $J(x_k)$ is the Jacobian of $\psi$ evaluated at $x_k$. To reduce numerical instability and computational time we do not find $(J(x_k) - I)^{-1}$ explicitly but instead solve
\begin{align*}
    (J(x_k) - I)(x_{k+1} - x_k) = - \psi(x_k) + x_k
\end{align*}
for $(x_{k+1} - x_k)$. The iterates are shown in Figure~\ref{fixed_point_full.fig}, and it appears by eye the algorithm was successful and a fixed point was found.

\begin{figure}
  \centering
    \includegraphics[width=0.9\textwidth]{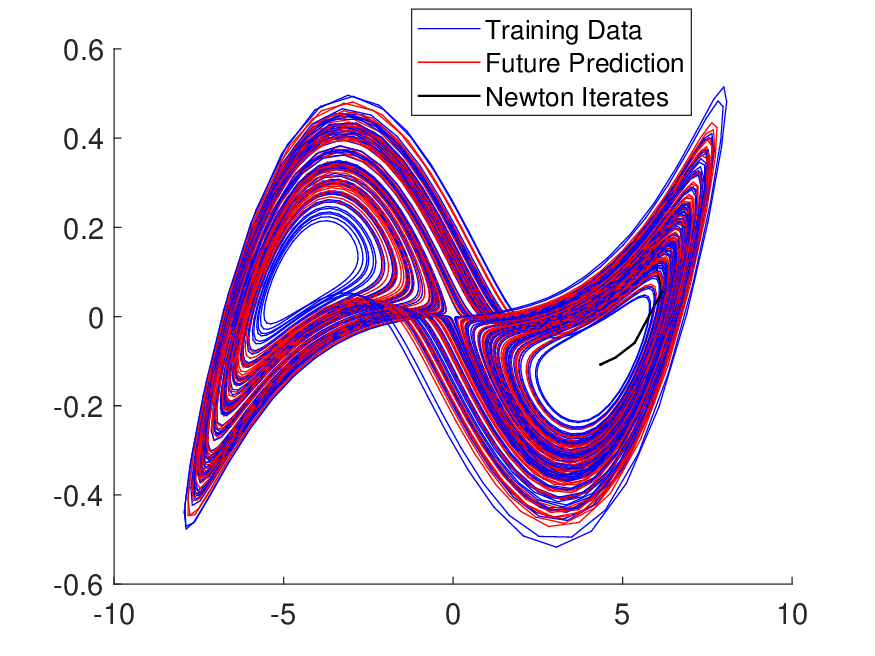}
    \caption{The driven reservoir dynamics are plotted in blue and autonomous dynamics are plotted in red. Both were projected onto the first three principal components of the driven dynamics, then the axes are rotated such that the projection appears on the first 2 components. The black line indicates the iterates of Newton's method, used to locate a fixed point - the method eventually converges to a fixed point in the middle of the right wing of the figure. We can see by eye that the reservoir dynamics appear by eye to be diffeomorphic to the Lorenz system. }
    \label{fixed_point_full.fig}
\end{figure}

Further, if the GS $f$ is a $C^1$ embedding of the original dynamics, we expect $f$ to preserve the stability of fixed points, i.e. we expect the eigenvalues of the linearisation of the autonomous phase to be preserved at every fixed point. Now, comparing the eigenvalues of the linearisation of the Lorenz system and autonomous phase at the respective fixed points requires some subtlety, because the Lorenz system is a continuous time flow, while the autonomous phase is a discrete time map. So, we began by considering one of the known fixed points found in the Lorenz attractor's wings
\begin{align}
    m^* = (6\sqrt{2},6\sqrt{2},27) \nn
\end{align}
 and noted the Jacobian $J$ of the continuous time Lorenz system evaluated at the fixed point $m^*$ is therefore
\begin{align}
J\Big|_{m^*} =
\begin{bmatrix}
    -10 & 10 & 0 \\
    1 & -1 & -6\sqrt{2} \\
    6\sqrt{2} & 6\sqrt{2} & -8/3
\end{bmatrix}. \nn
\end{align}
So the discretised Lorenz system
\begin{align*}
    \phi(\xi,\upsilon,\zeta) = (\xi,\upsilon,\zeta) + \int_0^{\tau} (\dot{\xi},\dot{\upsilon},\dot{\zeta}) \ dt.
\end{align*}
admits a linearisation $\phi^*$ about the fixed point $m^*$
\begin{align*}
    \phi^*(\xi,\upsilon,\zeta) = \exp\bigg(J\Big|_{m^*}\tau\bigg)
    \begin{bmatrix}
        \xi \\ 
        \upsilon \\
        \zeta
    \end{bmatrix}
    , 
\end{align*}
where $\exp : \mathbb{M}_{3 \times 3}(\mathbb{R}) \to \mathbb{M}_{3 \times 3}(\mathbb{R})$ denotes the matrix exponential defined 
\begin{align*}
    \exp(Q) = \sum_{n = 0}^{\infty} \frac{Q^n}{n!}.
\end{align*}
 The matrix $\exp(J\rvert_{m^*} \tau)$ has 3 eigenvalues, which we have compared with the ESN autonomous eigenvalues in Figure \ref{eigenvalues-fig}.

\begin{figure}
  \centering
    \includegraphics[width=0.9\textwidth]{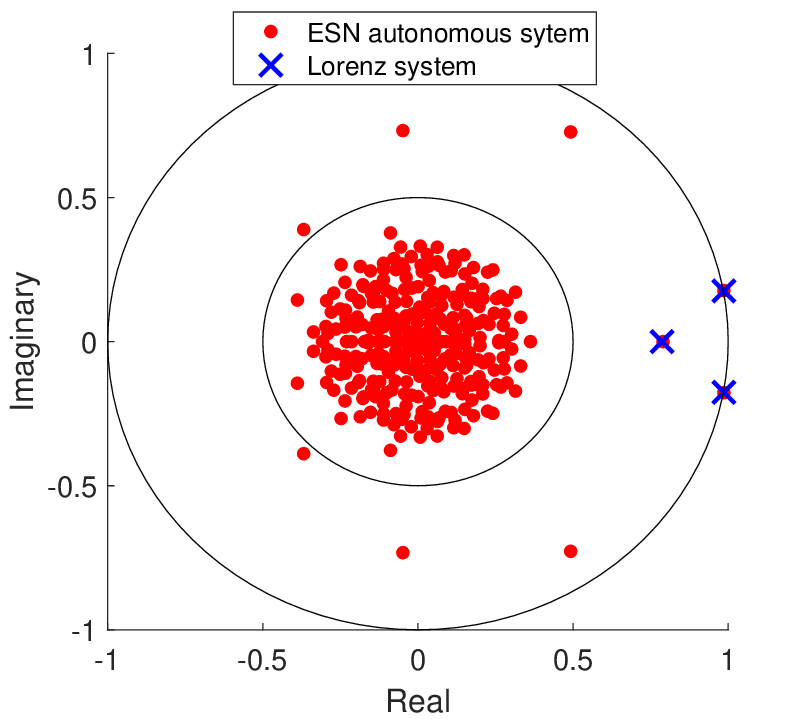}
      \caption{Comparison of the eigenvalues of the discrete time evolution operator for the Lorenz system and the autonomous ESN \eqref{autonomous_ESN}. The three eigenvalues of the linearisation of the Lorenz system on the fixed point inside one of the Lorenz attractor's wings are represented by blue crosses. The 300 eigenvalues of the linearisation of the ESN autonomous system at the fixed point found with Newton's method are represented by red dots. The concentric grey circles have radius 1/2 and 1 respectively.}
      \label{eigenvalues-fig}
\end{figure}

If the GS $f$ is indeed a $C^1$ embedding, the dynamics of the autonomous phase are diffeomorphic to the discrete time Lorenz system on some three dimensional submanifold. This manifold (linearised at the fixed point $m^*$) is spanned by three eigenvectors, each with an associated eigenvalue, which will coincide with the eigenvalues of the linearisation of the Lorenz system at the fixed point. Figure \ref{eigenvalues-fig} appears to show three overlapping eigenvalues, suggesting that the autonomous phase is diffeomorphic to the Lorenz system (at least in a neighbourhood of $m^*$). This is particularly remarkable because $m^*$ is distant from the training data. It appears that the ESN has inferred the existence, position and eigenvalues of a fixed point from training data which contains no fixed points. In machine learning parlance, we might say that the ESN has generalised patterns in the training data to an unseen region of the phase space. The remaining 297 eigenvalues of the ESN autonomous system have absolute value less than 1, suggesting that the fixed point is attractive with respect to them.

\section{Lyapunov Spectra}
\label{section::lyapunov}
Another invariant of the Lorenz system preserved under diffeomorphism is the Lyapunov spectrum, which captures how quickly very close trajectories diverge from eachother, and is used as a measure of chaos. To define the spectrum, let $J$ be the Jacobian of the evolution operator of a continuous time dynamical system. Let $Y$ be the solution of the ODE $\dot Y = JY$ with initial condition $Y(0) = m_0$. Then the Lyapunov spectrum of the invariant set containing $m_0$ is the spectrum of the matrix $\Lambda$ defined
    \begin{align}
        \Lambda = \lim_{t \to \infty} \frac{1}{2t} Y Y^{\top}. \nn
    \end{align}
 Each eigenvalue in the spectrum is called a \emph{Lyapunov exponent} to signify that two initially close trajectories diverge or converge exponentially fast with exponentiation constant in the direction of each eigenvector of $J$ given by a Lyapunov exponent. Further details and discussion is given in \cite{DARBYSHIRE1996287}. The Lyapunov spectrum for the Lorenz system was estimated by \cite{Sprott_2003} as $\{0.9056, 0, -14.5723\}$. In order to compare the Lorenz spectrum to the spectrum of the autonomous ESN, we computed the autonomous system's spectrum using the discrete time $QR$ method discussed in \cite{DARBYSHIRE1996287} and plotted each Lyapunov exponent against the known exponents of the Lorenz system in Figure \ref{Lorenz_Lyapunov}.
 The method is described by the following algorithm 
 
\begin{algorithm}
\caption{Compute the Lyapunov spectrum}
\begin{algorithmic}[1]
\State Let $Q_0$ be an arbitrary $N \times N$ orthogonal matrix
\State Let $J\lvert_{m_0}$ be the Jacobi matrix of the autonomous ESN evaluated at $m_0$
\State \textbf{for} each $j$ from $0$ to $n$
\Indent 
    \State Take the $QR$ decomposition $Q_{j+1} R_{j+1} = Q_j J\lvert_{\phi^j(m_0)}$
    \State Compute $\Lambda_{j+1} = \log(\text{diag}(R_{j+1}))$ where $\text{diag}(A)$ creates a vector from the diagonal entries of $A$)
\EndIndent
\State Compute $\Lambda = \frac{1}{n}\sum_{j=1}^{n} \Lambda_j$. The exact spectrum is obtained as $n \to \infty$
\end{algorithmic}
\end{algorithm}
 
 We found that the largest 2 exponents of the Lorenz system and the ESN map are in good agreement with each other while there was a significant difference between the next-largest in each case. This problem was also noted and encountered by \cite{Pathak2017}, and we do not have a satisfactory explanation for this.

\begin{figure}
  \centering
    \includegraphics[width=0.9\textwidth]{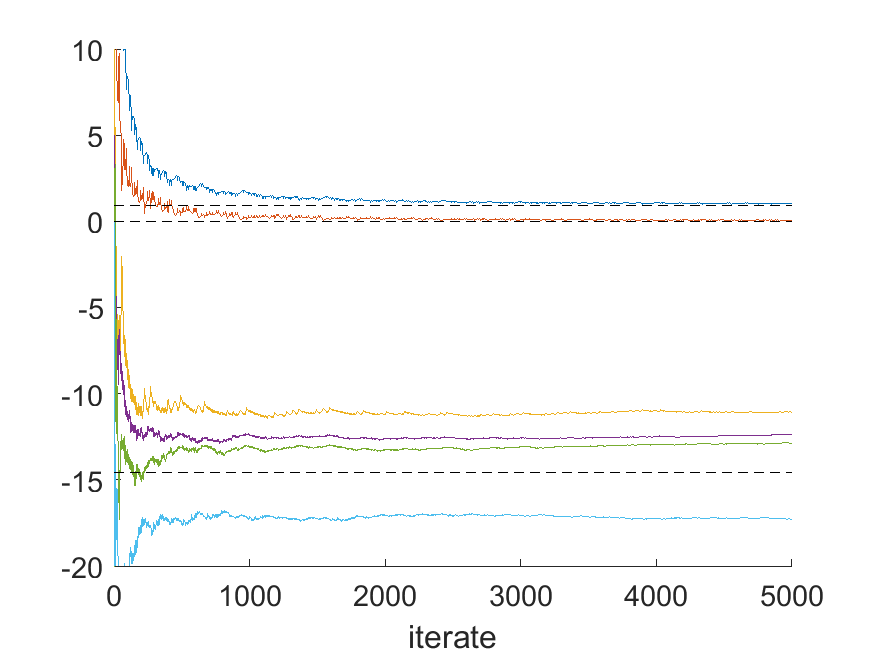}
    \caption{The Lyapunov spectrum of the autonomous phase as the iterates increases is shown. The true Lyapunov exponents of the autonomous phase is given by the limit of these exponents as the iterations tend to infinity. These autonomous exponents are compared to the black dotted lines representing the 3 exponents of the Lorenz system.}
    \label{Lorenz_Lyapunov}
\end{figure}

In the next section we will consider a somewhat different invariant that is preserved under the embedding $f$: the homology groups of the attractor (for a finite length of time intergration). We note that the true attractor has a fractal structure, and that the closure of the true attractor contains the fixed points at the origin and those in the middle of each wing. We therefore proceed carefully in the upcoming sections. We will compute the these groups using \emph{persistent homology}.

\section{Persistent homology}

\subsection{Introduction}

The objective of an unsupervised learning algorithm is to reveal the hidden relations and structures within a dataset, without the use of a training set. A classic unsupervised learning algorithm is principal component analysis (PCA), which reveals linear relations in high dimensional data. Another is $k$-means clustering which partitions a data set into $k$ clusters. In this section, we will explore a completely different kind of structure that a data set might have: homology.

Roughly speaking, homology is the study of objects under continuous deformation, which preserves their connected components, holes and voids. The sphere is made up of 1 connected component, and encloses a 3 dimensional void. The torus is also made up of 1 connected component, and also encloses a 3 dimensional void, but unlike the sphere contains a 2 dimensional hole, or tunnel. It is this distinction between the sphere and the torus that distinguish their homology. Given a constellation of data points then, what does it mean to say it has a hole or encloses a void? Astronomers have for thousands of years looked into the night sky, and in their minds' eye, connected the stars to create shapes with holes and voids. The emerging field of persistent homology offers a mathematical formalism for obtaining similar results. This subsection will outline the main results of persistent homology and explain how to practically compute the homology of a real data set. We draw heavily on the work of \cite{Ghrist2007} and \cite{Ghrist2014}.

\subsection{Theory}

\subsubsection{Simplicial complexes}

A data set can be transformed into a geometric object called a simplicial complex with computable homology. Every singleton set containing a single point is a point, also known as a 0-simplex. Any pair of points in the data set can be connected together to form a line, or 1-simplex. Any triple of points can be collected to form a triangular face, also known as a 2-simplex. Any quadruple of points can together form a tetrahedron, (a 3-simplex) and so on. A set of simplices form a simplicial complex. To make this idea rigorous we will introduce a few definitions.

\begin{defn}
    (Closed under restriction) A collection of sets $S$ is closed under restriction if any subset of a set in $S$ is also in $S$. 
\end{defn}

\begin{defn}
    (Simplicial Complex) Let $X$ be a finite set. A simplicial complex $S$ is a set of subsets of $X$ that is closed under restriction. 
\end{defn}

\begin{defn}
    (Simplex) Let $S$ be a simplicial complex. Then a non-empty $\sigma \in S$ containing $k$ elements is called a $k-1$ simplex.
\end{defn}

Given a data set, how do we decide whether a pair of points ought to be connected to form a 1-simplex? Or indeed how do we decide whether some set of $k$ points ought to be brought together to form a $k-1$ simplex? One method developed by Leopold \cite{Vietoris1927} and Eliyahu Rips supposes every data point is the centre of a ball of radius $\epsilon$. If a set of $k$ points are each contained in the ball of every other, then they form a $k-1$ simplex, otherwise they do not.

\begin{defn}
    (Data set) A data set $X$ is a finite list of elements in $\mathbb{R}^n$. 
\end{defn}

\begin{defn}
    (Vietoris-Rips Complex) The Vietoris-Rips (Rips for short) complex $\mathcal{R}_\epsilon$ of a data set $X$ is the set of subsets of $X$ defined as those subsets $\sigma \subset X$ with members whose pairwise distance (in the Euclidean metric) is less than $\epsilon$. 
\end{defn}

Another method of connecting the dots is called the \u{C}ech Complex, and supposes every data point is the centre of a ball of radius $\epsilon/2$. The division by $2$ is a useful convention for reasons that will soon be made clear. This time, we say a set of $k$ data points form a $k-1$ simplex if and only if their balls have a common intersection. 
\begin{defn}
    (\u{C}ech Complex) The \u{C}ech complex $\mathcal{C}_\epsilon$ of a data set $X$ is a set of subsets of $X$ defined by the statement that $\sigma \subset X$ is in $\mathcal{C}_\epsilon$ if and only if 
    \begin{align*}
        \bigcap\limits_{x\in \sigma} B_{\epsilon/2}(x) \neq \emptyset.
    \end{align*}
\end{defn}

\begin{remark}
The Rips complex and \u{C}ech Complex are clearly closed under restriction and are therefore simplicial complexes.
\end{remark}

The \u{C}ech Complex has the elegant feature of being homotopy equivalent to the union of $\epsilon/2$ balls centred at each data point. 

\begin{defn}
    (Homotopy) Let $X,Y$ be topological spaces and $f,g : X \to Y$ continuous maps. A homotopy between $f,g$ is a continuous function $H : X \times [0,1] \to Y$ such that $H(x,0) = f(x)$ and $H(x,1) = g(x)$.
\end{defn}

\begin{defn}
    (Homotopy Equivalent) Two topological spaces $X$, $Y$ are homotopy equivalent if there is a pair of continuous maps $f : X \to Y$ and $g : Y \to X$ such that $f \circ g$ is homotopic to the identity $\text{id}_{Y}$ and $g \circ f$ is homotopic to the identity $\text{id}_X$.
\end{defn}

One could say the \u{C}ech Complex is the nervous system found inside a muscular subspace of $\mathbb{R}^n$ composed of $\epsilon$-balls. To put this more formally we introduce homotopy and the Nerve Lemma. An example is shown in Figure \ref{nerve_fig}.

\begin{lemma}
    (Nerve Lemma) The \u{C}ech Complex $\mathcal{C}_\epsilon$ of a data set $X$ is homotopic to
    \begin{align*}
        \bigcup\limits_{x\in X} B_{\epsilon/2} (x) \subset \mathbb{R}^n.
    \end{align*}
\end{lemma}
\begin{proof}
    \cite{Alexandroff1928}.
\end{proof}

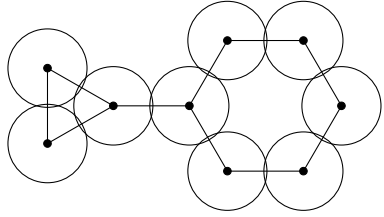
\begin{figure}
  \centering
    \begin{tikzpicture}
   \newdimen\R
   \R=1.0cm
   \newdimen\radius
   \radius=0.52cm
   \draw (0:\R) \foreach \x in {60,120,...,360} { --  (\x:\R) };
   \foreach \x/\l/\p in
     { 60/{1}/above,
      120/{6}/above,
      180/{5}/left,
      240/{4}/below,
      300/{3}/below,
      360/{2}/right
     }
     \node[inner sep=1pt,circle,draw,fill] at (\x:\R) {};
     \draw (-1,0) -- (-2,0) -- (-2.866,0.5) -- (-2.866,-0.5) -- (-2,0);
     \node[inner sep=1pt,circle,draw,fill] at (-2,0) {};
     \node[inner sep=1pt,circle,draw,fill] at (-2.866,0.5) {};
     \node[inner sep=1pt,circle,draw,fill] at (-2.866,-0.5) {};
     \draw (-1,0) circle [radius=\radius];
     \draw (-0.5,0.866) circle [radius=\radius];
     \draw (0.5,0.866) circle [radius=\radius];
     \draw (1,0) circle [radius=\radius];
     \draw (0.5,-0.866) circle [radius=\radius];
     \draw (-0.5,-0.866) circle [radius=\radius];
     \draw (-2,0) circle [radius=\radius];
     \draw (-2.866,0.5) circle [radius=\radius];
     \draw (-2.866,-0.5) circle [radius=\radius];
    \end{tikzpicture}
    \caption{The \u{C}ech complex is homotopic to the union of $\epsilon$-balls, which in this figure, is a genus 2 surface.}
    \label{nerve_fig}
\end{figure}

\subsubsection{Persistence complexes}

Given a data set and a particular $\epsilon$, we can create a Rips or \u{C}ech complex with holes and voids, and therefore a homology we can scrutinise. But the question remains as to how to chose $\epsilon$. A sufficiently small $\epsilon$ will not connect any points and result in a boring list of $0$ dimensional simplices, while a sufficiently large $\epsilon$ will connect every point to every other forming a single gigantic simplex. In persistent homology we consider all values of $\epsilon$, and observe which homological features of the data persist over long intervals of $\epsilon$ and which flutter rapidly in and out of existence. Those features which persist are considered real, while their more transient companions are dismissed as noise.

To make this idea concrete, we consider a data set $X$ and create a sequence of complexes called a \emph{persistence complex} as follows. The first complex $\mathcal{R}_{\epsilon_0}$ is a set of $0$-simplices; no data point is connected to any other. Starting from $\epsilon = 0$ we let $\epsilon$ grow until some pair of points become connected to form a $1$-simplex, giving rise to the next complex in the sequence $\mathcal{R}_{\epsilon_1}$. We let $\epsilon$ grow some until a new pair of points become connected, giving rise to the third complex in the sequence $\mathcal{R}_{\epsilon_2}$. We continue to add terms to the sequence until $\epsilon$ is large enough that every point in the data set belongs to the same simplex.
Whether some set of $k$ points are connected and form a $k-1$ simplex depends of course on whether we are considering a sequence of Rips complexes or \u{C}ech complexes. The following definitions formalise these ideas

\begin{defn}
    (Filtration) A filtration is a collection of sets $X_i$ with the property that if $i < j$, then $X_i \subset X_j$.
\end{defn}

\begin{defn}
    (Rips Persistence complex) Let $X$ be a data set and notice that the subsets of
    $\{ \mathcal{R}_\epsilon | \epsilon > 0 \}$ form a filtration $\mathcal{R}_{\epsilon_0} \subset \mathcal{R}_{\epsilon_1} \subset \mathcal{R}_{\epsilon_2} \subset \ldots \subset \mathcal{R}_{\epsilon_m}$ for $\epsilon_0 < \epsilon_1 < \ldots < \epsilon_m$. The sequence of inclusion maps
    \begin{align*}
        \mathcal{R}_{\epsilon_0} \xhookrightarrow{\iota_1} \mathcal{R}_{\epsilon_1} \xhookrightarrow{\iota_2} \mathcal{R}_{\epsilon_2} \xhookrightarrow{\iota_3} \ldots \xhookrightarrow{\iota_m} \mathcal{R}_{\epsilon_m}
    \end{align*}
    is called the Rips persistence complex.
\end{defn}

\begin{defn}
    (\u{C}ech Persistence complex) Let $X$ be a data set and consider that the subsets of
    $\{ \mathcal{C}_\epsilon | \epsilon > 0 \}$ form a filtration $\mathcal{C}_{\epsilon_0} \subset \mathcal{C}_{\epsilon_1} \subset \mathcal{C}_{\epsilon_2} \subset \ldots \subset \mathcal{C}_{\epsilon_m}$. The sequence of inclusion maps
    \begin{align*}
        \mathcal{C}_{\epsilon_0} \xhookrightarrow{\iota_1} \mathcal{C}_{\epsilon_1} \xhookrightarrow{\iota_2} \mathcal{C}_{\epsilon_2} \xhookrightarrow{\iota_3} \ldots \xhookrightarrow{\iota_m} \mathcal{C}_{\epsilon_m}
    \end{align*}
    is called the \u{C}ech persistence complex.
\end{defn}
For a given data set $X$ the Rips and \u{C}ech chain complexes are not the same in general, but we can always squeeze a \u{C}ech complex between two Rips complexes as follows.
\begin{lemma}
    (Squeezing Lemma) For any $\epsilon > 0$ there is a chain of inclusion maps
    \begin{align*}
        \mathcal{R}_{\epsilon} \xhookrightarrow{} \mathcal{C}_{\epsilon \sqrt{2}} \xhookrightarrow{} \mathcal{R}_{\epsilon \sqrt{2}}.
    \end{align*}
\end{lemma}
\begin{proof}
    \cite{diSilva}
\end{proof}
It follows from this Lemma that any homological feature of the Rips complex that persists from $\epsilon$ to $\epsilon \sqrt{2}$ is also a feature of the \u{C}ech complex $\mathcal{C}_{\epsilon \sqrt{2}}$. An illustration of this squeezing phenomenon is shown in Figure \ref{squeezing_fig}.

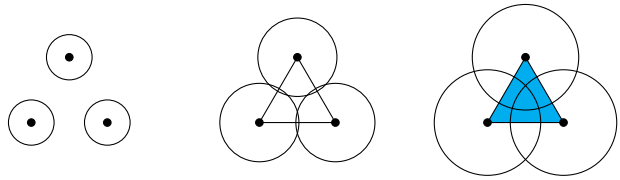
\begin{figure}
  \centering
    \begin{tikzpicture}
    \draw (0,0) circle [radius=0.3];
    \draw (1,0) circle [radius=0.3];
    \draw (0.5,0.866) circle [radius=0.3];
    \node[inner sep=1pt,circle,draw,fill] at (1,0) {};
    \node[inner sep=1pt,circle,draw,fill] at (0.5,0.866) {};
    \node[inner sep=1pt,circle,draw,fill] at (0,0) {};
    \draw (3,0) -- (4,0) -- (3.5,0.866) -- (3,0);
    \draw (3,0) circle [radius=0.52];
    \draw (4,0) circle [radius=0.52];
    \draw (3.5,0.866) circle [radius=0.52];
    \node[inner sep=1pt,circle,draw,fill] at (3,0) {};
    \node[inner sep=1pt,circle,draw,fill] at (3.5,0.866) {};
    \node[inner sep=1pt,circle,draw,fill] at (4,0) {};
    \draw [fill = cyan] (6,0) -- (7,0) -- (6.5,0.866) -- (6,0);
    \draw (6,0) circle [radius=0.7];
    \draw (7,0) circle [radius=0.7];
    \draw (6.5,0.866) circle [radius=0.7];
    \node[inner sep=1pt,circle,draw,fill] at (7,0) {};
    \node[inner sep=1pt,circle,draw,fill] at (6.5,0.866) {};
    \node[inner sep=1pt,circle,draw,fill] at (6,0) {};
    \end{tikzpicture}
    \caption{A filtration of \u{C}ech complexes for increasing $\epsilon$. We can see that for $\epsilon < 1$ the \u{C}ech complex is a collection of $3$ points, for $1 < \epsilon < \sqrt{2}$ the \u{C}ech complex forms a 1 dimensional loop, and for $\epsilon > 1$ the complex is a triangle. The Rips complex for the same 3 points is $3$ isolated points when $\epsilon < 1$ and a triangle otherwise. The one dimensional loop in the sequence of \u{C}ech complexes is squeezed between 2 Rips complexes; 3 isolated points and a triangle.}
    \label{squeezing_fig}
\end{figure}

\subsubsection{Simplicial homology}

We first seek the homology of a simplicial complex defined at a particular value of $\epsilon$. To this end, we note that a simplicial complex is composed of some $m$-simplices, $(m-1)$-simplices, $(m-2)$-simplices and so on, until a final set of $0$-simplices. Each $k$ simplex has a boundary, which is composed of $(k-1)$-simplices. The boundary of a 1-simplex (line) is composed of the two $0$-simplices (points) with one at each end. The boundary of a $2$-simplex (triangle) is composed of the three $1$-simplices (lines) that enclose it. The boundary of a $3$-simplex (a tetrahedron) is composed of the four $2$-simplices (triangles) that enclose it and so on. The boundary of a $0$-simplex is the empty set.

\begin{figure}[b!]
  \centering
    \begin{tikzpicture}
    \draw [fill = cyan] (0,0) -- (1,0) -- (1,1) -- (0,1) -- (0,0);
        \draw [fill = cyan] (1,0) -- (2,0) -- (2,1) -- (1,1) -- (1,0);
    \node[inner sep=1pt,circle,draw,fill,label={1}] at (0,1) {};
    \node[inner sep=1pt,circle,draw,fill,label={2}] at (1,1) {};
    \node[inner sep=1pt,circle,draw,fill,label={3}] at (2,1) {};
    \draw[fill] (2,0) circle [radius = 0.05];
    \draw[fill] (1,0) circle [radius = 0.05];
    \draw[fill] (0,0) circle [radius = 0.05];
    \node[below] at (2,0) {4};
    \node[below] at (1,0) {5};
    \node[below] at (0,0) {6};
    \end{tikzpicture}
    \caption{The boundary of this 2-complex is the set of 1-simplices $\{ (12),(23),(34),(45),(56),(61) \}$}
    \label{boundary_fig}
\end{figure}
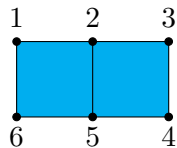

More generally, a set of $k$-simplices has a boundary composed of $(k-1)$-simplices. For example, a pair of triangles welded together along one side has a boundary composed of four $1$-simplices. A pair of squares welded together along one edge is shown in Figure \ref{boundary_fig}. A pair of tetrahedra welded together along one face has a boundary comprised of $6$ triangular faces. This operation of welding simplices together will be denoted with a $+$. Under the operation $+$ the $k$ simplices form a group.

\begin{defn}
    (Simplex Group) Let $s_k$ be the set of all $k$-simplices in a simplicial complex. Let $S_k = P(s_k)$: the powerset of $s_k$. Let $+$ be a binary operation on the elements of $S_k$ defined
    \begin{align*}
        \sigma + \tau = (\sigma \cup \tau) / (\sigma \cap \tau)
    \end{align*}
    for $H,G \in S_k$. Under the operation $+$ the set $S_k$ forms an Abelian group with identity $\emptyset$ called the simplex group.
\end{defn}

It therefore makes sense for each $k$ to define a map $\partial_k$ which sends a set of $k$ simplices to the set of $(k-1)$ simplices which comprise its boundary. It turns out this map is a group homomorphism.

\begin{lemma}
    Let $s_k$ be the set of all $k$-simplices in a simplicial complex. Let $S_k = P(s_k)$. Let $\partial_k : S_k \to S_{k-1}$ be defined 
    \begin{align*}
        \partial_k (\sigma) = s_{k-1} \cap \sum_{\varsigma \in \sigma} \varsigma 
    \end{align*}
    then $\partial_k$ is a homomorphism.
\end{lemma}

\begin{proof}
    The proof proceeds directly
    \begin{align*}
        \partial_k (\sigma) + \partial_k (\tau) = \big( \partial_k (\sigma) \cup \partial_k (\tau) \big) / \big( \partial_k (\sigma) \cap \partial_k (\tau) \big) \\
        = \Bigg( s_{k-1} \cap \sum_{\varsigma \in (\sigma \cup \tau)} \varsigma \Bigg) \Bigg/ \Bigg( s_{k-1} \cap \sum_{\varsigma \in ( \sigma \cap \tau )} \varsigma \Bigg) \\
        = s_{k-1} \cap \sum_{\varsigma \in (\sigma \cup \tau) / (\sigma \cap \tau)} \varsigma \\
        = s_{k-1} \cap \sum_{\varsigma \in (\sigma + \tau)} \varsigma \\
        = \partial_k (\sigma + \tau).
    \end{align*}
\end{proof}

A set of $k$ simplices that have an empty boundary is called a cycle. For example a set of $1$-simplices arranged in a pentagon form a loop, which has no $0$ dimensional boundary. Similarly a set of $2$-simplices welded together to form a cube form a closed surface, which has no $1D$ boundary. Formally, a set of $k$ simplices is a cycle if the boundary map $\partial_k$ sends the set to the empty set $\emptyset$. This motivates the following definition.
\begin{defn}
    (Cycle Group) Let $S$ be a simplicial complex. The group of $k$-cycles of $S$ is
    \begin{align*}
        Z_k = \text{ker}(\partial_k).
    \end{align*}
\end{defn}
Since the boundary map returns boundaries, the image of the $(k+1)$ boundary map is the group of boundaries among the group of $k$-simplices. This is formalised in the following definition.
\begin{defn}
    (Boundary Group) Let $S$ be a simplicial complex. The group of $k$-boundaries of $S$ is 
    \begin{align*}
        B_k = \text{Im}(\partial_{k+1})
    \end{align*} 
\end{defn}

We say that 2 cycles are equivalent if we can weld them together to form a boundary. Under this equivalence relation, the equivalence class of $k$-cycles can be interpreted as a $(k+1)$-dimensional hole. For example an equivalence class of $1$-cycles represents a $2$ dimensional hole - exactly the type of hole that distinguishes the torus from the sphere. An equivalence class of $2$-cycles can be interpreted as a $3$ dimensional void, like the volume enclosed by the torus, or the volume enclosed by the sphere. The equivalence classes of cycles form a group as defined below.

\begin{defn}
    (Homology Group) The $k$th homology group is the quotient of the cycles group by the boundary group:
    \begin{align*}
        H_k = \text{ker}(\partial_k) / \text{Im}(\partial_{k+1}) = Z_k / B_k.
    \end{align*}
\end{defn}

\begin{defn}
    (Betti Numbers) The $k$th Betti number is
    \begin{align*}
        \beta_k = \text{rank}(H_k).
    \end{align*}
\end{defn}

The Betti numbers $\beta_0, \beta_1 , \beta_2, ...$ of a simplicial complex characterise its homology. $\beta_0$ is the number of connected components, $\beta_1$ is the number of $2$ dimensional holes, $\beta_2$ is the number of $3$ dimensional voids, and the remaining $\beta_k$ represent higher dimensional voids. The Betti numbers of a simplicial complex are the coefficients of a polynomial called the Poincar\'{e} Polynomial.

\begin{defn}
    (Poincar\'{e} polynomial) The Poincar\'{e} polynomial of a simplicial complex $S$ is defined
    \begin{align*}
        p[X] = \sum_{i} \beta_i X^i.
    \end{align*}
\end{defn}

\subsubsection{Persistent homology}

In the previous section we introduced the homology group of a single simplicial complex - but we can do better! We will now present the persistent homology group which characterises the homology of a persistence complex that persists over a range of $\epsilon$.

\begin{defn}
    (Induced homomorphism) Let $i < j$ and $\mathcal{R}_{\epsilon_i}$, $\mathcal{R}_{\epsilon_j}$ be a pair of Rips complexes from the same Rips persistence complex. Let $Z^i_k$ and $Z^j_k$ be the $k$-cycle groups of the complex $\mathcal{R}_{\epsilon_i}$ and $\mathcal{R}_{\epsilon_j}$ respectively. Let $\iota : Z^i_k \to Z^j_k$ be the inclusion map. Then $\iota_* : H^i_k \to H^j_k$ is called the induced homomorphism of $\iota$.
\end{defn}

\begin{lemma}
    Let $x$ be a $k$-cycle in $\mathcal{R}_{\epsilon_i}$ so $x \in Z^i_k$. Let $[x]_i$ denote the equivalence class of $x$ under the quotient relation
    \begin{align*}
        H^i_k = Z^i_k / B^i_k.
    \end{align*}
    Then the induced homomorphism $\iota_*$ satisfies 
    \begin{align*}
        [ \iota(x) ]_j = \iota_*([x]_i).
    \end{align*}
\end{lemma}
\begin{proof}
    \cite{Ghrist2014}.
\end{proof}

\begin{defn}
    (Persistent Homology Group) Let $\mathcal{R}$ be a Rips persistence complex. The $(i,j)$ persistent $k$ homology group of $\mathcal{R}$ is
    \begin{align*}
        H^{i \to j}_k = \text{Im}(\iota_*).
    \end{align*}
\end{defn}

We should take a moment to celebrate. The $(i,j)$ persistent $k$ homology group $H^{i \to j}_k$ represents exactly those $(k+1)$ dimensional holes that persist over the interval $[\epsilon_i, \epsilon_j]$. For example $H^{1 \to 2}_2$ contains all $3$-dimensional voids that persist from $\epsilon = \epsilon_1$ to $\epsilon = \epsilon_2$. This group is exactly what we want.

\subsubsection{Computing the persistent homology}

We have presented an algebraic derivation of the persistent homology group. This section will provide some details about how to determine the group in practice. The key observation is that we can convert the group theory problem into a linear algebra problem.
Let $S$ be an $n$-simplex (and therefore a simplical complex).
Let the $k$th simplex in $S$ be the $k$th canonical unit vector of the $n$ dimensional vector space over the field of $2$ elements $\mathbb{F}^n_2$. i.e. the first simplex is $(1,0,0, ... , 0)$, the second is $(0,1,0, ... , 0)$ and so on.
It turns out the $k$th simplex group $S_k$ is isomorphic to a subspace of $\mathbb{F}^n_2$. The boundary map $\partial_k : S_k \to S_{k-1}$ turns out to be a linear map between subspaces. The cycle group $Z_k$ is isomorphic to a boundary space determined by computing $\text{ker}(\partial_k)$ using the standard techniques of linear algebra. The boundary group $B_k$ is just the column span of the matrix representation of $\partial_k$. The quotient
\begin{align*}
    H_k = Z_k / B_k
\end{align*}
is isomorphic to the column span of the vectors that are not in the span of the $B_k$. Using this vector space representation, we are ready to outline an algorithm for computing the persistent homology group. Suppose we have an $n$ dimensional data set $X$ comprising $m$ points. Then Algorithm \ref{persistent_homology_pseudocode} is pseudocode for computing the Rips persistent homology groups of the data set $X$. 

\begin{algorithm}
\caption{Compute $H^{i \to j}_k$}\label{persistent_homology_algorithm}
\begin{algorithmic}[1]
\State assign to every simplex in $P(X)$ a unique natural number
\State compute the distance between every pair of points in $X$
\State list the distances in ascending order
\State create a filtration of Rips complexes
\State \textbf{for} each $j$ indexing the $j$th Rips complex in the filtration
\Indent
    \State partition the simplices of the Rips complex into simplices of equal dimension $k$
    \State \textbf{for} all dimensions $k > 0$ 
    \Indent
        \State \textbf{for} each simplex in the set
        \Indent
            \State retrieve the unique natural number $\ell$ of the simplex
            \State set the $\ell$th column of matrix representing $\partial_k$ as the vector representing the boundary of the
            $\ell$th simplex
            \EndIndent
        \State find a basis for the vector space isomorphic to $Z_k = \text{ker}(\partial_k)$
        \State find a basis for the vector space isomorphic to $B_{k-1} = \text{Im}(\partial_k)$ 
        \State \textbf{if} $k > 1$
            \Indent \State find a basis for the vector space isomorphic to $H^j_{k-1} = Z_{k-1} / B_{k-1}$
            \EndIndent
        \State \textbf{for} all $i < j$
        \Indent
            \State \textbf{for} each vector $v$ in the basis of $H^i_k$
            \Indent
                \State \textbf{if} $v$ is in the span of $B^j_k$
                \Indent 
                    \State $v$ is not in the $(i,j)$ persistent $k$ homology group
                    \EndIndent
                \State \textbf{else}
                \Indent 
                    \State $v$ is in the $(i,j)$ persistent $k$ homology group
                \EndIndent
            \EndIndent
        \EndIndent
    \EndIndent
\EndIndent
\end{algorithmic}
\label{persistent_homology_pseudocode}
\end{algorithm}
This algorithm is na\"{i}ve, and if properly implemented would surely run slower than more efficient algorithms like Ripser developed by \cite{ctralie2018ripser}.

Computing the persistent homology group in practice is a confusing ordeal, so we will offer a hopefully illustrative example. Consider $6$ nodes arranged in a regular hexagon, and contemplate the resultant Rips filtration. This is illustrated in Figure \ref{example_hexagon_complex}. Tables \ref{del_2} and \ref{del_1} show the matrix representations $D_2$ and $D_1$ of the boundary maps $\partial_2$ and $\partial_1$. We can find a basis for the column span of $D_2$ and null space of $D_1$ and determine that the homology group
\begin{align*}
    H^{2}_1 = \text{ker}(\partial_1) / \text{Im}(\partial_2)
\end{align*}
is isomorphic to the space spanned by those vectors in the null space of $D_1$ not in the column span of $D_2$. We shall call this list of vectors $\alpha$. We can similarly compute a basis for a vector space isomorphic to $H^{3}_1$, which we will call $\beta$. The persistent homology group $H^{2 \to 3}_1$ is isomorphic to the span of those vectors $\alpha$ not in the span of the vectors in $\beta$. 

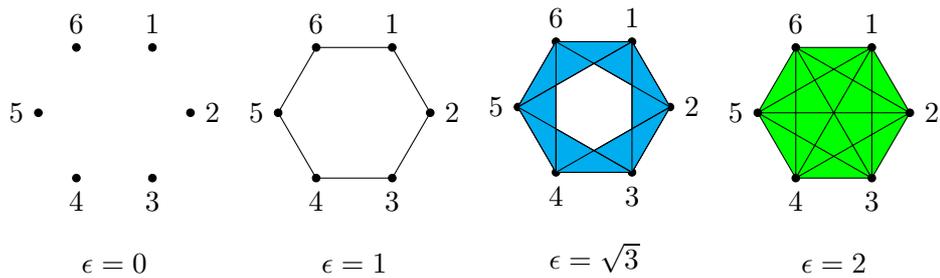
\begin{figure}
    \centering
        \begin{tikzpicture}
   \newdimen\R
   \R=1.0cm
   \draw (0:\R) \foreach \x in {60,120,...,360} {  (\x:\R) };
   \foreach \x/\l/\p in
     { 60/{1}/above,
      120/{6}/above,
      180/{5}/left,
      240/{4}/below,
      300/{3}/below,
      360/{2}/right
     }
     \node[inner sep=1pt,circle,draw,fill,label={\p:\l}] at (\x:\R) {};
     \node at (0,-2) {$\epsilon = 0$};
    \end{tikzpicture}
    \begin{tikzpicture}
   \newdimen\R
   \R=1.0cm
   \draw (0:\R) \foreach \x in {60,120,...,360} {  -- (\x:\R) };
   \foreach \x/\l/\p in
     { 60/{1}/above,
      120/{6}/above,
      180/{5}/left,
      240/{4}/below,
      300/{3}/below,
      360/{2}/right
     }
     \node[inner sep=1pt,circle,draw,fill,label={\p:\l}] at (\x:\R) {};
     \node at (0,-2) {$\epsilon = 1$};
    \end{tikzpicture}
    \begin{tikzpicture}
   \newdimen\R
   \R=1.0cm
   \draw (0:\R) \foreach \x in {60,120,...,360} {  -- (\x:\R) };
   \foreach \x/\l/\p in
     { 60/{1}/above,
      120/{6}/above,
      180/{5}/left,
      240/{4}/below,
      300/{3}/below,
      360/{2}/right
     }
     \node[inner sep=1pt,circle,draw,fill,label={\p:\l}] at (\x:\R) {};
     \draw [fill = cyan] (0.5,0.866) -- (-0.5,0.866) -- (1,0) -- (0.5,0.866);
     \draw [fill = cyan] (0.5,0.866) -- (1,0) -- (0.5,-0.866) -- (0.5,0.866);
     \draw [fill = cyan] (1,0) -- (0.5,-0.866) -- (-0.5,-0.866) -- (1,0);
     \draw [fill = cyan] (0.5,-0.866) -- (-0.5,-0.866) -- (-1,0) -- (0.5,-0.866);
     \draw [fill = cyan] (-0.5,-0.866) -- (-1,0) -- (-0.5,0.866) -- (-0.5,-0.866);
    \draw [fill = cyan] (-1,0) -- (-0.5,0.866) -- (0.5,0.866) -- (-1,0);
    \draw (-1,0) -- (0.5,0.866) -- (0.5,-0.866) -- (-1,0);
     \draw (-0.5,0.866) -- (1,0) -- (-0.5,-0.866) -- (-0.5,0.866);
     \node at (0,-2) {$\epsilon = \sqrt{3}$};
    \end{tikzpicture}
        \begin{tikzpicture}
   \newdimen\R
   \R=1.0cm
   \draw [fill = green] (0:\R) \foreach \x in {60,120,...,360} {  -- (\x:\R) };
   \foreach \x/\l/\p in
     { 60/{1}/above,
      120/{6}/above,
      180/{5}/left,
      240/{4}/below,
      300/{3}/below,
      360/{2}/right
     }
     \node[inner sep=1pt,circle,draw,fill,label={\p:\l}] at (\x:\R) {};
     \draw (-1,0) -- (0.5,0.866) -- (0.5,-0.866) -- (-1,0);
     \draw (-0.5,0.866) -- (1,0) -- (-0.5,-0.866) -- (-0.5,0.866);
     \draw (-0.5,0.866) -- (0.5,-0.866);
     \draw (0.5,0.866) -- (-0.5,-0.866);
     \draw (-1,0) -- (1,0);
     \node at (0,-2) {$\epsilon = 2$};
    \end{tikzpicture}
    \caption{A filtration of Rips complexes $\mathcal{R}_0 \xhookrightarrow{} \mathcal{R}_1 \xhookrightarrow{} \mathcal{R}_{\sqrt{3}} \xhookrightarrow{} \mathcal{R}_{2}.$
    }
    \label{example_hexagon_complex}
\end{figure}

\begin{table}
\begin{center}
 \begin{tabular}{||c c c c c c c||} 
 \hline
  & (123) & (234) & (345) & (456) & (561) & (612) \\ [0.5ex] 
 \hline\hline
 (12) & 1 & 0 & 0 & 0 & 0 & 1 \\ 
 (23) & 1 & 1 & 0 & 0 & 0 & 0 \\
 (34) & 0 & 1 & 1 & 0 & 0 & 0 \\
 (45) & 0 & 0 & 1 & 1 & 0 & 0 \\
 (56) & 0 & 0 & 0 & 1 & 1 & 0 \\ 
 (61) & 0 & 0 & 0 & 0 & 1 & 1 \\ 
  (13) & 1 & 0 & 0 & 0 & 0 & 0 \\ 
  (35) & 0 & 0 & 1 & 0 & 0 & 0 \\ 
  (51) & 0 & 0 & 0 & 0 & 1 & 0 \\ 
  (62) & 0 & 0 & 0 & 0 & 0 & 1 \\ 
  (24) & 0 & 1 & 0 & 0 & 0 & 0 \\ 
  (46) & 0 & 0 & 0 & 1 & 0 & 0 \\ 
  \hline
\end{tabular}
\end{center}
\caption{For $\sqrt{3} < \epsilon < 2$, the Rips complex in Figure \ref{example_hexagon_complex} has $6$ traingular faces (2-simplices) each denoted by listing the 3 nodes they contain, e.g. (123). The boundary map $\partial_2$ sends a face to a set of 3 edges ($1$-simplices). For example $\partial_2$ sends the face (123) to the set of edges \{ (12),(23),(13) \}. We can therefore represent $\partial_2$ as a matrix with $ij$th entry the $ij$th entry of this table.}
\label{del_2}
\end{table}

\begin{table}
\begin{center}
 \begin{tabular}{||c c c c c c c c c c c c c||}
 \hline
  & (12) & (23) & (34) & (45) & (56) & (61) & (13) & (35) & (51) & (62) & (24) & (46) \\ [0.5ex] 
 \hline\hline
 (1) & 1 & 0 & 0 & 0 & 0 & 1 & 1 & 0 & 1 & 0 & 0 & 0 \\ 
 (2) & 1 & 1 & 0 & 0 & 0 & 0 & 0 & 0 & 0 & 1 & 1 & 0 \\ 
 (3) & 0 & 1 & 1 & 0 & 0 & 0 & 1 & 1 & 0 & 0 & 0 & 0 \\ 
 (4) & 0 & 0 & 1 & 1 & 0 & 0 & 0 & 0 & 0 & 0 & 1 & 1 \\ 
 (5) & 0 & 0 & 0 & 1 & 1 & 0 & 0 & 1 & 1 & 0 & 0 & 0 \\ 
 (6) & 0 & 0 & 0 & 0 & 1 & 1 & 0 & 0 & 0 & 1 & 0 & 1 \\ 
 \hline
\end{tabular}
\end{center}
\caption{For $\sqrt{3} < \epsilon < 2$, the Rips complex in Figure \ref{example_hexagon_complex} has $12$ edges (1-simplices) each denoted by listing the 2 nodes they contain, e.g. (12). The boundary map $\partial_1$ sends an edge to a set of 2 nodes ($0$-simplices). For example $\partial_1$ sends the face (12) to the set of nodes \{ (1),(2) \}. We can therefore represent $\partial_1$ as a matrix with $ij$th entry the $ij$th entry of this table.}
\label{del_1}
\end{table}
At any Rips complex in the filtration, a basis element that has not featured in any previous complex may appear in an associated homology group. This basis element represents a hole or void. At some later point in the filtration, the element representing the hole or void will disappear. We can record the ordered pair of $\epsilon$-values at which the void is born and dies. Then, for a particular $k$, we can take all birth-death pairs for persistent homology groups of order $k$ and plot them on a 2-dimensional plane, called a persistence diagram.

\subsection{Persistent homology of the Lorenz attractor}

We compared the homology groups of the Lorenz attractor to the persistent homology groups of the autonomous and driven attractors.
We followed the lead of \cite{GARLAND201649} who computed the persistent homology of the Lorenz system reconstructed from a sequence of 1D observations of a Lorenz trajectory using the delay observation map described in Takens' Theorem. The authors used the open source software Javaplex created by \cite{Javaplex} to find the Witness Complex (outside the scope of this thesis) for the delay embedded Lorenz attractor and computed the homology of the complex. They discuss a few subtleties that arise, in particular that the Lorenz attractor is a fractal, whose structure cannot be reconstructed exactly from any finite number of sample points. The authors therefore satisfied themselves by approximating the Lorenz attractor with a branched manifold model presented by \cite{PMIHES_1979__50__73_0} which has the homology of the \emph{double loop} (which is topologically a drawing of the number 8). We made the same approximation, and expected to find that the application of persistent homology to the Lorenz system, driven ESN dynamics, and autonomous ESN dynamics would reveal that all three have the double loop homology groups. In particular the persistence diagrams of these three systems would exhibit a pair of $H_1$ persistent homology groups floating well above the diagonal. To verify this, we produced persistence diagrams using the open source software Ripser produced by \cite{ctralie2018ripser} and plotted the results in Figure \ref{Persistence_fig}.

In the Figure we plot the $H_1$ persistence diagrams of the driven ESN dynamics, autonomous ESN dynamics, and Lorenz dynamics as blue circles, red downward triangles, and purple upward triangles. Each symbol represents a homology group. The $x$-coordinate of the symbol represents the length $\epsilon$ at which a homology group appears (is born) and the $y$-coordinate represents the length at which the homology group vanishes (dies). All symbols are above the diagonal because all homology groups die after they are born. Symbols close to the diagonal represent homology groups that die shortly after they are born, and are the result of noise, or an artefact of finite data. Symbols far above the diagonal represent homology groups that persist long after they are born and therefore represent persistent homology groups.

We can see a pair of blue circles, red downward triangles, and purple upward triangles floating well above the diagonal, suggesting the existence of 2 persistent homology groups for the driven ESN dynamics, autonomous ESN dynamics, and Lorenz dynamics. This is consistent with our expectation that the driven ESN dynamics, autonomous ESN dynamics, and Lorenz dynamics all adopt the topology of the double loop, which of course has 2 holes. The blue circles, red downward triangles, and purple upward triangles are not exactly in the same location because the driven ESN dynamics, autonomous ESN dynamics, and Lorenz dynamics are not isometric.

The reader may wonder why we would use persistent homology to show that the Lorenz system, driven ESN dynamics, and autonomous ESN dynamics all have the homology of the double loop when this can clearly be seen in Figure \ref{fixed_point_full.fig}.  The homology of a 3D system is usually apparent from a plot, but persistent homology can reveal the holes, voids and higher dimensional hypervoids of high dimensional systems that cannot be easily visualised. For example \cite{MULDOON19931} computed the homology of a delay embedded time series from a fluid dynamics experiment, which could in general be of much higher dimension.

\begin{figure}
  \centering
    \includegraphics[width=0.9\textwidth]{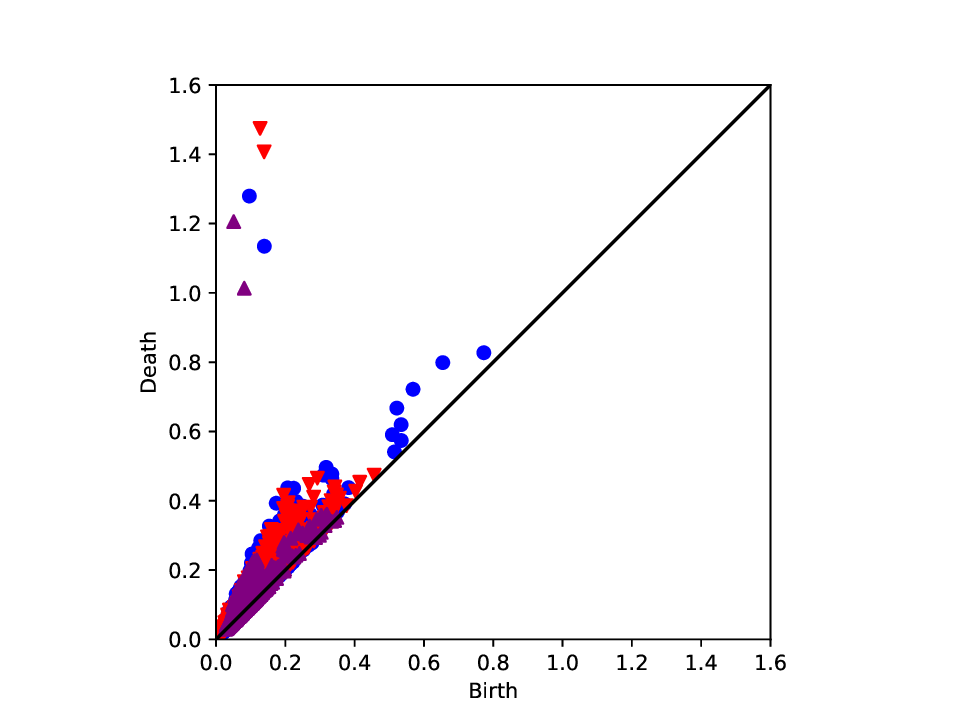}
    \caption{We have plotted the $H_1$ persistence diagrams of the driven ESN dynamics, autonomous ESN dynamics, and Lorenz dynamics as blue circles, red downward triangles, and purple upward triangles. If a homology group appears (is born) at time $\epsilon_0 > 0$ and then vanishes (dies) at time $\epsilon_1 > \epsilon_0 > 0$ then we plot this homology group at the point $(\epsilon_0,\epsilon_1) \in \mathbb{R}^2$ on a so-called persistence diagram. Points that are close to the diagonal die shortly after they are born, and are therefore transient homology groups, which we view as artefacts. Points far from the diagonal die long after they are born and are therefore persistent homology groups which we interpret as truly describing the topology of the underlying object. We can see that each of these 3 objects has a pair of points floating well above the diagonal, suggesting each has 2 holes. This is consistent with our expectation that all three adopt the topology of the double loop.}
    \label{Persistence_fig}
\end{figure}

\chapter{Generalised Synchronisation (GS)} %Include abbreviation in the title so that the abbreviation appears in the contents list at the top of the document.

\label{chapter::SSMs}

\section{Background}

In this Chapter we will examine the generalised synchronisation (GS) i.e. the map $f$ from an underlying dynamical system $(M,\phi)$ to the reservoir space $\mathbb{R}^N$. GSs appear in the seminal papers by \cite{PhysRevE.51.980}, \cite{PhysRevLett.76.1816}, \cite{doi:10.1063/1.166278}, \cite{stark_1999}, \cite{BOCCALETTI20021}, and a recent paper by \cite{doi:10.1080/00107514.2017.1345844}. Work on GS appears broadly in mathematics and physics \citep{PhysRevLett.64.821}, \citep{PhysRevA.44.2374}, \citep{doi:10.1080/00107514.2017.1345844}, and extensively in reservoir computing \citep{doi:10.1063/1.5022276}, \citep{Verzelli}, \citep*{chaos_on_compacta}, \citep{doi:10.1063/1.5128898}.

For the purpose of this thesis we will define the GS (in discrete time) in terms of a drive and response system.

\begin{defn}
    (Drive and response systems) Let $M$ be a smooth manifold and $\phi \in \text{Diff}^1(M)$ a diffeomorphism. This is called the \emph{drive system}. Let $G : \mathbb{R}^N \times M \to \mathbb{R}^N$, and define the system
    \begin{align*}
        x_{k+1} = G(x_k,\phi^k(m_0)).
    \end{align*}
    This is called the \emph{response system}.
\end{defn}

A reservoir computer is essentially a drive-response system where the underlying process $(M,\phi)$ is the drive system, and the evolution of the reservoir states (the reservoir dynamics) is the response system. We are interested in finding a map $f : M \to \mathbb{R}^N$ relating the drive to the response, and this is called a generalised synchronisation (GS).

\begin{defn}
    (Generalised Synchronisation) 
    Let $V \subseteq \mathbb{R}^N$.
    A generalised synchronisation (GS) between the drive system and response system is a map $f : M \to \mathbb{R}^N$ such that for any $x_0 \in V \subseteq \mathbb{R}^N$ and any $m_0 \in M$ we have $f\circ\phi^{k}(m_0) = G(x_k,\phi^{k}(m_0))$ for all $k \in \mathbb{Z}$. If $V$ is a proper subset of $\mathbb{R}^N$ then $f$ is called a \emph{local GS}, and if $V = \mathbb{R}^N$ then $f$ is called a \emph{global GS}.
\end{defn}

We will see in section \ref{section::ESP} that the concept of GS is closely related the \emph{Echo State Property} (ESP), which was coined by \citep{Jaeger2001}, and is discussed extensively in the reservoir computing world. The name evokes the image of past observations influencing the present moment with a fading intensity - just as the intensity of an echo fades with time.

\section{Echo State Property}
\label{section::ESP}

The existence of the GS $f$ is closely related to the Echo State Property (ESP) first introduced in the seminal paper by \cite{Jaeger2001}. There are several different formulations of the ESP, including in \cite{Jaeger2001}, \cite{YILDIZ20121}, and \cite{GRIGORYEVA2018495}. We will proceed with one inspired by \cite{YILDIZ20121} for now. To be as general as possible we introduce an arbitrary reservoir map $F : \mathbb{R}^N \times \mathbb{R}^d \to \mathbb{R}^N$ which reduces to an ESN as a special case when 
\begin{align*}
    F(x,z) = \sigma(Ax + Cz + b).
\end{align*}
\begin{defn}
    (Global Echo State Property) We say that a reservoir map $F: \mathbb{R}^N \times \mathbb{R}^d \to \mathbb{R}^N$ has the \emph{global ESP} if, given any sequence of inputs $(z_k \in \mathbb{R}^d)_{k \in \mathbb{N}}$, and initial reservoir states $x_0,y_0 \in \mathbb{R}^N$ the sequences 
    \begin{align*}
        x_{k+1} = F(x_k,z_k) \qquad \text{and} \qquad y_{k+1} = F(y_k,z_k)
    \end{align*}
    satisfy $\lVert x_{k+1} - y_{k+1} \rVert \to 0$ as $k \to \infty$.
\end{defn}

We can ensure that a reservoir map $F$ has the global ESP if $F$ is globally state contracting.

\begin{defn}
    A reservoir map $F:\mathbb{R}^N \times \mathbb{R}^d \to \mathbb{R}^N$ is called globally state contracting if there exists a $c \in (0,1)$ such that for any $x,y \in \mathbb{R}^N$ and $z \in \mathbb{R}^d$ it follows
    \begin{align*}
        \lVert F(x,z) - F(y,z) \rVert \leq c\lVert x - y \rVert.
    \end{align*}
\end{defn}

\begin{theorem}
    If a reservoir map $F:\mathbb{R}^N \times \mathbb{R}^d \to \mathbb{R}^N$ is globally state contracting, then $F$ has the global ESP.
\end{theorem}

\begin{proof}
    Consider that
    \begin{align*}
        \lVert x_{k+1} - y_{k+1} \rVert = \lVert F(x_k,z_k) - F(y_k,z_k) \rVert
        \leq c\lVert x_k - y_k \rVert
    \end{align*}
    so we can show by recursion that
    \begin{align*}
        \lVert x_{k+1} - y_{k+1} \rVert \leq c^k \lVert x_0 - y_0 \rVert
    \end{align*}
    so $\lVert x_{k+1} - y_{k+1} \rVert \to 0$ as $k \to \infty$.
\end{proof}

With this, we are ready to introduce our first major result. Suppose the inputs are a sequence of scalar observations from a dynamical system i.e. $z_k = \omega\circ\phi^{k}(m_0)$ where 
\begin{itemize}
    \item $m_0 \in M$ is an initial point on a smooth compact $q$-manifold $M$.
    \item $\phi \in \text{Diff}^1(M)$ is a smooth diffeomorphism on $M$
    \item $\omega \in C^1(M,\mathbb{R})$ is a smooth observation function.
\end{itemize}
Suppose further that $F$ is globally state contracting. Then, there is a continuously differentiable GS $f \in C^1(M,\mathbb{R}^N)$ synchronising the drive system system (represented by $\phi$) to dynamics in the reservoir space $\mathbb{R}^N$.

\section{A theorem for global GS}

Before we state the theorem, we will need some formalities. First we insist that the manifold $M$ is compact, connected, second countable and Hausdorff so that it can be endowed with a Riemannian metric $g$ \citep{Riemannian_Geometry}. Then for any $f \in C^1(M,\mathbb{R}^N)$ we define
\begin{align*}
    \lVert Df \rVert_{\infty} = \sup_{m \in M} \lVert Df(m) \rVert_2, \qquad \lVert Df(m) \rVert_2 = \sup_{v \in T_m M / \{0\}} \frac{\lVert Df(m) v \rVert_2}{\sqrt{g(m)(v,v)}}
\end{align*}
with $T_m M$ the tangent space of $M$ at $m$ and $Df(m)$ the differential of $f$ at $m$. Analogously for $\phi \in \text{Diff}^1(M)$ we define
\begin{align*}
    \lVert T\phi \rVert_{\infty} = \sup_{m \in M} \lVert T_m\phi \rVert, \qquad \lVert T_m\phi \rVert = \sup_{v \in T_m M / \{0\} } \sqrt{\frac{g(\phi(m))(T_m\phi v ,T_m\phi v)}{g(m)(v,v)}}
\end{align*}
where $T_m \phi$ is the tangent map for $\phi$ at $m$, and $(v,w)$ is the inner product on $T_m$. 

\begin{theorem}
    (\cite*{chaos_on_compacta}) Let $M$ be a compact, smooth, $q$-manifold (i.e a $q$-dimensional manifold) and $\phi \in \text{Diff}^1(M)$. Let $\omega \in C^1(M,\mathbb{R}^d)$ be the observation function on $M$. Suppose that the reservoir map $F \in C^2(\mathbb{R}^N \times \mathbb{R}^d, \mathbb{R}^N)$ is globally state contracting. Furthermore, suppose that
    \begin{align*}
        c_{x} := \sup_{(x,z) \in \mathbb{R}^N \times \mathbb{R}^d} \lVert D_x F(x,z) \rVert_2
    \end{align*}
    (where $D_x$ is the partial derivative with respect to $x$) satisfies
    \begin{align*}
        c_x < 1, \qquad c_x \lVert T\phi^{-1} \rVert_{\infty} < 1.
    \end{align*}
    Then, there exists a unique solution $f \in C^1(M,\mathbb{R}^N)$ of the equation
        \begin{align*}
            f = F( f \circ \phi^{-1} , \omega )
        \end{align*}
        such that for all initial $x_0 \in \mathbb{R}^N$ and $m_0 \in M$ the sequence
        \begin{align*}
            x_{k+1} = F(x_k, \omega \circ \phi^k(m_0))
        \end{align*}
        converges to $f \circ \phi^k(m_0)$ as $k \to \infty$.
        \label{special_SSM_thm}
\end{theorem}

The proof of Theorem \ref{special_SSM_thm} is omitted because the theorem is a special case of Theorem \ref{SSM_thm} (obtained by setting $V = \mathbb{R}^N$) which we will state and prove in section \ref{section::local_GS}. The more general result proved in section \ref{section::local_GS} concerns the case where the reservoir map $F$ admits multiple generalised synchronisations. To build some intuition for this, we will proceed with an example. Our example is closely related to the work by \cite{GREBOGI1984261} which studies \emph{strange attractors that are not chaotic}.

\section{Motivating example}

Let $F : \mathbb{R} \times \mathbb{R} \to \mathbb{R}$ be defined by
\begin{align*}
    F(x,z) = \tanh(2x + z).
\end{align*}
Under the constant driving input $z_k = 0$ for all $k$, we have the autonomous system $F(x,0) = \tanh(2x)$, which has 3 fixed points at $-x^*,0,x^*$ where $x^* \approx 0.958$ illustrated in Figure \ref{graph} below. $-x^*$ and $x^*$ are stable nodes, while $0$ is an unstable node.
\begin{figure}[b!]
\centering
\begin{tikzpicture}
\begin{axis}[
    width = 10cm,
    height = 8cm,
    axis x line=center,
    axis y line=none,
    ticks = none,
    xmin=-1.4,
    xmax=1.4,
    ymin=-2.2,
    ymax=1.2]
\addplot[color=black, samples = 100]{tanh(2*x)};
\addplot[color=black, samples = 100]{x};
\addplot[color=black, samples = 100, dotted]{1};
\addplot[color=black, samples = 100, dotted]{-1};
\filldraw [color = black] (axis cs:0,0) circle (2pt);
\filldraw [color = black] (axis cs:-0.958,-0.958) circle (2pt);
\filldraw [color = black] (axis cs:0.958,0.958) circle (2pt);
\filldraw [color = black] (axis cs:0,-2) circle (2pt) node[anchor=north] (un) {Untable Node};
\filldraw [color = black] (axis cs:-0.958,-2) circle (2pt) node[anchor=north] (snn) {Stable Node};
\filldraw [color = black] (axis cs:0.958,-2) circle (2pt) node[anchor=north] (snp) {Stable Node};
\draw [->] (axis cs: -0.08,-2) -- (axis cs: -0.44,-2);
\draw [->] (axis cs: -0.48,-2) -- (axis cs: -0.90,-2);
\draw [->] (axis cs: -1.4,-2) -- (axis cs: -1.01,-2);
\draw [->] (axis cs: 0.08,-2) -- (axis cs: 0.44,-2);
\draw [->] (axis cs: 0.48,-2) -- (axis cs: 0.90,-2);
\draw [->] (axis cs: 1.4,-2) -- (axis cs: 1.01,-2);
\end{axis}
\end{tikzpicture}
\caption{A graph of $y = \tanh(2x)$ intersecting $y = x$ at 3 points. These 3 points are fixed points of the map $F(x,0) = \tanh(2x)$.}
\label{graph}
\end{figure}
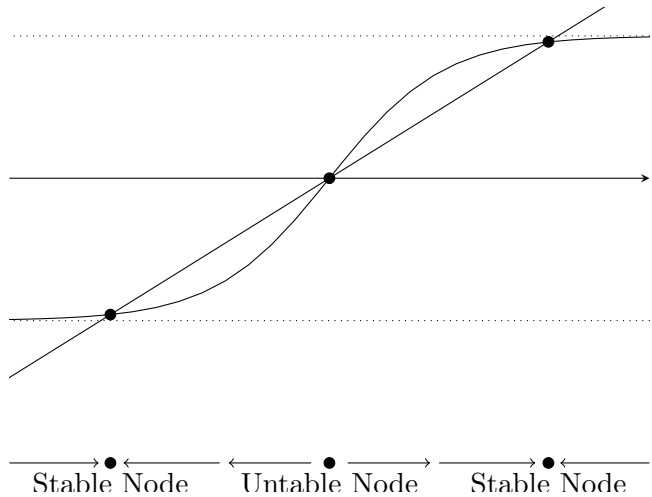
Multiple equilibria prove that $F$ does not have the global ESP for this input. Now, let us consider the behavior of $F$ under a more interesting driving input. Let $\phi \in \text{Diff}^1(S^1)$ be a diffeomorphism on the circle defined by $\phi(m) = m + \epsilon$ for some fixed $\epsilon > 0$. Let $\omega \in C^1(S^1,\mathbb{R})$ be defined to be $\omega(m) = \frac{1}{2}\sin(m)$. Now suppose the driving input $u_k = \omega \circ \phi^k(m_0)$ for some initial $m_0$. Now, any GS $f:S^1 \to \mathbb{R}$ satisfies, \ $\forall m \in S^1$,
\begin{align}
    f(m) &= F(f \circ \phi^{-1}(m) , \omega(m) ) \nonumber  \\
    &= \tanh\bigg(2f(m + \epsilon) + \frac{1}{2}\sin(m)\bigg). \label{SSM_relation}
\end{align}
Suppose we choose $\epsilon = 2\pi /100$, and take an initial $x_0 \in \mathbb{R}$, then compute $x_{k+1} = F(x_k,z_k)$, and plot the pairs $(\phi^k(m_0), x_k) \in \mathbb{R}^2$. It appears from numerical experiments, that for any initial $x_0$ in the vicinity of the fixed point $-x^*$, after a short transient period, we arrive at the same graph $G_-$ of the pairs $(\phi^k(m_0), x_k) \in \mathbb{R}^2$. Furthermore, the graph plausibly satisfies the relation \eqref{SSM_relation} of being a GS. Intriguingly, if we take any $x_0$ in the vicinity of $x^*$ this time, then after a short transient period, we arrive at a different graph $G_+$ of the pairs $(\phi^k(m_0), x_k) \in \mathbb{R}^2$.  This suggests the existence of (at least) two distinct (locally) stable GSs with respect to the input $z_k$. The graphs $G_-$ and $G_+$ are shown in Figure \ref{SSM_graphs}.

\begin{figure}
    \centering
    \begin{subfigure}[b]{\textwidth}
        \includegraphics[width=\textwidth]{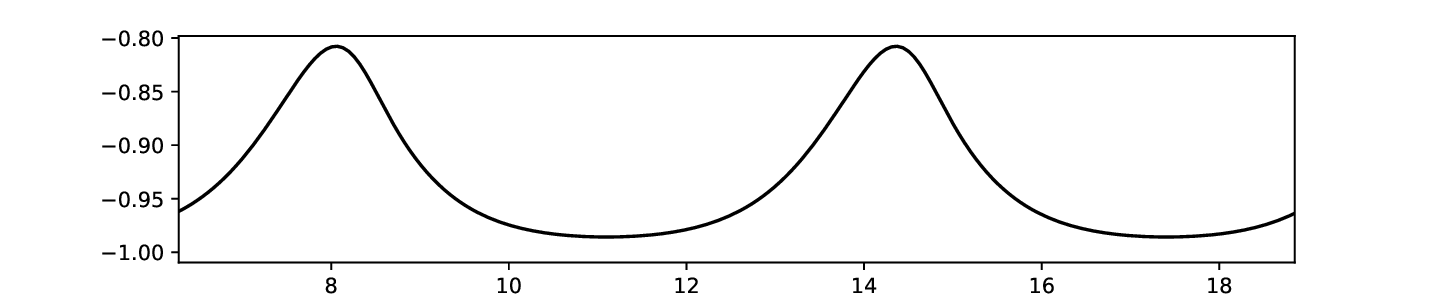}
        \caption{$G_-$}
    \end{subfigure}
    \begin{subfigure}[b]{\textwidth}
        \includegraphics[width=\textwidth]{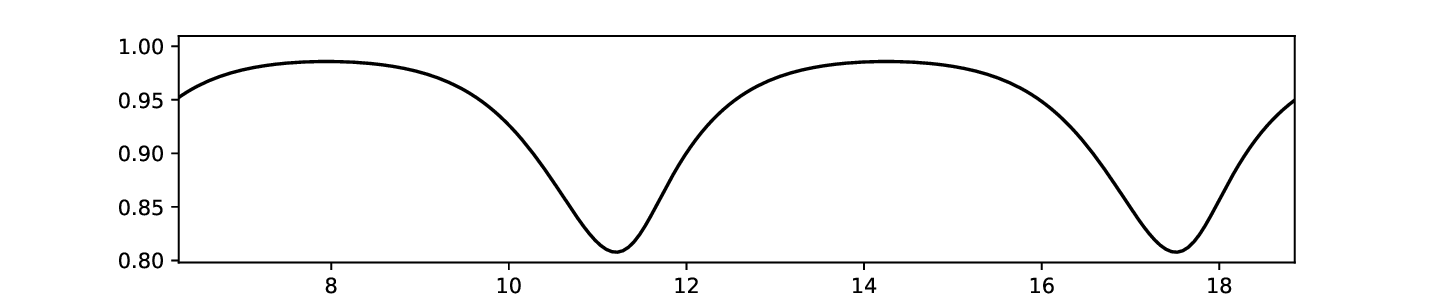}
        \caption{$G_+$}
    \end{subfigure}
    \caption{The graphs $G_-$ and $G_+$ of the points $(\phi^k(m_0),x_k)$.}\label{SSM_graphs}
\end{figure}

\section{Theorems for local GS}
\label{section::local_GS}

This numerical experiment suggests that there exist reservoir maps $F$ that do not have global ESP, and admit multiple GSs. To start to make sense of this, we will introduce the concept of local ESP with admits the global ESP as a special case.

\begin{defn}
    (Local ESP) Let $V \subset \mathbb{R}^N$ be closed under the Euclidean topology. We say that a reservoir map $F: \mathbb{R}^N \times \mathbb{R}^d \to \mathbb{R}^N$ has the $(V,W)$-local ESP if, given any sequence of inputs $(z_k \in W \subset \mathbb{R}^d)_{k \in \mathbb{N}}$, and initial reservoir states states $x_0,y_0 \in V$ the sequences 
    \begin{align*}
        x_{k+1} = F(x_k,z_k) \qquad \text{and} \qquad y_{k+1} = F(y_k,z_k)
    \end{align*}
    are contained in $V$ and $\lVert x_{k+1} - y_{k+1} \rVert \to 0$ as $k \to \infty$.
\end{defn}

We observe first of all that it is possible for $F$ to have $(V_i,W)$-local ESP on disjoint sets $V_i$ with $i = 1, ... , n$.
We can also ensure that a reservoir map $F$ has the $(V,W)$-local ESP by demanding that $F$ has the $(V,W)$-local state contracting property. 

\begin{defn}
    (Local state contracting) A reservoir map $F:\mathbb{R}^N \times \mathbb{R}^d \to \mathbb{R}^N$ is called $(V,W)$-locally state contracting if there exists a $c \in (0,1)$ such that for any $x,y \in V$ and $z \in W \subset \mathbb{R}^d$ it follows that
    \begin{itemize}
        \item $\lVert F(x,z) - F(y,z) \rVert \leq c\lVert x - y \rVert$ 
        \item $F(x,z) \in V$. 
    \end{itemize}
\end{defn}

\begin{theorem}
    If a reservoir map $F:\mathbb{R}^N \times \mathbb{R}^d \to \mathbb{R}^N$ is $(V,W)$-locally state contracting, then $F$ has the $(V,W)$-local ESP.
\end{theorem}

\begin{proof}
    First of all, since $F(x,z) \in V$, it follows that the sequence 
    \begin{align*}
        x_{k+1} = F(x_k,z_k)
    \end{align*}
    is contained by $V$. Next, for any $x_0,y_0 \in V$ we have
    \begin{align*}
        \lVert x_{k+1} - y_{k+1} \rVert = \lVert F(x_k,z_k) - F(y_k,z_k) \rVert
        \leq c\lVert x_k - y_k \rVert
    \end{align*}
    so we can show by recursion that
    \begin{align*}
        \lVert x_{k+1} - y_{k+1} \rVert \leq c^k \lVert x_0 - y_0 \rVert
    \end{align*}
    so $\lVert x_{k+1} - y_{k+1} \rVert \to 0$ as $k \to \infty$.        
\end{proof}

Now we are ready to state two theorems for the existence of local GS. The first result establishes the existence of a continuous GS $f$ on a topological space, while the second establishes a differentiable GS $f$ on a smooth manifold. A less general version of these results (with a less rigorous proof) appears in \cite{embedding_and_approximation_theorems} (Theorem 2.2.2) and a fully rigorous more general version appears in \cite*{chaos_on_compacta} (Theorem III).

\begin{theorem}
    (\cite*{chaos_on_compacta}) Let $M$ be a compact topological space and $\phi \in \text{Hom}(M)$. Let $\omega \in C^0(M,\mathbb{R}^d)$ be the observation function on $M$. Let $V \subset \mathbb{R}^N$ be closed and suppose that the reservoir map $F \in C^0(V \times \omega(M), \mathbb{R}^N)$ is $(V,\omega(M))$-locally state contracting with contraction coefficient $c$.
    Then, there exists a unique solution $f \in C^0(M,V)$ of the equation
        \begin{align*}
            f = F( f \circ \phi^{-1} , \omega )
        \end{align*}
        such that for all initial $x_0 \in V$ and $m_0 \in M$ the sequence
        \begin{align*}
            x_{k+1} = F(x_k, \omega \circ \phi^k(m_0))
        \end{align*}
        converges to $f \circ \phi^k(m_0)$ as $k \to \infty$.
        \label{SSM_weak_thm}
\end{theorem}

\begin{proof}
    Define the map $\Psi : C^0(M,V) \to C^0(M,V)$ by
    \begin{align*}
        \Psi(h) = F(h\circ \phi^{-1}, \omega).
    \end{align*}
    Notice that
    \begin{align*}
        \lVert \Psi(h) - \Psi(g) \rVert_{\infty} &= \lVert F(h\circ \phi^{-1}, \omega) - F(g\circ \phi^{-1}, \omega) \rVert_{\infty} \\
        &\leq c\lVert h\circ \phi^{-1} - g\circ \phi^{-1} \rVert_{\infty} = c \lVert h - g \rVert_{\infty}
    \end{align*}
by the definition of $c$. So $\Psi$ is a contraction mapping on the Banach space $C^0(M,V)$ and therefore (by the \cite{Banach1922} Fixed Point Theorem) admits a unique fixed point $f \in C^0(M,V)$ such that
    \begin{align*}
        \Psi(f) = F(f\circ\phi^{-1},\omega) = f.
    \end{align*}
    Furthermore, $F$ is $(V,\omega(M))$-locally contracting and therefore has the $(V,\omega(M))$-local ESP, so for all initial $x_0 \in V$ and $m_0 \in M$ and $y_0 = f\circ\phi^{-1}(m_0)$ the sequences
    \begin{align*}
        x_{k+1} = F(x_k, \omega \circ \phi^k(m_0)) \qquad \text{and} \qquad y_{k+1} = F(y_k, \omega \circ \phi^k(m_0))
    \end{align*}
    satisfy $\lVert x_{k+1} - y_{k+1} \rVert \to 0$ as $k \to \infty$. Since $y_{k+1} = f\circ\phi^k(m_0)$ it follows that $x_{k+1} \to f\circ\phi^k(m_0)$ and this completes the proof.
\end{proof}

\begin{theorem}
    (\cite*{chaos_on_compacta}) Let $M$ be a compact, smooth, $q$-manifold and $\phi \in \text{Diff}^1(M)$. Let $\omega \in C^1(M,\mathbb{R}^d)$ be the observation function on $M$. Let $V \subset \mathbb{R}^N$ be closed and convex. Suppose that the reservoir map $F \in C^2(V \times \omega(M), \mathbb{R}^N)$ is $(V,\omega(M))$-locally state contracting. Furthermore, suppose that
    \begin{align*}
        c_{x} := \sup_{(x,z) \in V \times \omega(M)} \lVert D_x F(x,z) \rVert_2
    \end{align*}
    (where $D_x$ is the partial derivative with respect to $x$) satisfies
    \begin{align*}
        c_x < 1, \qquad and \qquad c_x \lVert T\phi^{-1} \rVert_{\infty} < 1.
    \end{align*}
    Then, there exists a unique solution $f \in C^1(M,V)$ of the equation
        \begin{align*}
            f = F( f \circ \phi^{-1} , \omega )
        \end{align*}
        such that for all initial $x_0 \in V$ and $m_0 \in M$ the sequence
        \begin{align*}
            x_{k+1} = F(x_k, \omega \circ \phi^k(m_0))
        \end{align*}
        converges to $f \circ \phi^k(m_0)$ as $k \to \infty$.
        \label{SSM_thm}
\end{theorem}

The proof of the second result is rather more technical, but relies on a similar argument. We will need to first introduce for each $\delta > 0$ a norm $\lVert \cdot \rVert_{C^1(\delta)}$ on $C^1(M,V)$ defined by
\begin{align*}
    \lVert f \rVert_{C^1(\delta)} := \lVert f \rVert_{\infty} + \delta \lVert Df \rVert_{\infty}.
\end{align*}
It is shown in Chapter 2 of \cite{https://doi.org/10.1002/zamm.19690490315} that each of these endows $C^1(M,V)$ with a Banach space structure. Furthermore, \cite{https://doi.org/10.1002/zamm.19690490315} show that all of these norms induce the same topology, and that these topologies coincide with the $C^1$ topology described in \cite{Invariant_Manifolds}. With this, we will now prove 2 technical lemmas which are used to prove Theorem \ref{SSM_thm}.

\begin{lemma}
    Let $r > 0$ and $W \subset \mathbb{R}^{N}$ be closed in the Euclidean topology. Then, for any $\delta > 0$, the set
    \begin{align*}
        \Omega(r,W) = \{ f \in C^1(M,W) \ | \ \lVert Df \rVert_{\infty} \leq r \}
    \end{align*}
    is closed in the topology induced by the $\lVert \cdot \rVert_{C^1(\delta)}$ norm.
\end{lemma}

\begin{proof}
    $\Omega(r,W) = \Omega(r) \cap C^1(M,W)$ where
    \begin{align*}
        \Omega(r) = \{ f \in C^1(M,\mathbb{R}^N) \ | \ \lVert Df \rVert_{\infty} \leq r \}
    \end{align*}
    so it suffices to the show that both $\Omega(r)$ and $C^1(M,W)$ are closed. We will show $\Omega(r)$ is closed by proving that the complement $C^1(M,\mathbb{R}^N) \backslash \Omega(r)$ is open. We will show that for an arbitrary point in the complement $f \in C^1(M,\mathbb{R}^N) \backslash \Omega(r)$ there exists an $\epsilon > 0$ such that the ball of radius $\epsilon$ centred at $f$ denoted $B_{C^1(\delta)}(f,\epsilon)$ is a subset of the complement $C^1(M,\mathbb{R}^N) \backslash \Omega(r)$. We fix
    \begin{align*}
        \epsilon = \delta(\lVert Df \rVert_{\infty} - r)
    \end{align*}
    and then for any $g \in C^1(M,\mathbb{R}^N) / \Omega(r)$
    \begin{align*}
        \lVert Df \rVert_{\infty} &= \lVert Df + Dg - Dg \rVert_{\infty} \\
        &\leq \lVert Df - Dg \rVert_{\infty} + \lVert Dg \rVert_{\infty} \\
        &\leq \lVert Df - Dg \rVert_{\infty} + \frac{1}{\delta} \lVert f - g \rVert_{\infty} + \lVert Dg \rVert_{\infty} \\
        &= \frac{1}{\delta} \lVert f - g \rVert_{C^1(\delta)} + \lVert Dg \rVert_{\infty} \\
        &< \frac{\epsilon}{\delta} + \lVert Dg \rVert_{\infty} \\
        &= \lVert Df \rVert_{\infty} - r + \lVert Dg \rVert_{\infty}
    \end{align*}
    using the definition of the $C^1(\delta)$ norm directly. So $\lVert Dg \rVert_{\infty} > r$ hence $B_{C^1(\delta)}(f,\epsilon) \subset C^1(M,\mathbb{R}^N) \backslash \Omega(r)$. Next, we will show that $C^1(M,W)$ is closed by showing that for an arbitrary sequence $(f_n \in C^1(M,W))_{n \in \mathbb{N}}$ that converges to $f^* \in C^1(M,\mathbb{R}^N)$ it follows that $f^* \in C^1(M,W)$. If $(f_n \in C^1(M,W))_{n \in \mathbb{N}}$ converges to $f^* \in C^1(M,\mathbb{R}^N)$ then for all $\epsilon > 0$ there exists $N(\epsilon) \in \mathbb{N}$ such that, for any $m \in M$
    \begin{align*}
        \rVert f(m) - f^*(m) \lVert &\leq \rVert f - f \lVert_{\infty} \\
        &\leq \rVert f - f^* \lVert_{\infty} + \delta \lVert Df - Df^* \rVert_{\infty} \\
        &= \lVert f - f^* \rVert_{C^1(\delta)} < \epsilon
    \end{align*}
    so $f(m) \to f^*(m)$ and $W$ is closed so $f^*(m) \in W$. This holds for arbitrary $m \in M$ so it follows that $f^* \in C^1(M,W)$. 
\end{proof}

In what follows we will use the notation $D_x F(x,z)$, $D_z F(x,z)$ to denote the partial derivatives of the reservoir map $F$ with respect to $x$ and $z$, respectively.

\begin{lemma}
    Let $F \in C^1(V \times \omega(M), \mathbb{R}^N)$ be a $(V,\omega(M))$-locally contracting reservoir map. Let
    \begin{align*}
        c_x := \sup_{(x,z) \in V \times \omega(M)} \lVert D_x F(x,z) \rVert_2, \qquad c_z := \sup_{(x,z) \in V \times \omega(M)} \lVert D_z F(x,z) \rVert_2,
    \end{align*}
    and and choose $r$ such that
    \begin{align}
        r > \frac{c_z\lVert D\omega \rVert_{\infty}}{1-c_x \lVert T\phi^{-1}\rVert_{\infty}} . \label{r_condition}
    \end{align}
    Let $\Psi : \Omega(r,\omega(M)) \to C^1(M,\mathbb{R}^N)$ be defined by
    \begin{align*}
        \Psi(h) = F(h\circ\phi^{-1},\omega).
    \end{align*}
    Then $\Psi(\Omega(r,\omega(M))) \subset \Omega(r,\omega(M))$ and hence the map $\Psi : \Omega(r,\omega(M)) \to \Omega(r,\omega(M))$ is well defined 
\end{lemma}

\begin{proof}
    Let $h \in \Omega(r,\omega(M))$. It suffices to show that $\Psi(h)(m) \in \omega(M)$ for all $m \in M$ and that $\lVert D(\Psi(h)) \rVert_{\infty} \leq r$. We easily verify that $\Psi(h)(m) \in \omega(M)$ for all $m \in M$ because $F$ is $(V,\omega(M))$-locally contracting. Now, by direct computation of the derivative
    \begin{align*}
        &\lVert D(\Psi(h)) \rVert_{\infty} = \lVert DF(h\circ\phi^{-1},\omega) \rVert_{\infty} \\
        &= \sup_{m \in M}\lVert D_x F(h\circ\phi^{-1}(m),\omega(m)) (Dh)\circ\phi^{-1}(m)T_m\phi^{-1} + D_z F(h\circ\phi^{-1}(m),\omega(m)) D\omega(m) \rVert_2 \\
        &\leq c_x \lVert Dh \rVert_{\infty} \lVert T\phi^{-1} \rVert_{\infty} + c_z \lVert D\omega \rVert_{\infty} \\
        &= c_x r \lVert T\phi^{-1} \rVert_{\infty} + c_z \lVert D\omega \rVert_{\infty} < r. \text{ (by \eqref{r_condition})}
    \end{align*}
\end{proof}

We are almost ready to prove the result. For what follows we let $D_{xx}F(x,z)$ and $D_{xz}F(x,z)$ denote the second partial derivatives of $F$ and let
\begin{align*}
    c_{xx} := \sup_{(x,z) \in V \times \omega(M)} \lVert D_{xx} F(x,z) \rVert_2, \qquad \textrm{and} \qquad c_{xz} := \sup_{(x,z) \in V \times \omega(M)} \lVert D_{xz} F(x,z) \rVert_2.
\end{align*}
Then it follows from the mean value theorem, which applies because $V$ is convex, that for any $x,y \in V$ and $z \in \omega(M)$ that
\begin{align*}
    \lVert D_x(F(x,z) - F(y,z)) \rVert \leq c_{xx}\lVert x - y \rVert, \qquad \textrm{and} \qquad \lVert D_z(F(x,z) - F(y,z)) \rVert \leq c_{xz}\lVert x - y \rVert.
\end{align*}

\begin{proof}
    (Proof of Theorem \ref{SSM_thm}) We have shown that $\Psi : \Omega(r,\omega(M)) \to \Omega(r,\omega(M))$ is a well defined map on a closed subset of the Banach space $C^1(M,\mathbb{R}^N)$. We will now show that for sufficiently small $\delta_0 > 0$ the map $\Psi$ is a contraction in the $\lVert \cdot \rVert_{C^1(\delta_0)}$ norm, and therefore admits a unique fixed point $f \in \Omega(r,\omega(M)) \subset C^1(M,\mathbb{R}^N)$ by the \cite{Banach1922} Fixed Point Theorem. The fixed point $f$ satisfies
    \begin{align*}
        \Psi(f) = F(f\circ\phi^{-1},\omega) = f.
    \end{align*}
    Furthermore, $F$ is $(V,\omega(M))$-locally contracting and therefore has the $(V,\omega(M))$-local ESP, so for all initial $x_0 \in V$ and $m_0 \in M$ and $y_0 = f\circ\phi^{-1}(m_0)$ the sequences
    \begin{align*}
        x_{k+1} = F(x_k, \omega \circ \phi^k(m_0)) \qquad y_{k+1} = F(y_k, \omega \circ \phi^k(m_0))
    \end{align*}
    satisfy $x_{k+1} - y_{k+1} \to 0$ as $k \to \infty$. Now $y_{k+1} = f\circ\phi^k(m_0)$ so $x_{k+1} \to f\circ\phi^k(m_0)$. All that remains to show is that $\Psi$ is a contraction in $\lVert \cdot \rVert_{C^1(\delta_0)}$ for sufficiently small $\delta_0$. We do this by considering the two parts of the $C^1(\delta_0)$ norm separately. First, we observe that
    \begin{align*}
                \lVert \Psi(h) - \Psi(g) \rVert_{\infty} &= \lVert F(h\circ \phi^{-1}, \omega) - F(g\circ \phi^{-1}, \omega) \rVert_{\infty} \\
        &\leq c\lVert h\circ \phi^{-1} - g\circ \phi^{-1}) \rVert_{\infty} \\
        &= c \lVert h - g \rVert_{\infty}.
    \end{align*}
Second, we compute directly that
    \begin{align*}
        &\qquad \lVert D\Psi(g) -D\Psi(h) \rVert_{\infty} \\
        &= \lVert DF(g\circ\phi^{-1},\omega) -DF(h\circ\phi^{-1},\omega) \rVert_{\infty} \\
        &= \sup_{m \in M} \lVert D_xF(g\circ\phi^{-1}(m),\omega(m))(Dg)\circ\phi^{-1}(m)T_m\phi^{-1} + D_zF(g\circ\phi^{-1}(m),\omega(m))D\omega(m) \\
        &\qquad - D_xF(h\circ\phi^{-1}(m),\omega(m))(Dh)\circ\phi{-1}(m)T_{m}\phi^{-1} - D_zF(h\circ\phi^{-1}(m),\omega(m))D\omega(m) \rVert_{2} \\
        &\leq \sup_{m \in M}\lVert D_xF(g\circ\phi^{-1}(m),\omega(m))(Dg)\circ\phi^{-1}(m)T_{m}\phi^{-1} \\ 
        &\qquad - D_xF(h\circ\phi^{-1}(m),\omega(m))(Dh)\circ\phi^{-1}(m)T_{m}\phi^{-1} \rVert_{2} \\
        &\qquad + \sup_{m \in M}\lVert D_zF(g\circ\phi^{-1}(m),\omega(m))D\omega(m) - D_zF(h\circ\phi^{-1}(m),\omega(m))D\omega(m) \rVert_{2} \\
        &\leq \sup_{m \in M}\lVert D_xF(g\circ\phi^{-1}(m),\omega(m))(Dg)\circ\phi^{-1}(m)T_{m}\phi^{-1} \\
        &\qquad - D_xF(h\circ\phi^{-1}(m),\omega(m))(Dh)\circ\phi^{-1}(m)T_{m}\phi^{-1} \rVert_{2} \\ 
        &\qquad + \lVert D_z(F(g\circ\phi^{-1},\omega) - F(h\circ\phi^{-1},\omega)) \rVert_{\infty} \lVert D\omega \rVert_{\infty} \\
        &\leq \sup_{m \in M}\lVert D_xF(g\circ\phi^{-1}(m),\omega(m))(Dg)\circ\phi^{-1}(m)T_m\phi^{-1} \\ 
        &\qquad - D_xF(h\circ\phi^{-1}(m),\omega(m))(Dh)\circ\phi^{-1}(m)T_m\phi^{-1} \rVert_{2} \\ 
        &\qquad + c_{xz} \lVert g - h \rVert_{\infty} \lVert D\omega \rVert_{\infty} \\
        &= \sup_{m \in M}\lVert D_xF(g\circ\phi^{-1}(m),\omega(m))(Dg)\circ\phi^{-1}(m)T_{m}\phi^{-1} \\ 
        &\qquad-  D_xF(h\circ\phi^{-1}(m),\omega(m))(Dh)\circ\phi^{-1}(m)T_{m}\phi^{-1} \\
        &\qquad + D_xF(g\circ\phi^{-1}(m),\omega(m))(Dh)\circ\phi^{-1}(m)T_{m}\phi^{-1} \\ &\qquad - D_xF(g\circ\phi^{-1}(m),\omega(m))(Dh)\circ\phi^{-1}(m)T_m\phi^{-1} \rVert_{2} \\
        &\qquad + c_{xz} \lVert g - h \rVert_{\infty} \lVert D\omega \rVert_{\infty} \\
        &\leq \sup_{m \in M}\lVert D_xF(g\circ\phi^{-1}(m),\omega(m))(Dg)\circ\phi^{-1}(m)T_{m}\phi^{-1} \\
        &\qquad - D_xF(g\circ\phi^{-1}(m),\omega(m))(Dh)\circ\phi^{-1}(m)T_m\phi^{-1} \rVert_{2} \\
        &\qquad \sup_{m \in M}\lVert D_xF(h\circ\phi^{-1}(m),\omega(m))(Dh)\circ\phi^{-1}(m)T_{m}\phi^{-1} \\
        &\qquad - D_xF(g\circ\phi^{-1}(m),\omega(m))(Dh)\circ\phi^{-1}(m)T_{m}\phi^{-1}
        \rVert_{2} \\
        &\qquad + c_{xz} \lVert g - h \rVert_{\infty} \lVert D\omega \rVert_{\infty}
\end{align*}
\begin{align*}
        &\leq \lVert D_xF(g\circ\phi^{-1},\omega) \rVert_{\infty}\lVert Dg - Dh \rVert_{\infty} \lVert T\phi^{-1} \rVert_{\infty} \\
        &\qquad + \lVert D_x(F(h \circ \phi^{-1},\omega) -F(g \circ \phi^{-1},\omega)) \rVert_{\infty} \lVert Dh \rVert_{\infty} \lVert T\phi^{-1} \rVert_{\infty} \\
        &\qquad + c_{xz} \lVert g - h \rVert_{\infty} \lVert D\omega \rVert_{\infty} \\
        &\leq c_x \lVert Dg - Dh \rVert_{\infty} \lVert T\phi^{-1} \rVert_{\infty} \\
        &\qquad + c_{xx} \lVert h - g \rVert_{\infty} \lVert Dh \rVert_{\infty} \lVert T\phi^{-1} \rVert_{\infty} \\
        &\qquad + c_{xz} \lVert g - h \rVert_{\infty} \lVert D\omega \rVert_{\infty} \\
        &\leq c_x \lVert Dg - Dh \rVert_{\infty} \lVert T\phi^{-1} \rVert_{\infty} \\
        &\qquad + c_{xx} \lVert h - g \rVert_{\infty} r \lVert T\phi^{-1} \rVert_{\infty} \\
        &\qquad + c_{xz} \lVert g - h \rVert_{\infty} \lVert D\omega \rVert_{\infty} \\
        &= (r c_{xx} \lVert T\phi^{-1} \rVert_{\infty} + c_{xz} \lVert D\omega \rVert_{\infty})\lVert g - h \rVert_{\infty} + c_x \lVert T\phi^{-1} \rVert_{\infty} \lVert Dg - Dh \rVert_{\infty}. \\
    \end{align*}
    Now we choose
    \begin{align*}
        \delta_0 < \frac{1-c}{rc_{xx}\lVert T\phi^{-1} \rVert_{\infty} + c_{xz}\lVert D\omega \rVert_{\infty}}
    \end{align*}
    so that, combining the two estimates above we obtain
    \begin{align*}
        \lVert \Psi(g) -\Psi(h) \rVert_{C^1(\delta_0)} &= \lVert \Psi(g) -\Psi(h) \rVert_{\infty} + \delta_0 \lVert D\Psi(g) -D\Psi(h) \rVert_{\infty} \\
        &= (c+\delta_0(rc_{xx}\lVert T\phi^{-1} \rVert_{\infty} + c_{xz}\lVert D\omega \rVert_{\infty}))\lVert g - h \rVert_{\infty} \\ 
        &\qquad + \delta_0 c_x \lVert T\phi^{-1} \rVert_{\infty} \lVert Dg - Dh \rVert_{\infty} \\
        &\leq (c+\delta_0(rc_{xx}\lVert T\phi^{-1} \rVert_{\infty} + c_{xz}\lVert D\omega \rVert_{\infty})) \lVert g - h \rVert_{C^1(\delta_0)}.
    \end{align*}
    By our choice of $\delta_0$ the constant $(c+\delta_0(rc_{xx}\lVert T\phi^{-1} \rVert_{\infty} + c_{xz}\lVert D\omega \rVert_{\infty})) \in (0,1)$, so we have shown that $\Psi$ is a contracting map and the proof is complete.
\end{proof}

\section{Curious examples}

To create some further intuition for Theorem \ref{SSM_thm}, we present four more examples of reservoir maps $F$ which are locally state contracting and admit local GSs. These examples are related to the work by \cite{GREBOGI1984261}. In the first example (which appears in \cite*{chaos_on_compacta}) we consider the Lorenz system \eqref{eqn::Lorenz} once again. This time we feed the $\xi$ observations into a reservoir map that admits 8 distinct GSs.

\subsubsection{1. Finitely many GSs}

We simulated a 4000 point (40 time units) trajectory of the Lorenz system originating from the initial point $m_0 = (0,1,1.05)$ with time-step $h = 0.01$. Figure \ref{Lorenz_system} shows this trajectory for times $t \in (20,40)$. 
\begin{figure}
\centering
    \includegraphics[width=0.9\textwidth]{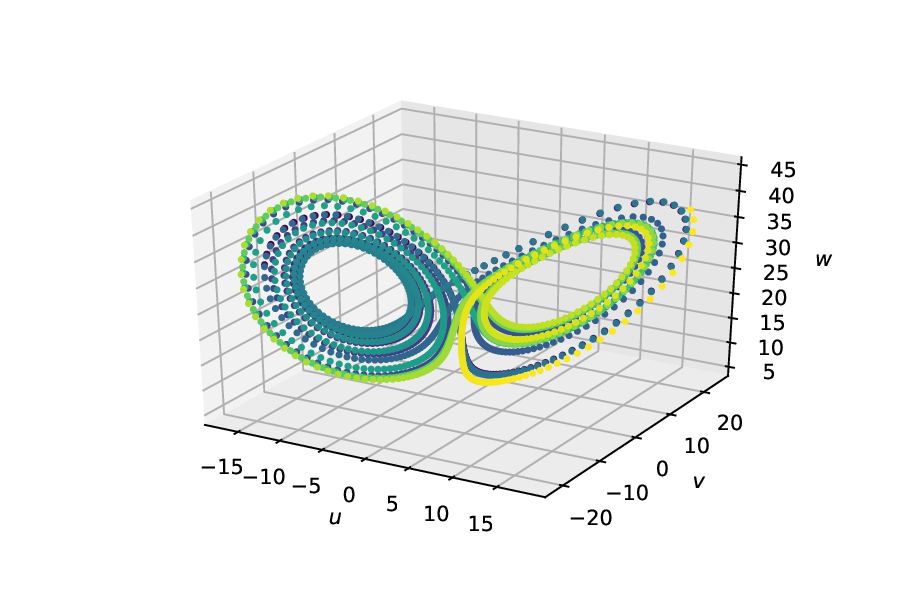}
\caption{\label{Lorenz_system}A trajectory of the Lorenz system. $(u,v,w)$ correspond to the variables $(\xi,\upsilon,\zeta)$ used in equations \eqref{eqn::Lorenz}.}
\end{figure}
If we observe only the $\xi$-component of this trajectory, then the observation function is $\omega(\xi,\upsilon,\zeta) = \xi$. The corresponding observed trajectory is illustrated in Figure \ref{x_trajectory}.

\begin{figure}
    \centering
    \includegraphics[width=0.9\textwidth]{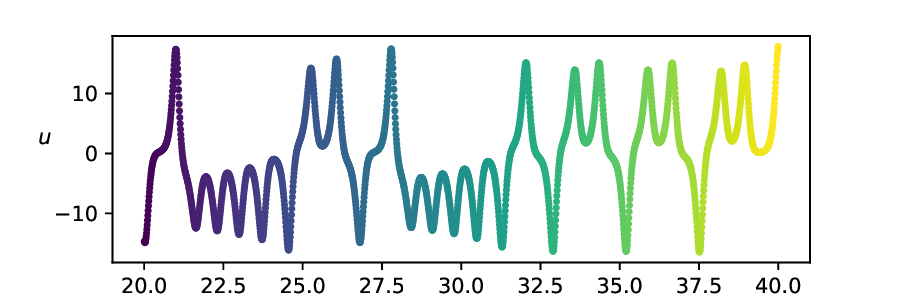}
\caption{\label{x_trajectory}$\xi$-component of a trajectory of the Lorenz system. The horizontal axis is time $t \in (20,40)$.}
\end{figure}

We now define the reservoir map  $F : \mathbb{R}^3 \times \mathbb{R} \to \mathbb{R}^3$ given by
\begin{align*}
    F(u,v,w;z)=(\mathrm{sgn}(u)|u|^\alpha,\mathrm{sgn}(v)|v|^\alpha,\mathrm{sgn}(w)|w|^\alpha) + \lambda(\sin(k z),\cos(k z), \sin^2(k z))
\end{align*}
with $\lambda, k \geq 0$ and $\alpha \in (0,1)$. If we set $\lambda = 0$ then the reservoir map is autonomous and has 8 stable fixed points at $(\pm 1, \pm 1, \pm 1)$ taking all sign combinations.
A cross section of the phase portrait of this autonomous system at the $w = 1$ plane is shown in Figure \ref{phase_portrait}. 

\begin{figure}
\includegraphics[width=0.9\textwidth]{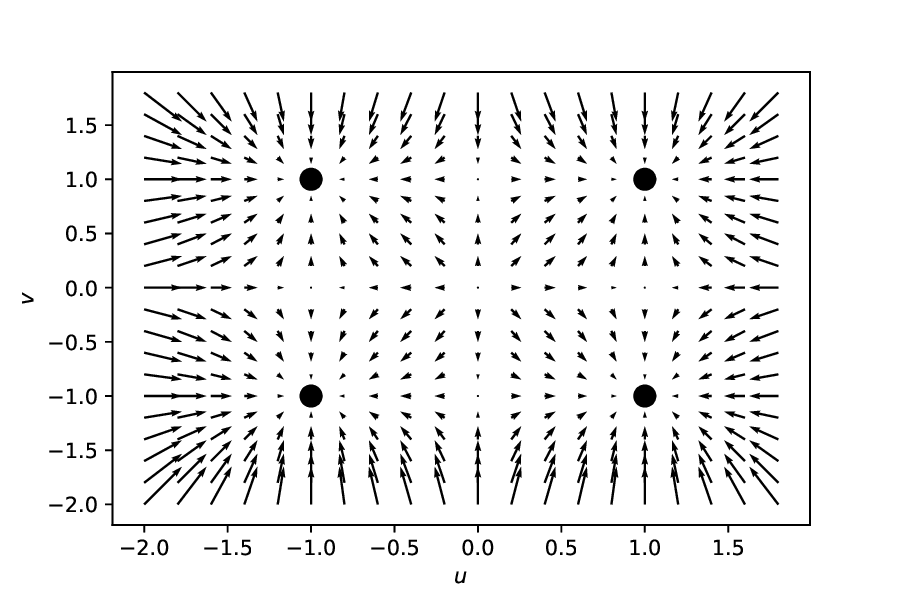}
\caption{\label{phase_portrait}
A phase portrait of the autonomous system $F(u,v,w;z) = (u^{\alpha}, v^{\alpha}, w^{\alpha})$ at the cross section $w = 1$. The four stable fixed points in this cross section are shown.}\end{figure}

If we now choose $\lambda = 0.009$, $k = 0.1$ and $\alpha = 0.9$, then we can construct a box $V_1 = [0.9,1.1] \times [0.9,1.1] \times [0.9,1.1]$ containing the fixed point $(1,1,1)$. We can create a similar box about each of the other seven fixed points and denote them $V_2, V_3, \ldots, V_8$. We have by construction for each $i = 1, \ldots , 8$ that
$
    F(V_i,\omega(\mathbb{R}^3)) \subset V_i
$
and that $F$ is state contracting on each box $V_i$. Thus, by Theorem \ref{SSM_thm} the reservoir map $F:\mathbb{R}^N \times \mathbb{R}^d \to \mathbb{R}^N $ is $(V_i,\omega(M))$-locally state contracting for each $i$, so there exist GSs $f_i \in C^1(M,V_i)$ for each $i$.

To observe the images of the GSs $f_i \in C^1(M,V_i)$ we computed the states
\begin{align}
    x_{k+1} = F(x_k,\omega(\phi^k(m_0))), \qquad y_{k+1} = F(y_k,\omega(\phi^k(m_0))), \label{eqn:ssequences}
\end{align}
from two different initial states $x_0 = (1,1,1) \in V_1$ and $y_0 = (-1,1,1) \in V_2$. For any other initial points $x_0' \in V_1$ and $y_0' \in V_2$ we expect that the sequences \eqref{eqn:ssequences} will converge to $f_1 \circ \phi^k(m_0)$ and $f_2 \circ \phi^k(m_0)$ respectively. Figure \ref{reservoir_lorenz} illustrates the states \eqref{eqn:ssequences} 
after a burn-in time  $t = (0,20)$ so that the state sequences very closely approximate $f_1 \circ \phi^k(m_0)$ and $f_2 \circ \phi^k(m_0)$ respectively.

\begin{figure}
\centering
    
  \begin{subfigure}[b]{0.95\textwidth}
        \includegraphics[width=\textwidth]{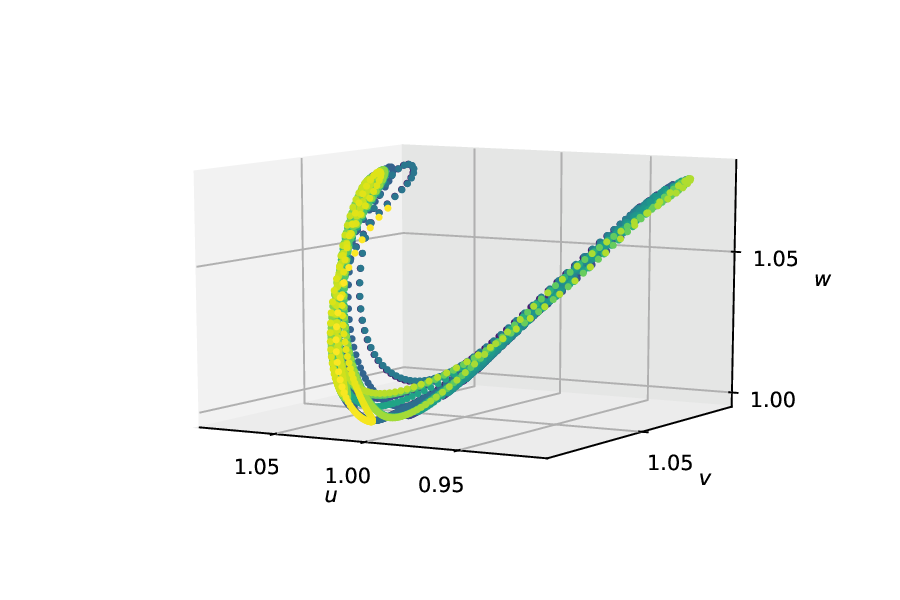}
        \caption{Image of the GS that contains the point $ (1,1,1)$ in its image.}
        \label{reservoir_lorenz_111}
  \end{subfigure}
  \begin{subfigure}[b]{0.95\textwidth}
        \includegraphics[width=\textwidth]{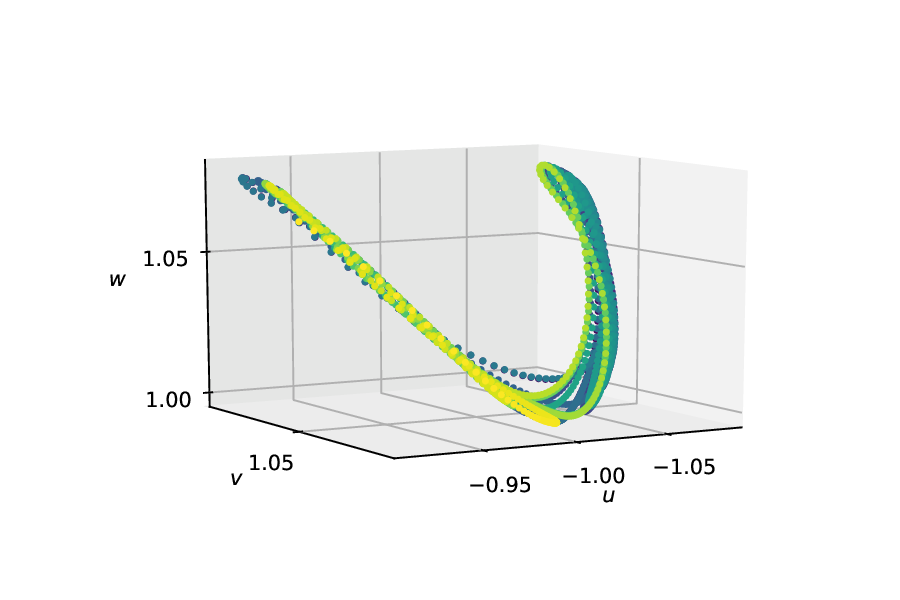}
        \caption{Image of the GS that contains the point $ (-1,1,1)$ in its image.}
        \label{reservoir_lorenz_-111}
  \end{subfigure}
\caption{\label{reservoir_lorenz}Image of the Lorenz solution under two different GSs.}
\end{figure}

\subsubsection{2. Uncountably many GSs}

Suppose $F : \mathbb{R}^2 \times \mathbb{R} \to \mathbb{R}^2$ is defined by
\begin{align*}
    F(\rho,\theta; z) = (\sqrt{\rho} + z , \theta + \delta)
\end{align*}
where $(\rho,\theta)$ is the polar coordinate representation of the state vector $x \in \mathbb{R}^2$, and $\delta \in \mathbb{R}$. Under the constant driving input $z_k = 0$ the state space system has a unstable node at the origin, and admits an attracting invariant set $\{ \rho = 1 \}$. A phase portrait of the dynamics is shown in Figure \ref{fig:stable_orbit}.

\begin{figure}
    \centering
    \begin{subfigure}[b]{0.48\textwidth}
        \includegraphics[width=\textwidth]{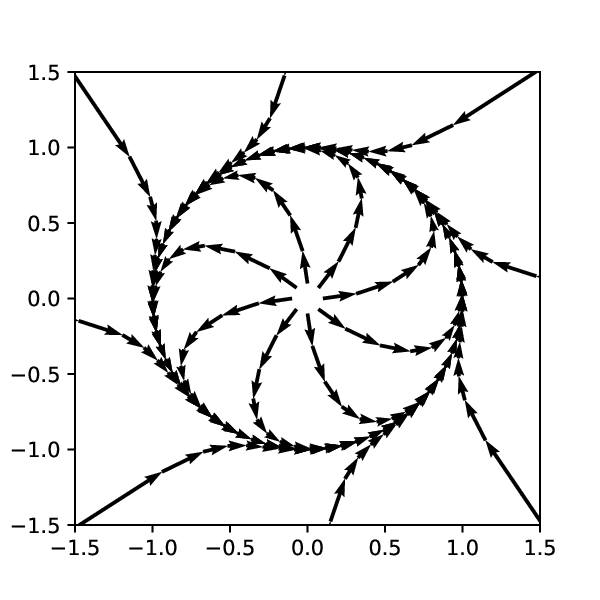}
        \caption{$F(\rho,\theta; z) = (\sqrt{\rho} + z , \theta + \delta)$.}
        \label{fig:stable_orbit}
    \end{subfigure}
    ~
    \begin{subfigure}[b]{0.48\textwidth}
        \includegraphics[width=\textwidth]{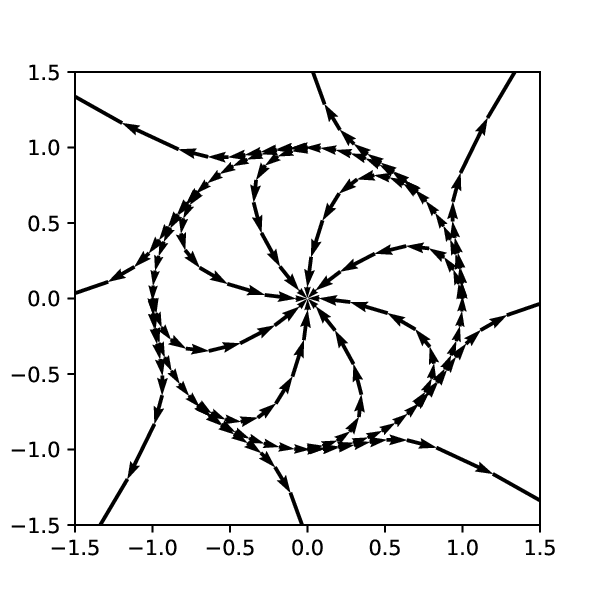}
        \caption{$F(\rho,\theta; z) = (\rho^2 + z , \theta + \delta)$.}
        \label{fig:unstable_orbit}
    \end{subfigure}
    \caption{Phase portraits of 2 reservoir maps $F : \mathbb{R}^2 \times \mathbb{R} \to \mathbb{R}^2$ under the constant driving input $z_k = 0$ for all $k$.}
    \label{fig:phase_portraits}
\end{figure}

Once again, suppose the driving input arises from $\phi \in \text{Diff}^1(S^1)$ a diffeomorphism on the circle defined by $\phi(m) = m + \epsilon$ for $\epsilon > 0$. This time, let the observation function $\omega \in C^1(S^1,\mathbb{R})$ be defined $\omega(m) = \frac{1}{5}\sin(m)$.
Suppose first of all that $\delta = 0$. For each $\theta_0 \in [-\pi,\pi)$ let
\begin{align*}
    V(\theta_0) = \{ (\rho,\theta) \ | \ \rho \geq 1/2 \ , \theta = \theta_0  \}.
\end{align*}
Note that $V(\theta_0) \subset \mathbb{R}^2$ is closed and that for all $x,y \in V(\theta_0)$ and $z \in \omega(S^1)$
    \begin{enumerate}
        \item $\exists \, c \in (0,1)$ such that $\lVert F(x,z) - F(y,z) \rVert \leq c \lVert x - y \rVert $
        \item $F(x,z) \in V(\theta_0)$
    \end{enumerate}
so we have the $(V(\theta_0),\omega(S^1))$-local state contraction property.
Hence, by Theorem \ref{SSM_thm}, there exists a unique solution $f_{\theta_0} \in C^1(S^1,V(\theta_0))$ of the equation
        \begin{align*}
            f_{\theta_0} = F( f_{\theta_0} \circ \phi^{-1} , \omega )
        \end{align*}
        such that for all initial $x_0 \in V_{\theta_0}$ and $m_0 \in S^1$ the sequence
        \begin{align*}
            x_{k+1} = F(x_k, \omega \circ \phi^k(m_0))
        \end{align*}
        converges to $f_{\theta_0} \circ \phi^k(m_0)$ as $k \to \infty$. 
This holds for each $\theta_0 \in [\pi,\pi)$, so there is an uncountable set of GSs parametrised by $\theta_0$.
\label{sec:uncountably_many_SSMs}

\subsubsection{3. A bounded trajectory that does not admit a GS}

Let $n \in \mathbb{N}$, $\epsilon = 2\pi / n$ and $\delta = 1$. Notice $\phi^n(m_0) = m_0$. If $x_0$ represents the polar angle of an initial reservoir state $x_0$, and proceeding states are defined in the usual manner
\begin{align*}
    x_{k+1} = F(x_k,\omega\circ\phi^k(m_0))
\end{align*}
then the polar angle of the $n$th reservoir state $x_n$ is $\theta_n = \theta_0 + n \delta$.

We will show that no GS $f$ exists for initial points
on $\{ \rho = 1 \}$. Suppose for contradiction that such a GS does exist. Then $f(m_0) = f \circ \phi^{n}(m_0)$ for all $m_0$, so the polar angles $\theta_0$ and $\theta_n$ of $f(m_0)$, and $f \circ \phi^{n}(m_0)$ must be equal. This leaves us with a contradiction, because $n \delta$ is not a multiple of $2\pi$. We can see however that the trajectory $\{x_k\}_{k \in \mathbb{N}}$ remains bounded.

\subsubsection{4. Some initial conditions converge to the image of a GS}

Suppose $F : \mathbb{R}^2 \times \mathbb{R} \to \mathbb{R}^2$ is defined by
\begin{align*}
    F(\rho,\theta; z) = (\rho^2 + z , \theta + \delta)
\end{align*}
where $\delta > 0$. A phase portrait of this system under the constant driving input $z_k = 0$ for all $k$ is shown in Figure \ref{fig:unstable_orbit}. Once more, suppose the driving input arises from $\phi \in \text{Diff}^1(S^1)$ a diffeomorphism on the circle defined by $\phi(m) = m + \epsilon$ for $\epsilon > 0$, and let the observation function $\omega \in C^1(S^1,\mathbb{R})$ be defined by $\omega(m) = \frac{1}{5}\sin(m)$. Let 
\begin{align*}
    V = \{ (\rho,\theta) \ | \ \rho \leq 1/2 \}.
\end{align*}
Note that $V \subset \mathbb{R}^2$ is closed and that for all $x,y \in V$ and $z \in \omega(S^1)$
    \begin{enumerate}
        \item $\exists \, c \in (0,1)$ such that $\lVert F(x ; z) - F(y ; z) \rVert \leq c \lVert x - y \rVert$, 
        \item $F(x;z) \in V$,
    \end{enumerate}
so we have the $(V,\omega(M))$-local state contracting property.
Hence, by Theorem \ref{SSM_thm}, there exists a unique solution $f \in C^1(S^1,V)$ of the equation
        \begin{align*}
            f = F( f \circ \phi^{-1} , \omega )
        \end{align*}
        such that for all initial $x_0 \in V$ and $m_0 \in S^1$ the sequence
        \begin{align*}
            x_{k+1} = F(x_k, \omega \circ \phi^k(m_0))
        \end{align*}
        converges to $f \circ \phi^k(m_0)$ as $k \to \infty$. 
This $f$ is a GS. Now suppose we choose an initial reservoir state $x_0$ such that
\begin{align*}
    \rho_0 := \lVert x_0 \rVert > 2.
\end{align*}
Then, for any initial point $m_0$, the reservoir states $x_k$ that satisfy
\begin{align*}
    x_{k+1} = F(x_k, \omega \circ \phi^k(m_0))
\end{align*}
grow without bound. Thus any GS with image intersecting
\begin{align*}
    V = \{ (\rho, \theta) \ | \ \rho > 2 \}
\end{align*}
must be unbounded. But this is impossible, as the GS is continuous and defined on the compact manifold $S^1$. Thus we must conclude that any GS that exists cannot have an image that intersects $V$. In other words, a sequence of states originating in $V$ will not converge to the image of a GS. 

\section{Relationship between GS and the Echo Index}

We were inspired to generalise the results surrounding global GSs to local GSs after reading a paper by \cite{CENI2020132609}. The authors present \emph{uniformly attracting entire solution} (UAES) to reservoir maps and introduce the \emph{echo index}, which is closely related to local GSs. We will present just enough theory to link the GS to the Echo Index and recommend the interested reader consult \cite{CENI2020132609} for further details.

\cite{CENI2020132609} consider reservoir maps 
\begin{align}
    F : X \times U \to X \label{eqn::echo_index_system}
\end{align}
that satisfy 3 assumptions
\begin{enumerate}
    \item $F$ is continuously differentiable i.e. $F \in C^1(X \times U, X)$.
    \item For all $z \in \mathbb{R}^d$ the map $F(\cdot,z) : X \to X$ is a local diffeomorphism onto its image.
    \item $U \subset \mathbb{R}^d$ is compact and $X \subset \mathbb{R}^N$ is usually the compact closure of a $N$-dimensional Cartesian product of real intervals.
\end{enumerate}
The authors introduce the shift map $T : U^\mathbb{Z} \to U^\mathbb{Z}$ defined $T(z)_k = z_{k+1}$ for $k \in \mathbb{Z}$, the canonical projection $\pi :  U^\mathbb{Z} \to U$ defined by $\pi(z) = z_0$, and the cocycle mapping which we define below.
\begin{defn}
    (Cocycle mapping) The reservoir map \eqref{eqn::echo_index_system} can be described using a cocycle mapping $\Phi : \mathbb{Z}_0^+ \times U^{\mathbb{Z}} \times X \to X$ defined by
    \begin{align*}
        \Phi(0,z,x_0) &:= x_0  &\forall z \in U^{\mathbb{Z}}, \ x_0 \in X, \\
        \Phi(n,z,x_0) &:= F(\Phi(n-1,z,x_0),\pi\circ T^n(z)) & \forall z \in U^{\mathbb{Z}}, \ x_0 \in X, n \geq 1.
    \end{align*}
\end{defn}
\begin{defn}
    (Entire Solutions) An entire solution for the reservoir map \eqref{eqn::echo_index_system} with input $z \in \mathbb{Z}$ is a bi-infinite sequence of states $x \in X^{\mathbb{Z}}$ that satisfies \eqref{eqn::echo_index_system}. In other words 
    \begin{align}
        \Phi(s,T^m(z),x_m) = x_{m+s}
    \end{align}
    for all $m \in \mathbb{Z}$ and $s \in \mathbb{Z}_0^+$
\end{defn}

\begin{defn}
    (Positively Invariant Nonautonomous Sets) A family of nonempty compact subsets $B$ is called a positively invariant nonautonomous set for input $z \in U^{\mathbb{Z}}$ (or simply a $z$-positively invariant set) if
    \begin{align*}
        \Phi(s,T^m(z),B_m) \subset B_{s+m}
    \end{align*}
    for all $m \in \mathbb{Z}$ and $s \in \mathbb{Z}_0^+$.
\end{defn}

\begin{defn}
    (Uniformly Attracting Entire Solution) Consider a fixed input sequence $z \in U^{\mathbb{Z}}$, and entire solution $x$, and a $z$-positively invariant nonautonomous set $B$ composed of compact sets.
    \begin{enumerate}
        \item If 
        \begin{align*}
            \lim_{k \to \infty} \bigg( \sup_{j \in \mathbb{Z}} h(\Phi(k,T^j(z),B_j),x_{j+k}) \bigg) = 0
        \end{align*}
        (where $h$ is the Hausdorff semi-distance) then we say $B$ is uniformly attracted to $x$.
        \item We say that $x$ is a UAES if there is a neighbourhood $B$ of $x$ that is uniformly attracted to $x$.
    \end{enumerate}
\end{defn}

\begin{defn}
\label{defn::decomposition}
    (Decompositions) We say that the reservoir map \eqref{eqn::echo_index_system} with input $z \in U^{\mathbb{Z}}$ admits a decomposition into $n \geq 1$ UAESs if there are $n$ UAESs $x^1, \ldots , x^n$ such that for all $\eta > 0$ and $i = 1, \ldots, n$ there are neighbourhoods $B^\eta_i$ uniformly attracted by $x$ and
    \begin{align*}
        \mu\bigg( \mathbb{R}^N - \bigcup_{i=1}^n (B_{i}^{\eta})_k \bigg) < \eta \qquad \forall k \in \mathbb{Z}
    \end{align*}
    where $\mu$ is the Lebesgue measure. We say this is a \emph{proper decomposition} if in addition
    \begin{align*}
        \inf_{k \in \mathbb{Z}}\lVert x^i_k - x^j_k \rVert > 0 \qquad \forall i \neq j.
    \end{align*}
\end{defn}

\begin{defn}
    (Echo Index) We say that the reservoir map \eqref{eqn::echo_index_system} driven by input $z \in U^{\mathbb{Z}}$ has echo index $n \geq 1$ and write 
    \begin{align*}
        \mathcal{I}(z) = n
    \end{align*}
    if it admits a proper decomposition into $n$ UAESs. In this case we say \eqref{eqn::echo_index_system} has the $n$-ESP for input $z$.
\end{defn}

We can see that the work in \cite{CENI2020132609} does not reference an underlying dynamical system $\phi : M \to M$ and therefore holds for a larger class of input sequences than those generated by observations of $\phi$. GSs on the other hand are defined explicitly in terms of the dynamical system $\phi$ and observation function $\omega$, so we will need to introduce these in order to connect GSs to the Echo Index.

\begin{theorem}
    Let $M$ be a compact topological space and $\phi \in \text{Hom}(M)$. Let $\omega \in C^0(M,\mathbb{R}^d)$ be the observation function on $M$. Suppose that \eqref{eqn::echo_index_system} admits a GS $f : M \to V \subset X \subset \mathbb{R}^N$. Then for each $m_0 \in M$, the sequence $(x_k)_{k \in \mathbb{Z}} = (f \circ \phi^k(m_0))_{k \in \mathbb{Z}}$ is a UAES attracted by the constant sequence $B_k = V$ for all $k$.
    \label{SSM_imples_UAES}
\end{theorem}

\begin{proof}
    For any $m_0 \in M$, let $(z_k)_{k \in \mathbb{Z}} = (\omega \circ \phi^k(m))_{k \in \mathbb{Z}}$. The constant sequence $B_k = V$ for all $k$ is uniformly attracted to $(x_k)_{k \in \mathbb{Z}} = (f_i^{\eta} \circ \phi^k(m))_{k \in \mathbb{Z}}$ by Theorem \ref{SSM_weak_thm}.
\end{proof}

\begin{theorem}
Let $M$ be a compact topological space and $\phi \in \text{Hom}(M)$. Let $\omega \in C^0(M,\mathbb{R}^d)$ be the observation function on $M$. Suppose there are exactly $n$ GSs $f_i : M \to f_i(M) \subset \mathbb{R}^N$, such that, for all $\eta > 0$, there exist $n$ sets $V^{\eta}_i \supset f_i(M)$ such that $f_i^\eta : M \to V^{\eta}_i$ is a GS and
    \begin{align}
        \mu\bigg(\mathbb{R}^N - \bigcup_{i=1}^n V^{\eta}_i\bigg) < \eta, \label{less_eta}
    \end{align}
    where $\mu$ is the Lebesgue measure.    
    Then for any $m_0 \in M$, the sequence $(z_k)_{k \in \mathbb{Z}} = (\omega \circ \phi^k(m_0))_{k \in \mathbb{Z}}$ has echo index $\mathcal{I}[z] = n$.
\end{theorem}

\begin{proof}
    Fix $m_0 \in M$. For each GS $f_i^{\eta} : M \to V^{\eta}_i$ the sequence $(x^i_k)_{k \in \mathbb{Z}} = (f^i_{\eta} \circ \phi^k(m_0))_{k \in \mathbb{Z}}$ is a UAES attracted by the constant sequence $(B_i^{\eta})_k = V_i^{\eta}$ for all $k$. This follows from Theorem \ref{SSM_imples_UAES}. Now, with respect to the input sequence $(z_k)_{k \in \mathbb{Z}} = (\omega \circ \phi^k(m_0))_{k \in \mathbb{Z}}$, the reservoir map clearly admits a \emph{decomposition} into $n$ UAESs, so all that remains is to show that this decomposition is \emph{proper}. The images $f_i(M)$ are disjoint compact manifolds and each sequence $(x^i_k)_{k \in \mathbb{Z}} = (f_i \circ \phi^k(m_0))_{k \in \mathbb{Z}}$ lies in exactly one of these images. Consequently 
    \begin{align*}
        \inf_{k \in \mathbb{Z}}\lVert x^i_k - x^j_k \rVert > 0 \qquad \forall \, i \neq j.
    \end{align*}
 which is precisely the requirement for the decomposition to be \emph{proper} in the sense of definition \ref{defn::decomposition}.
\end{proof}

\section{Biological interpretation}

These results about GSs admit a biological interpretation. We use the ESN reservoir sequence
\begin{align}
    x_{k+1} = \sigma(Ax_k + Cz_k + b) \label{ESN::reservoir}
\end{align}
to represent the dynamics of a brain, composed of neurons, that are connected together by dendrites and axons. Neuron $n$ has dendrites, which receive information from neighbouring neurons, causing neuron $n$ to fire a sequence of action potentials, thus influencing the other neurons \citep{izhikevich2007dynamical}. We (over)simplify this complex interaction by representing the neurons with nodes, and the connections between them with arcs. We represent the connection strength between neurons $i$ and $j$ with the entry $A_{ij}$ of the reservoir matrix $A$. If neurons $i$ and $j$ are not directly connected then $A_{ij} = 0$. Furthermore, the connections need not be symmetric, so $A_{ij} \neq A_{ji}$ in general. 

The $i$th neuron has an \emph{excitation level} at time $k$ given by the $i$th component of the reservoir state $x_k$. The excitation level of neuron $i$ governs the frequency of action potentials fired by neuron $i$. It therefore follows from \eqref{ESN::reservoir} that the excitation of each neuron depends on the excitation of its neighbours. If $\sigma = \tanh$ or some other saturating/squashing/sigmoidal function, then the contribution of a neuron's neighbours obeys a law of diminishing returns.

We view the inputs $z_k$ as electrical signals induced in the brain from an external source. These signals could arise in the body or from external sensory organs. The inputs interact with the neurons via connections $C$, and influence the excitation levels of the neurons in the brain. The results we have presented in this chapter suggest that, as long as the source of the inputs $z_k$ evolves deterministically (and the ESN has ESP) then the state of the brain (represented by the state vectors $x_k$) evolves in synchrony with the source system. We illustrate this interplay between the input connections and reservoir connections in Figure \ref{fig::biological1} which illustrates an ESN with 3 neurons.

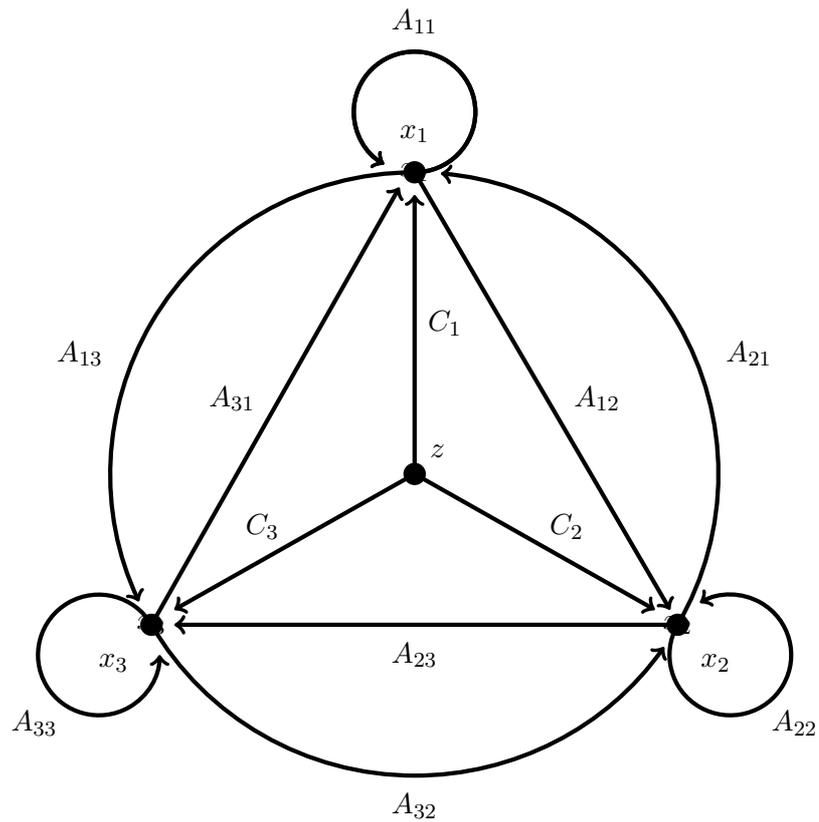
\begin{figure}
  \centering
   \begin{tikzpicture}
        \draw[ultra thick, ->] (0,4) -- (3.36,-1.8);
        \draw[ultra thick, ->] (3.46,-2) -- (-3.16,-2);
        \draw[ultra thick, ->] (-3.46,-2) -- (-0.2,3.8);
        \draw[ultra thick, ->] (0,4) arc (90:205:4);
        \draw[ultra thick, ->] (-3.46,-2) arc (210:325:4);
        \draw[ultra thick, ->] (3.46,-2) arc (330:445:4);
        \draw[ultra thick, ->] (0,4) arc (-90:240:0.8);
        \draw[ultra thick, ->] (0,4) arc (-90:240:0.8);
        \draw[ultra thick, ->] (3.46,-2) arc (150:480:0.8);
        \draw[ultra thick, ->] (-3.46,-2) arc (30:360:0.8);
        \filldraw (0,4) circle (4pt) node[label={[xshift=0, yshift=0.5]$x_1$}] {$x_1$};
        \filldraw (3.46,-2) circle (4pt) node[label={[xshift=0.5cm, yshift=-1cm]$x_2$}] {$x_2$};
        \filldraw (-3.46,-2) circle (4pt) node[label={[xshift=-0.5cm, yshift=-1cm]$x_3$}] {$x_3$};
        \filldraw (0,0) circle (4pt);
        \draw[ultra thick, ->] (0,0) -- (0,3.7);
        \draw[ultra thick, ->] (0,0) -- (3.16,-1.8);
        \draw[ultra thick, ->] (0,0) -- (-3.16,-1.8);
        \node(draw) at (0.3, 0.3) {$z$};
        \node(draw) at (2.4, 1) {$A_{12}$};
        \node(draw) at (-2.4, 1) {$A_{31}$};
        \node(draw) at (0, -2.4) {$A_{23}$};
        \node(draw) at (4.4, 1.6) {$A_{21}$};
        \node(draw) at (-4.4, 1.6) {$A_{13}$};
        \node(draw) at (0, -4.4) {$A_{32}$};
        \node(draw) at (0, 6) {$A_{11}$};
        \node(draw) at (5, -3.3) {$A_{22}$};
        \node(draw) at (-5, -3.3) {$A_{33}$};
        \node(draw) at (0.4, 2) {$C_1$};
        \node(draw) at (2, -0.7) {$C_2$};
        \node(draw) at (-2, -0.7) {$C_3$};
    \end{tikzpicture}
    \caption{An ESN driven by input $z$. The input communicates to three nodes with states $x_1,x_2,$ and $x_3$ via connections $C_1,C_2,$ and $C_3$. The nodes communicate with eachother via the three by three connectivity matrix $A$ with connection between node $i$ and $j$ given by $A_{ij}$. The direction of the arrows denotes the direction of communication.}
    \label{fig::biological1}
\end{figure}

Interestingly, if the ESN does not have ESP then the vectors $x_k$ may evolve erratically, and fail to synchronise with the source system. Some authors \citep{manjunath2017evolving} have suggested that this loss of the ESP describes an epileptic seizure, triggered by an input process $z_k$ that is outside a tolerated range (e.g rapidly flashing lights). We might imagine that a brain with $(V,W)$-local ESP undergoes an epileptic seizure if exposed to inputs outside of $W$. The existence of multiple GSs for different sets of inputs $W$ and sets of initial reservoir states $V$ suggests that the brain may adopt qualitatively different states under different conditions, that are each locally stable. This is all very speculative, but in future, it would be intriguing to explore these ideas in the context of careful neuroscience.

\chapter{Embeddings}
\label{chapter::embedding}

\section{Background}

Theorems \ref{SSM_weak_thm} and \ref{SSM_thm} establish conditions under which an evolution operator $\phi \in \text{Diff}^1(M)$, observation function $\omega : M \to \mathbb{R}^d$ and reservoir map $F : \mathbb{R}^N \times \mathbb{R}^d \to \mathbb{R}^N$ admit an associated GS $f_{(\phi,\omega,F)}$. Theorem \ref{SSM_thm} establishes conditions under which the GS $f_{(\phi,\omega,F)}$ is continuously differentiable. In this chapter we will establish additional conditions, which ensure that $f_{(\phi,\omega,F)}$ is a continuously differentiable \emph{embedding}. When $f_{(\phi,\omega,F)}$ is an embedding, the dynamics in the image of $f_{(\phi,\omega,F)}$ (the response system) replicate the dynamics of the drive system $(M,\phi)$ so that we have a chance of learning arbitrary target functions $g : M \to \mathbb{R}^s$ from observations alone. Out of the collection of all possible target maps, a target map in which we are particularly interested is the next step map $g := \omega \circ \phi$ which is well defined when $s = d$. For the rest of this chapter we assume for simplicity, and without loss of generality, that $d = 1$ so we have scalar observations. This chapter is based on \cite*{arXiv:2108.05024}, which builds on the earlier work in \cite{embedding_and_approximation_theorems}.

We will now define what it means for a map to be an embedding, and the closely related property of being an immersion.

\begin{defn}
    (Immersion) A map $f : M \to N$ between smooth manifolds $M,N$ is called an immersion at the point $m \in M$ if the tangent map at $m$ denoted $T_m f : T_m M \to T_{f(m)} N$ is injective. If $f$ is an immersion at all points $m \in M$ then $f$ is called an immersion.
\end{defn}

\begin{defn}
    (Embedding) A map $f : M \to N$ between smooth manifolds $M,N$ is called an embedding if $f$ is an injective immersion and a diffeomorphism onto its image.
\end{defn}

Notably, if $M$ is compact and $f : M \to N$ is an injective immersion then $f$ is an embedding. Hence we restrict ourselves to compact manifolds and consider the conditions on $\phi,\omega,F$ ensure that $f_{(\phi,\omega,F)}$ is an injective immersion. This problem is difficult in general, so we will start by recalling some fundamental theorems. The first is Whitney's embedding theorem.

\begin{theorem}
    \citep{whitney_embedding_thm}. Let $M$ be a compact $q$-manifold (i.e a manifold of dimension $q$) and $N \in \mathbb{N}$ such that $N > 2q$. Then a generic $f \in C^1(M,\mathbb{R}^N)$ is an embedding.
    \label{WWET}
\end{theorem}

To make sense of this theorem, we need to understand the term generic. Roughly speaking, a property is generic in some topological space if it holds for most elements, but excludes a small subset of exceptions. The concept of genericity is a topological one, which is analogous to (but not identical to) the concepts of \emph{almost everywhere} and \emph{prevalence} that appear in measure theory. 

\cite{Ghrist2014} provides some examples of genericity. In the appropriate topology:
\begin{itemize}
    \item a generic smooth curve $\gamma:\mathbb{R} \to \mathbb{R}^2$ self-intersects countably many times,
    \item a generic smooth curve $\gamma:\mathbb{R} \to \mathbb{R}^3$ does not self-intersect,
    \item a generic square matrix is invertible.
\end{itemize}

We state the formal definition of genericity below.

\begin{defn}
    (Generic property) Let $(X,\tau)$ be a topological space. A property that holds on a countable intersection of dense open subsets of $X$ is called a generic property of $(X,\tau)$.
    \label{defn::generic_property}
\end{defn}

Notably if $X$ is compact, then a property that holds on a dense open subset of $X$ is generic. We can now see that Whitney's embedding theorem implies that if we set $N > 2q$, where $M$ is a $q$ dimensional manifold, and we impose conditions on $\phi,\omega,F$ such that $f_{(\phi,\omega,F)} \in C^1(M,\mathbb{R}^N)$, then $f_{(\phi,\omega,F)}$ is either an embedding, or very close to an embedding. Heuristically, we expect that $f_{(\phi,\omega,F)}$ is very likely to be an embedding unless $\phi,\omega,F$ are deviously (or unfortunately) chosen such that the resulting GS $f_{(\phi,\omega,F)}$ is an exceptional map in $C^1(M,\mathbb{R}^N)$ that is not an embedding. This heuristic is quite compelling, but is not a proof that $f$ is an embedding for generic choices $\phi,\omega,F$. To progress toward this result, we will recall Takens' celebrated embedding theorem. The specific formulation and proof of the result is by \cite{Hukes_thm}, and it is this formulation which we will use in great detail in this Chapter.

\begin{theorem}
    \citep{TakensThm}, \citep{Hukes_thm}. Let $M$ be a compact manifold of dimension $q$. Suppose $\phi \in \text{Diff}^2(M)$ has the following two properties:
    \begin{itemize}
    \item[(1)] $\phi$ has only finitely many periodic points with periods
    less than or equal to $2q$.
    \item[(2)] For each period point $m \in M$ with period $n < 2q$ then the eigenvalues of the derivative $T_m \phi^n$ at $m$ are distinct. $\phi^n$ denotes the $n$-fold composition of $\phi$ with itself.
    \end{itemize}
    Then for a generic $C^2$ observation function $\omega \in C^2(M,\mathbb{R})$ the $(2q+1)$ delay observation map $\Phi_{(\phi,\omega)}:M \to \mathbb{R}^{2q+1}$ defined by
\begin{align*}
    \Phi_{(\phi,\omega)}(m) := ( \omega(m) , \omega \circ \phi (m) , \omega \circ \phi^2 (m) , \ldots , \omega \circ \phi^{2q}(m) )
\end{align*}
is an embedding.
\end{theorem}

Takens Theorem essentially states that: if we take equally spaced (in time) scalar observations from a dynamical system of dimension $q$, and the system $(M,\phi)$ satisfies conditions $(1)$ and $(2)$, then a vector formed from $2q + 1$ (or more) sequential observations fully captures the state of the dynamical system.

\section{Summary of novel results}

The connection between reservoir computing and Takens' Theorem becomes clear when we observe that the reservoir map $F : \mathbb{R}^N \times \mathbb{R} \to \mathbb{R}^N$ defined by
\begin{align*}
    F(x,z) = Ax + Cz
\end{align*}
admits as a GS the Takens delay map
\begin{align*}
    f_{(\phi,\omega,F)} = \Phi_{(\phi,\omega)}
\end{align*}
in the special case that $A \in \mathbb{M}_{N \times N}(\mathbb{R})$ is the lower shift matrix 
\begin{align*}
    A = 
    \begin{bmatrix}
    0 & 0 & 0 & \ldots & 0 \\
    1 & 0 & 0 & \ldots & 0 \\
    0 & 1 & 0 & \ldots & 0 \\
    \vdots & \vdots & \ddots & \vdots & \vdots \\
    0 & 0 & 0 & 1 & 0
    \end{bmatrix}
\end{align*}
and $C \in \mathbb{R}^N$ is defined to be $C = (1,0, \ldots, 0)$. This has been observed by \cite{4118282}, \cite{embedding_and_approximation_theorems}, \cite{doi:10.1063/5.0024890}, \cite*{arXiv:2108.05024} and others. It therefore seems natural to consider more general conditions on $A$ and $C$ for which the associated GS $f_{(\phi,\omega,F)}$ is an embedding for generic observation functions $\omega \in C^2(M,\mathbb{R})$. Such conditions are established in the next result.

\begin{theorem}
    \label{embedding_thm}
    Let $M$ be a compact manifold of dimension $q$. Suppose $\phi \in \text{Diff}^2(M)$ has the following two properties:
    \begin{itemize}
        \item[(1)] $\phi$ has only finitely many periodic points.
        \item[(2)] For each periodic point $m \in M$ with period $n$ the eigenvalues of the derivative $T_m \phi^n$ at $m$ are distinct.
    \end{itemize}
    Let $F : \mathbb{R}^N \times \mathbb{R}^d \to \mathbb{R}^N$ be defined by
    \begin{align*}
        F(x,z) = Ax + Cz
    \end{align*}
    for $A \in \mathbb{M}_{N \times N}(\mathbb{R})$ and $C \in \mathbb{R}^N$ satisfying the following properties:
    \begin{itemize}
        \item[(A)] $N > \max\{2q, \ell\}$ where $\ell \in \mathbb{N}$ is the lowest common multiple of the periods of all periodic points.
        \item[(B)] $\lambda_{\text{max}} \, \rho(A^{n_\text{min}}) < 1$ where $n_\text{min}$ is the minimal period over all periodic points and $\lambda_\text{max}$ is the maximal absolute value over all eigenvalues of all derivatives $T_m \phi^n$.
        \item[(C)] For each periodic point $m \in M$ with period $n$ the set of vectors
        \begin{align*}
            \bigg\{ (I-\lambda_j A^n)^{-1}(I - A)^{-1}(I-A)^n C \bigg\}_{j = 1, \ldots, q}
        \end{align*}
        where $\{ \lambda_j \}_{j = 1, \ldots q}$ are the eigenvalues of $T_m\phi^n$, are linearly independent.
        \item[(D)] The vectors $\{ A^j C \}_{j = 0, \ldots N-1}$
        are linearly independent.
    \end{itemize}
    Then there exists an associated GS $f_{(\phi,\omega,F)} \in C^2(M,\mathbb{R}^N)$ with explicit form
    \begin{align*}
        f_{(\phi,\omega,F)}(m) = \sum_{k=0}^{\infty}A^k C \omega \circ \phi^{-k}(m).       
    \end{align*}
    Furthermore, for generic $\omega \in C^2(M,\mathbb{R})$ the GS $f_{(\phi,\omega,F)}$ is an embedding.
\end{theorem}

\begin{remark}
    Condition $(1)$ appearing in Theorem \ref{embedding_thm} is stronger than condition $(1)$ appearing in Takens Theorem, so it may be possible to relax the former condition.
\end{remark}

\begin{remark}
    Condition $(D)$ appears to have interesting connections to controllability \citep{katsuhiko2010modern} and Krylov spaces \citep{nevanlinna1993convergence}.
\end{remark}

The proof proceeds in three major stages, and closely follows the proof of Takens Theorem presented by \cite{Hukes_thm}. We prove first of all that

\begin{enumerate}
    \item The associated GS $f_{(\phi,\omega,F)} \in C^2(M,\mathbb{R}^N)$ exists and has explicit form
    \begin{align*}
        f_{(\phi,\omega,F)}(m) = \sum_{k=0}^{\infty}A^k C \omega \circ \phi^{-k}(m).        
    \end{align*}
    \item Next, we show that the set of $\omega \in C^2(M,\mathbb{R})$ for which the associated GS $f_{(\phi,\omega,F)}$ is an embedding is an open subset of $C^2(M,\mathbb{R})$.
    \item Then all that remains is to show that a dense subset of the observation functions $C^2(M,\mathbb{R})$ yields an embedding. This process is split into 4 substages.
    \begin{itemize}
     \item We show that for arbitrary $\omega \in C^2(M,\mathbb{R})$ we can take an arbitrarily small perturbation $\omega' \in C^2(M,\mathbb{R})$ such that the perturbed GS $f_{(\phi,\omega',F)}$ is an immersion on the periodic points. It immediately follows that $f_{(\phi,\omega',F)}$ is also an immersion on some neighbourhood of the periodic points by the immersion theorem (Theorem \ref{theorem::immersion}).
    \item Immersions form an open subset of $C^1(M,\mathbb{R}^N)$, so we can create a second perturbation $\omega'' \in C^2(M,\mathbb{R})$ sufficently small that the open neighbourhood of the periodic points remains immersed, while immersing all remaining points. Then we have that $f_{(\phi,\omega'',F)}$ is a global immersion. All that remains is to establish injectivity.
    \item We take another perturbation $\omega'''$ sufficiently small that the immersion is preserved while ensuring that $f_{(\phi,\omega''',F)}$ is injective on the periodic points. By continuity of $f_{(\phi,\omega''',F)}$, it follows that there is a neighbourhood of the periodic points on which $f_{(\phi,\omega''',F)}$ is injective.
    \item We take one final perturbation $\omega''''$ sufficiently small that the immersion is preserved and $f_{(\phi,\omega'''',F)}$ is globally injective. This completes the proof.
    \end{itemize}
\end{enumerate}

The full proof is detailed in \cite*{arXiv:2108.05024} and section \ref{section::novel_results} of this thesis.
It is also shown in \cite*{arXiv:2108.05024} and in section \ref{section::system_isomorphism} of this thesis that conditions $(A)-(D)$ are invariant under \emph{system isomorphism}. Roughly speaking, we say that two reservoir systems with the echo state property are system isomorphic if when given the same input sequence the two reservoir systems (asymptotically) produce the same output sequence. This result is expressed in terms of \emph{filters} and \emph{functions} in Theorem \ref{linear_isomorphism} which appears in Chapter \ref{chapter::stochastic}.

In practice, the reservoir matrix $A$ and input vector $C$ are randomly generated, so we are interested in random variables $A,C$ that yield an embedding almost surely. We establish this in the following theorem. Before we state the result, we recall the following definition: that a real-matrix-valued random variable $X$ is non-singular if $\mathbb{P}[X \in \{A\}] = 0$ for all real matrices $A$.

\begin{theorem}
\label{random_thm}
    If $A \in \mathbb{M}_{N \times N}(\mathbb{R})$ and $C \in \mathbb{R}^N$ are randomly drawn from non-singular distributions then conditions $(C),(D)$ in Theorem \ref{embedding_thm} hold almost surely. 
\end{theorem}

We thank Friedrich Philipp for answering \href{https://math.stackexchange.com/questions/3774475/what-conditions-on-x-and-w-ensures-the-vectors-v-i-sum-infty-k-0}{\emph{this}} \footnote{Click the word \href{https://math.stackexchange.com/questions/3774475/what-conditions-on-x-and-w-ensures-the-vectors-v-i-sum-infty-k-0}{\emph{this}} to follow a link to the question} question we posted on \url{https://math.stackexchange.com}, which made the proof possible. The proof uses standard techniques in random matrix theory, and proceeds in three stages. 

\begin{enumerate}
    \item We show that a polynomial in $n$ variables, which is not identically zero, is non-zero almost everywhere.
    \item This implies that for linearly independent polynomials $p_1, p_2 , \ldots , p_n$ the random vectors
    \begin{align*}
        p_1(A)C,p_2(A)C, \ldots , p_n(A)C
    \end{align*}
    are linearly independent almost surely.
    \item Then the vectors described in $(C)$ and $(D)$ are manipulated until we can apply item 2. to yield the result.
\end{enumerate}

The proof of this result is detailed in \cite*{arXiv:2108.05024} and section \ref{section::novel_results} of this thesis.

\begin{remark}
    The results in this chapter apply to linear reservoir maps of the form $F(x,z) = Ax + Cz$; and we can easily extend the results to affine reservoir maps of the form $F(x,z) = Ax + Cz + b$. 
\end{remark}

A natural question is now: can prove similar results for more general reservoir maps like ESNs with nonlinear activations $F(x,z) = \sigma(Ax + Cz + b)$? We believe such a result exists by the weight of numerical evidence. For now, the answer is no. The techniques used is this present Chapter exploit the linearity of the reservoir map, as well as the explicit form of the GS
\begin{align*}
    f_{(\phi,\omega,F)}(m) = \sum^{\infty}_{k=0} A^k C \omega \circ \phi^{-k}(m)
\end{align*}
which we have not obtained for reservoir maps in general. To prove that a more general reservoir map admits an embedding GS, a different line of attack is required.

\section{An embedding allows for learning}

We have stated conditions under which a reservoir maps of the form 
\begin{align*}
    F(x,z) = Ax + Cz
\end{align*}
with randomly generated $A,C$ admits a GS $f_{(\omega,\phi,F)}$ which is an embedding. If an embedding is achieved, then we have sufficient conditions for learning. We recall that an embedding on a compact manifold is injective and immersive, and both of these properties have an interpretation in the context of learning. \cite{Verzelli} identified that when the GS $f_{(\omega,\phi,F)}$ is injective the reservoir map has learned the pointwise properties of the system, which allows pointwise approximation of target functions, such as the next step map for forecasting. We observe in this chapter that if the GS $f_{(\omega,\phi,F)}$ is immersive, then the reservoir map will learn deeper properties of the system, such as the eigenvalues of the linearisation of fixed points. 

Having achieved an embedding with a linear reservoir map, we are now in a position to approximate an arbitrary target function, for example, the next step map used to forecast the future time series of observations. We could, for example, train a feedforward neural to map the reservoir states $\{x_k\}_{k = 1, \ldots , \ell-1}$ onto the observations $\{z_{k+1}\}_{k = 1, \ldots, \ell}$. Training a feedforward neural network is a nonlinear optimisation problem, which is undesirable in the reservoir computing paradigm where training is usually linear regression. To avoid nonlinear training we can pass the reservoir states $\{x_k\}_{k = 1, \ldots , \ell-1}$ through a nonlinear function $g : \mathbb{R}^N \to \mathbb{R}^T$ (called a nonlinear kernel) and then find $W^* \in \mathbb{R}^T$ that solves the linear least squares problem
\begin{align*}
    W^* = \argmin_{W \in \mathbb{R}^T}\sum_{k=0}^{\ell-1} \lVert W^{\top} g(x_k) - z_{k+1} \rVert^2 + \lambda\lVert W \rVert^2.
\end{align*}
This is exactly the approach taken by \cite{doi:10.1063/5.0024890} and \cite{arXiv:2106.07688} in what the authors call \emph{next generation reservoir computing}. The authors demonstrate empirically that a polynomial kernel $g$ is suitable for forecasting the future trajectory of chaotic dynamical systems.

\section{Novel results in detail}
\label{section::novel_results}

\subsection{Existence of $C^2$ GS}

In this section we will prove Theorems \ref{embedding_thm} and \ref{random_thm}. The proofs closely follow those in \cite*{arXiv:2108.05024}. The proofs involve calculus on manifolds so we will start by reviewing some relevant theory of calculus on manifolds. Recall that if $f \in C^r(M,\mathbb{R}^N)$ then $Df : TM \to T(\mathbb{R}^N)$ is the differential of $f$, and $TM$ is the tangent bundle of $M$, and $T(\mathbb{R}^N) = \mathbb{R}^N \times \mathbb{R}^N$ is the tangent bundle of $\mathbb{R}^N$. Furthermore, if $\phi \in \text{Diff}^r(M)$ then $T\phi : TM \to TM$ is the tangent map of $\phi$. 

The differential $Df$ and tangent map $T\phi$ are generalisations of the first derivative.
Since the tangent bundle $TM$ of $M$ and the tangent bundle of $\mathbb{R}^N$ are themselves manifolds, we can define the second order differential
\begin{align*}
    D^2 f := D(Df) : T(TM) \to T(T(\mathbb{R}^N)),
\end{align*}
as the differential of the differential $Df$. Here $T(TM)$ is the tangent bundle of the tangent bundle of $M$, and $T(T(\mathbb{R}^N))$ is the tangent bundle of the tangent bundle of $\mathbb{R}^N$. Likewise, we can define the second order tangent map
\begin{align*}
    T^2 \phi := T(T\phi) : T(TM) \to T(TM)
\end{align*}
as the tangent map of the tangent map. We can proceed recursively to define for each $i \in \{ 1, \ldots , r \}$ the differential and tangent map of order $i$ denoted $D^i f$ and $T^i \phi$ respectively.

We are now ready to establish a lemma stating that the GS $f_{(\phi,\omega,F)}$ associated to the linear reservoir map exists, is $C^2$, and, for an open subset of observations functions, is an embedding.

\begin{lemma}
\label{generalized synch with spectral radius}
Let $\phi \in {\rm Diff}^1(M)$ be a dynamical system on the smooth manifold $M$ (not necessarily compact) and consider the observation map $\omega \in C ^1(M, \mathbb{R}) $. Let $F: \mathbb{R}^N \times \mathbb{R} \to \mathbb{R}^N  $ be a linear state map  given by $F(x, z):=A x+ C z $ with  $A \in \mathbb{M}_{N,N}(\mathbb{R})$, $C \in \mathbb{R}^N$, $N \in \mathbb{N}$.
\begin{description}
\item [(i)] If the spectral radius of $A$ satisfies that $\rho(A)<1$ and $\omega$ maps into a bounded subset of $\mathbb{R} $ then the GS $f_{(\phi, \omega,F)}: M \to \mathbb{R}^N$ associated to $F$ exists and is continuous.
\item [(ii)] Additionally, let $r \in \mathbb{N}$ and suppose that $\phi \in {\rm Diff}^r(M)$ and that there exist constants $k _1, \ldots, k _r \in \mathbb{N}  $ such that $\left\|A^{k _i}\right\| \left\|T ^i \phi ^{-k _i}\right\|_{\infty}<1   $, $\left\|T^i \phi ^{-1}\right\|_{\infty}< \infty$ for all $i \in  \left\{1, \ldots, r\right\}$.
Then for any $\omega \in C ^r(M, \mathbb{R}) $ such that  $\left\|D^i\omega\right\|_{\infty}< \infty $, for all $i \in  \left\{1, \ldots, r\right\}$, the map $f_{(\phi, \omega,F)} $ belongs to $ C^i(M, \mathbb{R}^N) $ and the higher order derivatives are given by:
\begin{equation}
\label{higher order derivatives of GS}
D^if_{(\phi, \omega,F)}(m)=\sum _{j=0}^{\infty}A ^j C D^i \left(\omega \circ \phi^{-j}\right) (m),  \mbox{ for all $i \in  \left\{1, \ldots, r\right\}$.}
\end{equation}
\item [(iii)] Suppose now that $M$ is compact. Under the hypotheses of points {\bf (i)} and {\bf (ii)} above, the map
\begin{equation}
\label{functor gs}
\begin{array}{cccc}
\Theta_{(\phi, F)}: &C ^r(M, \mathbb{R}) & \to & C ^r(M, \mathbb{R}^N)\\
	& \omega &\longmapsto & f_{(\phi, \omega,F)}
\end{array}
\end{equation}
is continuous. Moreover, the subsets $\Omega_{\text{imm}}$ and $\Omega_{\text{emb}}$ of $C ^r(M, \mathbb{R})$ for which the corresponding GSs are immersions and embeddings, respectively, are open.
\end{description}
\end{lemma} 

\begin{proof}
(i) This statement is obtained by combining the Weierstrass M-test and Gelfand's formula for the spectral radius, that is, 
$\lim\limits_{k \rightarrow \infty}\left\|{A ^k}\right\|^{1/k}=\rho(A)$. Since by hypothesis $\rho(A)<1 $, we can guarantee the existence of a number $k _0 \in \mathbb{N}$ such that $\left\|A ^{k _0}\right\|<1 $, for all $k\geq k _0 $. Consider now the series $\sum _{j=0}^{\infty}A ^j C \omega(\phi^{-j}(m))$ that defines $f_{(\phi, \omega,F)}(m)  $. Given that for any $j \in \mathbb{N}  $ there exist $l(j) \in \mathbb{N}  $  and $i \in \left\{0, \ldots, k _0-1\right\} $ such that $A ^j=A^{l(j) k _0+i}$, we then have that,
\begin{equation}
\label{sum for M}
\left\|A ^j C \omega(\phi^{-j}(m))\right\|\leq \left\|A^{k _0}\right\| ^{l(j)} \left\|C\right\|K _A K _\omega,
\end{equation}
with $K _A=\max \left\{1, \left\|A\right\|, \ldots, \left\|A^{k _0 -1}\right\| \right\}$ and $K _\omega \in \mathbb{R}  $ a constant that satisfies $\lVert \omega (m)\rVert \leq K _\omega  $ for any $m \in M $ and is available by the boundedness hypothesis on $\omega (M) $.

The inequality \eqref{sum for M} and the Weierstrass M-test guarantee that the series $\sum _{j=0}^{\infty}A ^j C \omega(\phi^{-j}(m)) $ converges absolutely and uniformly on $M$ and that
\begin{align*}
\left\|f_{(\phi, \omega,F)} (m)\right\|= \left\|\sum _{j=0}^{\infty}A ^j C \omega(\phi^{-j}(m))\right\| &\leq \sum _{j=0}^{\infty}\left\|A^{k _0}\right\| ^{l(j)} \left\|C\right\|K _A K _\omega \\ 
&= k_0 \sum _{l=0}^{\infty}\left\|A^{k _0}\right\| ^{l} \left\|C\right\|K _A K _\omega = \frac{k_0 \left\|C\right\|K _A K _\omega}{1-\left\|A^{k _0}\right\|}.
\end{align*}

%\jd{I'm still not completely convinced by this. If it turned out that there were 2 values $j$ that had the same value for $\ell(j)$ then there would be only one term for that value of $\ell$ in the expression after the line break (the sum over $\ell$), but 2 terms in the sum over $j$. So it might not be true that the sum over $\ell$ is always larger ?}

Finally, since each of the summands in the series is a continuous function then so is $f_{(\phi, \omega,F)}$.

\medskip

\noindent {\bf (ii)} The result that we just proved guarantees that if the differentials $D^if_{(\phi, \omega,F)}(m)$, $i \in  \left\{1, \ldots, r\right\}$, exist then they are given by the series $\sum _{j=0}^{\infty}A ^j C D^i \left(\omega \circ \phi^{-j}\right) (m)$ that, using again the Weierstrass M-test and the hypotheses in the statement, will be now shown to uniformly converge to a continuous map. Indeed, using again the decomposition $A ^j=A^{l(j) k _i+s}$ in terms of the element $k _i \in \mathbb{N}$ such that $\left\|A^{k _i}\right\| \left\| T^i \phi ^{-k _i}\right\|_{\infty}<1$ we can conclude that each summand of this series satisfies that
\begin{equation}
\label{bounds for derivatives}
\left\|A ^j C D^i \left(\omega \circ \phi^{-j}\right)(m)\right\|\leq \left(\left\|A^{k _i}\right\| \left\|T^i \phi^{-k _i}\right\| _{\infty} \right)^{l(j)} \left\|C\right\|K _A^i K _{T^i \phi^{-1}} \left\|D^i \omega\right\|_{\infty},
\end{equation}
with 
\begin{align*}
K _A^i:=\max \left\{1, \left\|A\right\|, \ldots, \left\|A^{k _i -1}\right\| \right\}, \text{ and } K _{T^i \phi^{-1}}:= \max \left\{1,\left\|T^i \phi^{-1}\right\|_{\infty}, \ldots,  \left\|T^i \phi^{-1}\right\|_{\infty}^{k_i-1}\right\}, 
\end{align*}
hence
\begin{equation}
\label{bounds for derivatives 2}
\left\|D^if_{(\phi, \omega,F)}(m)\right\|= \left\|\sum _{j=0}^{\infty}A ^j C D^i \left(\omega \circ \phi^{-j}\right) (m)\right\|\leq \frac{k_i \left\|C\right\|K _A^i K _{T^i \phi^{-1}} \left\|D^i \omega\right\|_{\infty}}{1-\left\|A^{k _i}\right\| \left\|T^i \phi^{-k _i}\right\| _{\infty} }. 
\end{equation}
which proves the desired convergence and that $D^if_{(\phi, \omega,F)}(m)=\sum _{j=0}^{\infty}A ^j C D^i \left(\omega \circ \phi^{-j}\right) (m)$. Moreover,

\medskip

\noindent {\bf (iii)} We start by noting that if the map \eqref{functor gs} is  continuous then the subsets $\Omega_{\text{imm}}$ and $\Omega_{\text{emb}}$ are open because by Theorems 1.1 and 1.4 in \cite{Hirsch:book} the immersions and the embeddings in $C ^r(M, \mathbb{R}^N) $ are open and hence $\Omega_{\text{imm}}$ and $\Omega_{\text{emb}}$ are the preimages of those open sets by the continuous map $\Theta_{(\phi, F)} $. We establish now the continuity of $\Theta_{(\phi, F)} $ by showing that if the sequence $\left\{\omega _n\right\}_{n \in \mathbb{N}} $  in $C ^r(M, \mathbb{R}) $ converges to some element $\omega \in C ^r(M, \mathbb{R})  $ then so does $\left\{\Theta_{(\phi, F)} (\omega _n)\right\}_{n \in \mathbb{N}} \subset C ^r(M, \mathbb{R}^N) $ with respect to  $\Theta_{(\phi, F)} (\omega) \in C ^r(M, \mathbb{R}^N)  $. Indeed, if $\omega_n \to \omega $ then, using the notation introduced in \eqref{bounds for derivatives 2}, we have that for a given $\epsilon> 0 $  and for $n$ sufficiently large 
\begin{equation*}
\frac{k_i \left\|C\right\|K _A^i K _{T^i \phi^{-1}}}{1-\left\|A^{k _i}\right\| \left\|T^i \phi^{-k _i}\right\| _{\infty} }  \left\|D^i \omega_n- D^i \omega\right\|_{\infty}< \epsilon/r.
\end{equation*}
Then,
\begin{multline*}
\left\|\Theta_{(\phi, F)} (\omega_n)- \Theta_{(\phi, F)} (\omega)\right\|_{C ^r(M, \mathbb{R}^N) }=\sum _{i=0} ^r \left\| D^if_{(\phi, \omega _n,F)}-  D^if_{(\phi, \omega ,F)} \right\| _{\infty}  \\
= \sum _{i=0} ^r
\left\|\sum _{j=0}^{\infty}A ^j C D^i \left((\omega _n- \omega) \circ \phi^{-j}\right) (m)\right\|\leq
\sum _{i=0} ^r \frac{k_i \left\|C\right\|K _A^i K _{T^i \phi^{-1}}}{1-\left\|A^{k _i}\right\| \left\|T^i \phi^{-k _i}\right\| _{\infty} }  \left\|D^i \omega_n- D^i \omega\right\|_{\infty} \\ 
<\frac{\epsilon}{r}+ \cdots+\frac{\epsilon}{r}= \epsilon,
\end{multline*}
as required.
\end{proof}

\subsection{Immersive GS}

Next we will prove that for generic observation functions $\omega$ the GS $f_{(\phi,\omega,F)}$ is an immersion.

\begin{theorem}
\label{Theorem immersion}
Let $\phi \in {\rm Diff}^2(M)$ be a dynamical system on a compact manifold $M$ of dimension $q$ that has only finitely many periodic orbits. Let $F: \mathbb{R}^N \times \mathbb{R} \to \mathbb{R}^N  $ be a linear state map with $N \geq 2q $ whose connectivity matrix satisfies that $\rho(A)<1 $ and such that for any observation map $\omega \in C ^2(M, \mathbb{R})$ the corresponding generalized synchronization  $f_{(\phi, \omega,F)} \in C ^2(M, \mathbb{R}^N) $ and, moreover, the map $\Theta_{(\phi, F)}: C ^2(M, \mathbb{R})  \to  C ^2(M, \mathbb{R}^N)$ introduced in \eqref{functor gs} is continuous. Suppose also that the two following conditions hold:
\begin{description}
\item [(i)] For each periodic orbit $m$ of $\phi $ with period $n \in \mathbb{N} $, the derivative $T_m \phi^{-n} $ has $q$ distinct eigenvalues $\lambda_1, \lambda _2, \ldots, \lambda _q $. Let $\lambda _{{\rm max}} $ be the eigenvalue with the highest absolute value among the eigenvalues of all those linear maps and let $n_{ {\rm min}} $ be the smallest period. Suppose that $\lambda _{{\rm max}} \rho(A^{n_{ {\rm min}}})<1  $ and that for any periodic point $m$, the vectors
\begin{equation}
\label{condition on A C for later}
\left\{\left(I- \lambda_j A ^n \right) ^{-1} \left(I - A\right)^{-1} \left(I- A ^n\right)C\right\}_{j \in  \left\{1, \ldots, q\right\}}, \  \mbox{ with $\lambda_j$ eigenvalue of  $ T_m \phi^{-n}$} 
\end{equation}
form a linearly independent set.
\item [(ii)] The vectors $\left\{A ^j C\right\}_{j \in \left\{0,1, \ldots, N-1 \right\}}$ form a linearly independent set.
\end{description}
Then, for generic $\omega \in C ^2(M, \mathbb{R}^N)$ the generalized synchronization $f_{(\phi, \omega,F)} \in C ^2(M, \mathbb{R}) $ is an immersion.
\end{theorem}

\begin{proof}
    We proceed in two steps. In the first one we show that $f_{(\phi, \omega,F)} \in C ^2(M, \mathbb{R}) $ is an immersion at periodic points and in the second one we take care of the remaining points. We emphasize that equilibria can be seen as periodic points with period $1$.

\medskip

\noindent {\bf Step 1. Immersion at periodic points.} We start this part with three preparatory lemmas.

\begin{lemma}
Consider a connectivity matrix that satisfies the conditions $\rho(A)<1 $ and also that $\rho(\lambda _{{\rm max}} A^{n_{ {\rm min}}})<1$ as in part {\bf (i) } of the statement of the theorem. Then, for any periodic point $m$ with period $n$ and any  eigenvalue $\lambda_j$ of  $ T_m \phi^{-n}$, we have that $\rho(\lambda _j A^{n})<1$ and 
\begin{equation}
\label{funny inverse for per}
(I - \lambda_j A^n)^{-1}=\sum_{k=0} ^{\infty} \lambda _j^kA^{nk}.
\end{equation}
\end{lemma}

\begin{proof} Firstly, recall that for any square matrix  $B$ such that $\rho(B)< 1  $ then $\left(I-B\right) ^{-1}=\sum_{j=0}^{\infty} B ^j $. Let now $m$ be a periodic point with period $n$ and let $\lambda_j$ be an  eigenvalue  of  $ T_m \phi^{-n}$. 
 This implies that in order for \eqref{funny inverse for per} to hold we just need to show that $\rho(\lambda _j A^{n})<1$. This is indeed true since  any element in the spectrum of $\lambda _j A^{n} $ can be written as $\lambda _j \mu _k ^n $ with $\mu_k \in \mathbb{C}$ an eigenvalue of $A$. Moreover, let $c < 1$ such that $\left| \lambda _j \right|=c \left| \lambda _{{\rm max}}  \right| $. Then
\begin{equation*}
\left| \lambda _j \mu _k ^n \right|=c\left|\lambda _{{\rm max}} \mu _k ^{n_{{\rm min}}}  \right|\left|  \mu _k ^{n-n_{{\rm min}}} \right|< 1, 
\end{equation*} 
as required. Notice that in the last inequality we used that $\rho(\lambda _{{\rm max}} A^{n_{ {\rm min}}})<1$ and that $\rho(A)<1 $.
\end{proof}

\begin{lemma}
\label{condition for linind}
In the hypotheses of the statement of the theorem, let $m \in M $ be a periodic point of $\phi \in {\rm Diff}^2(M)$ with period $n \in \mathbb{N} $. Let $\left\{\mathbf{v} _1, \ldots, \mathbf{v} _p\right\} $ be a basis of eigenvectors associated to the distinct eigenvalues $\lambda_1, \lambda _2, \ldots, \lambda _q $.  Suppose that the set
\begin{equation}
\label{linind condition}
\left\{(I - \lambda_j A^n)^{-1} \sum^{n-1}_{k = 0} A^k C D (\omega \circ  \phi^{-k})(m) \mathbf{v}_j\right\}_{j \in \left\{1, \ldots,q\right\}}
\end{equation}
is linearly independent.
Then $f_{(\phi, \omega,F)}$ is an immersion at the periodic point $m$ for generic $\omega \in C ^2(M, \mathbb{R}^N)$.
\end{lemma}

\begin{proof} Since the eigenvalues $\lambda_j$ are distinct and the eigenvectors $\mathbf{v}_j$ are hence linearly independent, it is enough to show that the set $\left\{Df_{(\phi, \omega,F)}(m) \mathbf{v}_j\right\}_{j \in \left\{1, \ldots,q\right\}}$ is linearly independent in order to conclude that $Df_{(\phi, \omega,F)}(m)$ is injective. Now, by the expression \eqref{higher order derivatives of GS}:
\begin{align*}
Df_{(\phi, \omega,F)}(m)\mathbf{v} _j &= \sum^{\infty}_{l = 0} A^l C D (\omega \circ \phi^{-l})(m) \mathbf{v} _j \\
&= \sum^{\infty}_{l = 0} \sum_{k = 0}^{n-1} A^{ln+k} C D(\omega \circ \phi^{-(ln+k)})(m)\mathbf{v} _j \\
&= \sum^{\infty}_{l = 0} \sum_{k = 0}^{n-1} A^{ln+k} C D(\omega \circ \phi^{-k})(m) [T\phi^{-n}(m)]^l \mathbf{v} _j \\
&= \sum^{\infty}_{l = 0} \sum_{k = 0}^{n-1} A^{ln+k} C D(\omega \circ  \phi^{-k})(m) \lambda^l_j \mathbf{v}_j \\
&= \sum^{\infty}_{l = 0} (\lambda_jA^n)^l \sum_{k = 0}^{n-1} A^k C D(\omega \circ  \phi^{-k})(m)  \mathbf{v}_j \\ 
&=  
(I - \lambda_j A^n)^{-1} \sum^{n-1}_{k = 0} A^k C D (\omega \circ  \phi^{-k})(m) \mathbf{v}_j,
\end{align*}
and so, since by assumption the set of vectors of the form of the RHS are linearly independent, so are the vectors of the form of the LHS,
which proves the statement.
\end{proof}

\begin{lemma}
\label{invertible perturbation}
Let $A$ and $B$ two square matrices of the same size such that $\det (A)=0 $ and $\det (B)\neq 0  $. Then, there exists $\varepsilon>0  $ such that 
\begin{equation*}
\det (A- \varepsilon B)\neq 0.
\end{equation*}
\end{lemma}

\begin{proof} Consider the singular matrix $B ^{-1}A $ and let $\lambda_0  $ be its non-zero eigenvalue that has the smallest absolute value. Then, for any $0< \varepsilon< | \lambda_0| $ we necessarily have that $\det  \left(B ^{-1}A - \varepsilon I \right)\neq 0 $ because otherwise $\varepsilon $ would be an eigenvalue of $B ^{-1}A $ which is impossible by the minimality of $\lambda_0 $. This implies that $C:= B ^{-1}A - \varepsilon I $ is invertible and hence so is $BC=A- \varepsilon B $, as required.

\end{proof}

\medskip

As a corollary of this lemma we can conclude that if $\mathcal{V}:= \left\{\mathbf{v}_1, \ldots, \mathbf{v}_n\right\} $ and $\mathcal{W}:= \left\{\mathbf{w}_1, \ldots, \mathbf{w}_n\right\} $ are two sets of vectors in ${\mathbb R}^n $, then there exists $\varepsilon>0  $ such that the set $\left\{\mathbf{v}_1+ \varepsilon \mathbf{w} _1, \ldots, \mathbf{v}_n+ \varepsilon \mathbf{w} _n\right\} $ is made of linearly independent vectors. This fact is used at the end of the proof of Step 1 of Theorem \ref{Theorem immersion}.

\medskip

We now use this result to show that, for generic $\omega \in C ^2(M, \mathbb{R}^N)$, the generalized synchronization $f_{(\phi, \omega,F)} \in C ^2(M, \mathbb{R}) $ is an immersion at the periodic points of $\phi$. Let $m _1, \ldots, m _P \in M$ be the distinct periodic points of $\phi$, having periods $n _1, \ldots, n _P \in \mathbb{N} $, respectively (the equilibria of $\phi$ are included in this list with periods equal to one). The term {\it distinct} means that none of those points are in the orbits of the others. We now choose $P$ disjoint open neighborhoods $B _i $ that contain each of the distinct periodic points $m _i $. Since there is a finite number of periodic points, the open sets $B _i $ can be chosen small enough so that, additionally, all the open sets
\begin{equation*}
\phi^{-t}(B _i) \quad \mbox{for all $t \in \left\{0, \ldots, n _i\right\} $ and $i \in \left\{1, \ldots, P\right\}$}
\end{equation*}
are disjoint.

Now, given any of the distinct periodic points $ m _i \in M $ on the list, we show that  $f_{(\phi, \omega,F)} $ for generic $\omega \in C ^2(M, \mathbb{R}^N)$, that is, the set of observation maps $\omega$ for which $f_{(\phi, \omega,F)} $ is an immersion at $m _i $ is open and dense in $C ^2(M, \mathbb{R}^N)$. The openness is a consequence of the hypothesis on the continuity of the map $\Theta_{(\phi, F)} $ and of an argument identical to the beginning of the proof of part {\bf (iii)} of Proposition \ref{generalized synch with spectral radius}. Regarding the density, we show that  if $f_{(\phi, \omega,F)} $ is not an immersion at $m _i  $, then there is a perturbation $\omega' $ of $\omega $ in $C^2(M, \mathbb{R})$ for which $f_{(\phi, \omega',F)} $ is an immersion at $m _i $. Indeed, set 
\begin{equation}
\label{perturbed observation map 1}
    \omega' = \omega + \sum_{l = 0}^{n _i-1}   \psi_l ^i
\end{equation}
where  $\psi_l ^i \in C^{\infty}(M,\mathbb{R})$ are bump functions whose  supports are contained in  $\phi^{-l}(B_i)$ and, additionally, are chosen to satisfy
\begin{align*}
D(\psi_l^i \circ \phi^{-l})(m _i) = \varepsilon \mathbf{v}^{\top}, \quad \mbox{$l \in \left\{0, \ldots, n _i-1\right\}$,}
\end{align*}
for some small constant $\varepsilon> 0$ and $\mathbf{v} \in \mathbb{R} ^p $  the unique vector that solves the linear system
\begin{equation}
\label{linear system for later}
\left(
\begin{array}{c}
\mathbf{v}_1^{\top}\\
\vdots\\
\mathbf{v}_p^{\top}
\end{array}
\right) \mathbf{v}=
\left(
\begin{array}{c}
1\\
\vdots\\
1
\end{array}
\right),
\end{equation}
with $\left\{\mathbf{v} _1, \ldots, \mathbf{v} _p\right\} $ a basis of eigenvectors of $T_{m _i}\phi^{-n _i} $. Note that by construction and for any $l \in \left\{0, \ldots, n _i-1\right\}$,
\begin{equation}
\label{another inter 1}
D(\omega'\circ \phi^{-l})(m _i) = D(\omega\circ \phi^{-l})(m _i) + D(\psi_l^i \circ  \phi^{-l}) = D(\omega\circ \phi^{-l})(m _i) + \varepsilon \mathbf{v}^{\top}.
\end{equation}
We now consider the vectors \eqref{linind condition} in Lemma \ref{condition for linind} with respect to the perturbed observation map in \eqref{perturbed observation map 1}. Indeed, by \eqref{another inter 1} and the way in which the vector $\mathbf{v}  $ has been constructed in \eqref{linear system for later}:
\begin{multline*}
(I - \lambda_j A^{n_i})^{-1} \sum^{n_i-1}_{k=0} A^k C D(\omega' \circ \phi^{-k})(m _i)\mathbf{v}_j \\
= (I - \lambda_j A^{n_i})^{-1} \sum^{n_i-1}_{k=0} A^k C D(\omega \circ  \phi^{-k})(m _i)\mathbf{v}_j + \varepsilon(I - \lambda_j A^{n_i})^{-1} \sum^{n_i-1}_{k=0} A^k C  \mathbf{v}^{\top} \mathbf{v}_j \\
= (I - \lambda_j A^{n_i})^{-1} \sum^{n_i-1}_{k=0} A^k C D(\omega \circ  \phi^{-k})(m _i)\mathbf{v}_j 
+ \varepsilon(I - \lambda_j A^{n_i})^{-1} (I - A)^{-1}(I - A^{n_i})C.
\end{multline*}
Given that when we vary $j \in \left\{1, \ldots, p \right\}$ in the previous expression the vectors in the second summand form by hypothesis a linearly independent set, we can use Lemma \ref{invertible perturbation} to choose $\varepsilon> 0 $ so that the family $\left\{(I - \lambda_j A^{n_i})^{-1} \sum^{n_i-1}_{k=0} A^k C D(\omega' \circ \phi^{-k})(m _i)\mathbf{v}_j\right\}_{j \in \left\{1, \ldots, p \right\}}$ forms a linearly independent set and, at the same time, $\omega'$ is as close to $\omega $ in $C^2(M, \mathbb{R})$ as desired. This shows by Lemma \ref{condition for linind} that $f_{(\phi, \omega',F)} $ is an immersion at $m _i $. 

The choice of the open sets $B _i $ implies that we can keep perturbing $\omega$ in order to make $f_{(\phi, \omega',F)} $ immersive at the other periodic points without spoiling that condition for the previous ones. This shows in particular that a perturbation of the type
\begin{equation}
\label{perturbed observation map final}
    \omega' = \omega + \sum_{i=1}^P\sum_{l = 0}^{n _i-1}   \psi_l ^i
\end{equation}
can be constructed so that $f_{(\phi, \omega',F)} $ is immersive at all the periodic points of $\phi $, as required.

\medskip

Before we continue, we will state the Immersion Theorem, which will be used several times in the proofs that follow.

\begin{theorem}
\label{theorem::immersion}
    \cite[Immersion Theorem, Theorem 3.5.7]{mta} Let $M,M'$ be smooth manifolds, and $f \in C^r(M,M')$ for $r \geq 1$. Then the following are equivalent:
    \begin{itemize}
        \item $f$ is an immersion at $m \in M$;
        \item there is a neighbourhood $U$ of $m$ such that $f(U)$ is a submanifold of $M'$ and $f$ restricted to $U$ is a diffeomorphism of $U$ onto $f(U)$. 
    \end{itemize}
\end{theorem}

\noindent {\bf Step 2. Immersion at the remaining points.} Having just proved that for generic $\omega \in C ^2(M, \mathbb{R}^N)$ the generalized synchronization $f_{(\phi, \omega,F)} \in C ^2(M, \mathbb{R}) $ is an immersion at the periodic points of $\phi$, the Immersion Theorem (Theorem \ref{theorem::immersion}) guarantees that the same holds for the open set  formed by the union of certain open neighborhoods around those points.

Now we let $\mathcal{M} \subset M$ be the compact subset of $M$ obtained by removing that immersed open set. Our goal is now to show that $f_{(\phi, \omega,F)} \in C ^2(M, \mathbb{R}) $ is also an immersion at $\mathcal{M}$ for generic $\omega \in C ^2(M, \mathbb{R}^N)$. 

The hypotheses that we imposed on $M$ in Chapter \ref{chapter::SSMs} imply that $M$ can be endowed with a Riemannian metric which makes $M$ a complete metric space by the Hopf and Rinow Theorem (see \cite[Theorem 7.7]{boothby2003introduction}). This implies in turn that the compact subset ${\cal M}\subset M $ is also a complete metric space which allows us to define open balls $B _r(m) $ of radius $r > 0$  around each point $m \in {\cal M} $. Using this notation, in the next paragraphs we show that for any $\omega \in C ^2(M, \mathbb{R}^N)$ and $m  \in {\cal M} $ we can find a $n (m) \in \mathbb{N}  $  and a perturbation $\omega' \in C ^2(M, \mathbb{R}^N)$ as close to $\omega$ as desired such that the restriction of  $f_{(\phi, \omega',F)} $ to $B _{2^{-n(m)}}(m) $ is an immersion.

Indeed, take an arbitrary $m \in \mathcal{M}$ and define a collection of balls $B_{2^{-n}}(m)$ centered at $m$ with radius $2^{-n}$,  $n \in \mathbb{N}$. For a fixed $n$ consider the infinite trajectory $\phi^{-t}(B_{2^{-n}}(m))$, $t  \in \mathbb{N}$.  Choose now $n_1 \in \mathbb{N}$ large enough so that, for any $n > n_1$ the balls $\phi^{-t}(B_{2^{-n}}(m))$   are disjoint for $t = 0, \ldots , N-1$ and $B_{2^{-n}}(m) \subset U$ where $(U, h) $ is an admissible chart of $M$. Given that $\phi \in {\rm Diff}^2(M)$, we note that the family $(U _t, h _t)$, $t \in \mathbb{N} $,  defined by $U _t= \phi^{-t}(U) $ and  $h _t :=h \circ \phi ^t  $ is made of admissible charts and that $\phi^{-t}(B_{2^{-n}}(m))\subset U_t$, for all $n> n _1 $. Let $T(n)$ denote the largest integer such that $\phi^{-t}(B_{2^{-n}}(m))$ are disjoint  for $t = 0, \ldots, T (n)-1$.

Now, for each $n > n _1$ and $t = 0, \ldots, N-1$ we define functions $\psi_{tn} \in C^{\infty}(M,\mathbb{R})$ that have their support included in $\phi^{-t}(B_{2^{-n}}(m))$ and satisfy
\begin{equation}
\label{condition on the derivative}
\frac{\partial (\psi_{tn}h_t^{-1})}{\partial u_j} = 1
\end{equation}
on $h_t\left(\phi^{-t}(B_{2^{-(n+1)}}(m))\right) =h\left( B_{2^{-(n+1)}}(m)\right)$. We impose further that $\psi_{tn} = \psi_{t(n+1)}$ on $\phi^{-t}(B_{2^{-(n+2)}}(m))$ for all $n> n _1$, and that there is some $\kappa > 0$ independent of $n$ and $t$ such that $\lVert \psi_{tn} h_t^{-1} \rVert_{C^1} \leq \kappa$.
These functions can be constructed by setting
\begin{align*}
    \psi_{tn}(m) = \lambda_{tn}(m) \sum_{j = 1}^{q}\xi_j(m)
\end{align*}
where $\xi_j$ is the $j$-th coordinate map for the chart $h _t$ and $\lambda_{tn} \in C^{\infty}(M,\mathbb{R})$ are bump functions that have support included in $\phi^{-t}(B_{2^{-n}}(m))$ and satisfy $\lambda_{tn}|_{\phi^{-t}(B_{2^{-(n+1)}}(m))} = 1$. Define now the perturbation $\omega _n $ of $\omega$ by
\begin{align}
\label{definition of omega perturb}
    \omega_n = \omega + \sum_{t = 0}^{N-1} \varepsilon_t \psi_{tn},
\end{align}
where $\varepsilon _t  $  are the components of a vector $\boldsymbol{\varepsilon} \in \mathbb{R}^N  $ with positive entries. By construction, for any $m' \in B_{2^{-n}}(m)$ and $t = 0 , \ldots,  N-1$, we have  that
\begin{align*}
    \omega_n \phi^{-t}(m') = \omega \phi^{-t}(m') + \varepsilon_t \psi_{tn}(m') 
\end{align*}
and moreover by \eqref{condition on the derivative} and for any $m' \in B_{2^{-(n+1)}}(m)$:
\begin{equation}
\label{component of delay}
    \frac{\partial (\omega_n \phi^{-t} h^{-1})}{\partial u_j}(h(m')) = 
    \frac{\partial (\omega \phi^{-t} h^{-1})}{\partial u_j}(h(m')) + \varepsilon_t.
\end{equation}
Let $\Phi: M \to \mathbb{R}^N$ be a backwards version of the Takens delay map, that is, 
\begin{align*}
    \Phi(m) := \left( \omega(m) , \omega\circ  \phi^{-1}(m) , \ldots, \omega\circ\phi^{-(N-1)}(m) \right)^{\top},
\end{align*}
and let $\Phi_n : M \to \mathbb{R}^N$ be its perturbed version defined by
\begin{align*}
    \Phi_n(m) := \left( \omega_n(m) , \omega_n\circ \phi^{-1}(m) , \ldots, \omega_n\circ\phi^{-(N-1)}(m) \right)^{\top}.
\end{align*}
Using these objects, we can rewrite \eqref{component of delay} in vector form as
\begin{equation}
\label{derivative in vector form bis}
\frac{\partial (\Phi_n h^{-1})}{\partial u_j}(h(m')) = \frac{\partial (\Phi h^{-1})}{\partial u_j} (h(m'))+ \boldsymbol{\varepsilon},
\end{equation}
for any $m' \in B_{2^{-(n+1)}}(m)$ and where $\boldsymbol{\varepsilon} \in \mathbb{R}^N$. Next, for any $t = N , \ldots, T(n) - 1$ notice that
\begin{equation}
\label{derivative in vector form bis1}
    \omega_n \phi^{-t}(m) = \omega \phi^{-t}(m).
\end{equation}
Finally, if $t \geq T(n)$ then 
\begin{equation}
\label{derivative in vector form bis2}
    \omega_n \phi^{-t}(m) = \omega \phi^{-t}(m) + \sum^{N-1}_{\tau = 0} \varepsilon_\tau \psi_{\tau n}\phi^{-t}(m).
\end{equation} 
We now consider the perturbed generalized synchronization $f_{(\phi, \omega_n,F)} :M \to\mathbb{R} ^N$  given by
\begin{align*}
    f_{(\phi, \omega_n,F)} &= \sum_{t = 0}^{\infty} A^t C \omega_n \phi^{-t} = \sum_{t = 0}^{N-1} A^t C \omega_n \phi^{-t} + \sum_{t = N}^{\infty} A^t C \omega_n \phi^{-t} = Q \Phi_n + \sum_{t = N}^{\infty} A^t C \omega_n\phi^{-t}
\end{align*}
where $Q$ is the $N \times N$ real matrix with $(t+1)$-th column $A^t C$. Now we take the partial derivatives with respect to $u_j$ at points in $h(B_{2^{-(n+1)}}(m))$ and observe that by \eqref{derivative in vector form bis}, \eqref{derivative in vector form bis1}, and \eqref{derivative in vector form bis2}:	
\begin{align}
    & \ \ \frac{\partial (f_{(\phi, \omega_n,F)} h^{-1})}{\partial u_j} \notag \\ &= Q \frac{\partial(\Phi_n h^{-1})}{\partial u_j} + \sum_{t = N}^{\infty} A^t C \frac{\partial (\omega_n\phi^{-t}h^{-1} )}{\partial u_j} \notag \\ 
    &= Q \frac{\partial(\Phi h^{-1})}{\partial u_j} + Q \boldsymbol{\varepsilon} + \sum^{\infty}_{t = N} A^t C \frac{\partial (\omega_n \phi^{-t}h^{-1})}{\partial u_j}\notag \\
    &=
    Q \frac{\partial(\Phi h^{-1})}{\partial u_j} + Q \boldsymbol{\varepsilon} + \sum^{T(n)-1}_{t = N} A^t C \frac{\partial (\omega_n \phi^{-t} h^{-1})}{\partial u_j} + 
    \sum^{\infty}_{t = T(n)} A^t C \frac{\partial (\omega_n \phi^{-t} h^{-1})}{\partial u_j}\notag \\
    &= Q \frac{\partial(\Phi h^{-1})}{\partial u_j} + Q \boldsymbol{\varepsilon} + \sum^{T(n)-1}_{t = N} A^t C \frac{\partial (\omega \phi^{-t} h^{-1})}{\partial u_j} + 
    \sum^{\infty}_{t = T(n)} A^t C \frac{\partial (\omega_n \phi^{-t} h^{-1})}{\partial u_j}\notag \\
    &= Q \frac{\partial(\Phi h^{-1})}{\partial u_j} + Q \boldsymbol{\varepsilon} + \sum^{T(n)-1}_{t = N} A^t C \frac{\partial (\omega \phi^{-t}h^{-1})}{\partial u_j} + 
    \sum^{\infty}_{t = T(n)} A^t C \frac{\partial (\omega \phi^{-t}h^{-1})}{\partial u_j}\notag \\
    &+ \sum^{\infty}_{t = T(n)} A^t C \bigg( \sum^{N-1}_{\tau = 0} \varepsilon_{\tau} \frac{\partial( \psi_{\tau n} \phi^{-t} h^{-1})}{\partial u_j} \bigg)\notag \\
    &= Q \frac{\partial(\Phi h^{-1})}{\partial u_j} + Q \boldsymbol{\varepsilon} + \sum^{\infty}_{t = N} A^t C \frac{\partial( \omega \phi^{-t}h^{-1})}{\partial u_j} + \sum^{\infty}_{t = T(n)} A^t C \bigg( \sum^{N-1}_{\tau = 0} \varepsilon_{\tau} \frac{\partial (\psi_{\tau n} \phi^{-t} h^{-1})}{\partial u_j} \bigg)\notag \\
    &= \sum_{t = 0}^{N-1}A^t C \frac{\partial (\omega \phi^{-t}h^{-1})}{\partial u_j} + Q \boldsymbol{\varepsilon} + \sum^{\infty}_{t = N} A^t C \frac{\partial (\omega \phi^{-t}h^{-1})}{\partial u_j} + \sum^{\infty}_{t = T(n)} A^t C \bigg( \sum^{N-1}_{\tau = 0} \varepsilon_{\tau} \frac{\partial (\psi_{\tau n} \phi^{-t} h^{-1})}{\partial u_j} \bigg)\notag \\
    &= \sum_{t = 0}^{\infty}A^t C \frac{\partial (\omega \phi^{-t}h^{-1})}{\partial u_j} + Q \boldsymbol{\varepsilon} + \sum^{\infty}_{t = T(n)} A^t C \bigg( \sum^{N-1}_{\tau = 0} \varepsilon_{\tau} \frac{\partial (\psi_{\tau n} \phi^{-t} h^{-1})}{\partial u_j} \bigg)\notag \\
    &= \frac{\partial(f_{(\phi, \omega,F)} h^{-1})}{\partial u_j} + Q \boldsymbol{\varepsilon} + \sum^{\infty}_{t = T(n)} A^t C \bigg( \sum^{N-1}_{\tau = 0} \varepsilon_{\tau} \frac{\partial (\psi_{\tau n} \phi^{-t} h^{-1})}{\partial u_j} \bigg)
\label{important identity for derivatives}.
\end{align}
In order to prove that  $f_{(\phi, \omega,F)}$ is an immersion  at the  points in $h(B_{2^{-(n+1)}}(m))$ for a generic observation $\omega$, we shall find an arbitrarily small vector $\boldsymbol{\varepsilon}$ for which the vectors corresponding to the $\boldsymbol{\varepsilon}$-perturbed observation $\omega _n$ 
\begin{align*}
    \left\{\frac{\partial \left(f_{(\phi, \omega_n,F)}  h^{-1}\right)}{\partial u_j} \right\}_{j \in  \left\{1, \ldots, q\right\}}
\end{align*}
are a linearly independent family. We will proceed inductively by showing that if we assume for some $s$ satisfying $1 \leq s < q$ that the vectors
\begin{align}
\label{proceeding_vectors}
    \left\{\frac{\partial \left(f_{(\phi, \omega,F)}  h^{-1}\right)}{\partial u_j} \right\}_{j \in  \left\{1, \ldots, s\right\}}
\end{align}
are linearly independent, then we can choose an arbitrarily small vector $\boldsymbol{\varepsilon}$ such that the family corresponding to the perturbed observation $\omega_n$ defined in  \eqref{definition of omega perturb} satisfies that 
\begin{align*}
    \left\{\frac{\partial \left(f_{(\phi, \omega_n,F)}  h^{-1}\right)}{\partial u_j} \right\}_{j \in  \left\{1, \ldots, s+1\right\}}
\end{align*}
is a linearly independent family.  To this end, we define the map $\Psi : \mathbb{R}^s \times h(U) \to \mathbb{R}^{N}$ as
\begin{align*}
    \Psi(\boldsymbol{\alpha} , {\bf u}) = \sum^s_{j = 1}\alpha_j \frac{\partial \left(f_{(\phi, \omega,F)}  h^{-1}\right)}{\partial u_j} -  \frac{\partial \left(f_{(\phi, \omega,F)}  h^{-1}\right)}{\partial u_{s+1}}.
\end{align*}
The hypothesis on the statement of the theorem about $f_{(\phi, \omega,F)} \in C ^2(M, \mathbb{R}^N) $ for any observation map $\omega \in C ^2(M, \mathbb{R})$ implies that $\Psi $ is of class $C^1$ and maps a manifold  of dimension  $s + q$ to a manifold of dimension $N$. 
Since by hypothesis $s + q< 2q\leq N$,  then the set  $\mathbb{R}^N \setminus   \Psi(\mathbb{R}^s \times h(U))$ is dense in $\mathbb{R}^N  $ (see \cite[Chapter 3, Proposition 1.2]{Hirsch:book}). This implies that we can choose an arbitrarily small vector 
$ \boldsymbol{\delta} \in \bigg(\mathbb{R}^N \setminus \Psi(\mathbb{R}^s \times h (U))\bigg)$ such that  if we set $\boldsymbol{\varepsilon} := Q^{-1}\boldsymbol{\delta} $ then we have that the vector
\begin{align*}
    \frac{\partial(f_{(\phi, \omega,F)} h^{-1})}{\partial u_{s+1}} + Q \boldsymbol{\varepsilon}
\end{align*}
is independent of the vectors in \eqref{proceeding_vectors} when evaluated at any point in $h(U) $. Since the linear independence is stable under small perturbations we can choose $\boldsymbol{\varepsilon} $ small enough so that it is actually the family
\begin{align*}
    \left\{\frac{\partial \left(f_{(\phi, \omega,F)}  h^{-1}\right)}{\partial u_j}+ Q \boldsymbol{\varepsilon} \right\}_{j \in  \left\{1, \ldots, s+1\right\}}
\end{align*}
that is linearly independent. Now, in view of the identity \eqref{important identity for derivatives} we note that the value $n \in \mathbb{N}$ can be chosen large enough so that the residual terms
\begin{align*}
    \sum^{\infty}_{t = T(n)} A^t C \bigg( \sum^{N-1}_{\tau = 0} \varepsilon_{\tau} \frac{\partial (\psi_{\tau n} \phi^{-t} h^{-1})}{\partial u_j} \bigg), \quad \mbox{$j \in  \left\{1, \ldots, s+1\right\}$,}
\end{align*} 
are small enough so that the family 
\begin{align*}
    &\left\{\frac{\partial \left(f_{(\phi, \omega,F)}  h^{-1}\right)}{\partial u_j}+ Q \boldsymbol{\varepsilon}+  \sum^{\infty}_{t = T(n)} A^t C \bigg( \sum^{N-1}_{\tau = 0} \varepsilon_{\tau} \frac{\partial (\psi_{\tau n} \phi^{-t} h^{-1})}{\partial u_j} \bigg)\right\}_{j \in  \left\{1, \ldots, s+1\right\}} \\
    =
  &\left\{\frac{\partial \left(f_{(\phi, \omega_n,F)}  h^{-1}\right)}{\partial u_j}\right\}_{j \in  \left\{1, \ldots, s+1\right\}}
\end{align*}
is linearly independent, as required. Notice that this equality is a consequence of \eqref{important identity for derivatives}. The possibility to shrink the residual term comes from the convergence of the series
\begin{align*}
    \sum^{\infty}_{t = 0} A^t C \bigg( \sum^{N-1}_{\tau = 0} \varepsilon_{\tau} \frac{\partial (\psi_{\tau n} \phi^{-t} h^{-1})}{\partial u_j} \bigg), \quad \mbox{$j \in  \left\{1, \ldots, s+1\right\}$,}
\end{align*} 
which is guaranteed by the hypothesis on the differentiability of $f_{(\phi, \omega,F)} $ for any observation map $\omega \in C ^2(M, \mathbb{R})$ and the expression \eqref{higher order derivatives of GS}. In this case the bump functions play the role of the observations for which we assumed the existence of a uniform bound $\kappa$ over $\tau$ and $n$ such that $\lVert \psi_{\tau n} \rVert_{C^1} < \kappa$.

If we recursively apply this procedure, we can conclude the existence of a small perturbation $\omega_n $ of $\omega $ obtained as a sequence of perturbations of the type \eqref{definition of omega perturb} for which the family 
\begin{equation*}
  \left\{\frac{\partial \left(f_{(\phi, \omega_n,F)}  h^{-1}\right)}{\partial u_j}\right\}_{j \in  \left\{1, \ldots, q\right\}}
\end{equation*}
is linearly independent when evaluated at $m \in {\cal M} $, which proves that 
$f_{(\phi, \omega_n,F)} \in C ^2(M, \mathbb{R} ^N ) $ is an immersion at $m \in {\cal M}$.

Finally, observe that we just showed that for any $m \in \mathcal{M}$, there exists an $n(m) \in \mathbb{N}$ such that the restriction of the perturbation $f_{(\phi, \omega_{n(m)},F)} $ to $B_{2^{-n(m)}}(m)$ is an immersion. We note that the union 
\begin{align*}
    \bigcup_{m \in \mathcal{M}} B_{2^{-n}}(m)
\end{align*}
is clearly an open cover of $\mathcal{M}$. Since $\mathcal{M}$ is compact, it admits a finite subcover. The finite subcover comprises sets for which, one at a time, we can construct an immersion using the procedure described earlier in this proof. For each set, we ensure that the perturbation is sufficiently small not to spoil the immersion on any other set. 

This argument completes the proof of the  immersion of the GS at the points of $\mathcal{M}$ and therefore, together with the Step 1, shows that there exists a small perturbation of $f_{(\phi, \omega_n,F)} \in C ^2(M, \mathbb{R} ^N ) $ of $f_{(\phi, \omega,F)}$ that is an immersion at all the points in $M$, as required.
\end{proof}

\subsection{Embedding GS}

Now that we have established immersivity, we can complete the proof of Theorem \ref{embedding_thm} by establishing injectivity.

\begin{proof}
    As in the previous theorem, we proceed in two steps.

\medskip

\noindent {\bf Step 1. Injectivity around the periodic set.} We start by showing that the observations corresponding to the globally immersive generalized synchronizations whose existence we proved in Theorem \ref{Theorem immersion} can be slightly perturbed in $ C ^2(M, \mathbb{R})$ so that the resulting GS is injective in an open subset $V _P  $ that includes all the periodic points of $\phi$. We start this part of the discussion with a preparatory lemma.

\begin{lemma}
\label{injectivity_lemma}
In the hypotheses of the theorem, let $m _1, \ldots, m _P \in M$ be the distinct periodic points of $\phi$, each of which have periods $n _1, \ldots, n _P \in \mathbb{N} $, respectively. Let $\ell \in \mathbb{N} $ be the lowest common multiple of all the periods and denote by $M _P $ the set of all periodic points of $\phi$ (that is, the set that comprises $ \left\{m _1, \ldots, m _P\right\} $ and all the corresponding orbits). Then, the restriction $\left.f_{(\phi, \omega,F)}\right|_{M _P} $ of a generalized synchronization $f_{(\phi, \omega,F)} \in C ^2(M, \mathbb{R}^N) $ to $M _P $  is injective if and only if the map $g_{(\phi, \omega,F)} : M _P \to \mathbb{R}^N$ defined by
    \begin{align*}
        g_{(\phi, \omega,F)} = \sum_{k = 0}^{\ell - 1}A^k C \left(\omega \circ  \phi^{-k}\right)
    \end{align*}
    is injective.
\end{lemma}

\begin{proof}
Let $m _1, m _2 \in M _P $ be such that  $f_{(\phi, \omega,F)}(m_1) = f_{(\phi, \omega,F)}(m_2)$. This equality is equivalent to the following expressions:
\begin{align*}
\sum_{t = 0}^{\infty} A^t C \omega \phi^{-t}(m_1) &= \sum_{t = 0}^{\infty} A^t C \omega \phi^{-t}(m_2), \\
\sum_{t = 0}^{\infty} \sum_{k = 0}^{\ell-1} A^{t\ell + k} C \omega \phi^{-(t\ell+k)}(m_1) &= \sum_{t = 0}^{\infty} \sum_{k = 0}^{\ell-1} A^{t\ell + k} C \omega \phi^{-(t\ell+k)}(m_2), \\
\sum_{t = 0}^{\infty} (A^\ell)^t \sum_{k = 0}^{\ell-1} A^{k} C \omega \phi^{-k}(m_1) &= \sum_{t = 0}^{\infty} (A^\ell)^t \sum_{k = 0}^{\ell-1} A^{k} C \omega \phi^{-k}(m_2). 
\end{align*}
Given that $\rho(A)<1 $ then $\rho(A^{\ell})<1 $ necessarily and hence this equality can be rewritten as
\begin{equation*}
(I - A^\ell)^{-1} \sum_{k = 0}^{\ell-1} A^{k} C \omega \phi^{-k}(m_1) =  (I - A^\ell)^{-1} \sum_{k = 0}^{\ell-1} A^{k} C \omega \phi^{-k}(m_2),
\end{equation*}
which is equivalent to 
$\sum_{k = 0}^{\ell-1} A^{k} C \omega \phi^{-k}(m_1) = \sum_{k = 0}^{\ell-1} A^{k} C \omega \phi^{-k}(m_2) $ and hence, by definition, to $g_{(\phi, \omega,F)}(m_1) = g_{(\phi, \omega,F)}(m_2)$, which proves the statement.
\end{proof}

\medskip

If we now define  $\Phi_{(\ell, \omega)}: M \to \mathbb{R}^l$ as 
\begin{align*}
    \Phi_{(\ell, \omega)}(m) := \left( \omega(m) , \omega\circ  \phi^{-1}(m) , \ldots, \omega\circ\phi^{-(\ell-1)}(m) \right)^{\top},
\end{align*}
we note that the map $g_{(\phi, \omega,F)}  $  can be rewritten as $g_{(\phi, \omega,F)} = Q \Phi_{(\ell, \omega)}  $,  where $Q \in \mathbb{M}_{N, \ell} $ is a matrix whose $(k+1) $th-column is set to the vector $A ^k C $. The hypotheses on the vectors $\left\{A ^j C\right\}_{j \in \left\{0,1, \ldots, N-1 \right\}}$ forming a linearly independent set and that $N> \ell $ guarantee that ${\rm rank}\, Q = N $ and hence that the associated linear map $Q: \mathbb{R}^{\ell}\to \mathbb{R}^N$ is injective. With this notation we now show that if $g_{(\phi, \omega,F)}  $ is not injective in $M _P $ then a perturbation $\omega ' \in C ^2(M, \mathbb{R}^N) $ of $\omega  $ can be chosen so that $g_{(\phi, \omega',F)}  $ is. More specifically, define
\begin{equation}
\label{perturbation for mp}
\omega':= \omega+\sum_{i=1}^P\sum_{j=1}^{n _i}\varepsilon_{ij}\Psi _{ij},
\end{equation}
where $\Psi _{ij}$ are bump functions with non-intersecting supports $U _{ij} $ such that $m _{ij}:=\phi^{-(j-1)}(m _i) \in U _{ij} $ and, moreover, $\Psi _{ij} \left(\phi^{-(j-1)}(m _i)\right)= \Psi _{ij}(m _{ij})=1/ \mathcal{L}(i,j)$. The symbol $\mathcal{L}(i,j) \in \mathbb{N} $  denotes the ordinal of the pair $(i,j) $ in lexicographic order.

We now show that the constants $\varepsilon_{ij} $ can be chosen so that $\omega'  $ is as close as we want to $\omega  $ and, at the same time,  $g_{(\phi, \omega',F)}  $ is injective. Firstly, it is easy to see that, by construction,
\begin{align*}
    \Phi_{(\ell, \omega')}(m) :=\Phi_{(\ell, \omega)}(m) +\sum_{i=1}^P\sum_{j=1}^{n _i}\varepsilon_{ij} \left( \Psi _{ij}(m) , \Psi _{ij}\circ  \phi^{-1}(m) , \ldots, \Psi _{ij}\circ\phi^{-(\ell-1)}(m) \right)^{\top}.
\end{align*}
Second, if $m_{i _1j _1} $ and $m_{i _2j _2} $ are two different periodic points then
\begin{equation}
\label{injectivity condition before lexi}
g_{(\phi, \omega',F)} (m_{i _1j _1})-g_{(\phi, \omega',F)} (m_{i _2j _2})=g_{(\phi, \omega,F)} (m_{i _1j _1})-g_{(\phi, \omega,F)} (m_{i _2j _2})+ Q\left(\varepsilon_{i _1 j _1}\mathbf{v}_{i _1 j _1}-\varepsilon_{i _2 j _2}\mathbf{v}_{i _2 j _2}\right) ,
\end{equation}
where the vectors $\mathbf{v}_{i _1 j _1} \in \mathbb{R}^{{\rm Card} M _P} $ have entries equal to zero except at the slots that are multiples of the period of the corresponding periodic point. More specifically, if the periodic point  $m_{i j }  $ has period $n _{ij}$, then
\begin{equation}
\label{definition of vs}
 (\mathbf{v}_{i  j }) _i:= \left\{
\begin{array}{l}
1/ \mathcal{L}(i,j) \quad \mbox{when $i=1$ or $i-1 $ is a multiple of $n_{ij}$,}\\
0 \quad \mbox{otherwise.}
\end{array}
\right.
\end{equation}
Using  the injectivity of $Q$ and Lemma \ref{injectivity_lemma}, we now show that we can choose the perturbation constants $\varepsilon_{i  j }$   so that the restriction of $g_{(\phi, \omega',F)} $ to $M _P $ is injective. Let
\begin{equation}
\label{choice of vareps}
\varepsilon_{i  j }:= \epsilon \|g_{(\phi, \omega,F)} (m_{i j })\|, \quad \mbox{for some constant $\epsilon>0$.} 
\end{equation}
We now show that if $\epsilon>0$ is chosen so that 
\begin{multline}
\label{choice of epsil}
\epsilon \max_{(i_1,j_1), (i_2,j_2)} \left\{\left\| Q\left(\left\|g_{(\phi, \omega,F)} (m_{i_1 j_1) }\right\|\mathbf{v}_{i _1 j _1}-\left\|g_{(\phi, \omega,F)} (m_{i_2 j_2} )\right\|\mathbf{v}_{i _2 j _2}\right)  \right\|\right\}\\
<\min_{(i_1,j_1), (i_2,j_2)} \left\{\left\| g_{(\phi, \omega,F)} (m_{i_1 j_1} )-g_{(\phi, \omega,F)} (m_{i_2 j_2} ) \right\|   \right\}
\end{multline}
then the injectivity of $g_{(\phi, \omega',F)} \mid _{M _P} $ is guaranteed. Indeed, consider first the case of two distinct periodic points $m_{i _1j _1} $  and  $m_{i _2j _2}$ for which $g_{(\phi, \omega,F)} $ fails to be injective, that is, $g_{(\phi, \omega,F)} (m_{i _1j _1})= g_{(\phi, \omega,F)} (m_{i _2j _2}) $. In that case, by \eqref{injectivity condition before lexi} and \eqref{choice of vareps} we have that
\begin{equation}
\label{second step}
g_{(\phi, \omega',F)} (m_{i _1j _1})-g_{(\phi, \omega',F)} (m_{i _2j _2})=
\epsilon \left\|g_{(\phi, \omega,F)} (m_{i _1j _1})\right\|Q\left(\mathbf{v}_{i _1 j _1}-\mathbf{v}_{i _2 j _2}\right).
\end{equation}
Given that $\mathbf{v}_{i _1 j _1}-\mathbf{v}_{i _2 j _2} \neq {\bf 0}$ (notice, for instance, that $\left(\mathbf{v}_{i _1 j _1}-\mathbf{v}_{i _2 j _2}\right)_1=1/ \mathcal{L}(i _1, j _1)- 1/ \mathcal{L}(i _2, j _2)\neq 0$) and $Q$ is injective then  $Q \left(\mathbf{v}_{i _1 j _1}-\mathbf{v}_{i _2 j _2}\right)\neq {\bf 0}$ and hence $g_{(\phi, \omega',F)} (m_{i _1j _1})\neq g_{(\phi, \omega',F)} (m_{i _2j _2}) $ necessarily by \eqref{second step}. In the case $g_{(\phi, \omega,F)} (m_{i _1j _1})\neq g_{(\phi, \omega,F)} (m_{i _2j _2}) $ the same conclusion can be drawn because the choice of $\epsilon>0$ in \eqref{choice of epsil} guarantees  that
\begin{equation*}
\left\|Q\left(\varepsilon_{i _1 j _1}\mathbf{v}_{i _1 j _1}-\varepsilon_{i _2 j _2}\mathbf{v}_{i _2 j _2}\right)\right\|<   \left\|g_{(\phi, \omega,F)} (m_{i _1j _1})-g_{(\phi, \omega,F)} (m_{i _2j _2})\right\|
\end{equation*}
which by \eqref{injectivity condition before lexi} ensures that, again, $g_{(\phi, \omega',F)} (m_{i _1j _1})\neq g_{(\phi, \omega',F)} (m_{i _2j _2}) $, as required.

We now show that if $\left. f_{(\phi, \omega,F)}\right|_{M _P}$ is injective then there exists an open set $V_P$ such that $M _P \subset V _P  $ and $\left. f_{(\phi, \omega,F)}\right|_{\overline{V _P}}$ is also injective.
By the Immersion Theorem, we know that there exists $n \in \mathbb{N}  $ such that the balls $B_{2^{-n}}(m _{ij})$ do not intersect and that the restriction of $f_{(\phi, \omega,F)} $ to each of them is a collection of injective maps. It could still be, however, that the images of different balls intersect. The continuity of $f_{(\phi, \omega,F)} $ and the fact that $\left. f_{(\phi, \omega,F)}\right|_{M _P}$ is injective implies that $n$ can be chosen sufficiently large so that this does not happen. Indeed, if this was not the case for the balls around the periodic points, say, $m_{i _1j _1} $ and $m_{i _2j _2} $, then it would be possible to construct two sequences $\{m_{i _1j _1,l}\}_{l \in \mathbb{N}} $ and $\{m_{i _2j _2, l}\}_{l \in \mathbb{N}} $ with limits $m_{i _1j _1} $ and $m_{i _2j _2} $ for which $f_{(\phi, \omega,F)}(m_{i _1j _1,l})=f_{(\phi, \omega,F)}(m_{i _2j _2,l}) $ for each $ l \in \mathbb{N}  $. By continuity this implies that $f_{(\phi, \omega,F)}(m_{i _1j _1})=f_{(\phi, \omega,F)}(m_{i _2j _2}) $ which is in contradiction with the injectivity of $\left. f_{(\phi, \omega,F)}\right|_{M _P}$ and hence proves the injectivity of $ f_{(\phi, \omega,F)} $ restricted to $V _P = \bigcup_{ij}B_{2^{-n}}(m _{ij}) $, with $n$ chosen so that the properties of the corresponding balls designed above are satisfied. Notice that by doubling $n$, if necessary, it is also easy to ensure the injectivity of $\left. f_{(\phi, \omega,F)}\right|_{\overline{V _P}}$.

\medskip

\noindent {\bf Step 2: Global injectivity.} This final step is quite involved and appeals to several results in differential topology. We will state the results here:

\begin{lemma}
\label{lemma1 with r}
If $M$ is a compact differentiable manifold endowed with a metric $d$ and $f : M \to \mathbb{R}^N$ is an immersion, then there exists a constant $r > 0$ such that for any $m \in M$ the restriction $f\rvert_{B_r(m)}$ of $f$ to the open ball $B_r(m) \subset M$ of radius $r$ and center $m$ is injective.
\end{lemma}

\begin{proof}
The Immersion Theorem (\cite[Theorem 3.5.7]{mta}) implies that  each $m \in M$ has an open neighborhood $U_m \subset M$ such that $f\rvert_{U_m}$ is injective. The collection of sets $\{ U_m\}_{ m \in M }$ forms an open cover of $M$. Then, by Lebesgue's number lemma \cite[Lemma 27.5]{Munkres:topology}, there exists a $\delta > 0$ such that every set of diameter $\delta$ is contained in some set in the family $\{ U_m\}_{ m \in M }$. The lemma is proved by choosing  $r = \delta/2$. 
\end{proof}

\medskip

\begin{defn}
\label{defn::reg_val}
\cite[Regular Value, Definition 3.5.3]{mta}
    Suppose $M$ and $M$ are smooth manifolds and $f \in C^r(M,M')$ for $r \geq 1$. A point $m' \in M'$ is called a \emph{regular value} of $f$ if for each $m \in f^{-1}\{m'\}$, the tangent map $T_m f$ is surjective with split kernel.
\end{defn}

\begin{theorem}
\label{theorem::submersion}
    \cite[Submersion Theorem, Theorem 3.5.4]{mta} Let $M,M'$ be smooth manifolds and $f \in C^{\infty}(M,M')$. Let $m'$ be a regular value. Then the level set
    \begin{align*}
        f^{-1}(m') = \{ m \in M \ | \ f(m) = m'  \}
    \end{align*}
    is a closed submanifold of $M$.
\end{theorem}

\begin{theorem}
\label{theorem::hirsch}
    \cite[Chapter 3, Proposition 1.2]{Hirsch:book} Let $M,M'$ be smooth manifolds with $\text{dim(M)} < \text{dim}(M')$. If $f \in C^1(M,M')$ then $M' \backslash f(M)$ is dense in $M'$.
\end{theorem}

We are now ready to prove global injectivity, and complete the proof.
Since $M$ is compact and $f_{(\phi, \omega, F)} : M \to \mathbb{R}^N$ is an immersion, Lemma \ref{lemma1 with r} implies the existence of a constant $r > 0$ such that for any $m \in M$ the restriction $f_{(\phi, \omega, F)}\rvert_{B_r(m)}$ of $f_{(\phi, \omega, F)}$ to the open ball $B_r(m)$ is  injective.
We now define the set $W \subset M \times M$ as follows using the open set $V_P$ whose existence we proved in Step 1.

\begin{equation*}
W := \{ (m_1 , m_2) \in (M \times M) \setminus (V_P \times V_P) \mid \ d(m_1,m_2) \geq r \}.
\end{equation*}
The set $W$ comprises pairs $(m_1,m_2) \in M$ whose entries satisfy one of two conditions:
\begin{enumerate}
\item Neither $m_1$ nor $m_2$ are in $V_P$.
\item One of $m_1$ and $m_2$ is in $V_P$ and the other is not.
\end{enumerate}
In view of this, the injectivity of $f_{(\phi, \omega, F)}\rvert_{V_P}$ proved in the Step 1 together with Lemma \ref{lemma1 with r} imply that if we show that $f_{(\phi, \omega, F)}(m _1)\neq f_{(\phi, \omega, F)}(m _2)$ for all $(m _1, m _2) \in W $ then $f_{(\phi, \omega, F)} $ is globally injective and the proof is concluded.

We start the proof of this fact by first defining, for each $m \in M$,  a collection of nested balls $\{ B_{2^{-n}}(m) \mid n \in \mathbb{N} \}$ centered at $m$ with radius $2^{-n}$. Let $(m_1,m_2) \in W$, and assume from now on without loss of generality that $m_1 \in W \setminus V_p$. Let $T(n,m_1,m_2)$ denote the largest integer such that the following two properties hold. Firstly, the sets
\begin{align*}
    \{B_{2^{-n}}(\phi^{-t}(m_1))\}_{t=0, \ldots, T(n,m_1,m_2)-1}
\end{align*}
are disjoint and secondly
\begin{align*}
    B_{2^{-n}}(\phi^{-t}(m_1)) \cap B_{2^{-n}}(\phi^{-s}(m_2)) = \emptyset \quad \mbox{for all} \quad t,s \in \{ 0 , \ldots , T(n,m_1,m_2) - 1\}.
\end{align*}
Notice now that by the continuity of $\phi$, for each $n \in \mathbb{N}$ and pair $(m_1,m_2) \in W$ there is an open neighbourhood $U_{(m_1,m_2)} \subset M \times M$ of $(m_1,m_2)$ such that $T(n,m_1',m_2') = T(n,m_1'',m_2'')$ for all $(m_1',m_2'),(m_1'',m_2'') \in U_{(m_1,m_2)}$. The collection $\{ U_{(m_1,m_2)} \ | \ (m_1,m_2) \in W \}$ covers $W$ and since it is a compact set we can extract a finite subcover $\{ U_a \ | \ a \in \mathcal{A} \}$, where $\mathcal{A}$ is a finite set. Then we can choose one pair $(m_1^a,m_2^a) \in U_a$ for each $a \in \mathcal{A}$ and notice that 
\begin{align*}
    \min_{(m_1,m_2) \in W} \{T(n,m_1,m_2)\} = \min_{\{(m_1^a,m_2^a) \ | \ a \in A\}} \{T(n,m_1^a,m_2^a)\}.
\end{align*}
The importance of this equality is that, since $\mathcal{A}$ is a finite set, the minimum on the right hand side is realized by a pair $(m_1^*,m_2^*) \in W$. Let $T(n) = T(n,m_1^*,m_2^*) = \min_{(m_1,m_2) \in W} T(n,m_1,m_2)$.
Observe that as $n \to \infty$ the families $\{ B_{2^{-n}}(\phi^{-t}(m_1^*)) \}_{t \in \mathbb{N}}$ and $\{ B_{2^{-n}}(\phi^{-t}(m_2^*)) \}_{t \in \mathbb{N}}$ converge to $\{\phi^{-t}(m_1^*)\}_{t \in \mathbb{N}}$ and $\{\phi^{-t}(m_2^*)\}_{t \in \mathbb{N}}$ respectively. The point $m_1^*$ is not periodic so the infinite orbit $\{\phi^{-t}(m_1^*)\}_{t \in \mathbb{N}}$ of singletons is disjoint, and furthermore does not intersect any point in $\{\phi^{-t}(m_2^*)\}_{t \in \mathbb{N}}$. This allows us to conclude that $T(n) \to \infty$ as $n \to \infty$.

The fact that we just proved guarantees the existence of a $\nu \in \mathbb{N}$ such that $T(\nu) = N$. Thus for all pairs $(m_1,m_2) \in W$, the collection
\begin{align*}
    \{B_{2^{-\nu}}(\phi^{-t}(m_1))\}_{t=0, \ldots, N-1}
\end{align*}
is disjoint and 
\begin{align*}
    B_{2^{-n}}(\phi^{-t}(m_1)) \cap B_{2^{-n}}(\phi^{-s}(m_2)) = \emptyset \quad \mbox{for all} \quad t,s \in \{ 0 , \ldots , N - 1\}.
\end{align*}
Now for any $n > \nu$ the collection
\begin{align*}
    \mathcal{C}_n = \{ B_{2^{-(n+1)}}(m) \ | \ m \in M \}
\end{align*}
forms on open cover of $M$ from which we can extract a finite subcover $\{B_i \ | \ i \in J_n\}$ for $J$ a finite set with cardinality  $\ell(n) \in \mathbb{N}$. Now define a partition of unity $\{ \lambda_i \ | \ i \in J_n \}$ subordinate to $\{B_i \ | \ i \in J_n\}$. We impose on this partition of unity the special property that for each $m \in M$ there exists an $i \in J_n$ such that $\lambda_i(m) \geq 1/2$. Now we define the perturbed observation function
\begin{align*}
    \omega_n = \omega + \sum^{\ell(n)}_{i=1} \epsilon_i \lambda_i
\end{align*}
where $\epsilon_i \in \mathbb{R}$ is the $i$th component of a vector $\boldsymbol{\epsilon} \in \mathbb{R}^{\ell(n)}$ with positive entries. Then we define $\Psi_n : M \times M \times \mathbb{R}^{\ell(n)} \to \mathbb{R}^N$ by
\begin{equation}
\label{perturbation lambda}
\Psi_n(m_1,m_2,\epsilon) = f_{(\phi, \omega_n, F)}(m_1) - f_{(\phi, \omega_n, F)}(m_2).
\end{equation}
Let $\Delta = \{ (m,m) \in M\times M \mid m \in M \}$ be the diagonal set. Given an arbitrary open neighborhood $\mathcal{N} \subset C^1(M,\mathbb{R})$ of the observation function $\omega \in C^2(M,\mathbb{R})$
our goal is to find $\boldsymbol{\epsilon} \in \mathbb{R}^{\ell(n)}$ such that $\omega_n \in \mathcal{N}$ and that for all $(m_1,m_2) \in (M \times M) \setminus \Delta$ we have that $\Psi_n(m_1,m_2,\boldsymbol{\epsilon}) \neq {\bf 0}$.

First of all, we observe that for any pair $(m_1,m_2) \in (M \times M) \setminus W$ either $d(m_1,m_2) < r$ or both $m_1,m_2 \in V_P$. In the former case, $\Psi_n(m_1,m_2,{\bf 0}) \neq  {\bf 0}$ unless $(m_1,m_2) \in \Delta$ by Lemma \ref{lemma1 with r}, and in the latter case, $\Psi_n(m_1,m_2,{\bf 0}) \neq {\bf 0}$ unless $(m_1,m_2) \in \Delta$ because $f_{(\phi, \omega, F)}\rvert_{V_P}$ is injective by the Step 1. Now $\Psi_n$ is continuous so there is an open neighbourhood $U_{{\bf 0}} \subset \mathbb{R}^{\ell(n)}$ of ${\bf 0} \in \mathbb{R}^{\ell(n)}$ such that for all $\boldsymbol{\epsilon} \in U_{{\bf 0}}$ we have $\Psi_n(m_1,m_2,\boldsymbol{\epsilon}) \neq {\bf 0}$ for all $(m_1,m_2) \in (M \times M) \setminus W$ unless $(m_1,m_2) \in \Delta$. So all that remains is to find  $\boldsymbol{\epsilon} \in U_{{\bf 0}} \subset \mathbb{R}^{\ell(n)}$ such that $\Psi_n(m_1,m_2,\boldsymbol{\epsilon}) \neq 0$ for all $(m_1,m_2) \in W$.

We start by noting that if ${\bf 0} \in \mathbb{R} ^N $  is not in the image of $\Psi_n \rvert_{W \times \{{\bf 0}\}}$ then we are done so we shall assume the opposite. In that case we proceed by showing that  $\Psi_n \rvert_{W \times \{{\bf 0}\}}$ is a submersion. If that is the case, then for some open set $X \subset (M \times M \times \mathbb{R}^{\ell(n)})$ containing $W \times \{{\bf 0}\}$ then the restriction $\Psi_n \rvert_{X}$ is also a submersion. Then by the Submersion Theorem \ref{theorem::submersion} the inverse image $\Psi_n\lvert_{X}^{-1}\left({\bf 0}\right)$ is a closed submanifold of dimension $2q+\ell(n) - N$ of the open submanifold $X\subset M \times  M \times \mathbb{R}^{\ell}$.

Moreover if $\pi : M \times M \times \mathbb{R}^{\ell(n)} \to \mathbb{R}^{\ell(n)}$ is the canonical projection defined by $\pi(m_1,m_2,\boldsymbol{\epsilon}) := \boldsymbol{\epsilon}$,  in these circumstances the complement $\mathbb{R}^{\ell(n)} \setminus \pi\left(\Psi_n\rvert_{X}^{-1}({\bf 0})\right)$ is a dense subset of $\mathbb{R}^{\ell(n)}$. Indeed, since $\pi$ is a continuously differentiable map, then so is its restriction  $\pi\rvert_{\Psi_n\lvert_{X}^{-1}\left({\bf 0}\right)\times \mathbb{R}^{\ell}}: \Psi_n\lvert_{X}^{-1}\left({\bf 0}\right)\to \mathbb{R}^{\ell} $ which by Theorem \ref{theorem::hirsch} guarantees the density of $\mathbb{R}^{\ell(n)} \setminus \pi\left(\Psi_n\rvert_{X}^{-1}({\bf 0})\right)$.  This implies that we can choose $\boldsymbol{\epsilon} \in \left(\mathbb{R}^{\ell(n)} \setminus \pi \left(\Psi_n\rvert_{X}^{-1}({\bf 0})\right)\right)$ as small as we want so that $\boldsymbol{\epsilon} \in U_{{\bf 0}}$ and $\omega_n \in \mathcal{N}$. We fix this $\epsilon$ and see that for any $(m_1,m_2) \in W$ the map $\Psi_n(m_1,m_2,\boldsymbol{\epsilon}) \neq  0 $, as required. 
Consequently, all that remains to be done is to find $n$ sufficiently large so that $\Psi_n \rvert_{W \times \{{\bf 0}\}}$ is a submersion, and then the proof will be complete.

We start by observing that by \eqref{perturbation lambda} 
\begin{equation*}
\omega_n \circ \phi^{-t} = \omega \circ \phi^{-t} + \sum_{i = 1}^{\ell(n)} \epsilon_i \lambda_i \circ \phi^{-t}
\end{equation*}
hence 
\begin{equation*}
\frac{\partial (\omega_n\circ  \phi^{-t})}{\partial \epsilon_j} = \lambda_j \circ \phi^{-t}.
\end{equation*}
Now we consider an arbitrary $(m_1,m_2) \in W$ assuming once again without loss of generality that $m_1 \in M \setminus V_p.$ For each point in the orbit $\{ \phi^{-t}(m_1) \}_{t = 0, \ldots , T(n)-1}$ there exists a $j(t) \in J_n$ such that $\lambda_{j(t)}\left( \phi^{-t}(m_1) \right)\geq 1/2$ by the special property that we imposed earlier  on the partition of unity $\{ \lambda_i \ | \ i \in J _n \}$. Now the support of $\lambda_{j(t)}$ is a ball $B_{j(t)}$ of radius $2^{-(n+1)}$ which contains $\phi^{-t}(m_1)$. Hence the ball $B_{j(t)} \subset B_{2^{-n}}(\phi^{-t}(m_1))$. Now since the sets in the family $\{B_{2^{-n}}(\phi^{-t}(m_1))\}_{t = 0 , \ldots , T(n) - 1}$ are disjoint then so are $\{ B_{j(t)} \}_{t = 0 , \ldots , T(n) - 1}$. 
Furthermore, since $B_{2^{-n}}(\phi^{-t}(m_1)) \cap B_{2^{-n}}(\phi^{-s}(m_2)) = \emptyset$ for all $t,s \in \{ 0 , \ldots , T(n) - 1\}$ hence $\lambda_{j(t)}(\phi^{-t}(m_2)) = 0$ for $t \in \{ 0, \ldots, T(n) - 1\}$. Thus
\begin{align*}
\frac{\partial (\omega_n\circ  \phi^{-t})}{\partial \epsilon_{j(t)}}(m_2) = 0.
\end{align*}
Now,
\begin{align*}
\Psi_n(m_1,m_2,\epsilon) &= \sum_{t = 0}^{T(n)-1} A^t C (\omega_n (\phi^{-t}(m_1)) - \omega_n (\phi^{-t}(m_2)) ) \\ &+
\sum_{t = T(n)}^{\infty} A^t C (\omega_n( \phi^{-t}(m_1)) - \omega_n( \phi^{-t}(m_2) ))
\end{align*}
hence for $t = 0, \ldots, T(n) - 1$
\begin{align*}
\frac{\partial \Psi_n}{\partial \epsilon_{j(t)}}(m_1,m_2,\epsilon) &= A^t C (\lambda_{j(t)} (\phi^{-t}(m_1))) \\ &+
\sum_{t = T(n)}^{\infty} A^t C (\lambda_{j(t)} (\phi^{-t}(m_1)) - \lambda_{j(t)}( \phi^{-t}(m_2) ) ). 
\end{align*}
By assumption $\{A^t C\}_{t = 0, ... , N-1}$ are linearly independent, hence the vectors 
\begin{align*}
    \{A^t C (\lambda_{j(t)} (\phi^{-t}(m_1)))\}_{t \in \{ 0 , \ldots, T(n)-1\}}
\end{align*}

necessarily span $\mathbb{R}^N$ because since $n> \nu  $  then $T(n)\geq N $. Crucially, for any $n$ the property $\lambda_{j(t)} (\phi^{-t}(m_1)) \geq 1/2$ holds and therefore, the residual term
\begin{align*}
\sum_{t = T(n)}^{\infty} A^t C (\lambda_{j(t)} \phi^{-t}(m_1) - \lambda_{j(t)} \phi^{-t}(m_2) )
\end{align*}
may only spoil the spanning property of the vectors $\{A^t C (\lambda_{j(t)} \phi^{-t}(m_1))\}_{t = 0, \ldots , N-1}$ if it is sufficiently large. Since by hypothesis $\rho(A)<1 $, the residual term converges uniformly over $(m_1,m_2) \in W$ to $0$ as $n$ grows. We choose consequently $n$ large enough so that for all $(m_1,m_2) \in W$ the residual term is too small to spoil the spanning property of $\{A^t C (\lambda_{j(t)} \phi^{-t}(m_1))\}_{t \in \{ 0, \ldots , N-1\}}$. With this choice of $n$ we have that $\Psi_n\rvert_{W \times \{ 0 \}}$ is a submersion and the proof is complete.
\end{proof}

\subsection{An embedding almost surely}

In this subsection we will prove Theorem \ref{random_thm}. 

\begin{proof}
    First we will establish two lemmas about random matrices. 
    \begin{lemma}
\label{first lemma polynomial}
Let $X_1, \ldots , X_n$ be independent real-valued non-singular random variables and let $p$ be a non-constant polynomial in $n$ complex variables. Then
\begin{align*}
    \mathbb{P}\left(p(X_1, \ldots, X_n) = 0\right) = 0.
\end{align*}
\label{non_zero_polynomial_lemma}
\end{lemma}

\begin{proof}
    Define $\mu_j(\cdot) := \mathbb{P}(X_j \in  \cdot),$ $j = 1, \ldots, n$, and let $Z = \{ \mathbf{x} \in \mathbb{C}^n \mid \ p(\mathbf{x}) = 0 \}$ be the set of complex roots of the polynomial $p$. Then, since  $X_1, \ldots , X_n$ are independent we have that
\begin{align*}
    \mathbb{P}\left(p(X_1, \ldots , X_n) = 0\right) = \mathbb{P}\left((X_1 , \ldots , X_n) \in Z\right) = (\mu_1 \ldots  \mu_n)(Z).
\end{align*}
We now proceed by induction over $n$. For $n = 1$ we have that $\mathbb{P}\left(p(X_1) = 0\right) = \mu_1(Z) = 0$ since $Z$ is finite and $X_1$ is non-singular. Let the claim be true for $n-1$. For fixed $x_1 \in \mathbb{C}$ set $p_{x_1}(x_2 , \ldots , x_n) := p(x_1 , \ldots , x_n)$ and let
\begin{align*}
    Z_{x_1} := \{ (x_2 , \ldots , x_n) \in \mathbb{C}^{n-1} \mid \ p_{x_1}(x_2 , \ldots , x_n) = 0\}.
\end{align*}
The set $F := \{ x_1 \in \mathbb{R} \mid \ p_{x_1} \equiv 0 \}$ is a finite set, so
\begin{equation*}
    \mathbb{P}\left(p(X_1, \ldots , X_n) = 0\right) = \int_{\mathbb{C}}(\mu_2  \ldots  \mu_n)(Z_{x_1}) d \mu_1(x_1) = \int_{\mathbb{C} - F} (\mu_2  \ldots  \mu_n)(Z_{x_1}) d \mu_1(x_1) = 0,
\end{equation*}
since we assumed that $X_1$ is non-singular and $(\mu_2  , \ldots ,  \mu_n)(Z_{x_1}) = 0$ for $x_1 \notin F$ by the induction hypothesis.
\end{proof}

\begin{lemma}
\label{polynomial_lemma}
Let $N \in \mathbb{R}^N $, let  $A$ be a real $N \times N$ matrix, and let $C$ be a random vector in $\mathbb{R}^N$.  Assume the entries of $A$ and $C$ have been drawn using independent non-singular real valued random variables. Moreover, let $p_1 , \ldots , p_n \in \mathbb{C}[x]$ be linearly independent polynomials in one variable of degree at most $n-1$. Then
\begin{align*}
    \mathbb{P}(\det\left(p_1(A)C, p_2(A)C, \cdots, p_n(A)C\right) = 0) = 0.
\end{align*}
Equivalently, the vectors $p_1(A)C , p_2(A)C , \ldots , p_n(A)C$ are linearly independent almost surely.
\end{lemma}

\begin{proof} 
The expression $\det\left(p_1(A)C, \cdots, p_n(A)C \right)$ is a polynomial $p$ in the $n^2 + n$ variables $a_{ij}$ and $b_j$, $i,j \in \left\{ 1 , \ldots , n\right\}$ that constitute the entries of $A$ and $C $, respectively, and which in turn are by hypothesis non-singular random variables. As long as the polynomial $p$ is not identically zero, the result follows directly from the Lemma \ref{non_zero_polynomial_lemma}. So all that remains is to show that $p$ is not identically zero, that is, that there exist particular choices of $A$ and $C $ such that $\det\left(p_1(A)C, \cdots, p_n(A)C\right)$ is non-zero. So, we choose $C = (1 , \ldots , 1)^{\top}$ and $A = \text{diag}(a_1 , \ldots , a_n)$ with distinct real numbers $a_1 , \ldots , a_n$. We expand the polynomials $p_j$ in terms of their coefficients $\gamma_{j1} , \ldots , \gamma_{jn}$ so
\begin{align*}
    p_j(x) = \sum_{k = 0}^{n-1} \gamma_{jk}x^k.
\end{align*}
The vectors $\gamma_j = (\gamma_{j1}, \ldots , \gamma_{jn})^{\top}$, $j \in \left\{1, \ldots , n\right\}$, are by hypothesis linearly independent. We now show that with these choices, the vectors $p_1(A)C , p_2(A)C , \ldots , p_n(A)C$ are linearly independent and hence $\det\left(p_1(A)C, p_2(A)C, \cdots, p_n(A)C\right) \neq 0$, as required. Indeed, let $c_1 , \ldots , c_n \in \mathbb{R}$ and suppose that $\sum_{j=1}^n c_j p_j(A) C ={\bf 0}$. Additionally, we can write
\begin{align*}
    \sum_{j=1}^n c_j p_j(A) C = \sum_{j=1}^n c_j \sum_{k=0}^{n-1} \gamma_{kj} A^k C = \sum_{k=0}^{n-1} \bigg( \sum^n_{j=1} c_j \gamma_{jk} \bigg) 
    \begin{bmatrix}
        a_1^k \\ a_2^k \\ \vdots \\ a_n^k
    \end{bmatrix}
    = Vx
\end{align*}
where 
\begin{equation*}
    V = 
\left(
\begin{array}{ccccc}
        1 & a_1 & a_1^2 & \ldots & a_1^{n-1} \\
        1 & a_2 & a_2^2 & \ldots & a_2^{n-1} \\
        \vdots & \vdots & \vdots & \ddots & \vdots \\
        1 & a_n & a_n^2 & \ldots & a_n^{n-1}
\end{array}
\right)
\quad \mbox{and} \quad
    x = \sum^n_{j=1} c_j \gamma_{j}.
\end{equation*}
Since the diagonal entries of $A$ are all different and the determinant of the Vandermonde matrix $V$ is given by 
\begin{equation*}
\det (V)=\prod_{1\leq i<j\leq n} \left(a _i- a _j\right)
\end{equation*}
we can conclude that $V$ is invertible and hence the identity $\sum_{j=1}^n c_j p_j(A) C =V x = {\bf 0}$ implies that $x= {\bf 0} $. By the linear independence of the vectors $\gamma_1, \ldots, \gamma _n $ we have that $c _1, \ldots, c _n=0 $ necessarily.
It follows that the vectors $p_1(A)C , p_2(A)C , \ldots , p_n(A)C$ are linearly independent, as required. 
\end{proof}

Now we may complete the proof of Theorem \ref{random_thm}. The vectors $C, AC, A^2 C, \ldots , A^{N-1}C$ are linearly independent if and only if
\begin{align*}
    \text{det}\left(C, AC, A^2C,  \cdots , A^{N-1}C\right) = 0
\end{align*} 
which, in the notation of Lemma \ref{polynomial_lemma}, can be written as
\begin{align*}
    \text{det}\left(p_0(A)C,p_1(A)C,p_2(A)C, \ldots , p_{N-1}(A)C\right) = 0
\end{align*}
 using the linearly independent polynomials $p_j(A) := A^j$, $j \in \left\{0, \ldots, N-1\right\}$. Part {\bf (i)} of the statement hence follows directly from Lemma \ref{polynomial_lemma}. Now we turn our attention to part {\bf (ii)}. First of all, $\lambda_j$ is an eigenvalue of $A$ if and only if $\lambda_j$ is a root of the characteristic polynomial of $A$. This event has probability $0$ by Lemma \ref{first lemma polynomial} and hence $1, \lambda_1 , \ldots , \lambda_q \notin \sigma(A)$ almost surely. On this event, the inverses $(I - \lambda_j A)^{-1}$ and $(I - A)^{-1}$ exist. Furthermore, the product
 \begin{align*}
     \prod_{i=1}^{q}(I - \lambda_i A)
 \end{align*}
 is an invertible matrix. Therefore, the vectors 
 \begin{align*}
      (I - \lambda_j A)^{-1}(I - A)^{-1}(I - A^N)C, \quad \mbox{with  $j = 1, ... , m$,} 
 \end{align*}
are linearly independent if and only if 
 \begin{align}
\label{vecs}
     &\prod_{i=1}^{q}(I - \lambda_i A)(I - \lambda_j A)^{-1}(I - A)^{-1}(I - A^N)C,  \quad \mbox{with  $j = 1, ... , m$,} 
 \end{align}
are linearly independent.
We can now rewrite the vectors in \eqref{vecs} as
 \begin{align*}
    \prod_{i=1}^{q}(I - \lambda_i A)(I - \lambda_j A)^{-1}(I - A)^{-1}(I - A^N)C 
    &=\prod_{i \neq j}^{q}(I - \lambda_i A)(I - A)^{-1}(I - A^N)C \\
     &= \prod_{i \neq j}^{q}(I - \lambda_i A)\sum_{k=0}^{N-1} A^k C,
 \end{align*}
where we used the relation 
\begin{equation*}
(I - A^N)=(I-A)\sum_{k=0}^{N-1}A ^k \quad \mbox{and hence that} \quad (I-A)^{-1}(I - A^N)=\sum_{k=0}^{N-1}A ^k.
\end{equation*}
Now, if we are able to show that the family
 \begin{align*}
     p_j(x) = \prod_{i \neq j}^{q}(1- \lambda_i x)\sum_{k=0}^{N-1} x^k , \quad \mbox{with } \quad j \in \left\{ 1, \ldots , m\right\}
 \end{align*}
is linearly independent, then we can conclude by Lemma \ref{polynomial_lemma} that the vectors \eqref{vecs} are linearly independent almost surely, which would complete the proof. This is indeed the case because if $\mu_1, \ldots, \mu_n \in \mathbb{R} $ are such that 
\begin{equation*}
\sum_{j=1}^n \mu _jp _j(x)=0 \quad \mbox{then} \quad  \left(\sum_{k=0}^{N-1} x^k\right)\left(\sum_{j=1}^n \mu _j\prod_{i \neq j}^{q}(1- \lambda_i x)\right)=0.
\end{equation*}
Given that the polynomial $\sum_{k=0}^{N-1} x^k $ is non-zero, the previous equality is equivalent to $\sum_{j=1}^n \mu _j\prod_{i \neq j}^{m}(1- \lambda_i x)=0 $ which, evaluated at $x=1/ \lambda _k  $, implies that
\begin{equation*}
0=\sum_{j=1}^n \mu _j\prod_{i \neq j}^{q}\left(1- \lambda_i \frac{1}{\lambda_k}\right)=\mu _k\prod_{i \neq k}^{q}\left(1- \lambda_i \frac{1}{\lambda_k}\right).
\end{equation*}
Given that, by hypothesis, the values $\lambda_1, \ldots, \lambda_q \in \mathbb{C}$ are all different, we can conclude that $\prod_{i \neq k}^{q}\left(1-  \frac{\lambda_i}{\lambda_k}\right)\neq 0 $ and hence $\mu_k=0 $, necessarily. Since procedure can be repeated to obtain that $\mu_1, \ldots, \mu_n =0$, the result follows. 
\end{proof}

After all this work, the proofs are complete. In summary: we have shown that reservoir maps of the form
\begin{align*}
    F(x,z) = Ax + Cz + b
\end{align*}
with randomly generated weights $A,C$ and biases $b$ admit a GS $f_{(\phi,\omega,F)}$ when driven by observations of a discrete time dynamical system $(M,\phi)$. If the dynamical system $(M,\phi)$ is well behaved at the fixed points, then for generic observation functions $\omega \in C^2(M,\mathbb{R})$, the GS $f_{(\phi,\omega,F)}$ is an embedding almost surely.

\chapter{Universal Approximation}
\label{chapter::universal_approximation}

We saw in Chapter \ref{chapter::introduction} that an ESN does a good job predicting the future trajectory of the Lorenz system, whose evolution is nonlinear. Furthermore, there is a large literature attesting to the performance of ESNs when tasked with learning highly nonlinear relationships between datasets and targets \citep{TANAKA2019100}. This suggests that ESNs may have some sort of universal approximation property. There is recent literature analysing the universal approximation of ESNs and reservoir maps more generally -  with different authors taking different approaches. One strand of work developed in \cite{GRIGORYEVA2018495}, \cite{GONON2020132721}, \cite{Gonon2020}, and \cite{GONON202110} presents universal approximation results in terms of filters and functionals (which are presented in this thesis in Chapter \ref{chapter::stochastic}). This builds upon the work by \cite{Matthews1993} who establishes the universal approximation of fading memory systems supported by neural networks. Many of these results hold for arbitrary input sequences, without any reference to the process that generated them. This is in contrast to the results that will appear in this chapter, which will hold for input sequences generated by an ergodic and determinstic dynamical system. Many of the results presented here also hold for stationary and ergodic random processes - and this is explored in Chapter \ref{chapter::stochastic}. The presentation of this present chapter closely follows \cite{embedding_and_approximation_theorems} and \cite{HART2021132882}.

\section{Universal Approximation Theorems}

Many (perhaps all) of the major approximation results for Echo State Networks and similar reservoir maps are built upon Hornik's (1990) universal approximation theorem for feedforward neural networks. The result states, under mild assumptions on the activation function $\sigma$, that feedforward neural networks are dense in the set of differentiable functions. To make this formal, we'll first introduce the mild assumption on $\sigma$.

\begin{defn}
    ($\ell$-finite) Let $\ell \in \mathbb{N}_0$. Then we say an $\ell$-times differentiable scalar function $\sigma \in C^{\ell}(\mathbb{R})$ is $\ell$-finite if
    \begin{align*}
        0 < \int_{\mathbb{R}} \bigg| \frac{d^{\ell} \sigma}{d x^{\ell}} \bigg| dx < \infty. 
    \end{align*}
\end{defn}

\begin{theorem}
    (Universal Approximation Theorem, \cite{HORNIK1990551}) Let $K \subset \mathbb{R}^n$ be a compact subset. If the activation function $\sigma$ is $\ell$-finite, $W_1, \ldots W_N \in \mathbb{R}$ and $C_1, \ldots , C_N \in \mathbb{R}^n$, and $b_1, \ldots, b_N \in \mathbb{R}$  then, for all $0 \leq r \leq \ell$, the set of functions $g : K \to \mathbb{R}$ of the form
    \begin{align*}
        g(x) = \sum_{j = 1}^N W_j \sigma ( C_j^{\top}x + b_j)
    \end{align*}
    are dense in $C^{r}(K,\mathbb{R})$.
\end{theorem}

There are many other universal approximation theorems found throughout the literature \citep{Cybenko1989,NIPS2017_32cbf687} which typically show that neural networks (of some variety) equipped with an activation function (that satisfies mild conditions) are dense in some set of functions (with certain regularity conditions) between topological spaces (that satisfy reasonable conditions - usually subsets of $\mathbb{R}^n$). There are many novel results \citep{Gonon2020,arXiv:1910.03344,arXiv:2006.02341,arXiv:2101.05390,Gonon2021} which determine approximation bounds for the number of neurons required to approximate a sufficiently regular function. The bounds are often explicit up to a constant, which is sometimes computed exactly. Throughout this thesis we do not work with approximation bounds. This simplifies the presentation considerably, but comes at the cost of leaving us with no idea how many neurons are required for a given task.

For the results that follow we use Hornik's seminal result because it allows the pointwise approximation of a function $g$ in addition to its derivatives up to any order. This smooth approximation is of crucial importance for the topological results to come; but it is over-engineered for others. The first result we will present states that, if the internal weights $C_j$ and biases $b_j$ are randomly generated, there exists a choice of outer weights $W_j$ so that the network $g$ will approximate the target function as closely as is required. This result and preliminary lemmas appear in \cite{embedding_and_approximation_theorems}. A variant of the result with approximation bounds appears in \cite{Gonon2020}. 
    \begin{lemma}
\label{RUAT_lemma}
    Let $(X_j)_{j \in \mathbb{N}}$ be a sequence of i.i.d. random variables and $S_1, \ldots, S_{\ell}$ be a list of $\ell$ events, and suppose that for each $i$ (and for any $j$ since they are i.i.d.) there exists $\theta_i$ such that $\mathbb{P}(X_j \in S_i) = \theta_i > 0$. Then for all $\alpha \in (0,1)$ there exists $N \in \mathbb{N}$ such that
    \begin{align*}
        \mathbb{P}\big( \exists \text{ injective } \phi : \{ 1, \ldots, \ell \} \to \{ 1, \ldots, N \} : X_{\phi(i)} \in S_i, \ \ \forall \ i \in \{ 1, \ldots, \ell \} \big) > \alpha.
    \end{align*}
\end{lemma}

\begin{proof}
        First, fix $\alpha \in (0,1)$.
        %and let $\mathbb{P}(X_j \in S_i) := \theta_i \ \forall j$.
        Then define the set $\{n_0, \ldots, n_\ell\}$ as follows.
        Set $n_0 = 0$ and for any $i \in \{ 1 , ... , \ell \}$ let
        \begin{align*}
            n_i - n_{i-1} := \text{ceil}\bigg( \frac{\log(1 - \alpha^{1 /     \ell})}{\log(1-\theta_i)} \bigg) + 1. 
        \end{align*}
        Finally, set $N=n_{\ell}$. Then we can calculate that
        \ba 
        & & \mathbb{P} \big( \exists \text{ injective } \phi 
        %: \{ 1, \ldots, \ell \} \to \{ 1, \ldots, n \} 
        : X_{\phi(i)} \in S_i \ \forall i \in \{1, \ldots, \ell \} \big) 
        \nn \\
        %\phantom{a} \hspace{-2.5cm} 
        & > &
        \mathbb{P} \big(\forall \ i \in \{ 1, \ldots, \ell \} \ \exists \ j \in \{1 + n_{i-1}, \ldots, n_i \} : X_j \in S_i \big) \nn \\ 
        & = &
        \prod_{i=1}^{\ell} \mathbb{P}\big( \exists j \in \{1 + n_{i-1}, \ldots, n_i \} : X_j \in S_i \big) \nn \\ 
        & = &
        \prod_{i=1}^{\ell} 1 - \mathbb{P}\big( X_j \notin S_i \ \forall j \in \{1 + n_{i-1}, \ldots, n_i \} \big) \nn \\ 
        & \geq & 
        \prod_{i=1}^{\ell} 1 - (1 - \theta_i)^{n_i - n_{i-1}} \nn \\ 
        & = &
        \prod_{i=1}^{\ell} 1 - (1 - \theta_i)^{\text{ceil}\big(\log(1 - \alpha^{1 / \ell}) / \log(1-\theta_i) \big) + 1} \nn \\ 
        & > &
        \prod_{i=1}^{\ell} 1 - (1 - \theta_i)^{\big(\log(1 - \alpha^{1 / \ell}) / \log(1-\theta_i) \big)} \nn \\ 
        & = & 
        \prod_{i=1}^{\ell} 1 - \exp\bigg(\frac{\log(1 - \alpha^{1 / \ell})}{\log(1-\theta_i)}\log(1 - \theta_i)\bigg) \nn \\ 
        & = & \prod_{i=1}^{\ell} 1 - (1 - \alpha^{1/\ell}) = \prod_{i=1}^{\ell} \alpha^{1/\ell} = \alpha. \nn
    \ea
    
\end{proof}

\begin{theorem} (Random Universal Approximation Theorem)
    Let $K \subset \mathbb{R}^n$ be a compact subset and $f \in C^\ell(K , \mathbb{R})$.
    Let $\sigma \in C^\ell(\mathbb{R})$ be $\ell$-finite, and let $(b_j \in \mathbb{R})_{j \in \mathbb{N}}$, $(C_j \in \mathbb{R}^n)_{j \in \mathbb{N}}$ be sequences of i.i.d. random variables with full support.  Then for any $0 \leq r \leq \ell$, and $\alpha \in (0,1)$ and $\epsilon > 0$ there exists some natural number $N \in \mathbb{N}$ such that, with probability greater than $\alpha$, there exist real numbers $W_1, \ldots, W_N \in \mathbb{R}$ such that the \emph{random neural network} $g : K \to \mathbb{R}$ defined by
    \begin{align}
        g(x) = \sum_{j=1}^{N} W_j \sigma (C_j^{\top} x + b_j) \nn 
    \end{align}
    satisfies
    \begin{align}
    \lVert f - g \rVert_{C^r} < \epsilon. \nn
    \end{align}
    \label{theorem::RUAT}
\end{theorem}
\begin{proof}
    First, fix $r \geq 0$. Then by the Universal Approximation Theorem we know that for any $\epsilon > 0$ there exists a neural network $\hat{g} : K \to \mathbb{R}$ with $N'$ neurons defined by
\begin{align}
        \hat{g} (x) = \sum_{i = 1}^{N'} \hat{W}_i \sigma(\hat{C}_i^{\top} x + \hat{b}_i) \nn 
\end{align}
    such that
    \begin{align}
    \lVert f - \hat{g} \rVert_{C^r} < \frac{\epsilon}{2}. \label{eqn:fghat} 
    \end{align}
Now, consider two sequences of i.i.d. random variables $(b_j)_{j \in \mathbb{N}}$ and $(C_j)_{j \in \mathbb{N}}$ with full support, and let $X_j := (b_j , C_j)$. Fix $\epsilon>0$  and define a collection of $N'$ events $S_{1} , ... , S_{N'}$ by
\begin{align}
    S_i := \bigg\{ (b,C) \in \mathbb{R} \times \mathbb{R}^n : \lVert \sigma(\hat{C}_i^{\top} \cdot + \hat{b}_i) - \sigma(C^{\top} \cdot + b) \rVert_{C^r} < \frac{\epsilon}{2 N' \max_k(\hat{W}_k)} \bigg\}, \nn 
\end{align}
where the weights $\hat{W}_k$ are given by the form of the network $\hat{g}$.
Observe that each of the $S_i$ have strictly positive measure, so there exists
$\theta_i>0$ such that $\mathbb{P}(X_j \in S_i) > \theta_i > 0 \ \forall j \in \mathbb{N}$. 
Hence it follows by Lemma~\ref{RUAT_lemma} that for all $\alpha \in (0,1)$ there exists $N \in \mathbb{N}$ such that
    \begin{align}
        \mathbb{P}\big( \exists \text{ injective } \phi : \{ 1, \ldots, N' \} \to \{ 1, \ldots, N \} : X_{\phi(i)} \in S_i \ \forall i \in \{1, \ldots, N' \} \big) > \alpha. \nn 
    \end{align}
Now, assuming the event 
\begin{align}
    \exists \text{ injective } \phi : \{ 1, \ldots, N' \} \to \{ 1, \ldots, N \} : X_{\phi(i)} \in S_i \ \forall i \in \{1, \ldots, N' \} \nn 
\end{align}
occurs, we define
\begin{align}
    W_j := 
    \begin{cases}
        \hat{W}_i \text{ if } \phi(i) = j \\
        0 \text{ otherwise }
    \end{cases} \nn 
\end{align}
for all $j \in \{ 1, \ldots, N \}$, and define the \emph{random neural network} $g:K \to \mathbb{R}$ by
\begin{align}
    g(x) = \sum^{N}_{j=1}W_j \sigma( C_j^{\top} x + b_j ). \nn
\end{align}
Now observe that
\ba
\lVert \hat{g} - g \rVert_{C^1}
    & = & \bigg\lVert \sum_{i = 1}^{N'} \hat{W}_i \sigma(\hat{C}_i^{\top} \cdot + \hat{b}_i) - \sum_{j = 1}^{N}W_j\sigma(C^{\top}_j \cdot + b_j) \bigg\rVert_{C^r} \nn \\
    & = & \bigg\lVert \sum_{i = 1}^{N'} \hat{W}_i \big( \sigma ( \hat{C}_i^{\top} \cdot + \hat{b}_i ) - \sigma \big( C_{\phi(i)}^{\top} \cdot + b_{\phi(i)} \big)  \big) \bigg\rVert_{C^r} \nn \\ 
    & \leq & \sum_{i = 1}^{N'} \hat{W}_i \bigg\lVert \big( \sigma ( \hat{C}_i^{\top} \cdot + \hat{b}_i ) - \sigma \big( C_{\phi(i)}^{\top} \cdot + b_{\phi(i)} \big)  \big) \bigg\rVert_{C^r} \nn \\
    & < & \sum_{i = 1}^{N'} \frac{\hat{W}_i \epsilon}{2 N' \max_k(\hat{W}_k)} < \frac{\epsilon}{2}. \nn
\ea
Combining this with~\eqref{eqn:fghat} and using the triangle inequality we obtain
\ba 
    \lVert f - g \rVert_{C^1}
    \leq \lVert f - \hat{g} \rVert_{C^r} + \lVert \hat{g} - g \rVert_{C^r}  < \frac{\epsilon}{2} + \frac{\epsilon}{2} = \epsilon, \nn
\ea 
which completes the proof.
\end{proof}

The above result states that if we randomly generate the internal weights of a feedforward neural network with sufficiently many neurons, then there exists a choice of outer weights $W \in \mathbb{R}^N$ that will yield a sufficiently good approximation. This existence result is non constructive in the sense that 
\begin{enumerate}
    \item We do not know how many neurons $N$ are required for a given error $\epsilon$ or for a given probability $\alpha$.
    \item We do not know how to choose the output weights $W \in \mathbb{R}^N$.
\end{enumerate}
\cite{Gonon2020} have made significant progress on item $1.$, so in what follows, we will focus on items $2.$ In particular, we will consider an input sequence arising from the scalar observations of an ergodic dynamical system evolving on a manifold $M$. We will prove that for $W \in \mathbb{R}^N$ obtained by linear regression, as in Chapter \ref{chapter::introduction}, the approximation can be made arbitrarily good by increasing the number of neurons $N \in \mathbb{N}$ and data points $\ell \in \mathbb{N}$. The presentation in this chapter closely follows \cite{HART2021132882}. The results rely on ergodic theory - which we will introduce next.

\section{Ergodic Theory}
\label{Training_Thm_for_ESNs}

We can view an ergodic process, whether it is deterministic or random, as a process where the time average of some feature through phase space converges toward the spatial average of that feature throughout the phase space, almost surely.

For example, one could imagine a room that is heated to some constant, spatially inhomogeneous temperature. Perhaps the top of the room is warmer than the bottom as a consequence of heat rising, and one side, perhaps with single glazed windows, is cooler than the other. Nevertheless, the room has a well defined, constant, spatial average temperature. Now, we can view, under idealised assumptions, the average temperature of the room as proportional to root mean square velocity of the particles the room. If the particle dynamics are ergodic, we may observe a single arbitrary particle and track its velocity as it traverses the room. We expect the particle to explore the whole room, moving with greater velocity in the warm areas, and lower velocity in cool areas, such that the particle's root mean square velocity over time converges toward the spatial temperature average of the room. This convergence of the \emph{time average} of the particle's velocity to the \emph{space average} over all particles in the room is the crucial property of ergodic systems that we will exploit. In this chapter we are interested in deterministic ergodic systems, and in Chapter \ref{chapter::stochastic} we will extend the results the stationary and ergodic random processes.

We require that the underlying dynamical system is ergodic so that minimising the mean square differences between observations and targets closely approximates the mapping between the observation space and target space. The ergodicity ensures that that training data generated from a trajectory initialised at almost any point $m_0 \in M$ will represent all dynamics on $M$. To make this formal, we will introduce the definition of ergodicity and the celebrated Ergodic Theorem. 

\begin{defn}
    ($\mu$-generic point, \cite{genericpts}, Definition 4.4.1) Suppose $\phi : M \to M$ is a measure preserving map with respect to the probability space $(M,\Sigma,\mu)$. We say that $m_0$ is a $\mu$-generic point in $M$ if
    \begin{align}
        \lim_{\ell \to \infty}\frac{1}{\ell} \sum_{k=0}^{\ell-1} s \circ \phi^{k}(m_0) = \int_{M} s \ d\mu \nn
    \end{align}
    for all $s \in C^0(\mu)$.
\end{defn}

\begin{remark}
    A $\mu$-generic point is not to be confused with a generic property defined in Definition \ref{defn::generic_property}. The overlap in terminology is indeed annoying.
\end{remark}

\begin{defn}
    (Ergodic) Let $\phi : M \to M$ be a measure preserving transformation on the probability space $(M,\Sigma,\mu)$. Then $\phi$ is ergodic if for every $\sigma \in \Sigma$ with $\phi^{-1}(\sigma) = \sigma$ either $\mu(\sigma) = 0$ or $\mu(\sigma) = 1$.
\end{defn}

\begin{theorem}
    (Ergodic Theorem \citep{Birkhoff656}) Suppose $\phi : M \to M$ is ergodic with respect to the probability space $(M,\Sigma,\mu)$ and $s \in L^1(\mu)$. Then $\mu$-almost all $m_0 \in M$ are $\mu$-generic hence for $\mu$-almost all $m_0 \in M$
    \begin{align}
        \lim_{\ell \to \infty}\frac{1}{\ell} \sum_{k=0}^{\ell-1} s \circ \phi^{k}(m_0) = \int_{M} s \ d\mu. \label{Ergodic_Theorem_eqn}
    \end{align}
\end{theorem}

The left hand side of \eqref{Ergodic_Theorem_eqn} is called the \emph{time average} taken from the initial point $m_0 \in M$, and the right hand side is called the \emph{space average}. The Ergodic Theorem then states that the time average taken from almost all initial points equals the space average.

\section{A Training Theorem for Echo State Networks}
\label{training_theorem}
\subsection{Regularised Least Squares Regression}

Suppose we have have an ergodic dynamical system $\phi : M \to M$, and we observe the dynamics via an observation map $g : M \to \mathbb{R}^P$ for some fixed $P \in \mathbb{N}$ and target map $u : M \to \mathbb{R}$. A trajectory originating from a $\mu$-generic point $m_0 \in M$ will ergodically explore the space $M$ and yield a sequence of observations $g \circ \phi^{k}(m_0)$ and targets $u \circ \phi^{k}(m_0)$ for $k = 0, 1 , 2 , ... , \ell$. 

Suppose we compute the vectors $W_{\ell} \in \mathbb{R}^P$ minimising the regularised least squares difference between the mapping of the observations $W^{\top} g \circ \phi^k(m_0)$ and the targets $u \circ \phi^k(m_0)$. We prove in the next lemma that as the number of data points $\ell$ grows large, the least squares solution $W_{\ell}$ minimises the ergodic average difference between the mapping on the observations $W^{\top} g \circ \phi^k(m_0)$ and the targets $u \circ \phi^k(m_0)$.

\begin{lemma}
    Let $(M,\Sigma)$ be a measurable space, and suppose that $\phi : M \to M$ is ergodic with invariant measure $\mu$. Let $m_0$ be a $\mu$-generic point in $M$. Let $g \in L^2(\mu)(M,\mathbb{R}^P)$ be an observation function and suppose that $u \in L^2(\mu)(M, \mathbb{R})$ is a target function we wish to approximate. 
        
    Let $\Lambda \in \mathbb{M}_{P \times P}(\mathbb{R})$. Define the sequence $(W_{\ell})_{\ell \in \mathbb{N}}$ such that, for each $\ell \in \mathbb{N}$, the vector $W_{\ell} \in \mathbb{R}^P$ is the unique minimiser of the regularised least squares difference
    \begin{align*}
        \frac{1}{\ell} \bigg( \sum_{k = 0}^{\ell-1} \rVert W^\top g \circ \phi^k(m_0) - u \circ \phi^k(m_0) \lVert^2 + \lVert \Lambda W \rVert^2 \bigg).
    \end{align*}
    Then, the sequence $(W_{\ell})_{\ell \in \mathbb{N}}$ converges to
    \begin{align*}
        W^* &= \bigg( \int_M g(m)g(m)^{\top} \ d \mu(m) + \Lambda \Lambda^{\top} \bigg)^{-1} \int_M u(m)g(m) \ d \mu(m)
    \end{align*}
    which is the unique minimiser (over $W \in \mathbb{R}^P$) of
    \begin{align*}
            \lVert W^{\top} g - u \rVert_{L^2(\mu)}^2 + \lVert \Lambda W \rVert^2.
    \end{align*}
    \label{W_ell_lemma}
\end{lemma}

\begin{proof}
    Consider the map $\Psi : \mathbb{R}^P \to \mathbb{R}$ defined by
    \begin{align*}
        \Psi(W) &= \lVert W^{\top} g - u \rVert_{L^2(\mu)}^2 + \lVert \Lambda W \rVert^2
        = \int_M \lVert W^{\top} g(m) - u(m) \rVert^2 \ d \mu(m) + \lVert \Lambda W \rVert^2.
    \end{align*}
    The minimiser of $\Psi$ satisfies $D \Psi = 0$ where $D$ is the derivative operator, so we consider
    \begin{align*}
        0 &= (D\Psi)(W) \\
        &= D \bigg( \int_M \lVert W^{\top} g(m) - u(m) \rVert^2 \ d \mu(m) + \lVert \Lambda W \rVert^2 \bigg) \\
        &= \int_M D \lVert W^{\top} g(m) - u(m) \rVert^2 \ d \mu(m) + D \lVert \Lambda W \rVert^2 \\
        &= \int_M 2( W^{\top} g(m) - u(m) )g(m)^{\top} \ d \mu(m) + 2 W^{\top} \Lambda \Lambda^{\top} \\
        &= \int_M ( W^{\top} g(m) - u(m) )g(m)^{\top} \ d \mu(m) + W^{\top} \Lambda \Lambda^{\top} \\
        &= W^{\top} \int_M g(m)g(m)^{\top} \ d \mu(m) - \int_M u(m)g(m)^{\top} \ d \mu(m)
        + W^{\top} \Lambda \Lambda^{\top} \\
        &= W^{\top}\bigg( \int_M g(m)g(m)^{\top} \ d \mu(m) + \Lambda \Lambda^{\top} \bigg)
        - \int_M u(m)g(m)^{\top} \ d \mu(m),
    \end{align*}
    which upon rearrangement yields 
    \begin{align*}
        W &= \bigg( \int_M g(m)g(m)^{\top} \ d \mu(m) + \Lambda \Lambda^{\top} \bigg)^{-1}
        \int_M u(m)g(m) \ d \mu(m).
    \end{align*}
    Since this is the unique solution to $0 = D\Psi(W)$, this stationary point is unique, and we will denote it $W^*$. We can see it is a minimum because the Hessian $H\Psi$ is positive definite. Next, define the map
    \begin{align*}
        \Phi : \{ y \in C^1(\mathbb{R}^P,\mathbb{R}) \ | \ y \text{ is strictly convex} \} \to \mathbb{R}^P
    \end{align*}
    as the mapping on the strictly convex $C^1$ functions that returns their unique minimum. We can see that $\Phi$ is continuous with respect to the $C^1$ topology and Euclidean topology on $\mathbb{R}$ respectively. We consider the family of functions $y_{\ell} \in \{ y \in C^1(\mathbb{R}^P, \mathbb{R}) \ | \ y \text{ is strictly convex}\}$ defined by
    \begin{align*}
        y_{\ell}(W) =
        \frac{1}{\ell} \bigg( \sum_{k = 0}^{\ell-1} \rVert W^\top g \circ \phi^k(m_0) - u \circ \phi^k(m_0) \lVert^2 + \lVert \Lambda W \rVert^2 \bigg),
    \end{align*}
    so that by definition
    $W_{\ell} = \Phi(y_{\ell}(W))$ and hence
        \begin{align*}
        \lim_{\ell \to \infty} W_{\ell} &= \lim_{\ell \to \infty} \Phi(y_{\ell}(W)) \\
        &= \Phi\bigg(\lim_{\ell \to \infty} y_{\ell}(W) \bigg) \\
        &= \Phi\bigg( \lVert W^{\top} g - u \rVert_{L^2(\mu)}^2 + \lVert \Lambda W \rVert^2 \bigg)
        = W^*
    \end{align*}
where we have used, respectively, continuity of $\Phi$ and the Ergodic Theorem.
\end{proof}

    \begin{remark}
        The expression
        \begin{align*}
            W^* &= \bigg( \int_M g(m)g(m)^{\top} \ d \mu(m) + \Lambda \Lambda^{\top} \bigg)^{-1} \int_M u(m)g(m) \ d \mu(m)
        \end{align*}
        is reminiscent of the Gauss normal equations for least squares regression.
    \end{remark}

\subsection{Linear Universal Approximators}

In order to approximate the arbitrary dynamics of $\phi$ via the observation function $\omega$ using state space systems, we require that
the state space maps $F$ possess some sort of universal approximation property. Thus, we will define a class of \emph{linear universal approximators} with respect to an arbitrary complete norm $\lVert \cdot \rVert$. Every class of linear universal approximators contains maps, which after composition with another suitable map, forms a state map.

\begin{defn}
    Let $\mathcal{F}$ be a sequence of maps $\{F_P\} : \mathbb{R}^N \times \mathbb{R}^d \to \mathbb{R}^P$. Let $C \subset \mathbb{R}^N$ and $K \subset \mathbb{R}^d$ be vectors and let $\Omega(C \times K, \mathbb{R})$ be a Banach space of real valued functions on $C \times K$, with norm denoted $\lVert \cdot \rVert_{\Omega}$. If, for any $g \in \Omega(C \times K,\mathbb{R})$ and any $\epsilon > 0$ there exists an $P_0 \in \mathbb{N}$ such that for any $P > P_0$ there exists a $W \in \mathbb{R}^P$ such that
    \begin{align*}
        \rVert W^{\top}F_P - g \lVert_{\Omega} < \epsilon
    \end{align*}
    then we say that $\mathcal{F}$ is a class of \emph{linear universal approximators} on $\Omega(C \times K, \mathbb{R})$
\end{defn}

A widely used class of linear universal approximators is the class of Echo State Networks with randomly initialised internal weights, as shown by the following result.

\begin{theorem}
    Let $\mathcal{F}$ denote the sequence of maps $\{ F_P \} : \mathbb{R}^N \times \mathbb{R}^d \to \mathbb{R}^P$ defined by
    \begin{align*}
        F_P(x,z) = \sigma(Ax + Cz + b)
    \end{align*}
    where
    \begin{itemize}
        \item $\sigma \in C^{1}(\mathbb{R})$ is $1$-finite
        \item $A$ is a $P \times N$ random matrix, where $P >N$ and the first $N$ rows of $A$ form an $N \times N$ random submatrix with 2-norm less than 1 almost surely. The $j^{\mathrm{th}}$ row of $A$ (where $j > N$), denoted $A_j$, is a random variable with full support on $(\mathbb{R}^N)^{\top}$
        \item $C$ is a $P \times d$ random matrix with $j^\mathrm{th}$ row $C_j$, a random variable with full support on $(\mathbb{R}^d)^{\top}$
        \item $b$ is a random $P$-vector with $j\mathrm{th}$ entry $b_j$, a random variable with full support on $\mathbb{R}$.
    \end{itemize}
    Let $C \times K$ be an arbitrary compact subset of $\mathbb{R}^N \times \mathbb{R}^d$. Then, almost surely, $\mathcal{F}$ is  a class of linear universal approximators on $L^2(C \times K, \mathbb{R})$.
    \label{ESNs_are_universal}
\end{theorem}

\begin{proof}
    Fix $g \in L^2(C \times K, \mathbb{R})$ and $\epsilon > 0$. Then for any $\alpha \in (0,1)$, it follows from the Random Universal Approximation (Theorem \ref{theorem::RUAT}) that there exists a $P_0 \in \mathbb{N}$ such that for any $T > T_0$ there exists a $W \in \mathbb{R}^P$ such that with probability at least $\alpha$,
    \begin{align*}
        \rVert W^\top F_P - g \lVert_{L^2} < \epsilon,
    \end{align*}
    hence $\mathcal{F}$ is a class of linear universal approximators. Since $\mathcal{F}$ is a class of linear universal approximators for any $\alpha \in (0,1)$, $\mathcal{F}$ is almost surely a class of linear universal approximators.
\end{proof}

To construct such an ESN in practice, we create a reservoir map $F : \mathbb{R}^P \times \mathbb{R}^d \to \mathbb{R}^P$ by defining
\begin{align*}
    F(x,z) = \sigma \big( [A,0]x + Cz + b \big)
\end{align*}
where $[A,0]$ is the $P \times P$ matrix where the first $N$ columns form the matrix $A$ and the remaining $P-N$ columns are 0. Suppose we truncate at $N$ the state vectors $x \in \mathbb{R}^P$ by applying the canonical projection $\pi : \mathbb{R}^P \to \mathbb{R}^N$ onto the first $N$ coordinates, ignoring the remainder, and denote the truncation $\pi(x) = \bar{x} \in \mathbb{R}^N$. The dynamics of the truncated vectors $\bar{x}$ are given by the (state contracting) state space system $\pi \circ F_P : \mathbb{R}^N \times \mathbb{R}^d \to \mathbb{R}^N$, which is also an ESN and is defined by
\begin{align*}
    \pi \circ F_P(\bar{x},z) = \sigma(\bar{A}\bar{x} + \bar{C}z + \bar{b}).
\end{align*}
Here, the $N \times N$ reservoir matrix $\bar{A}$ is created by truncating at $N$ the rows and columns of $A$. The $N \times d$ input matrix $\bar{C}$ is created by truncating at $N$ the rows of $C$. The $N$-vector $\bar{b}$ is created by truncating at $N$ the entries of $b$.
We conclude that Echo State Networks with (appropriately chosen) randomly generated internal weights are a class of linear universal approximators that each give rise to a state synchronisation map.

We demanded that the $P \times P$ reservoir matrix take the form $[A,0]$, whereas in practice, the reservoir matrix does not have this structure. We imposed this condition to simplify the proofs, but we believe, based on numerical evidence in the literature, that this choice of shape is not necessary. 

There is one more technical lemma we will include here before presenting the main theorem (Theorem \ref{least_sqs_thm}) of this section. Recall that topological spaces have a natural Borel sigma algebra and are therefore measurable spaces. On such spaces we can integrate real valued functions. If $A$ and $B$ are homeomorphic topological spaces, then integration on $A$ is essentially the same as integration on $B$. We use this observation in Theorem \ref{least_sqs_thm} to move between integration on the topological space $M$ to integration on the image $f(M)$. This demands the highly non-trivial assumption that the GS $f$ is a homeomorphism. The observation is made formal in the following lemma.

\begin{lemma}
    (Change of variables) Let $A,B$ be homeomorphic topological spaces and suppose $y \in \text{Hom}(A,B)$. The topologies on $A,B$ induce Borel sigma algebras $\mathscr{A},\mathscr{B}$ on $A,B$ respectively. Let $\mu_A$ be a measure on $A$ and $\mu_B$ a measure on $B$ (called the pushforward measure) defined by $\mu_B(b) = \mu_A(y^{-1}(b))$ for all $b \in \mathscr{B}$. Then for any $\mu_B$-measurable function $g : B \to \mathbb{R}$ it follows that
    \begin{align*}
        \int_A g \circ y \ d \mu_A = \int_B g \ d \mu_B.
    \end{align*}
    \label{gy_lemma}
\end{lemma}

\begin{proof}
    This is a special case of Theorem 3.6.1 in \cite{Bogachev}.
\end{proof}

\subsection{Linear Universal Approximators Trained by Least Squares}

Before we finally plunge into the statement and proof of the main theorem, we will describe the result in words.
Suppose we have an ergodic dynamical system $\phi : M \to M$, which we observe via the function $\omega : M \to \mathbb{R}^d$ and that our goal is to approximate a target function $u : M \to \mathbb{R}$.
Suppose we have at our disposal a class $\mathcal{F}$ of linear universal approximating state maps. For example, $\mathcal{F}$ could be a collection of arbitrarily high dimensional ESNs. Make the additional (and non trivial) assumption that the state maps give rise to a GS that is homeomorphic onto its image. Suppose then that the state map $F$ is driven with observations of a trajectory $z_k = \omega \circ \phi^k(m_0)$ originating from a $\mu$-generic point $m_0$. This creates a sequence of reservoir states $x_k$ that satisfy
\begin{align*}
    x_{k+1} = F(x_k,z_k).
\end{align*}
We also have a sequence of scalar targets $u \circ \phi^k(m_0)$, which are called \emph{labels} in the supervised learning parlance.

Suppose we use regularised least squares regression to minimise the difference between the linear mapping on the reservoir states $W^{\top} x_k$ (which are effectively our observations) and the targets $u \circ \phi^k(m_0)$. Then we can conclude that the ergodic average difference between the mapping on the data and the target map $u$ can be made smaller than the arbitrary threshold $\epsilon$. This requires that the trajectory length $\ell$ and state map dimension $P$ are sufficiently large, while ensuring the regularisation parameter $\lambda > 0$ is sufficiently small.

Lastly, before stating the theorem (Theorem \ref{least_sqs_thm}), we remark that a notable weakness of the result is its non-constructive nature; the actual values for $\ell$, $P$ and $\lambda$ are not computed in terms of $\epsilon$.

In the remainder of this subsection we state and prove the theorem itself.

\begin{theorem}
    Let $M$ be a topological space, and suppose that $\phi\in \text{Hom}(M)$ is ergodic with invariant measure $\mu$. Let $m_0$ be a $\mu$-generic point in $M$. Let $\omega \in C^0(M,\mathbb{R}^d)$ be the observation function and suppose that $u \in L^2(\mu)(M, \mathbb{R})$ is the target function we wish to approximate. 
    
    Suppose that $\mathcal{F} = \{ F_P \}_{P \in \mathbb{N}}$ is a class of linear universal approximators on $L^2(C \times K, \mathbb{R})$ on every compact $C \subset \mathbb{R}^N,K \subset \mathbb{R}^d$. Let $(s_P)_{P \in \mathbb{N}} : \mathbb{R}^P \to \mathbb{R}^N$ be a sequence of maps. Suppose (for each large enough $P$) that the reservoir map $s_P \circ F_P : \mathbb{R}^N \times \mathbb{R}^d \to \mathbb{R}^N$ admits a GS $f \in \text{Hom}(M,f(M))$.
    For each $P, \ell \in \mathbb{N}$, let $\Lambda \in \mathbb{M}_{P \times P}(\mathbb{R})$, let $\lambda > 0$, and let $W_{\ell} \in \mathbb{R}^P$ be the vector obtained by minimising the regularised least squares difference
    \begin{align*}
        \sum_{k = 0}^{\ell-1} \rVert W^\top F_P(f \circ \phi^{k-1}(m_0) , \omega \circ \phi^k(m_0)) - u \circ \phi^k(m_0) \lVert^2
        + \lambda \lVert \Lambda W \rVert^2.
    \end{align*}
    Then, for any $\epsilon > 0$, there exists $\lambda^* > 0$ and $\ell_0, P_0 \in \mathbb{N}$ such that for all $\lambda \in (0,\lambda^*)$ and $\ell > \ell_0, P > P_0 $
    \begin{align*}
        \lVert W_{\ell}^{\top} F_P(f \circ \phi^{-1}, \omega) - u \rVert_{L^2(\mu)}^2 < \epsilon.
    \end{align*}
    \label{least_sqs_thm}
\end{theorem}

\begin{proof}
    Let $y : M \to y(M) \subset \mathbb{R}^N \times \mathbb{R}^d$ be defined by
    \begin{align*}
        y(m) = ( f \circ \phi^{-1}(m) , \omega(m) ) \quad \forall \, m \in M
    \end{align*}
    and note that $F_P(f \circ \phi^{-1}, \omega) = F_P \circ y$ and that
    $y \in \text{Hom}(M, y(M) )$ because $f \in \text{Hom}(M , f(M))$. Now fix $\epsilon > 0$. Let $\mu'$ be a measure defined on $y(M) \subset \mathbb{R}^N \times \mathbb{R}^d$ by $\mu'(\sigma) = \mu(y^{-1}(\sigma))$ for all measurable subsets $\sigma$ of $f(M)$.
    Using the assumption that $\mathcal{F}$ is a class of linear universal approximators, we can choose $P_0$ sufficiently large that for any $P > P_0$ there exists $\bar{W} \in \mathbb{R}^P$ such that
    \begin{align*}
        \lVert \bar{W}^{\top} F_P - u \circ y^{-1} \rVert_{L^2(\mu')}^2 < \frac{\epsilon}{3},
    \end{align*}
    hence (by lemma \ref{gy_lemma})
    \begin{align*}
        \lVert \bar{W}^{\top} F_P \circ y - u \rVert_{L^2(\mu)}^2 = \lVert \bar{W}^{\top} F_P - u \circ y^{-1} \rVert_{L^2(\mu')}^2  < \frac{\epsilon}{3}.
    \end{align*}
    Now let
    \begin{align*}
        \lambda^* = \frac{\epsilon}{3 \lVert \bar{W} \rVert^2}
    \end{align*}
    and fix $\lambda \in (0,\lambda^*)$. Define the sequence $(W_{\ell})_{\ell \in \mathbb{N}}$ such that, for each $\ell \in \mathbb{N}$, the vector $W_{\ell} \in \mathbb{R}^P$ is the unique minimiser of the regularised least squares difference
    \begin{align*}
        \frac{1}{\ell} \bigg( \sum_{k = 0}^{\ell-1} \rVert W^\top F_P(f \circ \phi^{k-1}(m_0) , \omega \circ \phi^k(m_0)) - u \circ \phi^k(m_0) \lVert^2 
        + \lambda \lVert \Lambda W \rVert^2 \bigg).
    \end{align*}
    By lemma \ref{W_ell_lemma}, $(W_{\ell})_{\ell \in \mathbb{N}}$ converges as $\ell \to \infty$ to $W^*$ which minimises 
    \begin{align*}
        \lVert W^{\top} F_P(f \circ \phi^{-1}, \omega) - u \rVert_{L^2(\mu)}^2 + \lambda \lVert \Lambda W \rVert^2.
    \end{align*}
    Now we choose $\ell_0$ such that for all $\ell > \ell_0$ 
    \begin{align*}
        \rVert W_{\ell}^{\top} F_P(f \circ \phi^{-1}, \omega) - W^{*\top} F_P(f \circ \phi^{-1}, \omega) \lVert_{L^2(\mu)}^2 < \frac{\epsilon}{3}.
    \end{align*}
    Now the proof proceeds directly
    \begin{align*}
        &\lVert W_{\ell}^{\top} F_P(f \circ \phi^{-1},\omega) - u \rVert_{L^2(\mu)}^2 \\
        =&\lVert W_{\ell}^{\top} F_P(f \circ \phi^{-1}, \omega) - W^{*\top} F_P(f \circ \phi^{-1}, \omega)
        + W^{*\top} F_P(f \circ \phi^{-1},\omega) - u \rVert_{L^2(\mu)}^2 \\
        \leq&\lVert W_{\ell}^{\top} F_P(f \circ \phi^{-1}, \omega) - W^{*\top} F_P(f \circ \phi^{-1}, \omega) \rVert_{L^2(\mu)}^2
        + \lVert W^{*\top} F_P(f \circ \phi^{-1}, \omega) - u \rVert_{L^2(\mu)}^2 \\
        <& \frac{\epsilon}{3} + \lVert W^{*\top} F_P(f \circ \phi^{-1}, \omega) - u \rVert_{L^2(\mu)}^2 \\
        \leq&  \frac{\epsilon}{3} + \lVert W^{*\top} F_P(f \circ \phi^{-1}, \omega) - u \rVert_{L^2(\mu)}^2 + \lambda \lVert \Lambda W^* \rVert^2 \\
        \leq& \frac{\epsilon}{3} + \lVert \bar{W}^{\top} F_P(f \circ \phi^{-1}, \omega) - u \rVert_{L^2(\mu)}^2 + \lambda \lVert \Lambda \bar{W} \rVert^2 \\
        <& \frac{\epsilon}{3} + \frac{\epsilon}{3} + \lVert \bar{W}^{\top} F_P(f \circ \phi^{-1}, \omega) - u \rVert_{L^2(\mu)}^2 \\
        =& \frac{\epsilon}{3} + \frac{\epsilon}{3} + \lVert \bar{W}^{\top} F_P \circ y - u \rVert_{L^2(\mu)}^2 \\
        <& \frac{\epsilon}{3} + \frac{\epsilon}{3} + \frac{\epsilon}{3} = \epsilon.
    \end{align*}
    \end{proof}

\subsection{Convergence rate of the time average to the space average}

Theorem \ref{least_sqs_thm} guarantees, under appropriate conditions, that with sufficiently many neurons $P$ and a sufficiently many training data $\ell$ we can obtain an arbitrarily good $L^2(\mu)$ approximation of a target function $u$. It is natural to wonder how many training data are required to achieve a given $L^2(\mu)$ approximation. To answer this, we turn our attention to the convergence rate of the time average to the space average
    \begin{align}
        \lim_{\ell \to \infty}\frac{1}{\ell} \sum_{k=0}^{\ell-1} s \circ \phi^{k}(m_0) = \int_{M} s \ d\mu \tag{\ref{Ergodic_Theorem_eqn}}
    \end{align}
as the timespan over which training data is collected grows. 
We want a uniform estimate for the rate of convergence for $s$ over all ergodic maps $\phi$. Unfortunately, no such estimate exists.
\cite{Kachurovskii1996} presents negative results that (in the author's words) \emph{leave no hope that estimates of the rate of convergence depending only on the averaged function $s$ can be obtained in ergodic theorems}. The negative results presented by \cite{Kachurovskii1996} prove that the amount of training data required is strictly dependant on the dynamical system.

Though we cannot say exactly how many data points we need for a good $L^2(\mu)$ approximation, the central limit theorem for mixing dynamical systems suggests that for an initial point chosen uniformly at random over the invariant measure of $\phi$, the difference between the sum of random variables that is the finite time average and the space average converges in distribution to normal distribution with mean zero and standard deviation $1 / \sqrt{\ell}$. The notion of mixing is stronger than ergodicity, and defined in \ref{defn::mixing}. Before we state the theorem, we recall the definition of H{\"o}lder continuity.

\begin{defn}
    (H{\"o}lder continuous) Let $(M,d)$ be a metric space. A map $s : M \to \mathbb{R}$ is called \emph{H{\"o}lder continuous} if there exist constants $p \in (0,1]$ and $K > 0$ such that
    \begin{align*}
        \lVert s(m) - s(m') \rVert \leq K d(m,m')^p
    \end{align*}
    for all $m,m' \in M$.
\end{defn}

\begin{theorem}
    (Central limit theorem for ergodic dynamical systems) Let $\phi : M \to M$ be mixing (Definition \ref{defn::mixing}) with respect to the probability space $(M,\Sigma,\mu)$. Let $X_0$ be a uniform random variable with respect to the space $(M,\Sigma,\mu)$. Let $s \in L^1(\mu)(M,\mathbb{R})$ be H{\"o}lder continuous and denote the
    space average of $s$ by
    \begin{align*}
        \mathbb{E}[s] := \int_{M} s \ d\mu.
    \end{align*}
     Let the random variables $X_j:= \phi^j(X_0)$ for $j=0,\ldots,\ell-1$, and denote the partial sum
     $S_\ell:=s(X_0) + \cdots + s(X_{\ell-1})$.
     Then, for some $\beta > 0$, the partial sum
     $S_\ell$ satisfies the central limit theorem:
    \begin{align}
        \lim_{\ell \to \infty} \mu \bigg( \bigg\{ \frac{S_\ell - \ell \mathbb{E}[s]}{\sqrt{\ell}} \leq z \bigg\} \bigg) = \frac{1}{\sqrt{2\pi} \beta} \int_{-\infty}^z \mathrm{e}^{-\frac{\tau^2}{2 \beta^2}} \ d\tau \nn
    \end{align}
    almost surely, or in other words $(S_\ell - \ell \mathbb{E}[s])/\sqrt{\ell}$ converges in law to $\mathcal{N}(0,\beta^2)$.
\end{theorem}

\begin{proof}
    \cite{ErgodicTheory2010}.
\end{proof}

To see the connection between the central limit theorem and the work in this paper, we observe first of all that the CLT holds in higher dimensions. Then we consider the map $s$ that returns the matrix--vector pair 
\begin{align*}
    s(m_0) &= \bigg( \big[f(m_0)f^\top(m_0) + \Lambda \Lambda^{\top} \big], f(m_0) u(m_0) \bigg) \\
    &=: (B_0, v_0),
\end{align*}

and define a sequence of pairs $(B_\ell, v_\ell)_{\ell \in \mathbb{N}}$ with $\ell$th pair
\begin{align*}
    (B_\ell, v_\ell) := \frac{1}{\ell}\sum_{k=0}^{\ell-1} s \circ \phi^k(m_0).
\end{align*}
Then it follows that
\begin{align*}
    W_\ell = B_{\ell}^{-1} v_\ell
\end{align*}
is the linear readout layer obtained by regularised least squares regression using $\ell$ data points. Furthermore, it follows from the central limit theorem that
for random initial points $m_0$ (distributed uniformly with respect to the invariant measure $\mu$) the sequence $(B_\ell, v_\ell)_{\ell \in \mathbb{N}}$ converges in law to a (multivariate) normal distribution, with standard deviation converging to 0 with order $1/\sqrt{\ell}$, and mean $(B,v)$ which satisfies 
\begin{align*}
    W^* = B^{-1}v.
\end{align*}
We note that the convergence of $(B_\ell, v_\ell)_{\ell \in \mathbb{N}}$ to $(B, v)$ with order $1/\sqrt{\ell}$ does not necessarily imply that $(W_{\ell})_{\ell \in \mathbb{N}}$ converges to $W^*$ at the same rate.

\section{The Lorenz attractor is stably mixing}
\label{Lorenz_is_stably_mixing}

We have shown that we can approximate, in the $L^2(\mu)$ sense, any target function on an ergodic dynamical system using an ESN and Tikhonov least squares. This partially explains the success enjoyed by \cite{Jaeger2001}, \cite{1556081}, \cite{416607}, \cite{4118282}, \cite{5645205},  \cite{Pathak2017}, \cite{Lokse2017}, \cite{YEO2019671}, \cite{Ashesh_2019}, \cite{Vlac_2019}, and \cite{embedding_and_approximation_theorems}. Many authors including \cite{Ashesh_2019} successfully predict the future observations of the Lorenz system, while \cite{Pathak2017}, \cite{Vlac_2019}, and \cite{embedding_and_approximation_theorems} additionally recover invariants including Lyapunov exponents, fixed point eigenvalues and homology groups. The authors are successful in their numerical experiments because the Lorenz attractor is \emph{mixing} which implies it is ergodic, suggesting the conditions Theorem \ref{least_sqs_thm} hold and we can $L^2(\mu)$ approximate target functions on the Lorenz attractor.

The proof by \cite{Tucker2002} that the Lorenz attractor is mixing was a tremendous achievement, and resolved Smale's $14^\mathrm{th}$ problem 

\emph{`Is the dynamics of the ordinary differential equations of Lorenz (1963) that of
the geometric Lorenz attractor of Williams, Guckenheimer and Yorke? '}
\citep{Smale1998}.

Tucker's proof built upon the works of \cite{1977DoSSR.234..336A}, \cite{Structural_stability_of_Lorenz_attractors}, \cite{pesin_1992}, \cite{PMIHES_1979__50__73_0}, and \cite{TUCKER19991197}. The ideas discussed in these papers go far beyond the scope of this thesis, and I will restrict attention to simpler results about mixing dynamical systems that are directly relevant to the results presented here.

\begin{defn}
    (Mixing) Let $\phi : M \to M$ be a measure preserving transformation on the measure space $(M,\Sigma,\mu)$ with $\mu(M) = 1$. Then $\phi$ is mixing if for any $A,B \in \Sigma$
    \begin{align}
        \lim_{\ell \to \infty} \mu\big(A \cap \phi^{-\ell}(B)\big) = \mu(A)\mu(B). \nn
    \end{align}
    \label{defn::mixing}
\end{defn}

\begin{lemma}
    (Mixing implies ergodic) Let $\phi : M \to M$ be a measure preserving transformation on the measure space $(M,\Sigma,\mu)$ with $\mu(M) = 1$. Suppose $\phi$ is mixing, then $\phi$ is ergodic. \label{mixing_implies_ergodic}
\end{lemma}
\begin{proof}
    Suppose $\phi$ is mixing and $A,B \in \Sigma$. Then
    \begin{align}
        \lim_{\ell \to \infty} \mu\big(A \cap \phi^{-\ell}(B)\big) &= \mu(A)\mu(B) \nn \\
        \implies \lim_{\ell \to \infty} \frac{1}{\ell} \sum_{k = 0}^{\ell-1} \mu\big(A \cap \phi^{-k}(B)\big) &= \mu(A)\mu(B) \nn \\
        \implies \lim_{\ell \to \infty} \frac{1}{\ell} \sum_{k = 0}^{\ell-1} \mu\big(A \cap \phi^{-k}(A)\big) &= \mu(A)^2 \label{last_line},
    \end{align}
where we have set $A=B$ to derive the last line.
    Now suppose that $A$ is $\phi$-invariant, i.e. $A=\phi^{-1}(A)$. Then the left-hand side of \eqref{last_line} is precisely $\mu(A)$, and hence \eqref{last_line} reduces to the statement that $\mu(A) = \mu(A)^2$ hence $\mu(A) = 1$ or $\mu(A) = 0$, so any invariant set has either full or zero measure, and so $\phi$ is ergodic.
\end{proof}

\begin{defn}
    (Stably mixing)  Let $\phi : M \to M$ be a measure preserving transformation on the measure space $(M,\Sigma,\mu)$ with $\mu(M) = 1$. Then $\phi$ is stably mixing if sufficiently small $C^1$ perturbations of $\phi$ are mixing.
\end{defn}

\begin{theorem}
    The \cite{doi:10.1175/1520-0469(1963)020<0130:DNF>2.0.CO;2} system
\begin{equation}
\begin{array}{ll}
\dot{\xi}      & = \sigma(\upsilon - \xi) \\
\dot{\upsilon} & = \xi(\rho - \zeta) - \upsilon \\
\dot{\zeta}    & = \xi\upsilon - \beta \zeta
\end{array} \label{eqn:lorenz}
\end{equation}
with parameters $\sigma = 10$, $\beta = 8/3$, $\rho = 28$ admits a robust attractor that is stably mixing.
\label{thm:lorenz}
\end{theorem}

\begin{proof}
    \cite{Luzzatto2005}
\end{proof}

Since the Lorenz attractor is stably mixing, so is any sufficiently good $C^1$ approximation to the evolution operator $\phi$, obtained by numerical methods.
Consequently, a numerically approximated Lorenz system is ergodic, by Lemma \ref{mixing_implies_ergodic}.
Thus, we expect that an ESN, trained using least squares fitting, with Tikhonov regularisation, on a sequence of observations of a numerically integrated trajectory of the Lorenz attractor will $L^2(\mu)$--approximate arbitrary target functions on the attractor. 

\section{Online Learning}
\label{section::online_learning}

Most ESN practitioners estimate the minimiser $W^*$ by collecting a long trajectory of reservoir states $x_k$, then solve the regularised least squares problem to find $W_\ell$ using their favourite numerical algebra package. This is an example of \emph{offline learning} because the learning happens after the training is collected. Offline learning requires that the entire set of reservoir states $x_k$ are stored in memory, and subject to large matrix computations.

An alternative approach is \emph{online learning}, where learning occurs dynamically as new data as comes in. Online learning methods typically do not require that the entire training trajectory is stored in memory, nor do they require large (memory intensive!) matrix operations. This gives them an obvious advantage over offline methods.

Furthermore, online learning algorithms are a more a realistic model of biological neural networks - because the connection strength between biological neurons varies in real time in response to stimulus. Indeed, biologically speaking, offline learning is completely implausible, as this would entail the states of a biological brain being stored somewhere `external'. Operations would be applied to these externally stored neuronal states, and the results then be transmitted back into the brain. With this motivation, we will study the online learning algorithm
\begin{align}
    W_{k+1} &= (1-\alpha_k\lambda)W_k
    - \alpha_k g\circ\phi^{k}(m_0)(W^{\top}_k g\circ\phi^k(m_0) - u \circ \phi^k(m_0)) \label{online_learning}
\end{align}
with $(\alpha_k)_{k \in \mathbb{N}}$ a strictly positive real sequence, and remaining terms defined in lemma \ref{W_ell_lemma}. The algorithm is essentially stochastic gradient descent and is discussed in numerous works including
\cite{Adaptive_Algorithms}, \cite{10.1007/978-3-540-72927-3_23}, \cite{Borkar2009}, and \cite{Chen2019}. The behavior of the sequence $(\alpha_k)_{k \in \mathbb{N}}$ often called the \emph{learning rate} or \emph{learning parameter}. If $(\alpha_k)_{k \in \mathbb{N}} = \alpha > 0$ is a constant sequence, then we call $\alpha$ the learning parameter. In the case that $(\alpha_k)_{k \in \mathbb{N}}$ converges to 0 rapidly, the convergence is called a fast learning rate, and slow convergence is called a slow learning rate. If $(\alpha_k)_{k \in \mathbb{N}}$ diverges then something like the opposite of learning is achieved.

In what follows we will show that if $(\alpha_k)_{k \in \mathbb{N}} = \alpha > 0$ is a constant sequence, then the algorithm does not converge in general, but in the limit as $k \to \infty$ the algorithm returns $W^*$ \emph{on average} (in the ergodic sense). Furthermore, we establish a bound on the \emph{variance} of the weights $W_k$ as $k \to \infty$. Next we show that if $\alpha_k = 1/k$ then algorithm \eqref{online_learning} converges to $W^*$. We will proceed with preliminary results in the case where $(\alpha_k)_{k \in \mathbb{N}}$ is a constant sequence.

\subsection{Fixed learning parameter $\alpha$}

\begin{lemma}
\label{contracting_lemma}
Let $V \subset \mathbb{R}^P$ be a bounded set, $\lambda > 0$ be the regularisation parameter, and 
\begin{align*}
    0 < \alpha < \frac{1}{\lambda + \sup_{v \in V}\lVert v \rVert^2}.
\end{align*}
Then the reservoir map $F : \mathbb{R}^P \times (V \times \mathbb{R}) \to \mathbb{R}^P$ defined by
    \begin{align*}
        F(W;v,z) &= (1-\alpha\lambda)W -\alpha v(W^{\top}v - z)
    \end{align*}
is state contracting.
\end{lemma}

\begin{proof}
    First, note that upon rearrangement
    \begin{align*}
        F(W;v,z) = [I - \alpha(v v^{\top} + \lambda I)] W + \alpha vz.
    \end{align*}
    Then we can see that
    \begin{align*}
        \lVert F(W;v,z) - F(Y;v,z) \rVert
        =&\lVert [I - \alpha(vv^{\top} + \lambda I)](W - Y) \rVert \\
        \leq& \rho [I - \alpha(vv^{\top} + \lambda I)] \ \lVert W - Y \rVert
    \end{align*}
    where $\rho[\cdot]$ denotes the spectral radius. Now the matrix $[vv^{\top} + \lambda I]$ has exactly 2 distinct eigenvalues, $\lambda$ and $\lambda + \lVert v \rVert^2$. Therefore, the matrix $\alpha[vv^{\top} + \lambda I]$ has exactly 2 eigenvalues $\alpha \lambda$ and $\alpha (\lambda + \lVert v \rVert^2)$. Both of these lie in $(0,1)$ by construction, so $\rho [I - \alpha(vv^{\top} + \lambda I)] \in (0,1)$ and we are done.
\end{proof}

\begin{theorem}
    Let $(M,\Sigma)$ be a measurable space, and suppose that $\phi : M \to M$ is ergodic with invariant measure $\mu$. Let $m_0$ be a $\mu$-generic point in $M$. Let $g \in C^0(M,\mathbb{R}^P)$ be an observation function and suppose that $u \in L^2(\mu)(M, \mathbb{R})$ is a target function we wish to approximate. Suppose that
    \begin{align*}
        0 < \alpha < \frac{1}{\lambda + \sup_{m \in M} \lVert g(m) \rVert^2}.
    \end{align*}
    Then there exists a unique GS $h \in C^0(M,\mathbb{R}^P)$, such that for generic initial $m_0 \in M$ and initial $W_0 \in \mathbb{R}^P$ the iteration \eqref{online_learning} converges to $h\circ\phi^k(m_0)$ as $k \to \infty$. Furthermore the GS $h$ satisfies
        \begin{align*} 
            \int_M [g g^{\top} + \lambda I](W^*-h) \ d\mu = 0
        \end{align*}
    and
        \begin{align*}
        \lVert W^* - h \rVert_{L^2(\mu)} \leq \frac{\lVert [gg^{\top} + \lambda I]W^* - gu \rVert_{L^2(\mu)}}{\lambda}.
    \end{align*}
    \label{theorem::constant_alpha}
\end{theorem}

\begin{proof}
    From Theorem \ref{SSM_thm} and Lemma \ref{contracting_lemma} it follows that for generic initial states $m_0 \in M$ and $W_0 \in \mathbb{R}^P$ the iteration \eqref{online_learning} converges to $h\circ\phi^k(m_0)$ as $k \to \infty$. Furthermore, in the limit $k \to \infty$ we have
    \begin{align*}
        h \circ \phi &= (1-\alpha\lambda) h - \alpha g(h^{\top} g - u),
    \end{align*}
    which upon rearrangement yields
    \begin{align*}
    h \circ \phi &= h-\alpha[(g g^{\top} + \lambda I)] h + \alpha gu.
    \end{align*}
    Then integrating both sides we have
    \begin{align*}
        \int_M h \circ \phi \ d\mu &= \int_M h \ d\mu - \alpha \int_M [g g^{\top} + \lambda I]h \ d\mu
        + \alpha\int_M gu \ d\mu 
    \end{align*}
    so, using the $\phi$-invariance of the measure $\mu$, we have
    \begin{align*}
        \int_M h \ d\mu &= \int_M h \ d\mu - \alpha \int_M [g g^{\top} + \lambda I]h \ d\mu
        + \alpha\int_M gu \ d\mu 
    \end{align*}
    
    hence 
    \begin{align*}
        \int_M [g g^{\top} + \lambda I]h \ d\mu = \int_M gu \ d\mu.
    \end{align*}
    We now recall from Lemma \ref{W_ell_lemma} that
    \begin{align*}
        \int_M gu \ d\mu = \int_M [g g^{\top} + \lambda I]W^* \ d\mu
    \end{align*}
    so
    \begin{align*}
        \int_M [g g^{\top} + \lambda I]h \ d\mu = \int_M [g g^{\top} + \lambda I]W^* \ d\mu
    \end{align*}
    hence
    \begin{align*}
        0 &= \int_M [g g^{\top} + \lambda I](h-W^*) \ d\mu.
    \end{align*}
    Now we recall
    \begin{align*}
        h \circ \phi &= (1-\alpha\lambda) h - \alpha g(h^{\top} g - u),
    \end{align*}
    which upon rearrangement yields
    \begin{align*}
        h \circ \phi - W^* &= (I - \alpha [gg^\top + \lambda I])(h - W^*)
        - \alpha([gg^\top + \lambda I]W^* - gu)
    \end{align*}
    hence 
    \begin{align*}
        \lVert h - W^* \rVert_{L^2(\mu)}
        &= \lVert h \circ \phi - W^* \rVert_{L^2(\mu)} \\
        &= \lVert (I - \alpha [gg^\top + \lambda I])(h - W^*)
        + \alpha([gg^\top + \lambda I]W^* - gu) \lVert_{L^2(\mu)} \\
        &\leq \lVert (I - \alpha [gg^\top + \lambda I])(h - W^*) \rVert_{L^2(\mu)}  +
        \alpha \rVert [gg^\top + \lambda I]W^* - gu \lVert_{L^2(\mu)} \\
        &\leq \rho(I - \alpha [gg^\top + \lambda I])\lVert(h - W^*) \rVert_{L^2(\mu)}
        + \alpha \rVert [gg^\top + \lambda I]W^* - gu \lVert_{L^2(\mu)}.
    \end{align*}
    where $\rho(\cdot)$ is the spectral radius. Now the eigenvalues of $\alpha[gg^{\top} + \lambda I]$ are $\alpha\lambda$ and $\alpha(\lambda + \lVert g \rVert^2)$ which both lie in $(0,1)$ by construction, so
    \begin{align*}
        \rho(I - \alpha [gg^\top + \lambda I]) = 1 - \alpha \lambda.
    \end{align*}
    Therefore we have 
    \begin{align*}
        \lVert h - W^* \rVert_{L^2(\mu)}
        &\leq \lVert (1-\alpha\lambda)(h - W^*) \rVert_{L^2(\mu)}
        + \alpha \rVert [gg^\top + \lambda I]W^* - gu \lVert_{L^2(\mu)} \\
        &= (1-\alpha\lambda)\lVert(h - W^*) \rVert_{L^2(\mu)} 
        + \alpha \rVert [gg^\top + \lambda I]W^* - gu \lVert_{L^2(\mu)}.
    \end{align*}
    We consider the first and last lines together and then rearrange to yield
    \begin{align*}
        \lVert W^* - h \rVert_{L^2(\mu)} 
        \leq \frac{\lVert [gg^{\top} + \lambda I]W^* - gu \rVert_{L^2(\mu)}}{\lambda}.
    \end{align*}
\end{proof}

We showed in theorem \ref{theorem::constant_alpha} that if $(\alpha_k)_{k \in \mathbb{N}}$ is a sufficiently small constant sequence, algorithm \eqref{online_learning} converges to $W^*$ in some ergodic average sense. We will show next that if we replace all observations $g \circ \phi^k(m_0)$ and targets $u \circ \phi^k(m_0)$ with the ergodic averages $\int_M g \ d \mu$ and $\int_M u \ d \mu$ then algorithm \eqref{online_learning} converges to $W^*$.

\begin{theorem}
    Let $(M,\Sigma)$ be a measurable space, and suppose that $\phi : M \to M$ is ergodic with invariant measure $\mu$. Let $m_0$ be a $\mu$-generic point in $M$. Let $g \in C^0(M,\mathbb{R}^P)$ be an observation function and suppose that $u \in L^2(\mu)(M, \mathbb{R})$ is a target function we wish to approximate. Let $\lambda > 0$ and choose $\alpha > 0$ such that
\begin{align*}
    \rho \bigg(I - \alpha \int_M gg^{\top} + \lambda I \ d\mu \bigg) \in (0,1).
\end{align*}
    Then the map $\Phi : \mathbb{R}^P \to \mathbb{R}^P$ defined by
    \begin{align*}
        \Phi(W) = W - \alpha\bigg[\int_M gg^{\top} + \lambda I \ d\mu \bigg]W + \alpha \int_M gu \ d\mu
    \end{align*}
    is contracting, and admits a unique fixed point
    \begin{align*}
        W^* = \bigg[\int_M gg^{\top} + \lambda I \ d\mu \bigg]^{-1} \int_M gu \ d\mu.
    \end{align*}
    \label{theorem::fixed_alpha_contraction}
\end{theorem}

\begin{proof}
    The proof that $\Phi$ is contracting proceeds directly 
    \begin{align*}
        \lVert \Phi(W_1) - \Phi(W_2) \rVert &= \bigg\lVert \bigg[I - \alpha \int_M gg^{\top} + \lambda I \ d\mu \bigg](W_1 - W_2) \bigg\rVert \\
        &\leq  \bigg\lVert I - \alpha \int_M gg^{\top} + \lambda I \ d\mu \bigg\rVert_2 \lVert W_1 - W_2 \rVert \\
        &= \rho \bigg(I - \alpha \int_M gg^{\top} + \lambda I \ d\mu \bigg) \lVert W_1 - W_2 \rVert \\
        &= c \lVert W_1 - W_2 \rVert
    \end{align*}
    where
    \begin{align*}
        c := \rho \bigg(I - \alpha \int_M gg^{\top} + \lambda I \ d\mu \bigg) \in (0,1).
    \end{align*}
    We have shown that $\Phi$ is contracting, and therefore admits a unique fixed point $W^*$ which satisfies 
    \begin{align*}
        W^* = \Phi(W^*) = W^* - \alpha\bigg[\int_M gg^{\top} + \lambda I \ d\mu \bigg]W^* + \alpha \int_M gu \ d\mu
    \end{align*}
    which is easily rearranged to
    \begin{align*}
        W^* = \bigg[\int_M gg^{\top} + \lambda I \ d\mu \bigg]^{-1} \int_{M}gu \ d\mu.
    \end{align*}
\end{proof}

We can view theorem \ref{theorem::fixed_alpha_contraction} as a result for the discrete time dynamical system given by the proposed iterative scheme \eqref{online_learning} for $W$. This result has a continuous time analogue which we present next.

\begin{theorem}
     Let $(M,\Sigma)$ be a measurable space, and suppose that $\phi : M \to M$ is ergodic with invariant measure $\mu$. Let $m_0$ be a $\mu$-generic point in $M$. Let $g \in C^0(M,\mathbb{R}^P)$ be an observation function and suppose that $u \in L^2(\mu)(M, \mathbb{R})$ is a target function we wish to approximate. Let $\lambda > 0$.
     
     Then the ODE 
     \begin{align*}
         \dot{W} = - \bigg[\int_M gg^{\top} + \lambda I \ d\mu \bigg]W + \int_M gu \ d\mu
     \end{align*}
     has a globally asymptotic fixed point
     \begin{align*}
         W^* = \bigg[\int_M gg^{\top} + \lambda I \ d\mu \bigg]^{-1} \int_M gu \ d\mu.
     \end{align*}
\end{theorem}

\begin{proof}
    The matrix $-\bigg[\int_M gg^{\top} + \lambda I \ d\mu \bigg]$ is symmetric negative definite. The eigenvalues are therefore all real and strictly negative. The fixed point $W^*$ is therefore a globally asymptotic fixed point.
\end{proof}

\subsection{Learning rate $\alpha_k = 1/k$}

If the learning rate in the iteration \eqref{online_learning} is $(\alpha_k)_{k \in \mathbb{N}} = 1/k$ then the algorithm converges to exactly $W^*$. Roughly speaking, this is because the choice of learning rate $(\alpha_k)_{k \in \mathbb{N}} = 1/k$ is not too fast or too slow. If we choose a learning rate that is too fast, then the algorithm converges before enough training data has been collected to a non-optimal $W \in \mathbb{R}^P$. On the other hand if the learning rate is too slow, the continued influence of the data prevents the algorithm from converging at all. The formal version of this result follows almost immediately from a theorem by \cite{Gyorfi1980} which we present here.

\begin{theorem}
    (\cite{Gyorfi1980}) Let $\mathcal{H}$ be a real Hilbert space with norm $\lVert \cdot \rVert_{\mathcal{H}}$. Let $B : \mathcal{H} \to \mathcal{H}$ be an invertible, bounded, linear, symmetric and positive operator. Let $(B_k)_{k \in \mathbb{N}} : \mathcal{H} \to \mathcal{H}$ be a sequence of bounded linear operators and $(v_k \in \mathcal{H})_{k \in \mathbb{N}}$ a sequence in $\mathcal{H}$ such that
    \begin{align*}
        \bigg\lVert \lim_{\ell \to \infty} \frac{1}{\ell}\sum_{k = 0}^{\ell - 1} B_k - B \bigg\rVert_{\mathcal{H}} = 0, \qquad \bigg\lVert \lim_{\ell \to \infty} \frac{1}{\ell}\sum_{k = 0}^{\ell - 1} v_k - v \bigg\rVert_{\mathcal{H}} = 0
    \end{align*}
    for some $v \in \mathcal{H}$. Then for any initial $W_0 \in \mathcal{H}$ the sequence 
    \begin{align*}
        W_{k+1} = W_k - \frac{1}{k}(B_k W_k - v_k)
    \end{align*}
    converges to 
    \begin{align*}
        W^* = B^{-1}v.
    \end{align*}
    \label{theorem::gyorfi}
\end{theorem} 

\begin{corollary}
    Let $(M,\Sigma)$ be a measurable space, and suppose that $\phi : M \to M$ is ergodic with invariant measure $\mu$. Let $m_0$ be a $\mu$-generic point in $M$. Let $g \in C^0(M,\mathbb{R}^P)$ be an observation function and suppose that $u \in L^2(\mu)(M, \mathbb{R})$ is a target function we wish to approximate. Then, for any initial $W_0$, the iteration \eqref{online_learning} with $\alpha_k = 1/k$
    \begin{align*}
        W_{k+1} &= (1-\alpha_k\lambda)W_k
    - \alpha_k g\circ\phi^{k}(m_0)(W^{\top}_k g\circ\phi^k(m_0) - u \circ \phi^k(m_0))
    \end{align*}
    converges to
    \begin{align*}
        W^* = \bigg[\int_M gg^{\top} + \lambda I \ d\mu \bigg]^{-1} \int_M gu \ d\mu.
    \end{align*}
\end{corollary}

\begin{proof}
    Let
    \begin{align*}
        B := \bigg[\int_M gg^{\top} + \lambda I \ d\mu \bigg], \qquad v := \int_M gu \ d\mu,
    \end{align*}
    and 
    \begin{align*}
        B_k := \bigg[g\circ\phi^k(m_0)(g \circ \phi^k(m_0))^{\top} + \lambda I \bigg], \qquad v_k := g\circ\phi^k(m_0)u\circ\phi^k(m_0),
    \end{align*}
    and notice that 
    \begin{align*}
        \bigg\lVert \lim_{\ell \to \infty} \frac{1}{\ell}\sum_{k = 0}^{\ell - 1} B_k - B \bigg\rVert_{\mathbb{R}^P} = 0, \qquad \text{and} \qquad \bigg\lVert \lim_{\ell \to \infty} \frac{1}{\ell}\sum_{k = 0}^{\ell - 1} v_k - v \bigg\rVert_{\mathbb{R}^P} = 0,
    \end{align*}
    by the ergodic theorem. Then the iteration \eqref{online_learning} can be written as
    \begin{align*}
        W_{k+1} &= W_k + \alpha_k\bigg(\bigg[g\circ\phi^k(m_0)(g \circ \phi^k(m_0))^{\top} + \lambda I \bigg] W_k - g\circ\phi^k(m_0)u\circ\phi^k(m_0)\bigg) \\
        &= W_k + \frac{1}{k} (B_k W_k - v_k)
    \end{align*}
    which converges to 
    \begin{align*}
        W^* = B^{-1}v = \bigg[\int_M gg^{\top} + I \lambda \ d\mu \bigg]^{-1} \int_M gu \ d\mu
    \end{align*}
    by Theorem \ref{theorem::gyorfi}.
\end{proof}

In summary: the online learning algorithm (iteration \eqref{online_learning}) discussed in section \ref{section::online_learning} updates the weights $W_k$ at each time step $k$ according to the reservoir state $x_k$ and observation $z_k$ made at time $k$. This is in contrast to offline methods, like the SVD, where every past reservoir state $x_k$ is stored in memory, and the estimate for $W^*$ is computed using a single matrix computation involving every past reservoir state. In machine learning problems, offline methods are typically more time-efficient than online methods, while online methods demand much less memory.

\section{Value Functions and Control}

In the preceding section we have shown that whether we use online or offline methods we can find a $W^*$ such that $W^{*\top}g$ is a good approximation to the target function $u$. If $f$ is the GS associated to a reservoir map, we can choose $g = f$. Alternatively we could choose $g = f - \gamma f \circ \phi$ where $\gamma \in [0,1)$. To understand the significance of this, we consider the (so called) value function $V : M \to \mathbb{R}$ defined by
\begin{align*}
    V :&= \sum_{k=0}^{\infty} \gamma^k u \circ \phi^k
\end{align*}
which returns the present value of a collection of future observations, with a discount factor $\gamma$ which indicates that future observations further away in time are valued less than observations nearer in the future. The value function satisfies the following functional relationship

\begin{align*}
    V &= \sum_{k=0}^{\infty} \gamma^k u \circ \phi^k
    = u + \sum_{k=1}^{\infty} \gamma^k u \circ \phi^k 
    = u + \sum_{k=0}^{\infty} \gamma^{k+1} u \circ \phi^{k+1} \\
    &= u + \gamma \sum_{k=0}^{\infty} \gamma^{k} u \circ \phi^{k+1}
    = u + \gamma \bigg(\sum_{k=0}^{\infty} \gamma^{k} u \circ \phi^k \bigg) \circ \phi
    = u + \gamma V \circ \phi.
\end{align*}

Hence, the vector $W$, which we call $W^*$, that minimises the following
\begin{align*}
    \lVert W^{\top}g - u \rVert_{L^2(\mu)}^2 + \lambda \lVert W \rVert^2 
    =&\lVert W^{\top}(f - \gamma f \circ \phi) - u \rVert_{L^2(\mu)}^2 + \lambda \lVert W \rVert^2  \\ =&\lVert W^{\top}f - (\gamma W^{\top} f \circ \phi + u) \rVert_{L^2(\mu)}^2 + \lambda \lVert W \rVert^2  
\end{align*}
results in a function $W^{*\top}f$ which is a good approximation to the value function $V$. This value function $V$ is intimately related to infinite horizon, discrete time control theory. We do not develop the deterministic theory further here, but in Chapter \ref{chapter::stochastic} we develop a stochastic analogue to control theory more detail.

\section{Biological interpretation}

In this section we are interested in finding the learning paradigm that best describes the information processing of neurons in biological organisms. The offline process is a completely implausible model, as this would entail the states of a biological brain, represented by the reservoir states $x_k$ being stored somewhere \emph{external to the brain}. These externally stored states would be subject to computations and the resulting weights $W^*$ then would be transmitted back into the brain.

The online process on the hand seems much more plausible. The connection weights, which are the components of $W_k$, represent the connections strength between the reservoir neurons and an output neuron, which might influence motor function or some other physiological output \citep{izhikevich2007dynamical}. The weights $W_k$ change gradually in response to connected neurons firing more or less frequently. This is consistent with the Hebbian learning maxim that \emph{neurons that fire together wire together} \citep{lowel1992selection}. The states of the brain are represented by the reservoir states $x$, which influence the output weights $W$ connecting the brain $A$ to the output neuron $u$.

This interpretation is highly speculative - and quite vague - and would certainly benefit from further development with a careful application of neuroscience. That said, the general results about reservoir computing suggest that, whatever the details, a biological nervous system may satisfy the definitions of a reservoir map, and perform a type of online learning which admits a universal approximation property.

\section{Learning Diffeomorphic Dynamics}

The results in previous sections explain our success in Chapter \ref{chapter::introduction} in approximating the $\zeta$ component of the Lorenz system. Building on this success in approximating `past' values of the $\zeta$ component, we turn our attention to approximating the next step map $g := \omega \circ \phi$ in the $L^2(\mu)$ norm, then feeding these predictions into the autonomous ESN \eqref{eqn::auto} in order to generate a future trajectory. Since the errors are bounded in $L^2(\mu)$, we can guarantee a good $L^2(\mu)$ approximation over any finite time horizon.

That said, we cannot rule out the possibility that over a sufficiently long time horizons, the $L^2(\mu)$ errors will accumulate causing the predicted trajectory to diverge from the true trajectory. The predicted trajectory could grow without bound, or even exhibit blow up to infinity in a finite time. 

These outcomes are undesirable, but remarkably, do not often occur in practice. In fact, we saw in Chapter \ref{chapter::introduction} that the autonomous ESN dynamics appear to be diffeomorphic to the Lorenz dynamics, and the future components appear qualitatively similar to the true trajectory, even though the predictions and truth do diverge after a transient period of close approximation, as we would expect from the existence of a positive Lyapunov exponent. We will show that this qualitative similarity is a consequence of the structural stability of the Lorenz system. A structurally stable dynamical system is robust up to small $C^1$ perturbations. Another way to phrase this is that a structurally stable dynamical system is diffeomorphic to itself under sufficiently small $C^1$ perturbations. Therefore if we use an ESN to approximate the dynamics very closely in the $C^1$ norm, the autonomous ESN is guaranteed to follow diffeomorphic dynamics for all future time. 

Theorem \ref{least_sqs_thm} guarantees that the $W$ that we obtain through the linear least squares minimisation yields an approximation in the $L^2(\mu)$ norm, which is sadly weaker than the $C^1$ norm. That is to say, a sequence which converges in $C^1$ also converges in $L^2(\mu)$, but the converse does not hold in general. Theorem \ref{least_sqs_thm} is therefore too weak for the purpose of guaranteeing approximation by the ESN in the $C^1$ norm. That said, Theorem \ref{theorem::RUAT} does at least guarantee the existence of a set of weights $W$ which yields an arbitrarily good $C^1$ approximation. With this motivation, we will prove an existence result here, stating that if the evolution operator $\phi$ is structurally stable then there exist weights $W$ such that the autonomous ESN with this set of weights will exhibit dynamics diffeomorphic to $\phi$.

We will first prove preliminary results which are somewhat technical, so we now carefully explain their purpose before they are stated and proved. First, we assume that the GS $f \in C^1(M,\mathbb{R}^P)$ is an embedding. Then the dynamics on $M$ described by $\phi \in \text{Diff}^1(M)$ are diffeomorphic to the dynamics on the image $f(M)$ described by $f \circ \phi \circ f^{-1}$. Now, we require readout weights $W$ such that $W^\top f$ closely approximates the next step map $\omega \circ \phi$ in the $C^1$ norm. Since we are dealing only with approximations, the vector $W^{\top}$ will act on points in $\mathbb{R}^P$ that are not strictly contained in the embedded manifold $f(M) \subset \mathbb{R}^P$. In order to prevent an accumulation of errors causing the autonomous ESN trajectory to fly away from the manifold $f(M)$, we need that $W^{\top}$ sends points in $\mathbb{R}^P$ that are close to $f(M)$ back onto $f(M)$. To achieve this, we define a discrete time dynamical system $\eta$, i.e. a map, on an open set $\Omega \subset \mathbb{R}^P$ which admits $f(M)$ as a normally hyperbolic attracting submanifold, on which the dynamics of $\eta$ are equal to $f \circ \phi \circ f^{-1}$. Such a manifold attracts nearby points such that trajectories originating near the manifold are attracted to it and move increasingly in directions tangential to $M$. We will choose $W$ so that the autonomous ESN closely approximates the dynamical system $\eta$ in the $C^1$ norm. To make this idea rigorous, we will first introduce a normally hyperbolic attracting submanifold.

\begin{defn}
    (Normally Hyperbolic Attracting Submanifold) Let $\phi \in \text{Diff}^1(M)$. A $\phi$-invariant submanifold $\Lambda \subset M$ is a normally hyperbolic attracting submanifold if the restriction to $\Lambda$ of the tangent bundle of $M$ admits a splitting into a direct sum of two $T\phi$-invariant sub-bundles, the tangent bundle of $\Lambda$, and the stable bundle $E^s$. Furthermore, with respect to some Riemannian metric on $M$, the restriction of $T\phi$ to $E^s$ must be a contraction, and must be relatively neutral on $T\Lambda$. Thus, there exist constants $0 < \lambda < \mu^{-1} < 1$ and $c > 0$ such that 
    \begin{align}
        T_{\Lambda}M &= T\Lambda \oplus E^s \nn \\
        (T\phi)_mE^s_m &= E^s_{\phi(m)} \ \forall \ m \in \Lambda \nn \\
        \lVert T\phi^k v \rVert &\leq c\lambda^k \lVert v \rVert \ \forall \ v \in E^s, \quad \text{and} \quad \forall \ k \in \mathbb{N} \nn  \\
        \lVert T \phi^k v \rVert &\leq c \mu^{\lvert k \rvert} \lVert v \rVert \ \forall \ v \in T\Lambda, \quad \text{and} \quad \forall \ k \in \mathbb{Z} \nn.
    \end{align}
\end{defn}

Next we will prove that there exists a $C^1$ evolution operator $\eta$ defined on $\mathbb{R}^P$ that admits a normally hyperbolic attracting submanifold on which the dynamics of $\eta$ are diffeomorphic to $\phi$. The existence of this map $\eta$ is guaranteed by standard topological machinery which we recall briefly here, and which is presented in detail by \cite{WarnerManifolds}.

\begin{defn}
    (Cubic centred chart) A chart $(V,\varphi)$ belonging to a $P$-manifold is called a cubic chart if $\varphi(V)$ is an open cube centred about the origin in $\mathbb{R}^P$. If $x \in V$ and $\varphi(x) = 0$, then the chart $(V,\varphi)$ is centred at $x$.
\end{defn}

\begin{defn}
    (Slice coordinates) Suppose that $(V,\varphi)$ is a chart on a $P$-manifold $N$ with coordinate functions $\xi_1 , ... , \xi_P$ and that $q$ is an integer $0 \leq q \leq P$. Let $a \in \varphi(V)$ and let
    \begin{align}
        S = \{ x \in V \mid \xi_i(x) = a_i , i = q+1 , ... , P \}. \nn
    \end{align}
    The subspace $S$ of $N$ together with coordinate maps $\xi\rvert_{S}$ for $j = 1 , ... , m$
    forms a submanifold of $N$, called a slice of the chart $(V,\varphi)$.
\end{defn}

\begin{lemma}
    (Slice Lemma) Let $M$ be a compact $q$-manifold, let $f : M \to \mathbb{R}^P$ be an immersion, and let $m \in M$. Then there exists a cubic centred chart $(V,\varphi)$ about $f(m)$ and a neighbourhood $U$ of $m$ such that $f\rvert_{U}$ is injective and $f(U)$ is a slice of $(V,\varphi)$. 
\end{lemma}

\begin{proof}
    \cite{WarnerManifolds} page 28 prop 1.35.
\end{proof}

\begin{lemma}
    Let $P > m$ and $M$ be a compact $q$-manifold. Let $\phi \in \text{Diff}^1(M)$. Suppose $f \in C^1(M,\mathbb{R}^N)$ is a $C^1$ embedding. Then there is an open subset $\Omega \subset \mathbb{R}^P$ and $\eta \in \text{Diff}^1(\Omega)$ with $f(M)$ a normally hyperbolic attracting submanifold such that $\eta\lvert_{f(M)} = f \circ \phi \circ f^{-1}$ (where we have defined $f^{-1}$ on the image of $f$).
    \label{manifold_lemma}
\end{lemma}

\begin{proof}
    We will make a similar argument to \cite{WarnerManifolds} in the proof of his Proposition 1.36, on page 29. First let $m \in M$. Then by the Slice Lemma there exists a cubic centred chart $(V_m,\varphi_m)$ about $f(m)$ and a neighbourhood $U_m$ of $m$ such that $f(U_m)$ is a slice $(V_m,\varphi_m)$.    
    Let $\xi_1, \ldots, \xi_q$ be the slice coordinates in the chart $(V_m,\varphi_m)$ of points in $f(U_m)$.
    Then we can define a map $\eta_m \in \text{Diff}^1(V_m,\mathbb{R}^d)$ applying the map $f \circ \phi \circ f^{-1}$ on the slice co-ordinates and dividing the remaining co-ordinates by 2.   
    We can make this argument for every $m \in M$ hence define a collection of maps $\{ \eta_m \}$ over a collection of open sets $\{ V_m \}$ which cover $f(M)$. Now we let $\{ \alpha_j \ | \ j \in \mathbb{N} \}$ form a partition of unity subordinate to the cover $\{ V_m \}$. We take a subsequence $\{ \alpha_k \}$ such that $\text{supp}(\alpha_k) \cap f(M) \neq \emptyset$ and denote the collection of sets to which $\{ \alpha_k \}$ is subordinate
    by $\{ V_k \}$. We then define a map $\eta$ on a neighbourhood $\Omega : = \cup_{k} V_k$ of $f(M)$ by
    \begin{align}
        \eta = \sum_k \alpha_k \eta_m. \nn
    \end{align}
    By construction, $\eta\rvert_{f(M)} = f \circ \phi \circ f^{-1}$ and $\eta$ has a normally hyperbolic attracting submanifold $f(M)$.
    \end{proof}

Not only does the dynamical system $\eta$ exist, but, importantly, its normally hyperbolic attracting submanifold is preserved by any sufficiently good approximation. This is made formal in the Invariant Manifold Theorem, which we will use in the proof of the ESN Approximation Theorem.

\begin{theorem}
    (Invariant Manifold Theorem) Let $K$ be a compact manifold and $\eta \in \text{Diff}^1(K)$ with normally hyperbolic attracting submanifold $\Lambda$. Then, $\exists \ \epsilon > 0$ such that for any $u \in \text{Diff}^1(K)$ with $\lVert \eta - u \rVert_{C^1} < \epsilon$, the diffeomorphism $u$ has a normally hyperbolic attracting submanifold $U$ such that $\lVert U - \Lambda \rVert_{C^1} < \epsilon$.
    \label{IMT}
\end{theorem}

\begin{proof}
    \cite{Invariant_Manifolds}.
\end{proof}

\begin{lemma}
    Let $M$ be a smooth compact $q$-manifold and $\phi \in \text{Diff}^1(M)$. Let $\omega, u \in C^1(M,\mathbb{R})$ denote the scalar observation function and scalar target function respectively.
    Let $\mathcal{F}$ denote the sequence of maps $F_P : \mathbb{R}^N \times \mathbb{R} \to \mathbb{R}^P$ defined by
    \begin{align*}
        F_P(x,z) = \sigma(Ax + Cz + b)
    \end{align*}
    where
    \begin{itemize}
        \item $\sigma \in C^{1}(\mathbb{R})$ is $1$-finite
        \item $A$ is a $P \times N$ random matrix, where $P > N$ and the first $N$ rows of $A$ form an $N \times N$ random submatrix with 2-norm less than 1 almost surely. The $j^{\mathrm{th}}$ row of $A$ (where $j > N$), denoted $A_j$, is a random variable with full support on $(\mathbb{R}^N)^{\top}$
        \item $C$ is a random $P$-vector with $j^\mathrm{th}$ entry $C_j$ a random variable with full support on $\mathbb{R}$
        \item $b$ is a random $P$-vector with $j\mathrm{th}$ entry $b_j$, a random variable with full support on $\mathbb{R}$.
    \end{itemize}
    Suppose further that the reservoir map $\pi_N \circ F_P : \mathbb{R}^N \times \mathbb{R} \to \mathbb{R}^N$ admits a GS $f \in C^1(M,\mathbb{R}^N)$ which is an embedding.
    
    Then for all $\epsilon > 0$ and $\alpha \in (0,1)$ there exists a $P_0 \in \mathbb{N}$ such that for all $P > P_0$, there exists a $W \in \mathbb{R}^P$ such that 
    \begin{align*}
        \lVert W^{\top}F_P(f\circ\phi^{-1},\omega) - u \rVert_{C^1} < \epsilon
    \end{align*}
    with probability at least $\alpha$.
    \label{lemma::conjugacy}
\end{lemma}

\begin{proof}
    Let $y \in C^1(M,\mathbb{R}^{N} \times \mathbb{R})$ be defined by
    \begin{align*}
        y(m) = (f \circ \phi^{-1}(m), \omega(m)).
    \end{align*}
    Note that $y$ is an embedding because $f$ is an embedding, and $y$ therefore admits an inverse on its image $y(M) \subset \mathbb{R}^N \times \mathbb{R}$. Let $\Phi : C^1(\mathbb{R}^{N+1},\mathbb{R}) \to C^1(M,\mathbb{R})$ be defined 
    \begin{align*}
        \Phi(s) = s \circ y.
    \end{align*}
    Now $\Phi$ is continuous at $u \circ y^{-1}$ so for all $\epsilon > 0$ there exists a $\delta > 0$ such that for any $v \in C^1(\mathbb{R}^{N+1},\mathbb{R})$
    \begin{align*}
        \lVert v - u \circ y^{-1} \rVert_{C^1} < \delta \implies \lVert \Phi(v) - \Phi(u \circ y^{-1}) \rVert_{C^1} < \epsilon.
    \end{align*}
    Now let $g : \mathbb{R}^{N+1} \to \mathbb{R}^P$ be defined by
    \begin{align*}
        g(x) = \sigma([A,C]x + b)
    \end{align*}
    where $[A,C]$ is the random matrix obtained by appending the vector $C$ to the right of $A$. Now we fix $\alpha \in (0,1)$. Then by Theorem \ref{theorem::RUAT}, for all $\delta > 0$ there exists $P_0 \in \mathbb{N}$ such that for all $P > P_0$ there exists a $W \in \mathbb{R}^P$ such that
    \begin{align*}
        \lVert W^{\top}g - u \circ y^{-1} \rVert_{C^1} < \delta
    \end{align*}
    with probability at least $\alpha$.
    Then we proceed directly to estimate that
    \begin{align*}
        &\lVert W^{\top}F_P(f\circ\phi^{-1},\omega) - u \rVert_{C^1} \\ 
        = &\lVert W^{\top}\sigma(Af\circ\phi^{-1} + C\omega + b) - u \rVert_{C^1} \\
        = &\bigg\lVert W^{\top}\sigma\bigg(
        \begin{bmatrix}
            A & C
        \end{bmatrix}
        \begin{bmatrix}
            f \circ \phi^{-1} \\
            \omega
        \end{bmatrix}
        +b\bigg) - u \bigg\rVert_{C^1} \\
        = &\lVert W^{\top}\sigma([A,C]y + b) - u \rVert_{C^1} \\
        = &\lVert W^{\top}g \circ y - u \rVert_{C^1} \\
        = &\lVert \Phi(W^{\top}g) - \Phi(u \circ y^{-1}) \rVert_{C^1} < \epsilon.
    \end{align*} 
\end{proof}

\begin{theorem}
    Adopt the same assumptions as in lemma \ref{lemma::conjugacy}, and the further assumption that $\phi$ is structurally stable. Let $[A,0] \in \mathbb{M}_{P \times P}(\mathbb{R})$ denote the matrix whose first $N$ columns are those of $A$ and the remaining columns are $0$. For all $\epsilon > 0$ and $\alpha \in (0,1)$ there exists a $P_0 \in \mathbb{N}$ such that for all $P > P_0$, there exists with probability at least $\alpha$ a $W \in \mathbb{R}^P$ such that the autonomous ESN $\psi \in C^1(\mathbb{R}^P)$ defined
    \begin{align*}
        \psi(x) = \sigma([A,0]x + C(W^{\top}x) + b)
    \end{align*}
    admits a normally hyperbolic attracting submanifold $V$, of dimension $q$, such that $\psi\rvert_{V} \cong \phi$. The symbol $\cong$ denotes equivalence under diffeomorphism.
    \label{theorem::topological_isomorphiccy}
\end{theorem}

\begin{proof}
    For brevity, we will write for each $P \in \mathbb{N}$ 
    \begin{align*}
        \varphi_P := F_P(f\circ\phi^{-1},\omega).
    \end{align*}
    By lemma \ref{manifold_lemma} there is an open $\Omega \subset \mathbb{R}^P$ that contains $\varphi_P(M)$ such that the dynamical system $\eta \in \text{Diff}^1(\Omega)$ admits $\varphi_P(M)$ as a normally hyperbolic invariant submanifold such that $\eta\rvert_{\varphi_P(M)} = \varphi_P \circ \phi \circ \varphi_P^{-1}$.
    
    Let $K$ be a compact submanifold on $\mathbb{R}^N$ that contains $\varphi_P(M)$. By Theorem \ref{IMT} (Invariant Manifold Theorem) there exists $\epsilon > 0$ such that for any $u \in \text{Diff}^1(K)$ with $\lVert \eta - u \rVert_{C^1} < \epsilon$, the diffeomorphism $u$ has a normally hyperbolic attracting submanifold $V$ such that $\lVert V - \varphi_P(M) \rVert_{C^1} < \epsilon$. 
    Now let $\Phi : C^1(M,\mathbb{R}^P) \to C^1(\mathbb{R}^P,\mathbb{R}^P)$ and $\Sigma : C^1(M,\mathbb{R}^P) \to C^1(M,\mathbb{R}^P)$ be defined by
    \begin{align*}
        \Phi(s) = s \circ \varphi_P, \qquad \Sigma(u) = \sigma \circ s
    \end{align*}
    Since $\Phi \circ \Sigma$ is continuous at $[A,0]\varphi_P + C\omega\circ\phi + b$ there exists a $\delta > 0$ such that for any $v \in C^1(M,\mathbb{R}^P)$
    \begin{align*}
        \lVert v - [A,0]\varphi_P + C\omega\circ\phi + b \rVert_{C^1} &< \delta \\
        \implies \lVert \Phi \circ \Sigma(v) - \Phi \circ \Sigma([A,0]\varphi_P + C\omega\circ\phi + b) \rVert_{C^1} &< \epsilon.
    \end{align*}
    Fix $\alpha \in (0,1)$. Then by lemma \ref{lemma::conjugacy} there exists $P_0 \in \mathbb{N}$ such that for all $P > P_0$, there exists a $W \in \mathbb{R}^P$ such that 
    \begin{align*}
        \lVert W^{\top}\varphi_P - \omega \circ \phi \rVert_{C^1} < \delta 
    \end{align*}
    with probability at least $\alpha$.
   On this event we now have that
    \begin{align*}
        \lVert \psi\rvert_{\varphi_P(M)} - \eta\rvert_{\varphi_P(M)} \rVert_{C^1}
        &= \lVert \psi\rvert_{\varphi_P(M)} - \varphi_P \circ \phi \circ \varphi_P^{-1} \rVert_{C^1} \\
        &= \lVert \Phi(\psi \circ \varphi_P) - \Phi(\varphi_P \circ \phi) \rVert_{C^1} \\
        &= \lVert \Phi \circ \Sigma([A,0]\varphi + C(W^{\top}\varphi_P) + b) - \Phi \circ \Sigma([A,0]\varphi_P + C\omega\circ\phi + b) \rVert_{C^1} \\
        &< \epsilon.
    \end{align*}
    Hence, there is an open set $U \subset \mathbb{R}^P$ with $\varphi_P(M) \subset U \subset K$
    \begin{align*}
        \lVert \psi\rvert_{U} - \eta\rvert_{U} \rVert_{C^1} < \epsilon.
    \end{align*}
   Furthermore, $\psi\rvert_U$ admits a normally hyperbolic attracting submanifold $V$ such that $\lVert V - \varphi(M) \rVert_{C^1} < \epsilon$. Then by structural stability of $\phi$, and therefore of $\eta$, there exists $h \in \text{Diff}^1(\varphi_P(M),V)$ such that
   \begin{align*}
       \psi\rvert_V = h \circ \eta\rvert_{\varphi_P(M)} \circ h^{-1}\rvert \cong \eta\rvert_{\varphi_P(M)}
    \end{align*}
   hence
    \begin{align*}
        \psi\rvert_V = h \circ \eta\rvert_{\varphi_P(M)} \circ h^{-1} = h \circ \varphi_P \circ \phi \circ \varphi_P^{-1} \circ h^{-1} \cong \phi.
    \end{align*}
\end{proof}

The proof is now complete.

To conclude this chapter, we will summarise the major result. If $(M,\phi)$ is a structurally stable discrete time dynamical system, and we take a trajectory of scalar observations, and feed these into an ESN, then there exists a choice of weights $W$ such that the autonomous ESN dynamics are diffeomorphic to $(M,\phi)$. The diffeomorphism ensures that the topological and geometrical features of $(M,\phi)$ like homology groups, Lyapunov exponents, and linearisation of fixed points are learned by the ESN autonomous dynamics. This holds even if the ESN autonomous and $(M,\phi)$ trajectories diverge. We note that the exponential divergence of trajectories is unavoidable for systems such as the Lorenz system that have a positive Lyapunov exponent.

\chapter{Stochastic Dynamical Systems and Reinforcement Learning}
\label{chapter::stochastic}

In this chapter we consider ESNs, and reservoir maps more generally, which are trained on a realisation of a stationary ergodic stochastic process. This is in contrast to previous chapters where the inputs were deterministically generated. In the deterministic case we defined a GS $f : M \to \mathbb{R}^N$, and this was the fundamental object of study. It is not clear how to define a GS in the stochastic context, so the fundamental objects in the stochastic case are \emph{filters} and \emph{functionals}.
The work in the chapter is based on the paper by \cite{arXiv:2102.06258}.

\section{Filters and Functionals}

We will first introduce filters and functionals. A filter $U$ is a map that takes a bi-infinite sequence of real vectors (that is, a sequence of real vectors indexed by $\mathbb{Z}$) and returns another bi-infinte sequence of real vectors (possibly of different dimension). A functional $H$ takes a bi-infinite sequence of real vectors and returns a single real vector.

\begin{defn}
    (Filters and functionals) For any $d,N \in \mathbb{N}$ and $D_d \subset \mathbb{R}^d$ a filter is a map $U : D_d^{\mathbb{Z}} \to (\mathbb{R}^N)^{\mathbb{Z}}$ and a functional is a map $H : D_d^{\mathbb{Z}} \to \mathbb{R}^N$.
\end{defn}

As we will soon see, an ESN can be viewed either as a filter from the space of input sequences to the space of reservoir sequences, or as a functional from the space of input sequences to the current reservoir state. An example of a filter, which we will use extensively in this chapter, is the \emph{time shift operator} which shifts a sequence one step forward in time.

\begin{defn}
    (Time shift operator) The time shift operator $T_N : (\mathbb{R}^N)^{\mathbb{Z}} \to (\mathbb{R}^N)^{\mathbb{Z}}$ is defined by $T_N(z)_k := z_{k+1}$ for all $k \in \mathbb{Z}$.
\end{defn}

In many realistic scenarios we require that the output of a filter depends only on observations from the past and present, and not on the future. That is we require that the filter or functional depends on terms indexed by $k = \ldots n-2,n-1,n$ and does not depend on the terms indexed by $k = n+1, n+2, \ldots$. A filter with this property is called causal because the output is not influenced by the future observations; this would violate the usual notion of causality.

\begin{defn}
    (Causal) We say that a filter $U : D_d^{\mathbb{Z}} \to (\mathbb{R}^N)^{\mathbb{Z}}$ is causal if for all $v,w \in D^{\mathbb{Z}}$ 
    \begin{align*}
        v_k = w_k \ \forall \ k \leq n \implies U(v)_n = U(w)_n.
    \end{align*}
\end{defn}

Furthermore, we will work with filters that are time invariant.

\begin{defn}
    (Time invariant) We say that a filter $U : D_d^{\mathbb{Z}} \to (\mathbb{R}^N)^{\mathbb{Z}}$ is time invariant if it commutes with the time shift operator i.e. $T_N \circ U = U \circ T_d$.
\end{defn}

If a filter is both causal and time invariant then we refer to the filter as a causal time invariant (CTI) filter. It turns out that there is a bijection between the space of CTI filters and the space of functionals.

\begin{theorem}
    \citep{JMLR:v20:19-150} There is a bijection between the space of CTI filters $U : D_d^{\mathbb{Z}} \to (\mathbb{R}^N)^{\mathbb{Z}}$ and the space of functionals $H : D_d^{\mathbb{Z}} \to \mathbb{R}^N$. In particular, for any CTI filter $U$ we can define the functional $H_U := p_0 \circ U$ where $p_0 : (\mathbb{R}^d)^{\mathbb{Z}} \to \mathbb{R}^N$ is the natural projection onto the $0$th entry. For any functional $H$ we can define the filter $U_H$ by defining $U_H(z)_k := H\circ T^k(z)$ for all $k \in \mathbb{Z}$.
\end{theorem}

%\jd{[Some comment on the proof of this theorem would be useful. I can't imagine the proof is more than just a few lines long, indeed. So perhaps include it here for completeness?]}

The local Echo State Property (ESP) has an important characterisation in terms of CTI filters. This appears in \cite{Jaeger2001} and in \cite{YILDIZ20121}.

\begin{theorem}
    (\citep{YILDIZ20121}, Theorem 2.1) Let $V \subset \mathbb{R}^N$ and $W \subset \mathbb{R}^d$ be compact. A continuous reservoir map $F : \mathbb{R}^N \times \mathbb{R}^d \to \mathbb{R}^N$ has the $(V,W)$-local Echo State Property (ESP) if and only if, for any $(z_k)_{k \in \mathbb{Z}} \in W^{\mathbb{Z}}$ there exists a unique $(x_k)_{k \in \mathbb{Z}} \in V^{\mathbb{Z}}$ that satisfies the equation
    \begin{align*}
        x_{k+1} = F(x_k,z_k).
    \end{align*}
    The reservoir map $F$ therefore has local ESP if and only if it has an associated CTI filter $U_F : W^{\mathbb{Z}} \to V^{\mathbb{Z}}$, and associated functional $H_F : W^{\mathbb{Z}} \to V$.
\end{theorem}

The purpose of this reformulation of the ESP is to recognise that an ESN with ESP can be expressed as a functional. In the remainder of this chapter we will use the fact that ESNs are universal approximators in the space of functionals \citep{GRIGORYEVA2018495,Gonon2020} to show that ESNs trained by least squares on stationary ergodic processes can learn arbitrary target functions; which are often in this context called reward functions.

\begin{defn}
    (ESN filter and functional) If an ESN 
    \begin{align*}
        F(x,z) = \sigma(Ax + Cz + b)
    \end{align*}
    has the local ESP then we will write $H^{A,C,b}$ to denote the reservoir functional associated to the ESN. We will also write $H^{A,C,b}_W$ to denote the output functional $W^{\top} H^{A,C,b}$ (defined by left multiplication of $H^{A,C,b}$ by the linear readout layer).
    %so that our notation is consistent with \cite{Gonon2020}.
\end{defn}

\section{System isomorphism}
\label{section::system_isomorphism}

We will now take a slight detour and introduce the notion of a system isomorphism. We do not use this explicitly in the remainder of this chapter, but system isomorphisms are interesting in their own right, and have a curious connection to the linear reservoir systems studied in Chapter \ref{chapter::embedding}.
We say that two reservoir systems are system isomorphic if they define the same input-output systems.

\begin{defn}
    (System Isomorphic) A reservoir map $F : \mathbb{R}^N \times \mathbb{R}^d \to \mathbb{R}^N$ along with output map $h : \mathbb{R}^N \to \mathbb{R}^s$ is a pair $(F,h)$ called a reservoir system. Two reservoir systems $(F,h)$ and $(\bar{F},\bar{h})$ both with local ESP have associated CTI filters $U_F,U_{\bar{F}}$. We say that $(F,h)$ and $(\bar{F},\bar{h})$ are system isomorphic if for any input sequence $z \in (\mathbb{R}^d)^{\mathbb{Z}}$ 
    \begin{align*}
        h ( U_F(z)_k ) = \bar{h} ( U_{\bar{F}}(z)_k ) \ \forall k \in \mathbb{Z} 
    \end{align*}
\end{defn}

In the case of linear reservoir maps $F : \mathbb{R}^N \times \mathbb{R} \to \mathbb{R}^N$ of the form $F(x,z) = Ax + Cz$ for $A \in \mathbb{M}_{N \times N}(\mathbb{R})$ and $C \in \mathbb{R}^N$ the system isomorphisms admit an exact form.

\begin{theorem}
\label{linear_isomorphism}
    Two reservoir systems $(F,h)$ and $(\bar{F},\bar{h})$ such that
    \begin{align*}
        F(x,z) = Ax + Cz, \qquad \bar{F}(x,z) = \bar{A}x + \bar{C}z
    \end{align*}
    where $A,\bar{A} \in \mathbb{M}_{N \times N}(\mathbb{R})$ and $C,\bar{C} \in \mathbb{R}^N$
    are system isomorphic if and only if there exists an invertible $P \in \mathbb{M}_{N \times N}(\mathbb{R})$ such that
    \begin{align*}
        A = P\bar{A}P^{-1}, \qquad C = P\bar{C}
    \end{align*}
    and $h = \bar{h} \circ P^{-1}$.
\end{theorem}

\begin{proof}
    The proof proceeds directly.
    \begin{align*}
        h( U_F(z)_k ) &= h \bigg( \sum_{k'=0}^\infty A^k C z_{k - k'} \bigg) \\
        &= \bar{h} \circ P^{-1} \bigg( \sum_{k'=0}^\infty P \bar{A} P^{-1} P \bar{C} z_{k - k'} \bigg) \\
        &= \bar{h}\bigg( \sum_{k'=0}^\infty \bar{A}^k \bar{C} z_{k - k'} \bigg) \\
        &= \bar{h}( U_{\bar{F}}(z)_k ).
    \end{align*}
\end{proof}

We can observe a connection with Chapter 3 by noting that conditions $(A)-(D)$ in Theorem \ref{embedding_thm} (repeated below in theorem~\ref{embedding_thm_conditions}) for ease of presentation are invariant under system isomorphism.

\begin{theorem}
\label{embedding_thm_conditions}
    The conditions that appear in Theorem \ref{embedding_thm} 
    \begin{itemize}
        \item[(A)] $N > \max\{2q, \ell\}$ where $\ell \in \mathbb{N}$ is the lowest common multiple of the periods of all periodic points,
        \item[(B)] $\lambda_{\text{max}}\rho(A^{n_\text{min}}) < 1$ where $n_\text{min}$ is the minimal period over all periodic points and $\lambda_\text{max}$ is the maximal absolute value over all eigenvalues of all derivatives $T_m \phi^n$,
        \item[(C)] For each periodic point $m \in M$ with period $n$ the vectors
        \begin{align*}
            \bigg\{ (I-\lambda_j A^n)^{-1}(I - A)^{-1}(I-A)^n C \bigg\}_{j = 1, \ldots, q}
        \end{align*}
        where $\{ \lambda_j \}_{j = 1, \ldots q}$ are the eigenvalues of $T_m\phi^n$, are linearly independent,
        \item[(D)] The vectors $\{ A^j C \}_{j = 0, \ldots N-1}$
        are linearly independent,
    \end{itemize}
    are invariant under system isomorphism.
\end{theorem}

\begin{proof}
    Condition $(A)$ clearly invariant because $A$ and $\bar{A}$ are both of size $N$, while condition $(B)$ is invariant because the transformation $\bar{A} = P^{-1} A P$ preserves the spectral radius. To show that $(C)$ is invariant we observe that
    \begin{align*}
        &(I-\lambda_j \bar{A}^n)^{-1}(I - \bar{A})^{-1}(I-\bar{A})^n \bar{C} \\
        =& \bigg( \sum_{k=0}^{\infty} \lambda_j^k \bar{A}^{kn}\bigg)\bigg( \sum_{k=0}^{\infty} \bar{A}^k \bigg) (I-\bar{A})^n \bar{C} \\
        =& P^{-1} \bigg( \sum_{k=0}^{\infty} \lambda_j^k A^{kn}\bigg) P P^{-1} \bigg( \sum_{k=0}^{\infty} A^k \bigg) P (I-P^{-1} A P)^n P^{-1} C \\
        =& P^{-1} (I-\lambda_j A^n)^{-1}(I - A)^{-1}(I-A)^n C
    \end{align*}
 %\jd{[these edits are to align this calculation precisely with your definition of $P$ in the statement of the theorem]}
    and note that the vectors
    \begin{align*}
        \bigg\{ (I-\lambda_j A^n)^{-1}(I - A)^{-1}(I-A)^n C \bigg\}_{j = 1, \ldots, q}
    \end{align*}
    are linearly independent if and only if the vectors 
    \begin{align*}
        \bigg\{ P^{-1} (I-\lambda_j A^n)^{-1}(I - A)^{-1}(I-A)^n C \bigg\}_{j = 1, \ldots, q}
    \end{align*}
    are linearly independent.
    Finally we can see that if the vectors $\{ \bar{A}^j\bar{C} \}_{j = 0, \ldots , N-1}$ are linearly independent then so are the vectors
    \begin{align*}
        \{ P \bar{A}^j \bar{C} \}_{j = 0, \ldots , N-1}
    \end{align*} and then noting that
    \begin{align*}
        P \bar{A}^j \bar{C} = P \bar{A}^j P^{-1} P \bar{C} = A^jC
    \end{align*}
    completes the proof.
\end{proof}

\section{Supervised learning on stationary ergodic processes}

For an ESN to successfully learn from a stochastic process, the process must satisfy mild conditions. The following definition that appears in \cite{Gonon2020} outlines these conditions. For the remainder of this chapter, we will use boldface for random variables.

\begin{defn}
    (Admissible input process) A $(\mathbb{R}^d)^\mathbb{Z}$-valued random variable $\boldsymbol{Z}$ is called an admissible process if for any $T_0 \in \mathbb{N}$ there exists $M_{T_0}>0$ such that for all $k \in \mathbb{Z}$ the $d \times (T_0+1)$ random matrix $[\boldsymbol{Z}_{k-T_0} , \boldsymbol{Z}_{k-T_0+1} , \ldots  , \boldsymbol{Z}_k]$ satisfies
    \begin{align}
        \label{admissible}
        \lVert \boldsymbol{Z}_{k-T_0} , \boldsymbol{Z}_{k-T_0+1} , \ldots  , \boldsymbol{Z}_k \rVert \leq M_{T_0}
    \end{align}
    Lebesgue-almost surely.
\end{defn}

Furthermore, to ensure the learning is successful the matrices $\boldsymbol{A},\boldsymbol{C},\boldsymbol{b}$ must be drawn from appropriate distributions. In the following we describe the procedure introduced by \cite{Gonon2020} by which $\boldsymbol{A},\boldsymbol{C},\boldsymbol{b}$ are randomly generated. 

\begin{proc}
    Let $n,T_0 \in \mathbb{N}$, $R > 0$ be the input parameters for the procedure. Suppose that $\boldsymbol{Z}$ is an admissible input process. Consequently, for any $T_0 \in \mathbb{N}$ there exists $M_{T_0}$ such that (setting $k=0$ in (\ref{admissible}))
    \begin{align*}
        \lVert \boldsymbol{Z}_{-T_0} , \boldsymbol{Z}_{-T_0+1} , \ldots  , \boldsymbol{Z}_0 \rVert \leq M_{T_0}
    \end{align*}
    Lebesgue-almost surely.
    Then, for a given $T_0$, we initialise the ESN reservoir matrix $\boldsymbol{A}$, input matrix $\boldsymbol{C}$, and biases $\boldsymbol{b}$ according to the following procedure.
    \begin{enumerate}
        \item Draw $n$ i.i.d. samples $\boldsymbol{A}_1 , \ldots  , \boldsymbol{A}_n$ from the uniform distribution on $B_R \subset \mathbb{R}^{d(T_0+1)}$ where $B_R$ is the ball of radius $R$ and centre 0, and draw $N$ i.i.d. samples $\boldsymbol{b}_1 , \ldots  \boldsymbol{b}_N$ from the uniform distribution on $[-\max(M_{T_0} R, 1) , \max(M_{T_0} R, 1)]$.
        \item Let $S$ and $c$ be shift matrices defined by
        \begin{align*}
            S = 
            \begin{bmatrix}
                0_{d,dT_0} & 0_{d,d} \\
                I_{dT_0} & 0_{dT_0,d}
            \end{bmatrix}
            \qquad 
            c = 
            \begin{bmatrix}
                I_{d} \\
                0_{dT_0,d}
            \end{bmatrix}
        \end{align*}
       where $I_{dT_0}$ and $I_d$ denote the $dT_0 \times dT_0$ and $d \times d$ identity matrices, respectively, and the dimensions of the other (rectangular or square) matrices are given by each pair of subscripts,
        and set
        \begin{align*}
            \boldsymbol{a} =
            \begin{bmatrix}
                \boldsymbol{A}_1^{\top} \\
                \boldsymbol{A}_2^{\top} \\
                \vdots \\
                \boldsymbol{A}_n^{\top}
            \end{bmatrix}
            \qquad 
            \boldsymbol{\bar{A}} =
            \begin{bmatrix}
                S & 0_{d(T_0+1),n} \\
                \boldsymbol{a}S & 0_{n,n} 
            \end{bmatrix}
                \\
            \boldsymbol{\bar{C}} =
            \begin{bmatrix}
                c \\
                \boldsymbol{a}c 
            \end{bmatrix}
            \qquad
            \boldsymbol{\bar{\zeta}} = 
            \begin{bmatrix}
                0_{d(T_0+1)} \\
                \boldsymbol{b}_1 \\
                \vdots \\
                \boldsymbol{b}_n
            \end{bmatrix}
        \end{align*}
        so that we can define
        \begin{align*}
            \boldsymbol{A} = 
            \begin{bmatrix}
                \boldsymbol{\bar{A}} & -\boldsymbol{\bar{A}} \\
                -\boldsymbol{\bar{A}} & \boldsymbol{\bar{A}}
            \end{bmatrix}
            \qquad
            \boldsymbol{C} = 
            \begin{bmatrix}
                \boldsymbol{\bar{C}} \\
                -\boldsymbol{\bar{C}}
            \end{bmatrix}
            \qquad
            \boldsymbol{b} = 
            \begin{bmatrix}
                \boldsymbol{\bar{b}} \\
                -\boldsymbol{\bar{b}}
            \end{bmatrix}.
        \end{align*}
    \end{enumerate}
    \label{proc:procedure}
\end{proc}

With everything set up we are ready to introduce a result by \cite{Gonon2020} which we will build on later in the chapter. The result roughly assumes that we have an admissible process, and a target function we wish to approximate to a tolerance $\epsilon$. Then for large enough $N$, any randomly generated ESN of dimensional $N$ admits an output layer $W \in \mathbb{R}^N$ with which we can approximate the target function to within the tolerance $\epsilon$.

\begin{theorem}
[\cite{Gonon2020}] Suppose that $\boldsymbol{Z}$ is an admissible input process. Let $\mathcal{R} : (D_d)^{\mathbb{Z}} \to \mathbb{R}$ (where $D_{d}$ is a compact subset of $\mathbb{R}^d$) be causal and measurable with respect to some measure $\mu$ such that $\mathbb{E}_{\mu}[|\mathcal{R}(\boldsymbol{Z})|^2] < \infty$.

Then for any $\epsilon > 0$ and $\delta \in (0,1)$ there exist $n,T_0 \in \mathbb{N}$, and $R > 0$ such that with probability $(1-\delta)$ the ESN with parameters $\boldsymbol{A},\boldsymbol{C},\boldsymbol{b}$ generated by procedure~\ref{proc:procedure} (with inputs $n,T_0,R$) has the local ESP and admits a readout layer $W \in \mathbb{R}^{2(d(T_0+1)+n)}$ such that
\begin{align*}
\bigg(\mathbb{E}_{\mu}\left[ \left. \left\lVert H^{\boldsymbol{A},\boldsymbol{C},\boldsymbol{b}}_W(\boldsymbol{Z}) - \mathcal{R}(\boldsymbol{Z}) \right\rVert^2 \, \right| \, \boldsymbol{A} , \boldsymbol{C} , \boldsymbol{b} \right]\bigg)^{1/2}
\hspace{-0.25cm} :=
\bigg( \int_{(\mathbb{R}^d)^{\mathbb{Z}}} \left\lVert H^{\boldsymbol{A},\boldsymbol{C},\boldsymbol{b}}_W(z) - \mathcal{R}(z)\right\rVert^2 d\mu(z) \bigg)^{1/2}
\hspace{-0.25cm} < \epsilon .
\end{align*}
    \label{general_ESN_approximation_theorem}
\end{theorem}

One observation we make about this result is that the output layer $W \in \mathbb{R}^N$, where $N := 2(d(T_0+1) + n)$, is shown to exist, but is not constructed. Of course, we are in practice often interested in the conditions under which the output layer obtained by regularised least squares is close the output layer $W \in \mathbb{R}^N$ that approximates the target functional. However, from the above result we cannot claim that the regularised-least-squares-minimising output layer $W$ does actually satisfy the criterion in the conclusion of the theorem above.

In general, for the least squares solution to be accurate we necessarily require that the sample trajectory is representative of the entire input process $\boldsymbol{Z}$. We can ensure that any sufficiently long sample trajectory is representative of the entire process $\boldsymbol{Z}$ by insisting that $\boldsymbol{Z}$ is stationary and ergodic.

A stochastic process being stationary is analogous to a deterministic dynamical system being autonomous, and is defined below.

\begin{defn}
    (Stationary Process; \cite{mcgoff2015}) A stochastic process $(\boldsymbol{Z}_k)_{k \in \mathbb{Z}} \equiv \boldsymbol{Z}$ is stationary if for any $k \in \mathbb{N}$ and finite subset $I \subset \mathbb{Z}$ the joint distribution $(\boldsymbol{Z}_i)_{i \in I}$ is equal to the joint distribution $(\boldsymbol{Z}_{i+k})_{i \in I}$.
\end{defn}

Building on this definition, a stationary ergodic proccess is very much like an ergodic deterministic system. The definition is stated below.

\begin{defn}
    (Stationary Ergodic Process; \cite{mcgoff2015}) A stationary stochastic process $(\boldsymbol{Z}_k)_{k \in \mathbb{Z}} \equiv \boldsymbol{Z}$ is called ergodic if for every $i \in \mathbb{N}$ and every pair of Borel sets $A,B$
    \begin{align*}
        \lim_{\ell \to \infty} \frac{1}{\ell} \sum^{\ell-1}_{k=0} &\mathbb{P}\bigg( (\boldsymbol{Z}_1 , \ldots , \boldsymbol{Z}_{i}) \in A, (\boldsymbol{Z}_k , \ldots , \boldsymbol{Z}_{i+k}) \in B \bigg) \\
        = &\mathbb{P}\bigg( (\boldsymbol{Z}_1 , \ldots , \boldsymbol{Z}_{i}) \in A \bigg) \mathbb{P}\bigg( (\boldsymbol{Z}_1 , \ldots , \boldsymbol{Z}_{i}) \in B \bigg).
    \end{align*}
\end{defn}
%\jd{[I don't understand the displayed equation above. If this is a definition for every $\ell$ then why does the limit $\ell \to \infty$ appear? And then why is the LHS then independent of $\ell$ but the RHS still depends on $\ell$?]}

Every stationary ergodic processes $\boldsymbol{Z}$ satisfies the celebrated ergodic theorem, stating that the sample average of a trajectory converges to the expectation of the invariant distribution.

\begin{theorem}
    (Ergodic Theorem) If $(\boldsymbol{Z}_k)_{k \in \mathbb{Z}} \equiv \boldsymbol{Z}$ is a stationary ergodic process then for any $i \in \mathbb{Z}$
    \begin{align*}
        \mathbb{E}_{\mu}[\boldsymbol{Z}_i] = \lim_{\ell \to \infty} \frac{1}{\ell}\sum_{k=0}^{\ell-1}  \boldsymbol{Z}_{i+k}
    \end{align*}
    almost surely.
\end{theorem}

The ergodic theorem is the crucial ingredient for Theorem \ref{generalised_training_theorem}. The result roughly assumes that we have an admissible stationary ergodic process, and a target function we wish to approximate to a tolerance $\epsilon$. Then for large enough $N$ and $\ell$, any randomly generated ESN of dimensional $N$ trained by least squares regression on a sample trajectory of length $\ell$, explicitly yields an output layer $W \in \mathbb{R}^N$ which approximates the target function to tolerance $\epsilon$.

\begin{theorem}
    \label{generalised_training_theorem}
    \citep{arXiv:2102.06258} Suppose that $\boldsymbol{Z}$ is an admissible input process that is also stationary and ergodic, with invariant measure $\mu$. Let $\mathcal{R} : (D_d)^{\mathbb{Z}} \to \mathbb{R}$ (where $D_d$ is a compact subset of $\mathbb{R}^d$) be causal, $\mu$-measurable, and satisfy $\mathbb{E}_{\mu}[|\mathcal{R}(\boldsymbol{Z})|^2] < \infty$. Let $z$ be an arbitrary realisation of $\boldsymbol{Z}$.
    
    Then for any $\epsilon > 0$, and $\delta \in (0,1)$ there exist
    \begin{itemize}
        \item constants $n,T_0 \in \mathbb{N}, R, \lambda^* > 0$ and $\ell \in \mathbb{N}$,
        \item an ESN with parameters $\boldsymbol{A},\boldsymbol{C},\boldsymbol{b}$ generated by procedure \ref{proc:procedure} (with inputs $n,T_0,R$),
        \item an output layer $W^{*}_{\ell} \in \mathbb{R}^{2(d(T_0+1)+n)}$ which minimises (over $W\in \mathbb{R}^{2(d(T_0+1)+n)}$) the least squares problem
    \begin{align}
        \frac{1}{\ell}\sum_{k=0}^{\ell - 1} \left\lVert H_{W}^{\boldsymbol{A},\boldsymbol{C},\boldsymbol{b}} T^{-k}(z) - \mathcal{R} T^{-k}(z) \right\rVert^2 + \lambda \left\lVert W \right\rVert^2, \nonumber
    \end{align}
    where $\lambda \in (0,\lambda^*)$,
    \end{itemize} 
    such that, with probability $(1-\delta)$, the inequality
    \begin{align}
        \mathbb{E}_{\mu}\left[ \left. \left\lVert H_{W^*_{\ell}}^{\boldsymbol{A},\boldsymbol{C},\boldsymbol{b}}(\boldsymbol{Z}) - \mathcal{R}(\boldsymbol{Z}) \right\rVert^2 \right| \boldsymbol{A},\boldsymbol{C},\boldsymbol{b} \right] < \epsilon \nonumber
    \end{align}
    is satisfied.
\end{theorem}

\begin{proof}
    We state and prove a result later (Theorem \ref{offline_q_learning}) of which this present result is the special case (in which $\gamma = 0$).
\end{proof}

Though we have constructed $W \in \mathbb{R}^N$ explicitly, the proof fails to be fully constructive because we do not know how many neurons $N$ or sample points $\ell$ are required to provide an approximation with tolerance $\epsilon$. We consider first how the error decreases as we increase the number of sample points $\ell$. To this end we recall the central limit theorem (CLT) for stationary ergodic processes which states that the error between the time average and expectation of the invariant measure converges to a normal distribution with standard deviation $1/\sqrt{\ell}$ as the number of sample points $\ell$ grows to infinity. 

\begin{theorem}
    (Central Limit Theorem; \cite{mcgoff2015}) If $(\boldsymbol{Z}_k)_{k \in \mathbb{Z}}$ is a stationary ergodic process then there exists a covariance matrix $\Sigma$ such that for any $i \in \mathbb{Z}$ and Borel set $A$
    \begin{align*}
        \lim_{\ell \to \infty}\mathbb{P}\bigg(\frac{1}{\sqrt{\ell}}\sum_{k=0}^{\ell-1} (\boldsymbol{Z}_{i+k} - \mathbb{E}_{\mu}[\boldsymbol{Z}_i]) \in A \bigg) = \mathbb{P} \big( \mathcal{N}(0,\Sigma) \in A \big).
    \end{align*}
    In other words, the random variables
    \begin{align*}
        \frac{1}{\sqrt{\ell}}\sum_{k=0}^{\ell-1} (\boldsymbol{Z}_{i+k} - \mathbb{E}_{\mu}[\boldsymbol{Z}_i])
    \end{align*}
    converge in distribution to the multivariate normal $\mathcal{N}(0,\Sigma)$ as $\ell \to \infty$.
\end{theorem}

This suggests that the convergence of the error with $\ell$ is of order $1 / \sqrt{\ell}$, and the constant factor in the convergence estimate is related to the mixing time and the variance of the process $\boldsymbol{Z}$. It could be a fruitful direction of future work to explore this in more detail. 

%\jd{[Make sure that you mention this again in the Conclusion chapter when discussing future work!]}

We are also interested in the convergence of error as the number of neurons $N$ grows. There are several results \citep[e.g.][]{Gonon2020} which establish explicit approximation bounds of order $1/\sqrt{N}$ using the CLT for i.i.d. random variables. Such an approximation bound is likely to hold in this context too.

\section{Reinforcement learning on stationary ergodic processes}

In the previous section we considered a scenario where we use an ESN to learn a given target functional given a sample trajectory of labelled data. This is a classic supervised learning problem. In this current section we will extend these results to a reinforcement learning (RL) setting.

In the RL paradigm we have an agent that explores its environment and at each moment in time $k \in \mathbb{Z}$ executes an action $a_k$, makes an observation $\omega_k$ of its surroundings, and obtains a reward $r_k$. So at every time point the agent records a (reward, action, observation) triple $(r_k,a_k,\omega_k)$. The goal of the agent to choose a sequence of actions that will maximise the present value $\sum_{k=0}^{\infty}\gamma^kr_k $ of its future rewards, discounting rewards by a factors of $\gamma \in [0,1)$ when they occur further into the future.

For a given stochastic sequence of actions $(\boldsymbol{A}_k)_{k \in \mathbb{Z}}$ which depend on the random observations $(\boldsymbol{\Omega}_k)_{k \in \mathbb{Z}}$ and rewards $(\boldsymbol{R}_k)_{k \in \mathbb{Z}}$ we have a random process
$\boldsymbol{Z} \equiv (\boldsymbol{Z}_k)_{k \in \mathbb{Z}} = (\boldsymbol{R}_k,\boldsymbol{A}_k,\boldsymbol{\Omega}_k)_{k \in \mathbb{Z}}$. Now suppose an agent has a particular history of (reward, action, observation) triples leading up to the present moment. It now makes sense to define the \emph{value} of that history as the expectation of the discounted sum of future rewards conditional on the history of (reward, action, observation) triples. In particular we can define a causal value functional $V$ which takes a sequence of (reward, action, observation) triples $z$, and returns their value.

\begin{defn}
\label{defn::value_func}
    (Value functional) Let $D_d$ be a compact subset of $\mathbb{R}^d$ and $\mathcal{R} : (D_d)^{\mathbb{Z}} \to \mathbb{R}$ a $\mu$- measurable causal reward functional that satisfies $\mathbb{E}[\mathcal{R}(\boldsymbol{Z})^2] < \infty$. Let $T^k : \mathbb{R}^d \to \mathbb{R}^d$ denote the $k$-fold composition of the shift map with itself. We define the causal value functional $V : (D_d)^{\mathbb{Z}} \to \mathbb{R}$ (with respect to the process $\boldsymbol{Z}$) as
\begin{align*}
    V(z) &:= \mathbb{E}_{\mu}\bigg[ \sum_{k=0}^{\infty} \gamma^k \mathcal{R}T^k(\boldsymbol{Z}) \ \bigg| \ \boldsymbol{Z}_j = z_j \ \forall j \leq 0 \bigg].
\end{align*}
\end{defn}

We will see in Theorem \ref{V_is_fixed_pt} that $V$ is actually the unique fixed point of a contraction mapping called the Bellman operator, which appears in reinforcement learning and in stochastic control.   

\begin{defn}
    Let $\boldsymbol{Z}$ be a stationary ergodic process. Then we define the map $T_{\boldsymbol{Z}}$ as a CTI filter on the bi-infinite sequences $(D_d)^{\mathbb{Z}}$, which returns the random variable:
\begin{align*}
    T_{\boldsymbol{Z}}(z)_k =
    \begin{cases}
        z_{k+1} &\text{ if } k < 0 \\
        \boldsymbol{Z}_{k+1} \ | \ \boldsymbol{Z}_j = z_j \ \forall j \leq 0 &\text{ if } k \geq 0.
    \end{cases}
\end{align*}
\end{defn}

\begin{theorem}
\label{V_is_fixed_pt}
Let $\boldsymbol{Z}$ be a stationary ergodic process with invariant measure $\mu$. Consider the space of functionals $H : (D_d)^{\mathbb{Z}} \to \mathbb{R}$ that satisfy $\mathbb{E}_\mu[H(\boldsymbol{Z})^2] \leq \infty$. Now equip this space with the norm
    \begin{align*}
        \lVert H \rVert_{\mu} = \sqrt{\mathbb{E}_{\mu}[ H(\boldsymbol{Z})^2 ]}.
    \end{align*}
    Let $\mathcal{R} : (D_d)^{\mathbb{Z}} \to \mathbb{R}$ be a $\mu$ measurable causal reward functional that satisifes $\mathbb{E}_\mu[\mathcal{R}(\boldsymbol{Z})^2] < \infty$. Then for $\gamma \in [0,1)$ the operator 
    \begin{align}
    \label{Phi}
        \Phi(H)(z) = \gamma \mathbb{E}_{\mu}[HT_{\boldsymbol{Z}}(z)] + \mathcal{R}(z)
    \end{align}
   is a contraction mapping with Lipshitz constant $\gamma$ and unique fixed point $V$, which is the value functional in definition \ref{defn::value_func}. 
\end{theorem}

\begin{proof}
We will first show that $\Phi$ is a contraction mapping with Lipschitz constant $\gamma$. To see this observe that
\begin{align*}
    \lVert \Phi(H_1) - \Phi(H_2) \rVert_{\mu}^2 &= \mathbb{E}_{\mu}\bigg[(\mathcal{R}(\boldsymbol{Z}) + \gamma \mathbb{E}_{\mu}[H_1 T_{\boldsymbol{Z}}(\boldsymbol{Z})] - \mathcal{R}(\boldsymbol{Z}) - \gamma \mathbb{E}_{\mu}[H_2 T_{\boldsymbol{Z}}(\boldsymbol{Z})])^2\bigg] \\
    &= \gamma^2 \mathbb{E}_{\mu}\bigg[\mathbb{E}_{\mu}[H_1 T_{\boldsymbol{Z}}(\boldsymbol{Z}) - H_2 T_{\boldsymbol{Z}}(\boldsymbol{Z})]^2\bigg] \\
    &= \gamma^2 \mathbb{E}_{\mu}\big[ (H_1 T(\boldsymbol{Z}) - H_2 T(\boldsymbol{Z}))^2 \big] \text{ (by the law of total expectation)} \\
    &= \gamma^2\mathbb{E}_{\mu}\big[(H_1(\boldsymbol{Z}) - H_2(\boldsymbol{Z}))^2\big] \text{ (by stationary ergodicity of $\boldsymbol{Z}$)} \\
    &= \gamma^2\lVert H_1 - H_2 \rVert_{\mu}^2.
\end{align*}
Then by Banach's fixed point theorem the operator $\Phi$ admits a unique fixed point. We will now show that the value function $V$ is indeed this fixed point.
    Re-arranging the definition of $V(z)$, we have that:
\begin{align*}
    V(z) &= \mathbb{E}_{\mu}\bigg[ \sum_{k=0}^{\infty} \gamma^k \mathcal{R}T^k(\boldsymbol{Z}) \ \bigg| \ \boldsymbol{Z}_j = z_j \ \forall j \leq 0 \bigg] \\
    &= \mathbb{E}_{\mu}\bigg[ \sum_{k=1}^{\infty} \gamma^k \mathcal{R}T^k(\boldsymbol{Z}) \ \bigg| \ \boldsymbol{Z}_j = z_j \ \forall j \leq 0 \bigg] + \mathcal{R}(z) \\
    &= \gamma \mathbb{E}_{\mu}\bigg[ \sum_{k=0}^{\infty} \gamma^k \mathcal{R}T^{k+1}(\boldsymbol{Z}) \ \bigg| \ \boldsymbol{Z}_j = z_j \ \forall j \leq 0 \bigg] + \mathcal{R}(z) \\
    &= \gamma \mathbb{E}_{\mu}\bigg[ \sum_{k=0}^{\infty} \gamma^k \mathcal{R}T^k(\boldsymbol{Z}) \ \bigg| \ \boldsymbol{Z}_j = z_{j+1} \ \forall j < 0 \bigg] + \mathcal{R}(z) \\
\end{align*}
where we have carried out straightforward relabellings of the indexing of terms in the sum by $k$. Then by the law of total expectation we may write this last expression as
\begin{align*}    
    V(z) &=\gamma\mathbb{E}_{\mu}\bigg[\mathbb{E}_{\mu}\bigg[ \sum^{\infty}_{k=0}\gamma^k \mathcal{R}T^k(\boldsymbol{Z}) \ \bigg| \ \boldsymbol{Z}_j = T_{\boldsymbol{Z}}(z)_j \ \forall j \leq 0 \bigg]\bigg] + \mathcal{R}(z) \\
    &= \gamma \mathbb{E}_{\mu}[VT_{\boldsymbol{Z}}(z)] + \mathcal{R}(z)
    = \Phi(V)(z),
\end{align*}
which shows that $V$ is indeed a fixed point of $\Phi$, and so is the unique such, since $\Phi$ is a contraction.
\end{proof}

We are now ready to prove our next result, which generalises Theorem \ref{generalised_training_theorem} and applies to the following reinforcement learning context. We envision an agent that executes a random sequence of actions $(\boldsymbol{A}_k)_{k \in \mathbb{Z}}$ which depend on the random observations $(\boldsymbol{\Omega}_k)_{k \in \mathbb{Z}}$ and rewards $(\boldsymbol{R}_k)_{k \in \mathbb{Z}}$ hence we have a random process
$\boldsymbol{Z} \equiv (\boldsymbol{Z}_k)_{k \in \mathbb{Z}} = (\boldsymbol{R}_k,\boldsymbol{A}_k,\boldsymbol{\Omega}_k)_{k \in \mathbb{Z}}$ controlled by the actions $(\boldsymbol{A}_k)_{k \in \mathbb{Z}}$. We assume that the controlled process $\boldsymbol{Z}$ is stationary and ergodic. Then for any tolerance $\epsilon > 0$ if we train an ESN with sufficiently many neurons $N$ on a trajectory of sufficiently many sample points $\ell$, using regularised least squares, then we can approximate the fixed point of $\Phi$ (and therefore the value functional $V$) to within the tolerance $\epsilon$.

\begin{theorem}
\label{offline_q_learning} \citep{arXiv:2102.06258}
    Suppose that $\boldsymbol{Z}$ is an admissible input process, that is also stationary and ergodic with invariant measure $\mu$. Let $\mathcal{R} : (D_d)^{\mathbb{Z}} \to \mathbb{R}$ be causal, $\mu$-measurable and satisfy $\mathbb{E}[|\mathcal{R}(\boldsymbol{Z})|^2] < \infty$ and define $\Phi$ using \eqref{Phi}
    on the $\mu$-measurable functionals $H$ that satisfy $\mathbb{E}_{\mu}[|H(\boldsymbol{Z})|^2] < \infty$. Let $\gamma \in [0,1)$. Let $z$ be an arbitrary realisation of $\boldsymbol{Z}$.
    
    Then for any $\epsilon > 0$, and $\delta \in (0,1)$ there exist
    \begin{itemize}
        \item constants $n,T_0 \in \mathbb{N}, R, \lambda^* > 0$ and $\ell \in \mathbb{N}$,
        \item an ESN with parameters $\boldsymbol{A},\boldsymbol{C},\boldsymbol{b}$ generated by procedure \ref{proc:procedure} (with inputs $n,T_0,R$),
        \item an output layer $W^{*}_{\ell} \in \mathbb{R}^{2(d(T_0+1)+n)}$ minimising (over $W \in \mathbb{R}^{2(d(T_0+1)+n)}$) the least squares problem
    \begin{align}
        \frac{1}{\ell}\sum_{k = 0}^{\ell - 1} \left\lVert W^{\top} (H^{\boldsymbol{A},\boldsymbol{C},\boldsymbol{b}}T^{-k}(z) - \gamma H^{\boldsymbol{A},\boldsymbol{C},\boldsymbol{b}}T^{1-k}(z)) - \mathcal{R}(z) \right\rVert^2 + \lambda \lVert W \rVert^2 \nonumber
    \end{align}
    where $\lambda \in (0,\lambda^*)$,
    \end{itemize} 
     such that, with probability $(1-\delta)$, the inequality
    \begin{align}
        \mathbb{E}_{\mu}\left[ \left. \left\lVert H_{W^*_{\ell}}^{\boldsymbol{A},\boldsymbol{C},\boldsymbol{b}}(\boldsymbol{Z}) - \Phi H_{W^*_{\ell}}^{\boldsymbol{A},\boldsymbol{C},\boldsymbol{b}}(\boldsymbol{Z}) \right\rVert^2 \right| \boldsymbol{A},\boldsymbol{C},\boldsymbol{b} \right] < \epsilon \nonumber
    \end{align}
    is satisfied.
\end{theorem}
\begin{proof}
    First let $V$ be the unique fixed point of the contraction mapping $\Phi$; the existence and uniqueness of this fixed point is guaranteed by Banach's fixed point theorem. Recall the Lipschitz constant of $\Phi$ is $\gamma$.
    Fix $\epsilon > 0$ and $\delta \in (0,1)$. Then by Theorem \ref{general_ESN_approximation_theorem} there exists with probability $(1-\delta)$ a linear readout layer $W \in \mathbb{R}^{2(d(T_0+1) + n)}$ such that
    \begin{align}
        \mathbb{E}_{\mu}\left[ \left. \left\lVert H_W^{\boldsymbol{A},\boldsymbol{C},\boldsymbol{b}}(\boldsymbol{Z}) - V(\boldsymbol{Z}) \right\rVert^{2} \ \right| \boldsymbol{A} , \boldsymbol{C} , \boldsymbol{b} \right] < \frac{\epsilon}{5(1 + \gamma)}. 
        \label{HW-H*}
    \end{align}
    Then it follows that
    \begin{align*}
        \mathbb{E}_{\mu}[ \lVert H^{\boldsymbol{A},\boldsymbol{C},\boldsymbol{b}}_{W} - \Phi H^{\boldsymbol{A},\boldsymbol{C},\boldsymbol{b}}_{W} \rVert^2 | \boldsymbol{A},\boldsymbol{C},\boldsymbol{b} ]  \\
         & \dawesspace = \mathbb{E}_{\mu}[ \lVert H^{\boldsymbol{A},\boldsymbol{C},\boldsymbol{b}}_{W}(\boldsymbol{Z}) - \Phi H^{\boldsymbol{A},\boldsymbol{C},\boldsymbol{b}}_{W}(\boldsymbol{Z}) + V(\boldsymbol{Z}) - V(\boldsymbol{Z}) \rVert^2 | \boldsymbol{A},\boldsymbol{C},\boldsymbol{b} ] \\
        &  \dawesspace \leq \mathbb{E}_{\mu}[ \lVert H^{\boldsymbol{A},\boldsymbol{C},\boldsymbol{b}}_{W}(\boldsymbol{Z}) - V(\boldsymbol{Z}) \rVert^2 | \boldsymbol{A},\boldsymbol{C},\boldsymbol{b} ] + \mathbb{E}_{\mu}[ \lVert V(\boldsymbol{Z}) - \Phi H^{\boldsymbol{A},\boldsymbol{C},\boldsymbol{b}}_W(\boldsymbol{Z}) \rVert^2 | \boldsymbol{A},\boldsymbol{C},\boldsymbol{b} ] \\
        &  \dawesspace = \mathbb{E}_{\mu}[ \lVert H^{\boldsymbol{A},\boldsymbol{C},\boldsymbol{b}}_{W}(\boldsymbol{Z}) - V(\boldsymbol{Z}) \rVert^2 | \boldsymbol{A},\boldsymbol{C},\boldsymbol{b} ] + \mathbb{E}_{\mu}[ \lVert \Phi V(\boldsymbol{Z}) - \Phi H^{\boldsymbol{A},\boldsymbol{C},\boldsymbol{b}}_W(\boldsymbol{Z}) \rVert^2 | \boldsymbol{A},\boldsymbol{C},\boldsymbol{b} ] \\
        &  \dawesspace \leq \mathbb{E}_{\mu}[ \lVert H^{\boldsymbol{A},\boldsymbol{C},\boldsymbol{b}}_{W}(\boldsymbol{Z}) - V(\boldsymbol{Z}) \rVert^2 | \boldsymbol{A},\boldsymbol{C},\boldsymbol{b} ] + \gamma \mathbb{E}_{\mu}[ \lVert V(\boldsymbol{Z}) - H^{\boldsymbol{A},\boldsymbol{C},\boldsymbol{b}}_W(\boldsymbol{Z}) \rVert^2 | \boldsymbol{A},\boldsymbol{C},\boldsymbol{b} ] \\
        &  \dawesspace = (1 + \gamma) \mathbb{E}_{\mu}[ \lVert V(\boldsymbol{Z}) - H^{\boldsymbol{A},\boldsymbol{C},\boldsymbol{b}}_W(\boldsymbol{Z}) \rVert^2 | \boldsymbol{A},\boldsymbol{C},\boldsymbol{b} ]\\
        &  \dawesspace < (1 + \gamma) \frac{\epsilon}{5(1+\gamma)} \text{ by \eqref{HW-H*}} \\
        &  \dawesspace < \frac{\epsilon}{5}
    \end{align*}
    which yields the estimate 
    \begin{align}
        \mathbb{E}_{\mu}[ \lVert H^{\boldsymbol{A},\boldsymbol{C},\boldsymbol{b}}_{W} - \Phi H^{\boldsymbol{A},\boldsymbol{C},\boldsymbol{b}}_{W} \rVert^2 | \boldsymbol{A},\boldsymbol{C},\boldsymbol{b} ] < \frac{\epsilon}{5}. \label{H-PhiH}
    \end{align}
    Now, we can choose $\lambda^*$ such that for any $\lambda \in (0,\lambda^*)$
    \begin{align}
        \lambda \lVert W \rVert^2 < \frac{\epsilon}{5}.
        \label{W_lambda<eps}
    \end{align}
    Next we define a sequence of vectors $(W^{*}_j)_{j \in \mathbb{N}}$ by
     \begin{align*}
         W^*_{j} = \argmin_{U \in \mathbb{R}^{2(d(T_0+1) + N)}} \bigg( \frac{1}{j} \sum_{k=0}^{j - 1} \lVert H_{U}^{\boldsymbol{A}, \boldsymbol{C}, \boldsymbol{b}} T^{-k}(z) - \gamma H^{\boldsymbol{A},\boldsymbol{C},\boldsymbol{b}}_U T^{1-k}(z) - \mathcal{R}T^{-k}(z) \rVert^2 + \lambda \lVert U \rVert^2 \bigg).
     \end{align*}
     We may view $\argmin$ as continuous map on the space of strictly convex $C^1$ functions that returns their unique minimiser. The regularised linear least squares problem is a strictly convex $C^1$ problem, so we may define $W^*_{\infty} \in \mathbb{R}^{2d(T_0+1)+n}$ by
     \begin{align*}
        W^*_{\infty} &:= \argmin_{U} \bigg( \mathbb{E}_{\mu}[ \lVert H_{U}^{A, C, \zeta} (\boldsymbol{Z}) - \gamma H^{\boldsymbol{A},\boldsymbol{C},\boldsymbol{b}}_U T(\boldsymbol{Z}) - \mathcal{R}(\boldsymbol{Z}) \rVert^2 | \boldsymbol{A},\boldsymbol{C},\boldsymbol{b} ] + \lambda \lVert U \rVert^2 \bigg) \\
        &= \argmin_{U} \lim_{j \to \infty} \bigg( \frac{1}{j} \sum_{k=0}^{j - 1} \lVert H_{U}^{\boldsymbol{A}, \boldsymbol{C}, \boldsymbol{b}} T^{-k}(z) - \gamma H^{\boldsymbol{A},\boldsymbol{C},\boldsymbol{b}}_U T^{1-k} - \mathcal{R}T^{-k}(z) \rVert^2 + \lambda \lVert U \rVert^2 \bigg) \\
        &= \lim_{j \to \infty} \argmin_{U} \bigg( \frac{1}{j} \sum_{k=0}^{j - 1} \lVert H_{U}^{\boldsymbol{A}, \boldsymbol{C}, \boldsymbol{b}} T^{-k}(z) - \gamma H^{\boldsymbol{A},\boldsymbol{C},\boldsymbol{b}}_U T^{1-k} - \mathcal{R}T^{-k}(z) \rVert^2 + \lambda \lVert U \rVert^2 \bigg) \\ 
        &= \lim_{j \to \infty} W^*_j 
     \end{align*}
     where the second and third equalities hold by the Ergodic Theorem and continuity of $\argmin$ respectively.
    Now, we may choose $\ell \in \mathbb{N}$ sufficiently large that
    \begin{multline}
        \left| \mathbb{E}_{\mu}[\lVert W^{*\top}_{\ell} ( H^{\boldsymbol{A},\boldsymbol{C},\boldsymbol{b}}(\boldsymbol{Z}) - \gamma H^{\boldsymbol{A},\boldsymbol{C},\boldsymbol{b}}T(\boldsymbol{Z}) ) - \mathcal{R}(\boldsymbol{Z}) \rVert^2 | \boldsymbol{A},\boldsymbol{C},\boldsymbol{b}] \right. \\
        \left. - \ \mathbb{E}_{\mu}[\lVert W^{*\top}_{\infty} ( H^{\boldsymbol{A},\boldsymbol{C},\boldsymbol{b}}(\boldsymbol{Z}) - \gamma H^{\boldsymbol{A},\boldsymbol{C},\boldsymbol{b}}T(\boldsymbol{Z}) ) - \mathcal{R}(\boldsymbol{Z}) \rVert^2 | \boldsymbol{A},\boldsymbol{C},\boldsymbol{b}] \right| < \frac{\epsilon}{5},
        \label{EW_l-EW_inf<eps}
    \end{multline}
    and
    \begin{multline}
        \left| \lim_{j \to \infty} \bigg( \frac{1}{j} \sum_{k = 0}^{j-1} \lVert W^{*\top}_{j}(H^{\boldsymbol{A},\boldsymbol{C},\boldsymbol{b}}T^{-k}(z) - \gamma H^{\boldsymbol{A},\boldsymbol{C},\boldsymbol{b}}T^{1-k}(z) ) - \mathcal{R}T^{-k}(z) \rVert^2 + \lambda \lVert W^*_{j} \rVert^2 \bigg) \right. \\
        \left. - \ \frac{1}{\ell} \sum_{k = 0}^{\ell-1} \lVert W^{*\top}_{\ell}(H^{\boldsymbol{A},\boldsymbol{C},\boldsymbol{b}}T^{-k}(z) - \gamma H^{\boldsymbol{A},\boldsymbol{C},\boldsymbol{b}}T^{1-k}(z) ) - \mathcal{R}T^{-k}(z) \rVert^2 + \lambda \lVert W^*_{\ell} \rVert^2 \right| < \frac{\epsilon}{5},
        \label{W-W}
    \end{multline}
    and by the Ergodic Theorem
    \begin{multline}
        \left\lvert \frac{1}{\ell} \sum_{k = 0}^{\ell-1} \lVert W^\top (H^{\boldsymbol{A},\boldsymbol{C},\boldsymbol{b}}T^{-k}(z) - \gamma H^{\boldsymbol{A},\boldsymbol{C},\boldsymbol{b}}T^{1-k}(z)) - \mathcal{R}(z) \rVert^2 \right. \\
        \left. - \lim_{j \to \infty} \frac{1}{j} \sum_{k = 0}^{j-1} \lVert W^\top (H^{\boldsymbol{A},\boldsymbol{C},\boldsymbol{b}}T^{-k}(z) - \gamma H^{\boldsymbol{A},\boldsymbol{C},\boldsymbol{b}}T^{1-k}(z)) - \mathcal{R}(z) \rVert^2 \right\rvert
        < \frac{\epsilon}{5}.
        \label{l-liml<eps}
    \end{multline}
    Now the proof proceeds directly
\begin{align*}
     \mathbb{E}_{\mu}[ \lVert H_{W^*_{\ell}}^{\boldsymbol{A},\boldsymbol{C},\boldsymbol{b}}(\boldsymbol{Z}) - \Phi H_{W^*_{\ell}}^{\boldsymbol{A},\boldsymbol{C},\boldsymbol{b}}(\boldsymbol{Z}) \rVert^2 | \boldsymbol{A},\boldsymbol{C},\boldsymbol{b} ] \nonumber \\
        &  \dawesspace = \mathbb{E}_{\mu}[ \lVert H_{W^*_{\ell}}^{\boldsymbol{A},\boldsymbol{C},\boldsymbol{b}}(\boldsymbol{Z}) - \gamma H_{W^*_{\ell}}^{\boldsymbol{A},\boldsymbol{C},\boldsymbol{b}}T(\boldsymbol{Z}) - \mathcal{R}(\boldsymbol{Z}) \rVert^2 | \boldsymbol{A},\boldsymbol{C},\boldsymbol{b} ] \\
        & \dawesspace = \mathbb{E}_{\mu}[\lVert W^{*\top}_{\ell} ( H^{\boldsymbol{A},\boldsymbol{C},\boldsymbol{b}}(\boldsymbol{Z}) - \gamma H^{\boldsymbol{A},\boldsymbol{C},\boldsymbol{b}}T(\boldsymbol{Z}) ) - \mathcal{R}(\boldsymbol{Z}) \rVert^2 | \boldsymbol{A},\boldsymbol{C},\boldsymbol{b}].
        \end{align*}
Then we apply~\eqref{EW_l-EW_inf<eps} which yields
\begin{align*}
    \mathbb{E}_{\mu}[ \lVert H_{W^*_{\ell}}^{\boldsymbol{A},\boldsymbol{C},\boldsymbol{b}}(\boldsymbol{Z}) - \Phi H_{W^*_{\ell}}^{\boldsymbol{A},\boldsymbol{C},\boldsymbol{b}}(\boldsymbol{Z}) \rVert^2 | \boldsymbol{A},\boldsymbol{C},\boldsymbol{b} ] \\
    & \dawesspace < \mathbb{E}_{\mu}[\lVert W^{*\top}_{\infty} ( H^{\boldsymbol{A},\boldsymbol{C},\boldsymbol{b}}(\boldsymbol{Z}) - \gamma H^{\boldsymbol{A},\boldsymbol{C},\boldsymbol{b}}T(\boldsymbol{Z}) ) - \mathcal{R}(\boldsymbol{Z}) \rVert^2 | \boldsymbol{A},\boldsymbol{C},\boldsymbol{b}] + \frac{\epsilon}{5}.
\end{align*}
Then we apply the Ergodic Theorem
\begin{align*}
    \mathbb{E}_{\mu}[\lVert W^{*\top}_{\infty} ( H^{\boldsymbol{A},\boldsymbol{C},\boldsymbol{b}}(\boldsymbol{Z}) - \gamma H^{\boldsymbol{A},\boldsymbol{C},\boldsymbol{b}}T(\boldsymbol{Z}) ) - \mathcal{R}(\boldsymbol{Z}) \rVert^2 | \boldsymbol{A},\boldsymbol{C},\boldsymbol{b}] + \frac{\epsilon}{5} \nonumber \\
    & \dawesspace\dawesspace = \lim_{j \to \infty} \bigg( \frac{1}{j} \sum_{k = 0}^{j-1} \lVert W^{*\top}_{\infty}(H^{\boldsymbol{A},\boldsymbol{C},\boldsymbol{b}}T^{-k}(z) - \gamma H^{\boldsymbol{A},\boldsymbol{C},\boldsymbol{b}}T^{1-k}(z) ) - \mathcal{R}T^{-k}(z) \rVert^2 \bigg) + \frac{\epsilon}{5} \\
    & \dawesspace\dawesspace \leq \lim_{j \to \infty} \bigg( \frac{1}{j} \sum_{k = 0}^{j-1} \lVert W^{*\top}_{\infty}(H^{\boldsymbol{A},\boldsymbol{C},\boldsymbol{b}}T^{-k}(z) - \gamma H^{\boldsymbol{A},\boldsymbol{C},\boldsymbol{b}}T^{1-k}(z) ) - \mathcal{R}T^{-k}(z) \rVert^2 \bigg) + \lambda \lVert W^*_{\infty} \rVert^2 + \frac{\epsilon}{5} \\
    & \dawesspace\dawesspace = \lim_{j \to \infty} \bigg( \frac{1}{j} \sum_{k = 0}^{j-1} \lVert W^{*\top}_{j}(H^{\boldsymbol{A},\boldsymbol{C},\boldsymbol{b}}T^{-k}(z) - \gamma H^{\boldsymbol{A},\boldsymbol{C},\boldsymbol{b}}T^{1-k}(z) ) - \mathcal{R}T^{-k}(z) \rVert^2 + \lambda \lVert W^*_{j} \rVert^2 \bigg) + \frac{\epsilon}{5} \\
    &\text{ then apply \eqref{W-W}} \\
    & \dawesspace\dawesspace < \frac{1}{\ell} \sum_{k = 0}^{\ell-1} \lVert W^{*\top}_{\ell}(H^{\boldsymbol{A},\boldsymbol{C},\boldsymbol{b}}T^{-k}(z) - \gamma H^{\boldsymbol{A},\boldsymbol{C},\boldsymbol{b}}T^{1-k}(z) ) - \mathcal{R}T^{-k}(z) \rVert^2 + \lambda \lVert W^*_{\ell} \rVert^2 + \frac{2\epsilon}{5} \\
    & \dawesspace\dawesspace \leq \frac{1}{\ell} \sum_{k = 0}^{\ell-1} \lVert W^{\top}(H^{\boldsymbol{A},\boldsymbol{C},\boldsymbol{b}}T^{-k}(z) - \gamma H^{\boldsymbol{A},\boldsymbol{C},\boldsymbol{b}}T^{1-k}(z) ) - \mathcal{R}T^{-k}(z) \rVert^2 + \lambda \lVert W \rVert^2 + \frac{2\epsilon}{5} \\
    &\text{ then apply \eqref{l-liml<eps}} \\
    & \dawesspace\dawesspace < \lim_{j \to \infty} \bigg( \frac{1}{j} \sum_{k = 0}^{j-1} \lVert W^{\top}(H^{\boldsymbol{A},\boldsymbol{C},\boldsymbol{b}}T^{-k}(z) - \gamma H^{\boldsymbol{A},\boldsymbol{C},\boldsymbol{b}}T^{1-k}(z) ) - \mathcal{R}T^{-k}(z) \rVert^2 \bigg) + \lambda \lVert W \rVert^2 + \frac{3\epsilon}{5} \\
    &\text{ then apply \eqref{W_lambda<eps}} \\
    & \dawesspace\dawesspace < \lim_{j \to \infty} \bigg( \frac{1}{j} \sum_{k = 0}^{j-1} \lVert W^{\top}(H^{\boldsymbol{A},\boldsymbol{C},\boldsymbol{b}}T^{-k}(z) - \gamma H^{\boldsymbol{A},\boldsymbol{C},\boldsymbol{b}}T^{1-k}(z) ) - \mathcal{R}T^{-k}(z) \rVert^2 \bigg) + \frac{4\epsilon}{5} \\
    & \dawesspace\dawesspace \text{ Then apply the Ergodic Theorem again} \\
    & \dawesspace\dawesspace = \mathbb{E}_{\mu}[ \lVert W^{\top}(H^{\boldsymbol{A},\boldsymbol{C},\boldsymbol{b}}(\boldsymbol{Z}) - \gamma H^{\boldsymbol{A},\boldsymbol{C},\boldsymbol{b}}T(\boldsymbol{Z}) - \mathcal{R}(\boldsymbol{Z})) \rVert^2 | \boldsymbol{A},\boldsymbol{C},\boldsymbol{b} ] + \frac{4\epsilon}{5} \\
    & \dawesspace\dawesspace = \mathbb{E}_{\mu}[ \lVert H^{\boldsymbol{A},\boldsymbol{C},\boldsymbol{b}}_{W} - \Phi H^{\boldsymbol{A},\boldsymbol{C},\boldsymbol{b}}_{W} \rVert^2 | \boldsymbol{A},\boldsymbol{C},\boldsymbol{b} ]  + \frac{4\epsilon}{5} \\ 
    & \dawesspace\dawesspace \text{ then apply \eqref{H-PhiH}} \\
    & \dawesspace\dawesspace < \epsilon.
    \end{align*}
\end{proof}

\section{Online learning on stationary ergodic processes}

Theorem \ref{offline_q_learning} outlines an offline learning algorithm because the least squares training occurs after all the data has been collection.
In some reinforcement learning applications, it is useful - or even essential - for the optimisation of $W$ to occur dynamically as new data comes in; such algorithms are called \emph{online} learning algorithms. We analysed a simple online learning algorithm driven by deterministic inpits in section \ref{section::online_learning}. In this present section, we will consider the same algorithm; this time driven instead by stochastic inputs. 

We will first introduce a lemma, stating that, under reasonable conditions, the ODE
    \begin{align}
        \frac{d}{dt}W = -h(W) := -\mathbb{E}_{\mu}\bigg[H^{\boldsymbol{A},\boldsymbol{C},\boldsymbol{b}}(\boldsymbol{Z})\big( H^{\boldsymbol{A},\boldsymbol{C},\boldsymbol{b}}_{W}(\boldsymbol{Z}) - \Phi H^{\boldsymbol{A},\boldsymbol{C},\boldsymbol{b}}_{W}(\boldsymbol{Z}) \big)   \bigg]
        \label{associated_ODE}
    \end{align}
converges exponentially quickly to a globally asymptotic fixed point $W^*$, for which the associated ESN functional $H_{W^*}^{\boldsymbol{A},\boldsymbol{C},\boldsymbol{b}}$ is \emph{as close as possible} to the unique fixed point of $\Phi$. By \emph{as close as possible} we mean that the orthogonal projection of $\Phi H_{W^*}^{\boldsymbol{A},\boldsymbol{C},\boldsymbol{b}}$ onto the finite dimensional vector space of functionals $\{ H_{W}^{\boldsymbol{A},\boldsymbol{C},\boldsymbol{b}} \ | \ W \in \mathbb{R}^d \}$ is $H_{W^*}^{\boldsymbol{A},\boldsymbol{C},\boldsymbol{b}}$. Unlike the previous result (Theorem \ref{offline_q_learning}) we do not need to assume that the contraction mapping satisfies $\Phi(H) = \mathcal{R} + \gamma \mathbb{E}[HT_{\boldsymbol{Z}}]$. We could (possibly) choose a Bellman operator that admits the optimal value function as the fixed point.

\begin{lemma}
    \citep{arXiv:2102.06258} Let $\boldsymbol{Z}$ be an admissible input process. Let $\boldsymbol{A},\boldsymbol{C},\boldsymbol{b}$ be $N \times N$, $N \times d$, and $N \times 1$ dimensional random matrices for the reservoir matrix, input matrix and bias vector, respectively. Let $H^{\boldsymbol{A},\boldsymbol{C},\boldsymbol{b}}$ and $H_{W}^{\boldsymbol{A},\boldsymbol{C},\boldsymbol{b}}$ denote the associated ESN functionals. Let $\Phi$ be a contraction mapping, with Lipschitz constant $0 \leq \gamma < 1$, on the space of functionals $H : (D_d)^{\mathbb{Z}} \to \mathbb{R}$ that are $\mu$-measurable and satisfy $\mathbb{E}[H(\boldsymbol{Z})^2] < \infty$. Suppose further that $0 \leq \gamma < \kappa^{-1} $ where $\kappa$ is the condition number of the autocorrelation matrix
    \begin{align*}
        \Sigma = \mathbb{E}_{\mu}\left[ \left. H^{\boldsymbol{A},\boldsymbol{C},\boldsymbol{b}}(\boldsymbol{Z})
        \left( H^{\boldsymbol{A},\boldsymbol{C},\boldsymbol{b}}(\boldsymbol{Z}) \right)^\top \ \right| \ \boldsymbol{A},\boldsymbol{C},\boldsymbol{b} \right].
    \end{align*}
    %\jd{[In the above I have just moved the $\top$ to be at the far right, and inserted large parentheses, so that the structure of $\Sigma$ is more obvious. The $\top$ symbol was rather hidden in the superscript of $H$ along with all the other superscripts, I felt.]}
    Then there exists a $\delta > 0$ such that the ODE 
    \begin{align*}
        \frac{d}{dt}W = -h(W) := -\mathbb{E}_{\mu}\bigg[H^{\boldsymbol{A},\boldsymbol{C},\boldsymbol{b}}(\boldsymbol{Z})\big( H^{\boldsymbol{A},\boldsymbol{C},\boldsymbol{b}}_{W}(\boldsymbol{Z}) - \Phi H^{\boldsymbol{A},\boldsymbol{C},\boldsymbol{b}}_{W}(\boldsymbol{Z}) \big) \ \bigg| \ \boldsymbol{A},\boldsymbol{C},\boldsymbol{b}   \bigg]
    \end{align*}
    satisfies 
    \begin{align}
        \frac{d}{dt} \lVert W - W^* \rVert \leq - \delta \lVert W - W^* \rVert
        \label{exp_conv_cond}
    \end{align}
    where $W^*$ is a globally asymptotic fixed point. $W^*$ enjoys the further property that
    \begin{align*}
        H_{W^*}^{\boldsymbol{A},\boldsymbol{C},\boldsymbol{b}} = \mathcal{P}\Phi H_{W^*}^{\boldsymbol{A},\boldsymbol{C},\boldsymbol{b}}
    \end{align*}
    where $\mathcal{P}$ denotes the $L^2(\mu)$ orthogonal projection operator on the $\mu$-measurable filters $H$ satisfying $\mathbb{E}[H(\boldsymbol{Z})^2] < \infty$ and is defined
    \begin{align*}
        \mathcal{P}H(z) := \left(H^{\boldsymbol{A},\boldsymbol{C},\boldsymbol{b}}(z)\right)^\top \Sigma^{-1} \mathbb{E}_{\mu}\left[ \left. H^{\boldsymbol{A},\boldsymbol{C},\boldsymbol{b}}(\boldsymbol{Z}) H(\boldsymbol{Z})\ \right| \ \boldsymbol{A},\boldsymbol{C},\boldsymbol{b} \right].
    \end{align*}
%\jd{[see previous comment.]}
\end{lemma}

\begin{proof}
    To show that $W^*$ is a globally asymptotic fixed point it suffices to show that there exists a $\delta > 0$ such that
     \begin{align*}
         (W - W^*) \cdot (h(W^*) - h(W)) \leq - \delta \lVert (W - W^*) \rVert ^2
     \end{align*}
     as this implies
     \begin{align*}
         \frac{d}{dt} \lVert W - W^* \rVert \leq - \delta \lVert W - W^* \rVert.
     \end{align*}
    To construct this $\delta$, we first note that
    \begin{align*}
        h(W) = \Sigma W - \mathbb{E}_{\mu} \bigg[ H^{\boldsymbol{A},\boldsymbol{C},\boldsymbol{b}}(\boldsymbol{Z}) \Phi H^{\boldsymbol{A},\boldsymbol{C},\boldsymbol{b}}_{W}(\boldsymbol{Z}) \big) \ \bigg| \ \boldsymbol{A},\boldsymbol{C},\boldsymbol{b}  \bigg]
    \end{align*}
    so, by a direct computation we have
    \begin{align*}
         (W - W^*)  \cdot  (h(W^*) - h(W))  \\
         & \dawesspace = (W - W^*) \cdot \bigg( \mathbb{E}_{\mu} \bigg[ H^{\boldsymbol{A},\boldsymbol{C},\boldsymbol{b}}(\boldsymbol{Z}) \Phi H^{\boldsymbol{A},\boldsymbol{C},\boldsymbol{b}}_{W}(\boldsymbol{Z}) \big) \ \bigg| \ \boldsymbol{A},\boldsymbol{C},\boldsymbol{b} \bigg] - \mathbb{E}_{\mu} \bigg[ H^{\boldsymbol{A},\boldsymbol{C},\boldsymbol{b}}(\boldsymbol{Z}) \Phi H^{\boldsymbol{A},\boldsymbol{C},\boldsymbol{b}}_{W^*}(\boldsymbol{Z}) \big) \ \bigg| \ \boldsymbol{A},\boldsymbol{C},\boldsymbol{b} \bigg] \bigg) \\
        & \dawesspace - (W - W^*) \cdot \bigg( \Sigma W - \Sigma W^* \bigg) \\
        & \dawesspace =  (W - W^*) \cdot \bigg( \mathbb{E}_{\mu} \bigg[ H^{\boldsymbol{A},\boldsymbol{C},\boldsymbol{b}}(\boldsymbol{Z}) \Phi H^{\boldsymbol{A},\boldsymbol{C},\boldsymbol{b}}_{W}(\boldsymbol{Z}) \big) \ \bigg| \ \boldsymbol{A},\boldsymbol{C},\boldsymbol{b} \bigg] - \mathbb{E}_{\mu} \bigg[ H^{\boldsymbol{A},\boldsymbol{C},\boldsymbol{b}}(\boldsymbol{Z}) \Phi H^{\boldsymbol{A},\boldsymbol{C},\boldsymbol{b}}_{W^*}(\boldsymbol{Z}) \big) \ \bigg| \ \boldsymbol{A},\boldsymbol{C},\boldsymbol{b} \bigg] \bigg) \\
        & \dawesspace - (W - W^*)^{\top} \Sigma \big( W - W^* \big) \\
        & \dawesspace \leq
        (W - W^*) \cdot \bigg( \mathbb{E}_{\mu} \bigg[ H^{\boldsymbol{A},\boldsymbol{C},\boldsymbol{b}}(\boldsymbol{Z}) \Phi H^{\boldsymbol{A},\boldsymbol{C},\boldsymbol{b}}_{W}(\boldsymbol{Z}) \big) \ \bigg| \ \boldsymbol{A},\boldsymbol{C},\boldsymbol{b} \bigg] - \mathbb{E}_{\mu} \bigg[ H^{\boldsymbol{A},\boldsymbol{C},\boldsymbol{b}}(\boldsymbol{Z}) \Phi H^{\boldsymbol{A},\boldsymbol{C},\boldsymbol{b}}_{W^*}(\boldsymbol{Z}) \big) \ \bigg| \ \boldsymbol{A},\boldsymbol{C},\boldsymbol{b} \bigg] \bigg) \\
        & \dawesspace - \sigma \lVert W - W^* \rVert^2 \ \text{ where $\sigma >0$ is the smallest eigenvalue of $\Sigma$} \\
        &\dawesspace = 
        (W - W^*) \cdot \bigg( \mathbb{E}_{\mu} \bigg[ H^{\boldsymbol{A},\boldsymbol{C},\boldsymbol{b}}(\boldsymbol{Z}) \Phi H^{\boldsymbol{A},\boldsymbol{C},\boldsymbol{b}}_{W}(\boldsymbol{Z}) \big) - H^{\boldsymbol{A},\boldsymbol{C},\boldsymbol{b}}(\boldsymbol{Z}) \Phi H^{\boldsymbol{A},\boldsymbol{C},\boldsymbol{b}}_{W^*}(\boldsymbol{Z}) \big) \ \bigg| \ \boldsymbol{A},\boldsymbol{C},\boldsymbol{b} \bigg] \bigg) \\
        & \dawesspace - \sigma \lVert W - W^* \rVert^2 \\
        & \dawesspace \leq 
        (W - W^*) \cdot \bigg( \mathbb{E}_{\mu} \bigg[ H^{\boldsymbol{A},\boldsymbol{C},\boldsymbol{b}}(\boldsymbol{Z}) H^{\boldsymbol{A},\boldsymbol{C},\boldsymbol{b}}_{W}(\boldsymbol{Z}) \big) - H^{\boldsymbol{A},\boldsymbol{C},\boldsymbol{b}}(\boldsymbol{Z}) H^{\boldsymbol{A},\boldsymbol{C},\boldsymbol{b}}_{W^*}(\boldsymbol{Z}) \big) \ \bigg| \ \boldsymbol{A},\boldsymbol{C},\boldsymbol{b} \bigg] \bigg) \gamma \\
        & \dawesspace - \sigma \lVert W - W^* \rVert^2 \ \text{ because $\gamma$ is the Lipschitz constant for $\Phi$} \\
        & \dawesspace = \gamma (W - W^*)^{\top} \Sigma (W - W^*) - \sigma \lVert W - W^* \rVert^2 \\
        & \dawesspace \leq \gamma \rho \lVert W - W^* \rVert^2 - \sigma \lVert W - W^* \rVert^2 \ \text{ where $\rho >0$ is the largest eigenvalue of $\Sigma$} \\
        & \dawesspace = -(\sigma - \gamma \rho) \lVert W - W^* \rVert^2,
    \end{align*}
    so we can set $\delta := \sigma - \gamma \rho$ and notice $\delta > 0$ because $0 \leq \gamma < \kappa^{-1} = \sigma / \rho$. Next, to show that
    \begin{align*}
        H_{W^*}^{\boldsymbol{A},\boldsymbol{C},\boldsymbol{b}} = \mathcal{P}\Phi H_{W^*}^{\boldsymbol{A},\boldsymbol{C},\boldsymbol{b}}
    \end{align*}
    we observe that since $W^*$ is an equilibrium point of the ODE 
    \begin{align*}
        \dot{W} = -h(W)
    \end{align*} 
    it follows that $h(W^*)=0$ and therefore
    \begin{align*}
    0 &= \mathbb{E}_{\mu}\bigg[H^{\boldsymbol{A},\boldsymbol{C},\boldsymbol{b}}(\boldsymbol{Z})\big( H^{\boldsymbol{A},\boldsymbol{C},\boldsymbol{b}}_{W^*}(\boldsymbol{Z}) - \Phi H^{\boldsymbol{A},\boldsymbol{C},\boldsymbol{b}}_{W^*}(\boldsymbol{Z}) \big) \ \bigg| \ \boldsymbol{A},\boldsymbol{C},\boldsymbol{b} \bigg] \\
    \implies 0 &= \mathbb{E}_{\mu}\bigg[H^{\boldsymbol{A},\boldsymbol{C},\boldsymbol{b}}(\boldsymbol{Z}) \left( H^{\boldsymbol{A},\boldsymbol{C},\boldsymbol{b}}(\boldsymbol{Z}) 
    \right)^\top \ \bigg| \ \boldsymbol{A},\boldsymbol{C},\boldsymbol{b} \bigg] W^* \\
    &\qquad- \mathbb{E}_{\mu}\bigg[ H^{\boldsymbol{A},\boldsymbol{C},\boldsymbol{b}}(\boldsymbol{Z}) \Phi H^{\boldsymbol{A},\boldsymbol{C},\boldsymbol{b}}_{W^*}(\boldsymbol{Z}) \ \bigg| \ \boldsymbol{A},\boldsymbol{C},\boldsymbol{b} \bigg] \\
    \implies 0 &= \Sigma W^* - \mathbb{E}_{\mu}\bigg[ H^{\boldsymbol{A},\boldsymbol{C},\boldsymbol{b}}(\boldsymbol{Z}) \Phi H^{\boldsymbol{A},\boldsymbol{C},\boldsymbol{b}}_{W^*}(\boldsymbol{Z}) \ \bigg| \ \boldsymbol{A},\boldsymbol{C},\boldsymbol{b} \bigg] \\
    \text{so, } \ \Sigma W^* 
    &= \mathbb{E}_{\mu}\bigg[ H^{\boldsymbol{A},\boldsymbol{C},\boldsymbol{b}}(\boldsymbol{Z}) \Phi H^{\boldsymbol{A},\boldsymbol{C},\boldsymbol{b}}_{W^*}(\boldsymbol{Z}) \ \bigg| \ \boldsymbol{A},\boldsymbol{C},\boldsymbol{b} \bigg] \\
          \text{so, } \ W^* 
          &= \Sigma^{-1} \mathbb{E}_{\mu}\bigg[ H^{\boldsymbol{A},\boldsymbol{C},\boldsymbol{b}}(\boldsymbol{Z}) \Phi H^{\boldsymbol{A},\boldsymbol{C},\boldsymbol{b}}_{W^*}(\boldsymbol{Z}) \ \bigg| \ \boldsymbol{A},\boldsymbol{C},\boldsymbol{b} \bigg] \\
          \text{so, } \ H^{\boldsymbol{A},\boldsymbol{C},\boldsymbol{b}}_{W^*} 
          &= H^{\boldsymbol{A},\boldsymbol{C},\boldsymbol{b} \top} \Sigma^{-1} \mathbb{E}_{\mu}\bigg[ H^{\boldsymbol{A},\boldsymbol{C},\boldsymbol{b}}(\boldsymbol{Z}) \Phi H^{\boldsymbol{A},\boldsymbol{C},\boldsymbol{b}}_{W^*}(\boldsymbol{Z}) \ \bigg| \ \boldsymbol{A},\boldsymbol{C},\boldsymbol{b} \bigg] \\
          &= \mathcal{P}\Phi(H^{\boldsymbol{A},\boldsymbol{C},\boldsymbol{b}}_{W^*}). \\
     \end{align*}
\end{proof}

One rather restrictive condition of this lemma is that the Lipschitz constant $\gamma$ of the contraction $\Phi$ must be less than the reciprocal condition number $\kappa^{-1}$. Now, we can interpret $\kappa$ as a measure of \emph{how orthonormal} the columns of the autocorrelation matrix $\Sigma$ are. In particular, if the columns are indeed orthonormal, then $\kappa = 1$ and this condition ceases to be restrictive at all. If the columns are close to being linearly dependant, then $\kappa$ is large so the requirement that $\kappa^{-1}$ is small could become troublesome. If indeed there is a linear dependence, the matrix $\Sigma$ is not even invertible and the theorem breaks down completely. If we interpret $H^{\boldsymbol{A},\boldsymbol{C},\boldsymbol{b}}(\boldsymbol{Z})$ as a vector of features, then $\kappa$ grows with the correlation between features. Higher correlation between the features imposes a greater constraint on the Lipschitz constant $\gamma$. If we have no inter-feature correlation (e.g independent features) then $\kappa = 1$ and we have no restriction at all on $\gamma$.

In the special case that $\Phi$ is the Bellman operator
\begin{align}
    \Phi(H) = \gamma \mathbb{E}_\mu[HT_{\boldsymbol{Z}}] + \mathcal{R} \tag{\ref{Phi}}
\end{align}
the Lipschitz constant $\gamma$ is the discount factor. So in this case, we conclude that high interfeature correlation, which results in small $\kappa^{-1}$, forces us to choose small $\gamma$, which strongly discounts the contributions of future rewards.

Now, to actually solve ODE \eqref{associated_ODE} we would need to compute
\begin{align}
    h(W) := \mathbb{E}_{\mu}\bigg[H^{\boldsymbol{A},\boldsymbol{C},\boldsymbol{b}}(\boldsymbol{Z})\big( H^{\boldsymbol{A},\boldsymbol{C},\boldsymbol{b}}_{W}(\boldsymbol{Z}) - \Phi H^{\boldsymbol{A},\boldsymbol{C},\boldsymbol{b}}_{W}(\boldsymbol{Z}) \big) \ \bigg| \ \boldsymbol{A},\boldsymbol{C},\boldsymbol{b}  \bigg] \label{h_definition}
\end{align}
which may, or may not, be practical. For example, if the process $\boldsymbol{Z}$ is ergodic, we can approximate \eqref{h_definition} by taking a sufficiently long time average of 
\begin{align*}
    H^{\boldsymbol{A},\boldsymbol{C},\boldsymbol{b}} T^{k}(z)\big( H^{\boldsymbol{A},\boldsymbol{C},\boldsymbol{b}}_{W} T^{k}(z) - \Phi H^{\boldsymbol{A},\boldsymbol{C},\boldsymbol{b}}_{W} T^{k}(z) \big).
\end{align*}
Alternatively, we may approach the problem of solving \eqref{associated_ODE} by first considering the explicit Euler method (with time-steps $\alpha_k > 0$)
\begin{align*}
    W_{k+1} &= W_k - \alpha_k h(W_k) \\
            &= W_k - \alpha_k \mathbb{E}_{\mu}\bigg[H^{\boldsymbol{A},\boldsymbol{C},\boldsymbol{b}}(\boldsymbol{Z})\big( H^{\boldsymbol{A},\boldsymbol{C},\boldsymbol{b}}_{W_k}(\boldsymbol{Z}) - \Phi H^{\boldsymbol{A},\boldsymbol{C},\boldsymbol{b}}_{W_k}(\boldsymbol{Z}) \big) \ \bigg| \ \boldsymbol{A},\boldsymbol{C},\boldsymbol{b}  \bigg],
\end{align*}
    then we might (heuristically) expect the algorithm (highly reminiscent of the iteration studied in section \ref{section::online_learning})
\begin{align}
    W_{k+1} = W_k - \alpha_k H^{\boldsymbol{A},\boldsymbol{C},\boldsymbol{b}} T^{k}(z)\big( H^{\boldsymbol{A},\boldsymbol{C},\boldsymbol{b}}_{W_k} T^{k}(z) - \Phi H^{\boldsymbol{A},\boldsymbol{C},\boldsymbol{b}}_{W_k} T^{k}(z) \big) 
    \label{stochastic_algorithm}
\end{align}
to converge to $W^*$, where $\alpha_k$ are positive definite real numbers that satisfy
\begin{align*}
    \sum_{k=1}^{\infty} \alpha_k = \infty \qquad \sum_{k=1}^{\infty} \alpha_k^2 < \infty.
\end{align*} 

We believe this heuristic could be made rigorous under mild assumptions, because algorithm \eqref{stochastic_algorithm} closely resembles the major algorithm extensively studied in \citep{Adaptive_Algorithms_and_Stochastic_Approximations} and \citep{Borkar2009}
for which similar results hold. Theorems 17 and 2.1.1. appearing in \citep{Adaptive_Algorithms_and_Stochastic_Approximations} and \citep{Borkar2009} respectively suggest that an algorithm much like \eqref{stochastic_algorithm} converges almost surely to $W^*$ if its associated ODE (reminiscent of  \eqref{associated_ODE}) satisfies condition \eqref{exp_conv_cond}, and the input process $\boldsymbol{Z}$ is strongly mixing. The conjecture that algorithm \eqref{stochastic_algorithm} converges to $W^*$ is also reminiscent of Theorem 3.1 by \cite{10.1007/978-3-540-72927-3_23}, and related results by \cite{Chen2019}. These results are closely related to Q-learning and stochastic gradient descent. We note that (sadly) finding the fixed point of the general contraction mapping $\Phi$ renders the estimation of $W$ a nonlinear problem.

We will close this chapter with a brief summary of its contents. We showed in the supervised learning context, that given a time series of observations and targets (rewards) realised from a stationary ergodic process, an ESN trained by linear least squares will learn the functional that relates the observations to the targets. Furthermore, in the reinforcement learning context, we showed that if an agent explores its environment using a given policy, such that the reward-action-observation triples are a stationary ergodic process, then an ESN trained by least squares will learn the value functional associated to that that policy. Learning the value functional of a given policy is a useful first step in establishing a complete reinforcement learning algorithm, which iterative improves policy, (ideally) until the optimal policy is found.

\chapter{Partial Differential Equations}
\label{chapter::PDEs}

\section{Introduction}

In previous chapters we took the view that the governing equations of our system were unknown, and that an agent must learn them using a sample trajectory. In this chapter we take a different view. We assume that the governing equations, which are partial differential equations (PDEs), are known, but that the solution is numerically expensive to compute. We will use reservoir computing to approximate the solutions. Though the setting is different, many of the ideas discussed in previous chapters apply here.

For motivation, we seek a numerical solution to Laplace's equation
\begin{align*}
    \Delta \phi = 0
\end{align*}
on a domain $\Omega \subset \mathbb{R}^d$ subject to the Dirichlet boundary condition $\phi = h$ on $\partial \Omega \subset \mathbb{R}^{d-1}$. This is known as the (inhomogeneous) Dirichlet problem for $\phi$.

We will approximate the solution $\phi(z)$ with a feedforward reservoir computer $W^{\top} \sigma(Cz + b)$. The weights and biases $C \in \mathbb{M}_{N,d}(\mathbb{R})$ and $b \in \mathbb{R}^N$ are randomly generated, and only $W \in \mathbb{R}^N$ is trained. To lighten notation, we let $f(z):= \sigma(Cz + b)$. Feedforward reservoir computers are universal approximators, so given sufficiently many neurons (i.e. sufficiently large $N$) there exists a $W$ such that $W^{\top} f$ approximates $\phi$ to the required precision. In particular, we seek $W$ such that
\begin{itemize}
    \item $\Delta (W^{\top} f) \approx 0$ on $\Omega$
    \item $W^{\top}f \approx h$ on $\partial \Omega$.
\end{itemize}
 We make the crucial observation here that $\Delta(W^{\top} f) = W^{\top} (\Delta f)$ by linearity of the Laplacian.
We begin by evaluating $\Delta f$ and $f$ on a finite set of points in the interior $\Omega$ and boundary $\partial \Omega$ respectively. We then seek $W$ that minimises the (regularised) mean square differences $\lVert W^{\top}(\Delta f) \rVert$ and $\lVert W^{\top}f - h \rVert$ over the points in the interior $\Omega$ and boundary $\partial \Omega$ respectively.

In many practical applications, we want more samples in some regions of the domain $\Omega \cup \partial \Omega = \bar{\Omega}$ than others. This is because the solution may be less regular, or more interesting, in some regions in comparison to others.
So we define probability measures $\mu$ and $\mu'$ on the measurable subsets of $\Omega$ and $\partial \Omega$. We then generate sample points $Z_0, Z_1, \ldots Z_{\ell-1} \in \Omega$ and $Z'_0, Z'_1, \ldots Z'_{\ell'-1} \in \partial \Omega$ on the interior and boundary from stationary ergodic processes $Z$ and $Z'$ with invariant measures $\mu$ and $\mu'$. The most straightforward approach is to choose $Z$ and $Z'$ as collections of i.i.d. uniform random variables with respect to $\mu$ and $\mu'$. If this is computationally infeasible (which is often the case for high dimensional problems) we can use Markov Chain Monte Carlo (MCMC) to generate $Z$ and $Z'$.

The remainder of this chapter is laid out as follows: In Section \ref{Theory} we lay out some theory which formalises the ideas in the introduction. In Section \ref{Numerics} we demonstrate the theory by numerically solving Laplace's equation on the disc using a feedforward reservoir computer, and compare this to the analytic solution. We then conclude the chapter and consider directions for future work in Section \ref{Conclusion}. 

\section{Theory}
\label{Theory}

The Dirichlet problem is posed on an open, bounded, and connected domain $\Omega \subset \mathbb{R}^d$ that satisfies the \emph{exterior ball property}. This property ensures the boundary $\partial \Omega$ is regular enough to support a unique solution to the Dirchlet problem.

\begin{defn}
    (Exterior ball property) An open set $\Omega$ has the exterior ball property at point $p \in \partial \Omega$ if there exists a ball $B \subset \mathbb{R}^d \backslash \Omega$ such that $p \in \partial B$. If $\Omega$ has the exterior ball property at every $p \in \partial B$ then we say that $\Omega$ has the exterior ball property.
\end{defn}

\begin{theorem}
    (Existence and Uniqueness for the Dirichlet Problem) Let $\Omega \subset \mathbb{R}^d$ be open, bounded, connected, and satisfy the exterior ball property. Let $h \in C^0(\partial \Omega)$. Then the Dirichlet problem 
    \begin{align*}
        \Delta \phi &= 0 \text{ on } \Omega \\
        \phi &= h \text{ on } \partial \Omega
    \end{align*}
    has a unique solution $\phi \in C^2(\Omega) \cup C^0(\bar{\Omega})$. Further if $h \in C^2(\partial \Omega)$ then $\phi \in C^2(\bar{\Omega})$
    \label{thm::Dirichlet}
\end{theorem}

Now, suppose that the samples $Z_0, Z_1 , \ldots , Z_{\ell-1}$ and $Z'_0, Z'_1 , \ldots , Z'_{\ell'-1}$ on the interior $\Omega$ and boundary $\partial \Omega$ are drawn from stationary ergodic processes $Z$ and $Z'$.

\begin{lemma}
    Consider the Dirichlet problem defined in Theorem \ref{thm::Dirichlet}.
    Let $Z, Z'$ be stationary ergodic procesess on $\Omega$ and $\partial \Omega$ with invariant measures $\mu$ and $\mu'$ respectively. Let $f \in L^1(\partial\Omega,\mathbb{R}^N)$ and $g \in L^1( \Omega,\mathbb{R}^N)$. Let $\Lambda \in \mathbb{M}_{N \times N}(\mathbb{R})$ be an invertible matrix. Let $W^*_{\ell\ell'}$ be the unique minimiser (over $W \in \mathbb{R}^N$) of
        \begin{align*}
        \frac{1}{\ell}\sum_{k=0}^{\ell-1} \lVert W^{\top}g(Z_k) \rVert^2 + \frac{1}{\ell'}\sum_{k=0}^{\ell'-1}\lVert W^{\top}f(Z'_k) - h(Z'_k) \rVert^2 + \frac{2}{\ell+\ell'}\lVert \Lambda W \rVert^2.
    \end{align*}
    Then $(W^*_{\ell\ell})_{\ell \in \mathbb{N}}$ converges (almost surely) as $\ell \to \infty$ to 
    \begin{align*}
        W^* := \bigg( \int_{\Omega} gg^{\top} \ d\mu + \int_{\partial \Omega} ff^{\top} \ d\mu' + \Lambda\Lambda^{\top} \bigg)^{-1} \int_{\partial \Omega} hf \ d\mu'.
    \end{align*}

    which is the unique minimiser (over $W$) of
    \begin{align*}
        \lVert W^{\top} g \rVert_{L^2(\mu)}^2 + \lVert W^{\top} f - h \rVert_{L^2(\mu')}^2 + \lVert \Lambda W \rVert^2.
    \end{align*}
\end{lemma}

\begin{proof}
    Consider the map $\Psi : \mathbb{R}^N \to \mathbb{R}$ defined
    \begin{align*}
        \Psi(W) & : = \lVert W^{\top} g \rVert_{L^2(\mu)}^2 + \lVert W^{\top} f - h \rVert_{L^2(\mu')}^2 + \lVert \Lambda W \rVert^2 \\
        &= \int_\Omega \lVert W^{\top} g \rVert^2 \ d \mu + \int_{\partial \Omega} \lVert W^{\top} f - h \rVert^2 \ d \mu' + \lVert \Lambda W \rVert^2.
    \end{align*}
    The minimiser of $\Psi$ satisfies $D \Psi = 0$ where $D$ is the derivative operator, so we consider
    \begin{align*}
        0 &= (D\Psi)(W) \\
        &= D \bigg( \int_\Omega \lVert W^{\top} g \rVert^2 \ d \mu + \int_{\partial \Omega} \lVert W^{\top} f - h \rVert^2 \ d \mu' + \lVert \Lambda W \rVert^2 \bigg) \\
        &= \int_\Omega D\lVert W^{\top} g \rVert^2 \ d \mu + \int_{\partial \Omega} D\lVert W^{\top} f - h \rVert^2 \ d \mu' + D\lVert \Lambda W \rVert^2 \\
        &= \int_\Omega 2W^{\top} gg^{\top} \ d \mu + \int_{\partial \Omega} 2(W^{\top} f - h)f^{\top} \ d \mu' + 2 \Lambda^{\top}\Lambda W \\
        &= \int_\Omega W^{\top} gg^{\top} \ d \mu + \int_M (W^{\top} f - h)f^{\top} \ d \mu' + W^{\top}\Lambda\Lambda^{\top} \\
        &= W^{\top}\bigg(\int_\Omega gg^{\top} \ d\mu + \int_{\partial \Omega} ff^{\top} \ d\mu' + \Lambda\Lambda^{\top} \bigg) - \int_{\partial \Omega} hf^{\top} \ d\mu'\\
    \end{align*}
    which upon rearrangement yields 
    \begin{align*}
        W &= \bigg( \int_\Omega gg^{\top} \ d\mu + \int_{\partial \Omega} ff^{\top} \ d\mu' + \Lambda\Lambda^{\top} \bigg)^{-1} \times \int_{\partial \Omega} hf \ d\mu'.
    \end{align*}
    Since this is the unique solution to $0 = D\Psi(W)$, this stationary point is unique, and we will denote it by $W^*$. We can see it is a minimum because the Hessian $H\Psi$ is positive definite.
    Next, define the map
    \begin{align*}
        \Phi : \{ y \in C^1(\mathbb{R}^N,\mathbb{R}) \ | \ y \text{ is strictly convex} \} \to \mathbb{R}^N
    \end{align*}
    as the mapping on the strictly convex $C^1$ functions that returns their unique minimum. We can see that $\Phi$ is continuous with respect to the $C^1$ topology and Euclidean topology on $\mathbb{R}$ respectively. We consider the family of functions $y_{\ell\ell'} \in \{ y \in C^1(\mathbb{R}^N, \mathbb{R}) \ | \ y \text{ is strictly convex}\}$ defined by
    \begin{align*}
        y_{\ell\ell'}(W) =
        \frac{1}{\ell}\sum_{k=0}^{\ell-1} \lVert W^{\top}g(Z_k) \rVert^2 + \frac{1}{\ell'}\sum_{k=0}^{\ell'-1}\lVert W^{\top}f(Z'_k) - h(Z'_k) \rVert^2 + \frac{2}{\ell + \ell'}\lVert \Lambda W \rVert^2,
    \end{align*}
    so that by definition
    $W_{\ell\ell} = \Phi(y_{\ell\ell}(W))$ and hence
        \begin{align*}
        \lim_{\ell \to \infty} W_{\ell\ell} &= \lim_{\ell \to \infty} \Phi(y_{\ell\ell}(W))
        = \Phi\bigg(\lim_{\ell \to \infty} y_{\ell\ell}(W) \bigg)
        = W^*.
    \end{align*}
almost surely, where we have used, respectively, continuity of $\Phi$ and the Ergodic Theorem.
\end{proof}

\begin{defn}
    Let $\sigma : \mathbb{R}^N \to \mathbb{R}^N$ be $\ell$-finite, with $\ell = 2$. Let $C$ be an $\mathbb{M}_{N \times d}(\mathbb{R})$-valued random variable with full support, and $b$ a $\mathbb{R}^N$-valued random variable with full support. Then we define the random neural network
    $f : \mathbb{R}^d \to \mathbb{R}^N$ by
    \begin{align*}
        f(z) := \sigma(Cz + b).
    \end{align*}
    \label{defn::feedforward}
\end{defn}

\begin{lemma}
    Consider the Dirichlet problem defined in Theorem \ref{thm::Dirichlet} with boundary data $h \in C^2(\partial \Omega)$. Let $f$ be the random neural network defined in Theorem \ref{defn::feedforward}. Then for all $\alpha \in (0,1)$ and $\epsilon > 0$ there exists $N_0 \in \mathbb{N}$ such that, with probability at least $\alpha$, for all $N > N_0$ there exists a vector $\bar{W} \in \mathbb{R}^N$ such that
    \begin{align*}
        \lVert \Delta(\bar{W}^{\top}f) \rVert_{L^2(\mu)} + \lVert \bar{W}^{\top}f - h \rVert_{L^2(\mu')} < \epsilon.
    \end{align*}
\end{lemma}

\begin{proof}
    Fix $\epsilon > 0$. Now the Laplacian $\Delta$ is continuous at $\phi$ (the unique solution of the Dirichlet problem) so there is a $\delta > 0$ such that for all $s \in C^2(\bar{\Omega})$ 
    \begin{align*}
        \lVert s - \phi \rVert_{C^2} < \delta \implies \lVert \Delta s - \Delta \phi \rVert_{\infty} < \epsilon/2.
    \end{align*}
    
    %\jd{[I'm worried about this. Function values can be close, but their gradients may still not be close. Given that this is a linear problem, I considered trying to establish the statement `given $\epsilon>0$, find a $\delta>0$ such that $\| s(x)\|_\infty<\delta$ implies $\|\Delta s\|_\infty<\epsilon$'. If you take the sequence $s_k(x):=(1/k)\sin(kx)$ on the 1D interval $x \in [0,2\pi]$ then
    %the sup-norm of the $s_k$ functions goes to zero as $k\to \infty$ but the sup-norm of their
    %second derivatives does not go to zero. Does this mean you need to work in a different space, e.g. $H^2$, to ensure that the Laplacian is bounded (implies continuous) w.r.t a different norm?]}
    
    Fix $\alpha \in (0,1)$. Then by Theorem \ref{theorem::RUAT} there exists $N_0 \in \mathbb{N}$ such that, with probability at least $\alpha$, for all $N > N_0$ there is a vector $\bar{W} \in \mathbb{R}^N$ such that
    \begin{align*}
        \lVert \bar{W}^{\top}f - \phi \rVert_{\infty} < \text{min}(\epsilon/2,\delta).
    \end{align*}
    Then we can see on the boundary that
    \begin{align*}
        \lVert \bar{W}^{\top}f - h \rVert_{L^2(\mu')} \leq \lVert \bar{W}^{\top}f - h \rVert_{\infty} = \lVert \bar{W}^{\top}f - \phi \rVert_{\infty} < \epsilon / 2.
    \end{align*}
    Furthermore we have in the interior that
    \begin{align*}
        \lVert \bar{W}^{\top}(\Delta f) \rVert_{L^2(\mu)} \leq \lVert \bar{W}^{\top}(\Delta f)\rVert_{C^2} = \lVert \Delta(\bar{W}^{\top}f) - \Delta\phi \rVert_{C^2} < \epsilon / 2.
    \end{align*}
    Then it follows that
    \begin{align*}
        \lVert \Delta(\bar{W}^{\top}f) \rVert_{L^2(\mu)} + \lVert \bar{W}^{\top}f - h \rVert_{L^2(\mu')} < \epsilon / 2 + \epsilon / 2 = \epsilon.
    \end{align*}
\end{proof}

\begin{lemma}
    Consider the Dirichlet problem defined in Theorem \ref{thm::Dirichlet} with boundary data $h \in C^2(\partial \Omega)$. Let $Z,Z'$ be stationary ergodic processes on $\Omega$ and $\partial \Omega$ with invariant measures $\mu$ and $\mu'$ respectively. Then $\phi \in C^2(\bar{\Omega})$ is the unique solution to the Dirichlet problem if and only if
    \begin{align*}
        \lVert \Delta \phi \rVert_{L^2(\mu)} + \lVert \phi - h \rVert_{L^2(\mu')} = 0
    \end{align*}
\end{lemma}

\begin{proof}
    First, suppose that $\phi \in C^2(\bar{\Omega})$ solves the Dirichlet problem, then
    \begin{align*}
        \Delta \phi = 0 \ \text{ on $\Omega$} \qquad \text{and} \qquad \phi - h = 0 \ \text{ on $\partial \Omega$}.
    \end{align*}
    so clearly 
    \begin{align*}
        \lVert \Delta \phi \rVert_{L^2(\mu)} + \lVert \phi - h \rVert_{L^2(\mu')} = 0 + 0 = 0.
    \end{align*}
    Conversely, suppose that 
    \begin{align*}
        0 &= \lVert \Delta \phi \rVert_{L^2(\mu)} + \lVert \phi - h \rVert_{L^2(\mu')} \\
        &= \lVert \Delta \phi \rVert_{L^2(\mu)}^2 + \lVert \phi - h \rVert_{L^2(\mu')}^2 \\
        &= \int_{\Omega} (\Delta \phi)^2 \ d \mu + \int_{\partial \Omega} (\phi - h)^2 \ d \mu'.
    \end{align*}
    Now we observe that $(\Delta\phi)^2 \geq 0$ and $(\phi - h)^2 \geq 0$ are continuous, so
    \begin{align*}
        (\Delta \phi)^2 = 0 \ \text{ on $\Omega$} \qquad (\phi - h)^2 = 0 \ \text{ on $\partial \Omega$}
    \end{align*}
    hence
    \begin{align*}
        \Delta \phi = 0 \ \text{ on $\Omega$} \qquad \phi - h = 0 \ \text{ on $\partial \Omega$}
    \end{align*}
    so $\phi$ is the unique solution of the Dirichlet problem.
\end{proof}

\begin{theorem}
    Consider the Dirichlet problem defined in Theorem \ref{thm::Dirichlet} with boundary data $h \in C^2(\partial \Omega)$.
    Let $Z,Z'$ be a stationary ergodic processes with invariant measures $\mu$ and $\mu'$ defined on $\Omega \subset \mathbb{R}^d$ and $\partial \Omega$ respectively. Let $f \in L^1(\Omega,\mathbb{R}^N)(\mu)$ be the feedforward random neural network in Definition \ref{defn::feedforward}. Let $\Lambda \in \mathbb{M}_{N \times N}(\mathbb{R})$ be the invertible regularisation matrix and $\lambda > 0$.
    Suppose that $W^*_{\ell\ell'}$ minimises (over $W \in \mathbb{R}^N$)
    \begin{align*}
        \frac{1}{\ell}\sum_{k=0}^{\ell-1} \lVert W^{\top}\Delta f(Z_k) \rVert^2 + \frac{1}{\ell'}\sum_{k=0}^{\ell'-1}\lVert W^{\top}f(Z'_k) - h(Z'_k) \rVert^2 + \frac{2\lambda}{\ell+\ell'}\lVert \Lambda W \rVert^2.
    \end{align*}
    Then, for all $\alpha \in (0,1)$ and $\epsilon > 0$ there exist constants $\ell_0, \lambda_0, N_0$ such that, with probability at least $\alpha$, for all $\ell > \ell_0, \lambda > \lambda_0, N > N_0$
    \begin{align*}
        \lVert \Delta(W^{*\top}_{\ell\ell}f) \rVert_{L^2(\mu)} + \lVert W^{*\top}_{\ell\ell}f - h \rVert_{L^2(\mu')} < \epsilon.
    \end{align*}
    \label{theorem::Dirichlet_approx}
\end{theorem}

\begin{proof}
    Fix $\epsilon > 0$. Then, by Theorem \ref{theorem::RUAT} there exists a $N_0 \in \mathbb{N}$ such that, with probability at least $\alpha$, for all $N > N_0$ the vector $\bar{W} \in \mathbb{R}^N$ satisfies
    \begin{align*}
        \lVert \Delta(\bar{W}^{\top}f) \rVert_{L^2}(\mu) + \lVert \bar{W}^{\top}f - h \rVert_{L^2}(\mu') < \frac{\epsilon}{3}.
    \end{align*}
    Let $\lambda_0 > 0$ be sufficiently small that for any $\lambda \in (0,\lambda_0)$ 
    \begin{align*}
        \lambda \lVert \Lambda \bar{W} \rVert < \frac{\epsilon}{3}.
    \end{align*}
    Now let
    \begin{align*}
        W^* = \bigg( \int_\Omega (\Delta f)(\Delta f)^{\top} \ d\mu + \int_{\partial \Omega} ff^{\top} \ d\mu' + \Lambda\Lambda^{\top} \bigg)^{-1} \int_{\partial \Omega} hf \ d\mu'.
    \end{align*}

    Then $W^*_{\ell\ell} \to W^*$ so we choose $\ell_0 \in \mathbb{N}$ sufficiently large that for any $\ell > \ell_0$
    \begin{align*}
        & \lVert \Delta(W^{*\top}_{\ell\ell}f) \rVert_{L^2(\mu)} + \lVert W^{*\top}_{\ell\ell}f - h \rVert_{L^2(\mu')} + \lambda\lVert \Lambda W^{*\top}_{\ell\ell} \rVert \\
        < & \lVert \Delta(W^{*\top}f) \rVert_{L^2(\mu)} + \lVert W^{*\top}f - h \rVert_{L^2(\mu')} + \lambda\lVert \Lambda W^* \rVert + \frac{\epsilon}{3}.
    \end{align*}
    Now the proof proceeds directly 
    \begin{align*}
        &\lVert \Delta(W^{*\top}_{\ell\ell}f) \rVert_{L^2(\mu)} + \lVert W^{*\top}_{\ell\ell}f - h \rVert_{L^2(\mu')} \\
        \leq & \lVert \Delta(W^{*\top}_{\ell\ell}f) \rVert_{L^2(\mu)} + \lVert W^{*\top}_{\ell\ell}f - h \rVert_{L^2(\mu')} + \lambda\lVert \Lambda W^{*\top}_{\ell\ell} \rVert \\
        < & \lVert \Delta(W^{*\top}f) \rVert_{L^2(\mu)} + \lVert W^{*\top}f - h \rVert_{L^2(\mu')} + \lambda\lVert \Lambda W^* \rVert + \frac{\epsilon}{3} \\
        \leq & \lVert \Delta(\bar{W}^{\top}f) \rVert_{L^2(\mu)} + \lVert \bar{W}^{\top}f - h \rVert_{L^2(\mu')} + \lambda\lVert \Lambda \bar{W} \rVert + \frac{\epsilon}{3} \\
        < & \lVert \Delta(\bar{W}^{\top}f) \rVert_{L^2(\mu)} + \lVert \bar{W}^{\top}f - h \rVert_{L^2(\mu')} + \frac{\epsilon}{3} + \frac{\epsilon}{3} \\
        < & \frac{\epsilon}{3} + \frac{\epsilon}{3} + \frac{\epsilon}{3} = \epsilon.
    \end{align*}
\end{proof}

We have shown that with a reservoir computer, and least squares regression, we can obtain $W_{\ell\ell}$ such that $W_{\ell\ell}^{*\top}f$ closely approximates the boundary data, and closely approximates the derivative condition on the interior. But does this imply that  $W_{\ell\ell}^{*\top}f$ closely approximates the unique solution $\phi$ of the Dirichlet problem i.e. can we ensure that $\lVert W_{\ell\ell}^{*\top}f - \phi \rVert_{L^2(\mu)} < \epsilon$? We are not sure. This difficulty has likely emerged because we are working in $C^2(\bar{\Omega})$, and could probably be avoided by working in the Sobolev space $H^1$. Analysing the problem in $H^1$ could be a fruitful direction of future work.

Theorem \ref{theorem::Dirichlet_approx} is an offline learning result where the reservoir computer is trained using least squares regression, once all the data has been collection. We can prove a similar online result where the weights $W_k$ are updated as new data comes in.

\begin{theorem}
    Consider the Dirichlet problem defined in Theorem \ref{thm::Dirichlet} with boundary data $h \in C^2(\partial \Omega)$.
    Let $Z,Z'$ be a stationary ergodic processes with invariant measures $\mu$ and $\mu'$ defined on $\Omega \subset \mathbb{R}^d$ and $\partial \Omega$ respectively. Let $f \in L^1(\Omega,\mathbb{R}^N)(\mu)$ be the feedforward random neural network in Definition \ref{defn::feedforward}. Let $\Lambda \in \mathbb{M}_{N \times N}(\mathbb{R})$ be the invertible regularisation matrix.
    Then for $\alpha_k = 1/k$ and any initial $W_0 \in \mathbb{R}^N$ the algorithm
    \begin{align}
        W_{k+1} &= (I - \alpha_k \Lambda \Lambda^{\top})W_k - \alpha_k\bigg( (\Delta f)(Z_k)W_k^{\top}(\Delta f)(Z_k) + f(Z'_k)(W^{\top}_kf(Z'_k) - h(Z'_k)) \bigg)
        \label{algorithm::PDE}
    \end{align}
    converges to
    \begin{align*}
        W^* = \bigg( \int_\Omega (\Delta f)(\Delta f)^{\top} \ d\mu + \int_{\partial \Omega} ff^{\top} \ d\mu' + \Lambda\Lambda^{\top} \bigg)^{-1} \times \int_{\partial \Omega} fh \ d\mu'
    \end{align*}
    \label{theorem::Dirichlet_approx_online}
    almost surely.
\end{theorem}

\begin{proof}
    Using that $\alpha_k = 1/k$ and rearranging algorithm \eqref{algorithm::PDE} we have
    \begin{align*}
        W_{k+1} &= W_k - \frac{1}{k}\bigg(\bigg[ (\Delta f)(Z_k)(\Delta f)(Z_k)^{\top} + f(Z'_k)f(Z'_k)^{\top} + \Lambda\Lambda^{\top} \bigg] W_k - f(Z'_k)h(Z'_k) \bigg) \\
        &= W_k - \frac{1}{k}(B_k W_k - v_k) 
    \end{align*}
    where 
    \begin{align*}
        B_k = (\Delta f)(Z_k)(\Delta f)(Z_k)^{\top} + f(Z'_k)f(Z'_k)^{\top} + \Lambda\Lambda^{\top}, \qquad v_k = f(Z'_k)h(Z'_k).
    \end{align*}
    Now it follows from the ergodic theorem that
    \begin{align*}
        \bigg\lVert \lim_{\ell \to \infty} \frac{1}{\ell}\sum_{k = 0}^{\ell - 1} B_k - B \bigg\rVert_{\mathbb{R}^N} = 0, \qquad \bigg\lVert \lim_{\ell \to \infty} \frac{1}{\ell}\sum_{k = 0}^{\ell - 1} v_k - v \bigg\rVert_{\mathbb{R}^N} = 0
    \end{align*}
    for 
    \begin{align*}
        B = \int_\Omega (\Delta f)(\Delta f)^{\top} \ d\mu + \int_{\partial \Omega} ff^{\top} \ d\mu' + \Lambda\Lambda^{\top}, \qquad v = \int_{\partial \Omega} fh \ d\mu'
    \end{align*}
    almost surely. Then it follows from Theorem \ref{theorem::gyorfi} that algorithm \eqref{algorithm::PDE} converges to
    \begin{align*}
        W^* = B^{-1}v = \bigg(\int_\Omega (\Delta f)(\Delta f)^{\top} \ d\mu + \int_{\partial \Omega} ff^{\top} \ d\mu' + \Lambda\Lambda^{\top}\bigg)^{-1} \int_{\partial \Omega} fh \ d\mu'
    \end{align*}
    almost surely.
\end{proof}

\section{Laplace's equation on the disc}
\label{Numerics}

\subsection{Analytic solution}

We will demonstrate the theory presented in this chapter so far by solving a Dirichlet problem on the unit disc, first analytically, then with a reservoir computer. We let 
\begin{align*}
    \Omega = \{ (r,\theta) \ | \ 0 \leq  r < 1 \} \qquad \text{and} \qquad \partial \Omega = \{ (r,\theta) \ | \ r = 1 \}
\end{align*}
and seek $\phi \in C^2(\Omega,\mathbb{R})$ that satisfies 
\begin{itemize}
    \item $\Delta \phi = 0$ on $\Omega$
    \item $\phi = h$ on $\partial \Omega$.
\end{itemize}
The general solution is
\begin{align*}
    \phi(r,\theta) = \alpha_0 + \sum_{n=1}^{\infty}\alpha_n \cos(n\theta)r^n + \sum_{n=1}^{\infty}\beta_n \sin(n\theta)r^n
\end{align*}
with 
\begin{align*}
        \alpha_0 &= \frac{1}{2\pi}\int_0^{2 \pi} h(\theta) \ d\theta,  \\
    \alpha_n &= \frac{1}{\pi}\int^{2\pi}_0 h(\theta) \cos(n \theta) \ d\theta, \\
    \beta_n &= \frac{1}{\pi}\int^{2\pi}_0 h(\theta) \sin(n \theta) \ d\theta.
\end{align*}
We will set the boundary data $h(\theta) := \cos(4\theta)$. The orthonormality of the Fourier basis implies that $\alpha_4 = 1$ and all other coefficients $\alpha_n$ and $\beta_n$ vanish so 
\begin{align*}
    \phi(r,\theta) = r^4 \cos(4\theta).
\end{align*}
A plot of the solution is shown in Figure \ref{fig::analytic_soln}.

\begin{figure}
  \centering
    \includegraphics[width=0.90\textwidth]{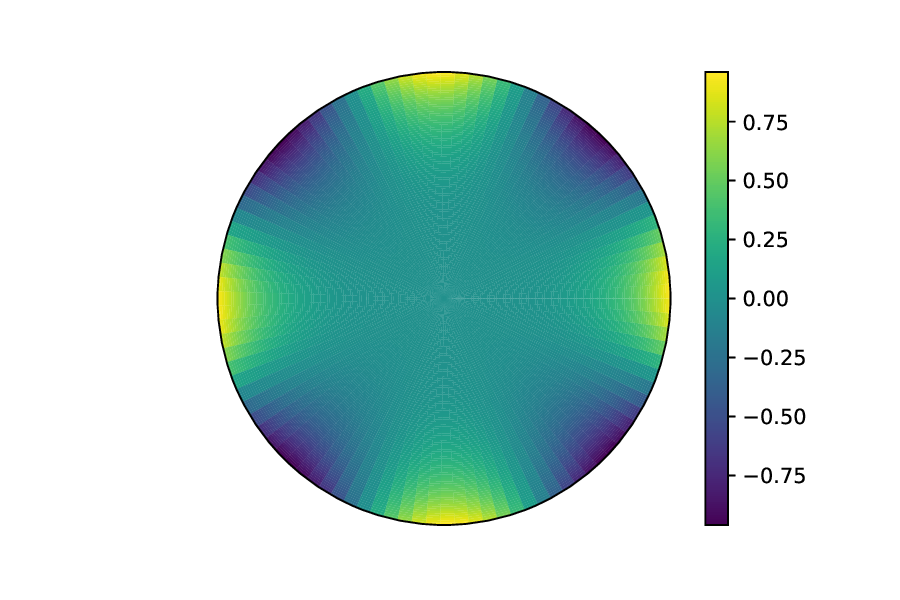}
    \caption{The analytic solution $\phi(r,\theta)=r^4 \cos 4\theta$ to the Dirichlet problem on the disc, taking the boundary data $h(\theta)=\cos 4\theta$.}
    \label{fig::analytic_soln}
\end{figure}

\subsection{Reservoir computer solution}

We define processes $Z$ and $Z'$ on the interior $\Omega$ and boundary $\partial \Omega$ of the unit disc. The terms in $Z$ and $Z'$ are i.i.d uniform distributions on the unit disc and unit circle. We take $\ell = 50$ samples on the disc, and $\ell' = 50$ samples on the circle respectively. We have maps $f, \ \Delta f : \Omega \to \mathbb{R}^N$, with $N = 50$, with $i$th components $f_i, \ (\Delta f)_i : \Omega \to \mathbb{R}$ defined by
\begin{align*}
    f_i(z) = \tanh(C_i^{\top}z + b_i), \qquad (\Delta f)_i(z) = -2\tanh(C_i^{\top} z + b_i) \, \text{sech}^2(C_i^{\top} z + b_i)
\end{align*}
respectively. The input matrix $C$ and bias vector $b$ are i.i.d uniform random variables $\sim U[-0.05,0.05]$. We set the Tikhonov regularisation parameter $\lambda$ to take the value $\lambda=10^{-6}$.

We seek $W^* \in \mathbb{R}^N$ that minimises
\begin{align*}
    \sum_{k=0}^{\ell-1} \lVert W^{\top}\Delta f(Z_k) \rVert + \sum_{k=0}^{\ell'-1}\lVert W^{\top}f(Z'_k) - h(Z'_k) \rVert + \lambda \lVert W \rVert.
\end{align*}
We can recast this as a matrix equation,
with $X^{\top}$ an $N \times (\ell+\ell')$ real matrix with $k$th column
\begin{align*}
    X^{\top}_k = 
    \begin{cases}
        \Delta f(Z_k) &\text{ if }  1 \leq k \leq \ell \\
        f(Z'_{k-\ell}) &\text{ if } \ell < k \leq \ell+\ell'
    \end{cases}
\end{align*}
and $Y^{\top}$ a real $(\ell + \ell')$-vector with $k$th component
\begin{align*}
    Y^{\top}_k = 
    \begin{cases}
        0 &\text{ if } 1 \leq k \leq \ell \\
        h(Z'_{k-\ell}) &\text{ if } \ell < k \leq \ell + \ell'
    \end{cases}
\end{align*}
so that
\begin{align*}
    W^* &= \argmin_{W \in \mathbb{R}^N}\bigg[\bigg\lVert 
    W^{\top}\bigg(
    \begin{bmatrix}
        \Delta f(Z_1)
    \end{bmatrix}
    \hdots
    \begin{bmatrix}
        \Delta f(Z_\ell)
    \end{bmatrix}
    \begin{bmatrix}
        f(Z'_1)
    \end{bmatrix}
    \hdots
    \begin{bmatrix}
        f(Z'_{\ell'})
    \end{bmatrix}
    \bigg) \\
    &\qquad - \bigg(
    \begin{bmatrix}
        0
    \end{bmatrix}
    \hdots
    \begin{bmatrix}
        0
    \end{bmatrix}
    \begin{bmatrix}
        h(Z'_1)
    \end{bmatrix}
    \hdots
        \begin{bmatrix}
        h(Z'_{\ell'})
    \end{bmatrix}
    \bigg)
    \bigg\rVert + \lambda \lVert W \rVert \bigg] \\
    &= \argmin_{W \in \mathbb{R}^N}\big[ \lVert W^{\top}X^{\top} - Y^{\top} \rVert + \lambda \rVert W \rVert \big].
\end{align*}    
Then we obtain $W^*$ by taking the singular value decomposition (SVD) of $X$
\begin{align*}
    X = U \Sigma V^{\top}.
\end{align*}
We denote the singular values by $\sigma_k$, and the columns of $U$ by $U_k$. Then the solution $W^*$ has the explicit form
\begin{align*}
    W^* = \sum_{k=1}^N \frac{\sigma_k U_k^{\top}Y}{\sigma_k^2 + \lambda}.
\end{align*}

%\jd{[There's a factor of $V$ missing on the RHS of the above, maybe as the first factor on the RHS.]}

The approximate solution to the Dirichlet problem $W^{*\top}f : \Omega \to \mathbb{R}$ is computed twice, each time under a different choice of parameters:
\begin{itemize}
    \item Number of interior points $\ell = 50$ or $\ell=500$
    \item Number of boundary points $\ell' = 50$ or $\ell' = 500$
    \item Number of neurons $N = 50$ or $N=500$
    \item Regularisation parameter $\lambda = 10^{-6}$ or $\lambda=0$
    \item For both runs the input matrix $C$ and bias vector $\zeta$ are i.i.d uniform random variables $\sim U[-0.05,0.05]$.
\end{itemize}
The outcomes are shown in Figures \ref{fig::res_comp_sol} and \ref{fig::res_comp_sol_2}. We can see that increasing the number of sample points $\ell, \ell'$ and neurons $N$ dramatically improves the quality of the solution.

\begin{figure}
    \centering
    \begin{subfigure}[b]{0.9\textwidth}
        \includegraphics[width=\textwidth]{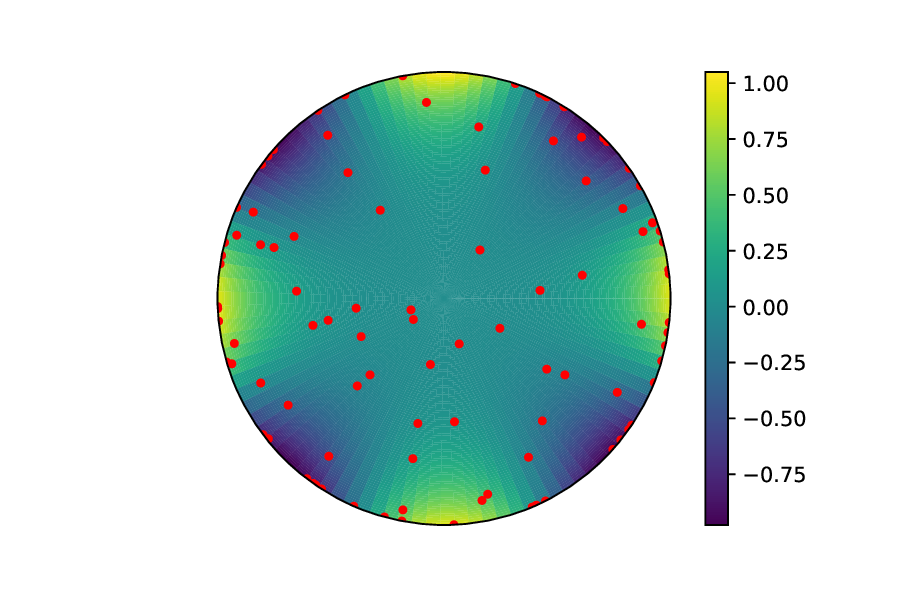}
        \caption{Approximate solution using a reservoir computer}
    \end{subfigure}
    \begin{subfigure}[b]{0.9\textwidth}
        \includegraphics[width=\textwidth]{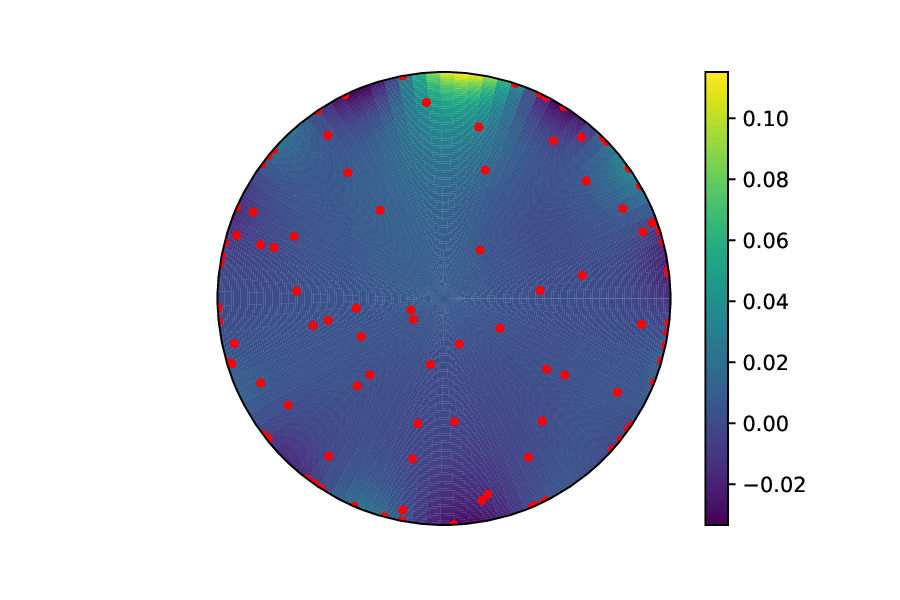}
        \caption{Difference between the analytic solution and approximate solution}
    \end{subfigure}
    \caption{Reservoir computer solution for $\ell = 50, \ \ell' = 50, \ N = 50, \ \lambda = 10^{-6}$. The red points are the sample points.}
    \label{fig::res_comp_sol}
\end{figure}

\begin{figure}
    \centering
    \begin{subfigure}[b]{0.9\textwidth}
        \includegraphics[width=\textwidth]{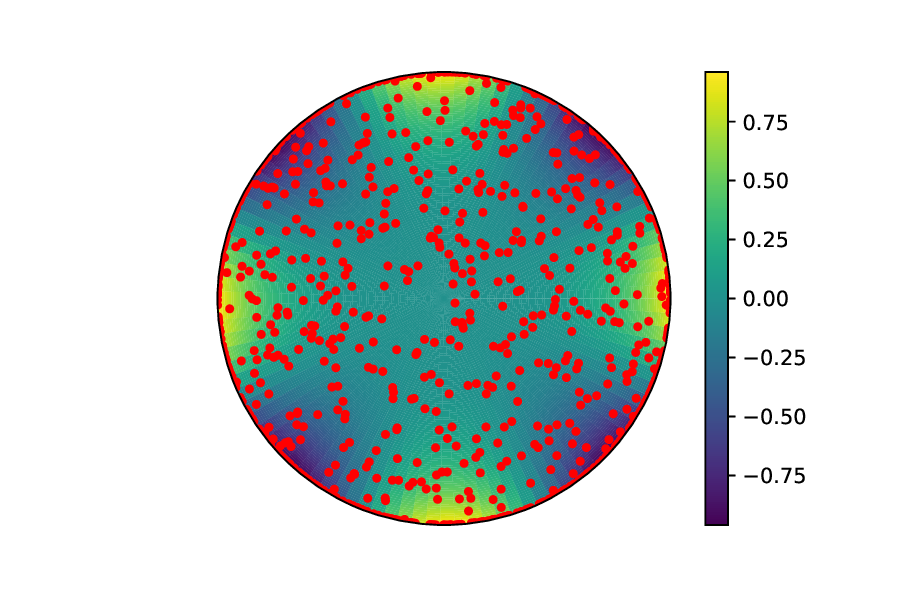}
        \caption{Approximate solution using a reservoir computer}
    \end{subfigure}
    \begin{subfigure}[b]{0.9\textwidth}
        \includegraphics[width=\textwidth]{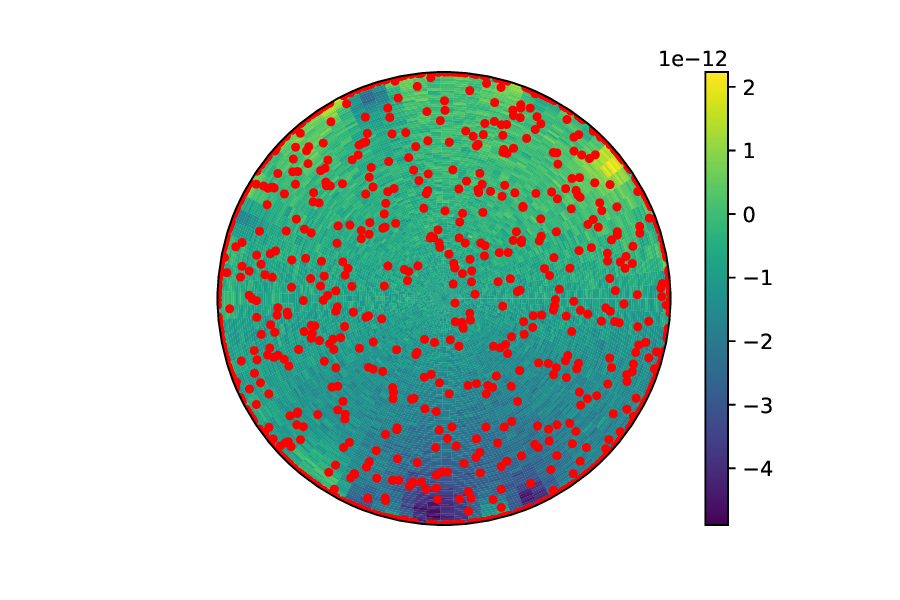}
        \caption{Difference between the analytic solution and approximate solution}
    \end{subfigure}
    \caption{Reservoir computer solution for $\ell = 500, \ \ell' = 500, \ N = 500, \ \lambda = 0$. The red points are the sample points.}
    \label{fig::res_comp_sol_2}
\end{figure}

\section{Future Work and Open Questions}
\label{Conclusion}

In this chapter we showed that we can approximate the solution to the Dirichlet problem $\phi$ using a feedforward reservoir computer $W^{\top}f$, which can be trained by offline or online methods. Our analysis exploits the linearity of the Laplacian, but does not exploit much else. This suggests that the results could be extended to the more general boundary value problem
\begin{itemize}
    \item $L(\phi) = s$ on $\Omega$
    \item $\phi = h$ on $\partial \Omega$
\end{itemize}
for continuous functions $s \in C^0(\Omega,\mathbb{R}), h \in C^0(\partial \Omega,\mathbb{R})$ where $L$ is an arbitrary linear differential operator. The results could also be extended to time dependent linear PDEs like the heat or wave equation, which need not even be posed on a bounded domain. For any such extension, the well posedness of solutions would have to be established. Moreover, the operator $L$ could be an integral operator such as convolution, or the Radon transform, that appears in problems like de-blurring or medical image reconstruction. 

It would be extremely powerful to generalise these linear results to the case of nonlinear PDEs, but this is much more complicated. We can rule out the offline method immediately because it relies explicitly on the linearity of the derivative operator. On the other hand the online method described by algorithm \eqref{algorithm::PDE} could be modified to yield the iteration
    \begin{align*}
        W_{k+1} &= (I - \alpha_k \Lambda \Lambda^{\top})W_k \\ 
        &\qquad- \alpha_k\bigg( L[f](Z_k)(L [W_k^{\top}f](Z_k) - s(Z_k)) + f(Z'_k)(W^{\top}_kf(Z'_k) - h(Z'_k)) \bigg)
    \end{align*}
for $L$ a nonlinear operator. This is highly reminiscent of the major algorithm studied in \cite{Borkar2009}, which is essentially stochastic gradient descent. Using stochastic gradient descent to train a neural network to solve (nonlinear) PDEs is not new, and is discussed in \cite{Han8505} and \cite{arXiv:1910.01547}. These authors train both the external weights $W$ and internal weights $C,b$ of a (deep) neural network, in contrast to the reservoir computing paradigm where only $W$ is trained. 

Furthermore, \cite{doi:10.1137/19M125649X}, \cite{arXiv:1809.02362}, \cite{arXiv:1809.07321} have provided an excellent reason to use deep neural networks to solve PDEs, namely that they can overcome the curse of dimensionality. The complexity of classic grid based methods like finite element and finite differences grows exponentially with the dimension of the PDE, while neural network based methods have polynomial complexity. These results are likely to hold in the case of reservoir computing; especially given the recent results by \cite{Gonon2021}. The results by Gonon show that a feedforward reservoir computer (called a random neural network by the author) trained on observations of a Black-Scholes type PDE overcomes the curse of dimensionality. The error on the approximation is shown to converge with order $1/\sqrt{N}$, and the bounding constant is computed explicitly. This suggests that random feedforward neural networks are well suited to solving very high dimensional linear PDEs, and the problem is amenable to mathematical analysis.

Another idea for future work is perhaps to establish (more easily in the linear case) uniform bounds on the number of sample points $\ell$ and neurons $N$ required for the solution of a PDE to be approximated to a given tolerance. This could be achieved using arguments inspired by \cite{Gonon2020} or \cite{Gonon2021}. Such bounds could be compared to empirical plots for the quality of the solution as the number of sample points $\ell,\ell'$ and neurons $N$ grows. 

The results that appear in this chapter could perhaps be generalised to domains $\Omega$ that are not so well behaved. The infamous $L$-shaped domain (imagine a square with a smaller square cut out of the top right hand corner) does not have the exterior ball property on the point at the inside crease of the $L$-shaped domain. The solutions to some boundary value problems are singular at this point, and cause numerical problems. It may be fruitful to explore how reservoir computing methods fare at the approximating the solutions to PDEs on domains such as these.

\chapter{Conclusions and Future Work}

In this final chapter we will review all previous chapters, and discuss directions for future work as we go.
We started in Chapter 1 with a brief introduction to neural networks, machine learning and how these are connected to reservoir computing. We then described an experiment undertaken at the start of the PhD which yielded results that motivated much of the work that appears in this thesis. This experiment was to take a trajectory of scalar observations of the Lorenz system, feed these observations into an ESN, and attempt to learn the future dynamics of the Lorenz system using least squares regression. We observed that the autonomous ESN was successful at forecasting the future, and furthermore that the autonomous ESN appeared to have learned topological and geometrical invariants of the Lorenz system.

To verify this, we spent Chapter 2 describing computational methods to obtain the linearisations of fixed points, Lyaponov exponents, and homology groups. This is all part of a larger field called computational topology which is slightly peripheral to the main focus of this PhD, but which could form an interesting area for future work. For example if a practitioner is interested in the geometrical or topological properties of a dynamical system, but is deprived of the equations, and limited only to trajectories of scalar observations, then their primary goal is obtain a good quality embedding, perhaps with reservoir computing, and then apply the tools of computational topology.

In Chapter 3, we proved that a state contracting reservoir map trained on observations of a deterministic system adopts dynamics that are synchronised to the drive system via a continuously differentiable generalised synchronisation (GS) $f : M \to \mathbb{R}^N$. We assumed that the drive system evolves on a compact manifold $M$, but we can likely relax this assumption by imposing suitable bounds on the observation function $\omega : M \to \mathbb{R}$ and evolution operator $\phi : M \to M$. It may be interesting in future to contemplate how the smoothness of the GS influences the quality of learning phase. The results by \cite{Mhaskar96neuralnetworks}, \cite{Poggio2017} establish that the number of neurons required to approximate a given target map decreases as the smoothness of the target map increases. These results strongly suggest that for many practical purposes a smoother GS is a better GS.

In Chapter 4, we showed that in the special case of a linear reservoir map 
\begin{align*}
    F(x,z) = Ax + Cz
\end{align*}
the associated GS $f : M \to \mathbb{R}^N$ is an embedding for generic observation functions $\omega : M \to \mathbb{R}$ and almost all matrices $A,C$. This result admits the \cite{TakensThm} embedding theorem as a special case when $A$ is the lower shift matrix and $C = (1,0,\ldots,0)$. Numerical experiments suggest that for a much broader class of reservoir maps, including ESNs with nonlinear activation functions
\begin{align*}
    F(x,z) = \sigma(Ax + Cz + b)
\end{align*}
the associated state synchronisation map $f : M \to \mathbb{R}^N$ is an embedding. Generalising this result to hold for nonlinear reservoir systems such as the ESN appears difficult because the proof in the linear case relies extensively on the linearity. In order to exploit the linearity as much as possible, we could perhaps linearise a nonlinear reservoir map at the fixed points of $\phi$, and prove some partial results in that regime.

The synchronisation and embedding results in Chapters 3 and 4 are set in discrete time, and it is likely possible to develop analogous results in continuous time. To see this, suppose that $\{ \phi^t \in \text{Diff}^1(M) \ | \ t \in \mathbb{R} \}$ forms a group under composition such that $\phi^{t_1 + t_2} = \phi^{t_1} \circ \phi^{t_2}$. Then we can define for each $m \in M$ the continuous time reservoir ODE
\begin{align}
    \dot{x}(t) = F(x(t),\omega\circ\phi^{-t}(m)).
    \label{linear_continuous_time}
\end{align}
In the special case that $F(x,z) = Ax + Cz$ we have
\begin{align*}
    \dot{x}(t) = Ax(t) + C\omega\circ\phi^{-t}(m)
\end{align*}
which admits for any initial $x_0 \in \mathbb{R}^N$ the unique solution
\begin{align}
    x(t) = \int_{0}^{\infty}e^{-A\tau}C\omega\circ\phi^{-\tau+t}(m) \ d\tau + e^{-At} x_0. \label{solution}
\end{align}
Now, for any point on the manifold $m \in M$ and initial reservoir state $x_0 \in \mathbb{R}^N$ the solution $x(t)$ converges to $f\circ\phi^t(m)$ for $f : M \to \mathbb{R}^N$ defined by
\begin{align}
    f(m) = \int_{0}^{\infty}e^{-A\tau}C\omega\circ\phi^{-\tau}(m) \ d\tau. \label{GS}
\end{align}
The map $f$ is therefore a continuous time GS highly analogous to the discrete time GSs discussed in this thesis. We proved in Chapter 4 that the discrete analogue of the continuous time GS in \eqref{GS} is an embedding, so it may be possible to prove that $f$ in \eqref{GS} is an embedding using similar techniques. Furthermore, the continuous time GS defined in \eqref{GS} necessarily satisfies
\begin{align}
    \frac{d}{dt}(f\circ\phi^t) = F(f\circ\phi^t,\omega\circ\phi^t).
    \label{state_ODE}
\end{align} 
This observation suggests, though we have not proved anything, that for a large class of generally nonlinear, and suitably contracting, state systems $F \in C^0(\mathbb{R}^N \times \mathbb{R},\mathbb{R}^N)$ there is a unique continuous time GS $f : M \to \mathbb{R}^N$ that satisfies \eqref{state_ODE}. Furthermore, for any $m \in M$ and $x(t)$ that solves \eqref{linear_continuous_time} it follows $x(t) \to f\circ\phi^t(m)$. Bridging the gap between discrete and continuous time could be a intriguing direction of future work.

In Chapter 5, we discussed the universal approximation capabilities of the special reservoir map called an ESN. Several results in the literature including those by \cite{GRIGORYEVA2018495}, \cite{Gonon2020}, and \cite{GONON202110} have shown that ESNs are universal approximators, in the sense that for a given target function, there exist a set of optimal weights which will approximate the target function as closely as required. Though the existence of the weights is established, the weights themselves are not explicitly constructed. We show that if the observations are drawn from an ergodic dynamical system, and the output weights $W \in \mathbb{R}^N$ are obtained using regularised least squares regression, then a sufficiently good $L^2(\mu)$ approximation will be obtained. This is analogous to the stochastic setting, where we prove in Chapter 6 that if the observations are drawn from a stationary ergodic stochastic process, and the output weights $W \in \mathbb{R}^N$ are obtained using regularised least squares regression, then a sufficiently good $L^2(\mu)$ approximation will also be obtained. 

In both the deterministic and stochastic case, we analysed the convergence of the approximation error as the number of data points $\ell$ grows. The convergence is likely of the order $1 / \sqrt{\ell}$ as a result of the central limit theorem for ergodic systems. The constant factor in the convergence estimate is related to the mixing time of the dynamical system $(M,\phi)$ or process $\boldsymbol{Z}$ respectively in the deterministic or stochastic setting. It could be a fruitful direction of future work to explore this in more detail, and establish bounds in certain cases. 

In both the deterministic and stochastic case, we discussed an online learning algorithm for finding the least squares solution $W$. The online learning algorithms do not require every reservoir state be stored in memory, and they are therefore less memory intensive then offline analogues like the SVD. Furthermore, the online methods are much more plausible descriptions of learning in biological nervous systems. The online learning algorithm is essentially stochastic gradient descent, and can therefore be applied to non-convex optimisation problems that linear regression procedure cannot solve at all. Though it is much harder to guarantee convergence to an optimal solution in the non-convex case, the online algorithms are promising methods for solving nonlinear PDEs (discussed in Chapter 7) or finding the \emph{optimal} value functional for control problems (discussed in Chapter 6)

An interesting field related to this thesis is \emph{transfer learning}, which was discussed in the context of reservoir computing recently by \cite{PhysRevE.102.043301}. In the transfer learning paradigm, one has limited real data, which is supplemented by data produced by a toy model. Another related idea is the \emph{physics informed neural network}, where the physical properties of the system are integrated with the machine learning paradigm. For example \cite{arXiv:1906.01563} demonstrate that a neural network trained to learn a system's Hamiltonian will adopt solutions that closely conserve energy. This approach is in contrast with most of the results in this thesis where there is almost no modelling of the drive system. Developing ideas like transfer learning or physics informed neural networks within the reservoir computing framework seems like a fruitful and interesting strand of future work. These connections between reservoir computing, physics, and mechanical systems, enable greater applications of reservoir computing to fields like robotics and engineering.

\bibliographystyle{agsm}
\bibliography{references}

\end{document}